\documentclass[a4paper]{amsbook}
\usepackage{amssymb, enumitem}
\usepackage[chatter]{rotating}
\usepackage{epsfig}
\usepackage{fancyvrb}

\usepackage{nomencl}
\usepackage{multicol}
\renewcommand{\nompreamble}{\begin{multicols}{2}}
\renewcommand{\nompostamble}{\end{multicols}}
\makenomenclature

\PassOptionsToPackage{hyphens}{url}
\usepackage[hidelinks]{hyperref}

\usepackage{bbm}

\usepackage{makecell}
\setcellgapes{3pt}

\usepackage{multirow}
\usepackage{array} 
\newcolumntype{A}{ >{} r <{} @{} >{{}} l <{} } 

\newtheorem{lemma}{Lemma}[chapter]
\newtheorem*{mth}{Main Theorem}
\newtheorem*{theo}{Classification Theorem of the Symmetric Algebras}
\newtheorem{teo}{Theorem}
\newtheorem{mcor}{Corollary}

\newtheorem{theorem}[lemma]{Theorem}
\newtheorem{remark}[lemma]{Remark}
\newtheorem{cor}[lemma]{Corollary}
\newtheorem{proposition}[lemma]{Proposition}

\newtheorem{hyp}[lemma]{Hypothesis}

\numberwithin{table}{chapter}
\numberwithin{equation}{chapter}

\newcommand{\N}{{\mathbb N}}

\newcommand{\ba}{{\bf a}}
\newcommand{\bs}{{\bf s}}
\newcommand{\bv}{{\bf v}}

\DeclareMathOperator{\ch}{char}

\newcommand{\Vo}{{\bf V}}
\newcommand{\ad}{{\rm ad}}
\DeclareMathOperator{\supp}{supp}

\newcommand{\F}{{\mathbb F}}
\newcommand{\Z}{{\mathbb Z}}

\newcommand{\lm}{\lambda}

\newcommand{\lmu}{\lambda_1}

\newcommand{\lmf}{\lambda_1^f}

\newcommand{\lmd}{\lambda_2}

\newcommand{\lmdf}{\lambda_2^f}

\newcommand{\al}{\alpha}
\newcommand{\bt}{\beta}
\newcommand{\gm}{\gamma}
\newcommand{\dl}{\delta}
\newcommand{\M}{\mathcal M}

\newcommand{\Mab}{\mathcal M(\al,\bt)}

\newcommand{\bu}{{\bf u}}

\newcommand{\ti}{\bar \imath}
\newcommand{\tj}{\bar \jmath}
\newcommand{\tr}{\bar r}
\newcommand{\tT}{\bar t}
\newcommand{\tz}{\bar 0}
\newcommand{\tu}{\bar 1}
\newcommand{\td}{\bar 2}
\newcommand{\tih}{\bar h}
\newcommand{\tik}{\bar k}

\newcommand{\ttr}{\tilde r}
\newcommand{\ttT}{\tilde t}
\newcommand{\ttu}{\tilde 1}
\newcommand{\ttd}{\tilde 2}
\newcommand{\ttz}{\tilde 0}
\newcommand{\ttx}{\tilde x}
\newcommand{\tta}{\tilde a}
\newcommand{\ttb}{\tilde b}

\newcommand{\V}{\mathcal{V}}
\newcommand{\W}{\mathcal{W}}

\newcommand{\cupdot}{\mathbin{\mathaccent\cdot\cup}}

\newcommand{\A}{\mathrm{A}}
\newcommand{\C}{\mathrm{C}}
\newcommand{\B}{\mathrm{B}}
\newcommand{\Y}{\mathrm{Y}}
\newcommand{\J}{\mathrm{J}}
\newcommand{\IY}{\mathrm{IY}}
\newcommand{\QQ}{\mathrm{Q}}
\newcommand{\cH}{\mathcal{H}}
\newcommand{\hatH}{\hat{\cH}}

\newcommand\la{\langle}
\newcommand\ra{\rangle}
\newcommand\lla{\langle\!\langle}
\newcommand\rra{\rangle\!\rangle}

\newcommand{\1}{\mathbbm{1}}

\DeclareMathOperator{\adim}{Adim}
\DeclareMathOperator{\Ann}{Ann}
\DeclareMathOperator{\Miy}{Miy}
\DeclareMathOperator{\Aut}{Aut}

\renewcommand{\phi}{\varphi}
\renewcommand{\epsilon}{\varepsilon}

\newcommand{\ie}{i.e.\ }

\SetEnumitemKey{enum_note}{label={\upshape \arabic*.}, ref=\arabic*} 
\SetEnumitemKey{enum_arabic}{label={\upshape (\arabic*)}}
\SetEnumitemKey{enum_alph}{label={\upshape (\alph*)}}
\SetEnumitemKey{enum_roman}{label={\upshape (\roman*)}}

\setlist[enumerate, 1]{enum_alph}
\setlist[enumerate, 2]{enum_roman}

\newsavebox{\savepar}
\usepackage{xcolor}

\title[$2$-generated Axial Algebras]{{\sc The Classification of the $2$-generated Primitive  Axial Algebras of Monster Type}}
\author{Clara Franchi}
\address{Dipartimento di Matematica e Fisica,
Universit\`a Cattolica del Sacro Cuore,
Via della Garzetta 48
I-25133 Brescia, Italy}
\email{clara.franchi@unicatt.it}
\author{Mario Mainardis}
\address{Dipartimento di Scienze Matematiche, Informatiche e Fisiche, 
Universit\`a degli Studi di Udine, via delle Scienze 206,
I-33100 Udine, Italy}
\email{mario.mainardis@uniud.it}
\author{Justin M\textsuperscript{c}Inroy}
\address{The Alan Turing Institute, British Library, 96 Euston Road, London, NW1 2DB, UK, and School of Mathematics, University of Bristol, Fry Building, Woodland Road, Bristol, BS8 1UG, UK}
\email{justin.mcinroy@bristol.ac.uk}
\author{Michael Turner}
\address{Independent Researcher, Chichester, UK}
\curraddr{}
\email{}
\date{\today}

\author{}
\address{2020 Subject Classification: 17A60, 17A99, 17C27, 20F29}

\author{}
\address{Abstract. {\rm Axial algebras of Monster type $(\al,\bt)$ arise naturally as a generalisation of the weight $2$ components of $OZ$-type vertex operator algebras by axiomatising some of their relevant properties. 
Indeed, the motivating examples of axial algebras of Monster type are the Conway-Norton-Griess algebra, which is the weight $2$ component of the Moonshine VOA and whose automorphims group is the Monster simple group, and its Norton-Sakuma subalgebras.  This class also includes Matsuo algebras, related to $3$-transposition groups, and Jordan algebras. Moreover, axial behaviour has also been detected elsewhere, in Chayet-Garibaldi algebras which are related to  algebraic groups and whose fusion law is remarkably close to the Monster type fusion law, and in Hsiang algebras, arising from PDEs.
A central issue has been the classification of the $2$-generated primitive axial algebras of Monster type.  This problem was originally attacked by Rehren around 2015. A huge milestone was accomplished by Yabe in 2023 leading, with additional cases completed by Franchi, Mainardis, and M\textsuperscript{c}Inroy, in 2022 and 2024, to the classification in the symmetric case.
In this paper we deal with the non-symmetric case and complete the classification, also giving a full description of all these algebras.}}

\usepackage{float}
\makeindex
\begin{document}

\maketitle
\tableofcontents



\chapter*{Introduction}

A long standing issue in finite simple groups is to find a unified theory whereby the sporadic groups could be treated within the same setup as other (possibly all) classes of finite simple groups, hopefully unveiling new connections between different classes of finite simple groups. 

This was one of the key motivations for Jon Hall, Felix Rehren, and Sergey Shpectorov to introduce in~\cite{HRS} the class of primitive axial algebras of Monster type $(\al, \bt)$ (in short $\M(\al,\bt)$-axial algebras, see definition in Section \ref{basicdef}).  The motivating example is the Conway-Norton-Griess algebra, which is an $\M(\tfrac{1}{4},\tfrac{1}{32})$-axial algebra.
Their work generalised Alexander A. Ivanov's theory of Majorana algebras (see~\cite{IvBook}), which most impressed John H. Conway (see~\cite{IvInd}).  Majorana theory has produced the explicit constructions of several important subalgebras of the Conway-Norton-Griess algebra, in particular those associated to subgroups of the Monster group generated by Fischer involutions, (see~\cite{IPSS, Iv11b, A67, IS12, L32, Dec, Alonso, FIM1, FIM2, A12, Why, BI, Gev}), bringing new insight into this algebra and into the Monster group.  

The class of $\M(\al,\bt)$-axial algebras is a rich class which realises several finite simple groups as their automorphism groups.
In particular, Jordan algebras, related to the classical groups and $G_2$; Matsuo algebras, related to the $3$-transposition groups; and the Conway-Norton-Griess algebra~\cite{Griess}, the originating example, and its subalgebras, related to the Monster sporadic simple group and the sporadic groups in the Happy Family. An axial behaviour has also been detected by Tkachev in Hsiang algebras, related to PDEs~\cite{Tka}, and in Chayet-Garibaldi algebras (see~\cite{CG, Des, ShGo}), arising from algebraic groups, whose fusion law is remarkably close to the Monster type fusion law (see~\cite[\S 6.6]{ShGo}).

Similarly to other classes of algebras, Lie algebras being a notable example, a key step towards the understanding and the possible applications of the theory of axial algebras of Monster type is the classification of the $2$-generated objects. 

This story for axial algebras can be traced back to an important result by Simon Norton,
who classified in~\cite{Nor} the $2$-generated subalgebras of the Conway-Norton-Griess algebra (which is an axial algebra of Monster type $(\frac{1}{4}, \frac{1}{32})$  \cite{IvBook}). 
The Conway-Norton-Griess algebra can be viewed as the weight $2$ component $V^\sharp_{(2)}$ of the Moonshine Module $V^\sharp$, which is a Vertex Operator Algebra (VOA) that arises as a  representation of the Virasoro algebra (Virasoro VOAs). The class of VOAs was  axiomatized by Richard Borcherds
in~\cite{Bor} in order to give his celebrated  solution of the McKay-Thompson Moonshine conjecture. Using the fact that Virasoro VOAs are $\Z_2$-graded, in 1996 Masahiko Miyamoto introduced Ising vectors for an $OZ$-type VOA $V$ and proved that every Ising vector $u$ of $V$ is associated with an involutory automorphism $\tau_u$ of $V$, now called a {\it Miyamoto automorphism} \cite{Miya}. In the case of the Moonshine module (which is a $OZ$-type VOA), the Ising vectors are precisely the Conway axes and the Miyamoto automorphisms are the Conway involutions of the Monster (\ie those whose centralisers are isomorphic to the double cover of the Baby Monster). The properties of Ising vectors and  Miyamoto automorphisms turned out to be central in developing the theory of axial algebras.  Indeed, based on Miyamoto's work, one of his students, Shinya Sakuma, classified the isomorphism types of the $OZ$-type VOA generated by two Ising vectors \cite{Sakuma}, essentially obtaining the same classification as Norton's.  For this reason these algebras are now called {\it Norton-Sakuma algebras}\index{Norton-Sakuma algebras!in characteristic zero}. They  fall into nine isomorphism classes which are labelled  accordingly to ATLAS notation for the conjugacy class in the Monster of the product of the Miyamoto involutions associated to the generating Ising vectors. We display them here as the nodes of McKay's extended $E_8$ graph:

\begin{picture}(80,60)(-60,-38)
\put(-5,5){$1\A$}
\put(0,0){\circle{4}}
\put(2,0){\line(1,0){26}} 
\put(25,5){$2\A$}
\put(30,0){\circle{4}}
\put(32,0){\line(1,0){26}} 
\put(55,5){$3\A$}
\put(60,0){\circle{4}}
\put(62,0){\line(1,0){26}} 
\put(85,5){$4\A$}
\put(90,0){\circle{4}}
\put(92,0){\line(1,0){26}} 
\put(115,5){$5\A$}
\put(120,0){\circle{4}}
\put(122,0){\line(1,0){26}}
\put(145,5){$6\A$}
\put(150,0){\circle{4}}
\put(152,0){\line(1,0){26}}
\put(175,5){$4\B$}
\put(180,0){\circle{4}}
\put(182,0){\line(1,0){26}}
\put(205,5){$2\B$}
\put(210,0){\circle{4}}
\put(150,-2){\line(0,-1){26}}
\put(150,-30){\circle{4}}
\put(155,-32){$3\C$}
\end{picture}

\noindent For a detailed description of these algebras see~\cite[Table~3]{IPSS}.

A careful analysis of  Sakuma's proof, led Ivanov to axiomatize the class of Majorana algebras, which are particular $\M(\tfrac{1}{4}, \tfrac{1}{32})$-axial algebras over the real numbers \cite{IvBook}.  Within this surprisingly simple axiomatics, Ivanov reproved the Norton-Sakuma classification \cite{IPSS}. Ivanov's result, later extended in~\cite{FMS3} to every $\mathcal{M}(\tfrac{1}{4}, \tfrac{1}{32})$-axial algebra over every field of characteristic $0$, was the starting point for the theory of Majorana algebras and Majorana representations.

Evidently, the analogue of the Norton-Sakuma classification for all $2$-generated $\mathcal{M}(\al, \bt)$-axial algebras, obtained in this paper, is also a necessary step to extend the machinery of Majorana Theory to the wider class of 
$\mathcal{M}(\al, \bt)$-axial algebras and their automorphism groups. 


 The investigation  of the  
  $2$-generated $\M(\al,\bt)$-axial algebras over a field $\F$, of characteristic other than $2$,  for generic distinct  parameters $\al$ and $\bt$ in $\F\setminus\{0,1\}$, was initiated by Rehren in~\cite{FelixPaper}.  (Note that, by an argument of Mathias Stout~\cite{Stout}, see also~\cite[Lemma 2.1]{HW}, in characteristic $2$ all axial algebras of Monster type are associative, whence the axial condition brings no relevant information.) In~\cite{FelixPaper}, Rehren showed implicitly that we can split into three disjoint cases that behave somewhat differently ({\it Rehren's Tricotomy}\index{Rehren's Tricothmy}), namely:
\begin{enumerate}[enum_arabic]
    \item $\al=2\bt$,
    \item $\al=4\bt$,
    \item all other cases (the {\it regular}\index{Rehren's regular case} case),
\end{enumerate}  
and showed that, in the regular case, these algebras have dimension at most $8$.

A major breakthrough was due to Takahiro Yabe, who in~\cite{Yabe} gave an almost complete classification of the $2$-generated $\Mab$-axial algebras $(V,\{a_0, a_1\})$, under the assumption that there exists an involutory automorphism (called a {\it flip}) of the algebra $V$ that swaps the two generating axes $a_0$ and $a_1$. In this case the $\Mab$-axial algebra $(V, \{a_0, a_1\})$ is called  {\it symmetric}. The remaining subcases of the symmetric case left open by Yabe were eventually completed by Franchi and Mainardis \cite{FM}, and Franchi, Mainardis and M\textsuperscript{c}Inroy \cite{HWQ}.

\begin{theo}[\cite{Yabe, FM, HWQ}]\label{symmetric}
Let $\F$ be a field of characteristic other than $2$,
  $\al, \bt\in \F\setminus \{0,1\}$ with $\al \neq \bt$,   and $(V, \{a_0, a_1\})$ be a \textup{(}primitive\textup{)} symmetric $2$-generated $\Mab$-axial algebra. Then $(V, \{a_0, a_1\})$ is isomorphic to a quotient of one of the following:
  \begin{enumerate}
      \item an axial algebra of Jordan type $\al$, or $\bt$; 
      \item an axial algebra in one of the following families: \label{symmetric_2}
      \begin{enumerate}
          \item $3\A (\al,\bt )$, $4\A(\tfrac{1}{4},\bt)$, $4\B(\al, \tfrac{\al^2}{2} )$, $4\J(2\bt ,\bt)$, $4\Y(\tfrac{1}{2},\bt)$, $4\Y(\al, \tfrac{1-\al^2}{2} )$, \\
          $5\A(\al, \tfrac{5\al-1}{8})$, $6\A(\al,-\tfrac{\al^2}{4(2\al-1)} )$, $6\J(2\bt,\bt)$, and $6\Y(\tfrac{1}{2},2)$;
     \item $\IY_3(\al, \tfrac{1}{2} ; \mu)$ and $\IY_5(\al, \tfrac{1}{2} )$;
      \end{enumerate}
      \item the Highwater algebra $\mathcal H$, or its characteristic $5$ cover $\hatH$. 
      \end{enumerate} 
\end{theo}

In this paper, we complete the classification of the $2$-generated $\mathcal{M}(\al,\bt)$-axial algebras by considering the non-symmetric algebras and prove the following result which gives a complete answer to~\cite[\S 10, Question~(c)]{TA} and \cite[Problem~5.15]{survey}.

\begin{mth}
Let $\F$ be a field of characteristic other than $2$, $\al, \bt\in \F\setminus \{0,1\}$ with $\al \neq \bt$, and $(V, \{a_0, a_1\})$ be a \textup{(}primitive\textup{)} $2$-generated $\Mab$-axial algebra over $\F$. Then $(V, \{a_0, a_1\})$ is either symmetric or isomorphic to one of the following:
\begin{enumerate}
    \item $4\QQ(2\bt,\bt)$;
    \item $4\QQ(-1,-\tfrac{1}{2})^\times$;
    \item $3\C^\prime(\al,1-\al)$, for $\al \neq \tfrac{1}{2}$;
    \item $3\QQ^\prime(\tfrac{1}{3},\tfrac{2}{3})$;
    \item $4\B(-1,\tfrac{1}{2}; \nu)^\times$. 
\end{enumerate}
\end{mth}
The above algebras, including the symmetric ones, will be defined and discussed in Chapter~\ref{known}. For the symmetric ones, we use the notation in \cite{survey}.  Note also that we use a prime in our notation of $3\QQ^\prime(\tfrac{1}{3},\tfrac{2}{3})$ and $3\C^\prime(\al,1-\al)$ since these algebras have a skew axet. Elsewhere these are called $\QQ_2(\tfrac{1}{3},\tfrac{2}{3})$ and $3\C(\al,1-\al)$, respectively.  All the different notation in the literature can be found in the tables in Section~\ref{thetables}.

\medskip

For more information and motivation on axial algebras and Majorana algebras, we refer to the survey papers~\cite{survey, IvInd, surveyI, Novosibirsk}. In particular, \cite{survey} is a good place to find out who first constructed the algebras mentioned in this paper.

We now state some consequences of the Main Theorem.

\begin{mcor}\label{mcor:FF}
All \textup{(}primitive\textup{)} $2$-generated axial algebras of Monster type admit a Frobenius form. 
\end{mcor}
This supports the conjecture that every axial algebra of Monster type admits a Frobenius form (see~\cite[Conjecture 6.1]{survey}). This has been proven for every axial algebra of Jordan type in~\cite[Theorem~4.1]{HSS}, where in addition, it is shown that the Frobenius form may be chosen in such a way that all axes have length $1$. The same holds for symmetric $2$-generated axial algebras of Monster type. On the contrary, in the non-symmetric algebras axes may necessarily have different lengths, even $0$ in a couple of cases (see Table~\ref{table3Cskew}). 


The algebras in the Main Theorem can fail to be symmetric in three different ways.  Firstly, the axet may be skew (the two orbits of axes have different sizes): both $3\C'(\al,1-\al)$ and $3\QQ^\prime(\tfrac{1}{3},\tfrac{2}{3})$ have skew axet $X'(3)$.  The remaining algebras have a regular axet (either one orbit of axes, or two orbits of the same size).  Secondly, the projections of one generating axis onto the other can be different (and so the axes have different lengths with respect to the Frobenius form).  This occurs in $4\QQ(2\bt, \bt)$ and $4\QQ(-1, -\frac{1}{2})^\times$.  Finally, even if the axes have the same projection onto each other (and so length), we may fail to have an algebra automorphism, as in $4\B(-1, \tfrac{1}{2}; \nu)^\times$.


Looking at the axets, we have the following:

\begin{mcor}\label{mcor:bt<>1/2}
If $\bt \neq \frac{1}{2}$, then the axet \textup{(}set of axes\textup{)} of a \textup{(}primitive\textup{)} $2$-generated axial algebra of Monster type has size $1$, $2$, $3$, $4$, $5$, or $6$.
\end{mcor}

Thus, for $\bt\neq \tfrac{1}{2}$, we get an analogous result to the $6$-transposition property found in~\cite{Sakuma}.

\begin{mcor}\label{cor6trans}
Let $\bt\neq \tfrac{1}{2}$ and $\V$ be a \textup{(}primitive\textup{)} $\Mab$-axial algebra. Then $\Miy(\V)$ is a $6$-transposition group. 
\end{mcor}

Note that if $\al\neq \tfrac{1}{2}\neq \bt$ and the algebra is finite dimensional, then by \cite[Corollary 3.4]{Auto}, its Miyamoto group is finite.

By~\cite[Theorem 1.5]{FMS3}, every (primitive) $2$-generated $\M(\tfrac{1}{4},\tfrac{1}{32})$-axial algebra in characteristic zero is isomorphic to a Norton-Sakuma algebra. We extend this definition to any characteristic  by saying that the {\it Norton-Sakuma algebras}\index{Norton-Sakuma algebras!in any  characteristic} are the algebras  
\[1\A,\: 3\C(\tfrac{1}{4}), \: 3\A(\tfrac{1}{4}, \tfrac{1}{32}), \: 4\A(\tfrac{1}{4}, \tfrac{1}{32}), \: 5\A(\tfrac{1}{4}, \tfrac{1}{32}), \: 6\A(\tfrac{1}{4}, \tfrac{1}{32}), \: 2\B, \: 4\B(\tfrac{1}{4}, \tfrac{1}{32}), \: \mbox{ and } \: 3\C(\tfrac{1}{32}).\footnote{Note that the Norton-Sakuma algebra $2\A$ is the algebra $3\C(\tfrac{1}{4})$ over the reals.}
\]
 By the Main Theorem, we can determine the (primitive) $2$-generated $\M(\tfrac{1}{4},\tfrac{1}{32})$-algebras in positive characteristics (see Lemma~\ref{JM}). In particular we get 
\begin{mcor}\label{mcor:NS char}
Suppose $\ch(\F)\notin \{2,3,5,7,11,23,31\}$. Every \textup{(}primitive\textup{)} $2$-generated $\M(\tfrac{1}{4},\tfrac{1}{32})$-axial algebra over $\F$ is isomorphic to one of the Norton-Sakuma algebras.
\end{mcor}

This proves \cite[Conjecture 5.2]{survey} and explicitly gives $\ch(\F)>31$. Interestingly, since $\M(\tfrac{1}{4},\tfrac{1}{32})$ does not exist when $\ch(\F)\in \{2,3,7,31\}$, the bound $31$, is down to the fusion law being properly defined. Therefore any extra conditions on the algebra will not reduce the bound on the characteristic. 


\newpage
\section*{Overview}
In Chapter~\ref{ch1}, we give the basic definitions and properties of $2$-generated axial algebras, in particular that of the universal algebra $\mathfrak {V}$, and prove several general lemmas. 

In Chapter~\ref{known}, we describe all the $2$-generated $\Mab$-axial algebras and the properties of these algebras which are relevant for this paper, including the symmetric and non-symmetric algebras.

The proof of the Main Theorem is given in Chapters~\ref{H}-\ref{proof} and follows the scheme given in Table~\ref{MCS}, which we shall now describe.
\begin{table}[H]

\begin{picture}(256,230)(30,150)

\put (131, 351) {\fbox{\tiny $\V$}}
\put (131, 348) {\vector ( -2,-1) {21}}
\put (143, 348) {\vector ( 2,-1) {21}}
\put (80,330) {\fbox{\tiny$\al=2\bt$}}
\put(164, 330){\fbox{\tiny $\al\neq 2\bt$}}
\put (164, 326) {\vector ( -2,-1) {21}}
\put (193, 326) {\vector ( 2,-1) {21}}
\put(100 , 308){\fbox {\tiny not $\mathcal H$-type}}

\put(214, 308){\fbox { {\tiny$\mathcal H$-type} }}
\put (100, 304) {\vector ( -2,-1) {20}}
\put(9, 282){\fbox { \parbox{0.85in}{\tiny $\V$ is a quotient of a symmetric algebra}}}
\put(163, 283){\fbox {\tiny  \parbox{0.95in}{$\V$ is not a quotient of a symmetric algebra}}}
\put (143, 304) {\vector ( 2,-1) {20}}
\put (163, 275) {\vector ( -2,-1) {20}}
\put ( 255, 257) { \fbox {\tiny  $\V$ has skew axet}}
\put (238, 275) {\vector ( 2,-1) {20}}
\put (73, 257) { \fbox {\tiny $\V$ has regular axet}}
\put (76, 252) {\vector ( -2,-1) {20}}
\put(4,234){\fbox {\tiny  $V\in\{V_e, V_o\}$}}
\put (143, 252) {\vector (2,-1) {20}}
\put(163,234){\fbox {\tiny $V\not \in\{V_e, V_o\}$}}
\put (163, 229) {\vector ( -2,-1) {20}}
\put (111, 211) {\fbox{\tiny $|X|\geq 8$}}
\put (215, 229) {\vector ( 2,-1) {20}}
\put (235, 211) {\fbox{\tiny $|X|< 8$}}
\put (111, 206) {\vector ( -2,-1) {20}}
\put (143, 206) {\vector ( 2,-1) {20}}
\put (163,188) {\fbox{\tiny $\adim(\V_e), \adim(\V_o)< 4$}}
\put (38,188) {\fbox{\tiny $\adim(\V_e)\geq  4$}}

\end{picture}

\caption{The main case subdivision.}\label{MCS}
\end{table} 

Let $\F$ be a field of characteristic other than $2$, $\al$ and $\bt$ be distinct elements of $\F \setminus \{0,1\}$, and  $\V:=(V, \{a_0, a_1\})$ be a primitive $2$-generated $\mathcal{M}(\al,\bt)$-axial algebra over $\F$. 

For $k\in \{0,1\}$,  let $\tau_k$ be the Miyamoto involution associated to $a_k$ and let $\Miy(\V)$ be the Miyamoto group $\langle \tau_0, \tau_1\rangle$ of $V$ associated to the set $\{a_0, a_1\}$. Set $$\rho:=\tau_0\tau_1$$ and, for $i\in \Z$, let
\[
a_{2i}:=a_0^{\rho^i} \quad \mbox{and}\quad a_{2i+1}:=a_1^{\rho^i}.
\]
Since $\rho$ is an automorphism of $\V$, for every $j\in \Z$,  $a_j$ is an axis.
Moreover, as  $\tau_1$ (respectively $\tau_0$) swaps $a_0$ and $a_2$ (respectively $a_{-1}$ and $a_1$), the subalgebras 
 \[
     \V_e:=(V_e, \{a_0, a_2\})\quad  \mbox{and}\quad \V_o:=(V_o, \{a_{-1}, a_1\})
 \]
 generated by $a_0$ and $a_2$, and, respectively, by $a_{-1}$ and $a_1$, are $2$-generated primitive symmetric axial algebras of Monster type $(\al, \bt)$. We call $\V_e$ and $\V_o$ the \index{subalgebra!even}{\em even}~\index{even subalgebra} and the \index{subalgebra!odd}{\em odd}\index{odd subalgebra} {\em subalgebra} of $\V$, respectively. 
 
 For each end node of the tree given in Table~\ref{MCS}, let $\V$ be an algebra satisfying the conditions stated in that node. The Classification Theorem of the Symmetric Algebras gives us a list of possibilities for the pair $(\V_e, \V_o)$, which  is further restricted by a result of M\textsuperscript{c}Inroy and Shpectorov (see Theorem~\ref{axetthm} and Corollary~\ref{axetcor}). The proof of the Main Theorem is then accomplished by a detailed analysis of each possible configuration of $(\V_e, \V_o)$. Key results for this analysis are certain relations  between $\V_e$ and $\V_o$ which are obtained from the construction of the universal algebra (see Sections~\ref{universal}, \ref{relVeVo}, and \ref{evaluation}).

According to Table~\ref{MCS}, the first case subdivision is whether $\al=2\bt$ or $\al\neq 2\bt$. The former case  has been accomplished in~\cite{FMS2} and gives the following:

\begin{teo} \cite[Theorem 1.1]{FMS2}
\label{al=2bt}
Assume $\al=2\bt$. Then $\V$ is either symmetric or isomorphic to one of the following algebras:
\begin{enumerate}
    \item $4\QQ(2\bt, \bt)$;
     \item $4\QQ(-1, -\tfrac{1}{2})^\times$;
    \item $3\C^\prime(\tfrac{2}{3},\tfrac{1}{3})$.
\end{enumerate}
\end{teo}

So we may now assume that $\al \neq 2\bt$.  As $\V$ is primitive, for every $j\in \Z$, there is a linear function $\lambda_{a_j}\colon V\to \F$  such that every $x\in V$ can be written in a unique way as 
$$
x=\lambda_{a_j}(x)a_j+x_{0,j}+x_{\al,j}+x_{\bt,j},
$$
where $x_{0,j}$, $x_{\al, j}$, $x_{\bt,j}$ are $0$-, $\al$-, and $\bt$-eigenvectors for $\ad_{a_j}$, respectively. For $i\in \Z_{\geq 1}$, set 
\begin{equation*}\label{deflmlmf1}
    \lm_i:=\lambda_{a_0}(a_i) \quad  \mbox{and}\quad  \lm_i^f:=\lambda_{a_1}(a_{1-i}).
\end{equation*}
We say that $\V$ is of \emph{$\mathcal H$-type}\index{H-type@$\mathcal H$-type} if 
$$(\al, \bt)=(2,\tfrac{1}{2})\quad  \mbox{and}\quad \{\lmu, \lmf, \lmd, \lmdf\}=\{1\}
.$$ 
In Chapter~\ref{H}, extending the methods used in~\cite{HW, Yabe}, we prove 

\begin{teo}\label{HWthm}
Assume  $\V$ is of $\mathcal H$-type. Then, either $\V$ is isomorphic to a quotient of $\mathcal H$, or $\F$ has characteristic $5$ and $V$ is isomorphic to a quotient of $\hatH$.    
\end{teo}

The quotients of $\mathcal H$ and $\hatH$ have been determined in~\cite{HWQ} and are all symmetric.
The ideals and the quotients of the twelve families of primitive $2$-generated symmetric axial algebras of Monster type $(\al, \bt)$ as in case \ref{symmetric_2} of the classification theorem of the symmetric algebras above have been determined in~\cite{MM}. By those results and Theorem~\ref{HWthm} and we get:

\begin{teo}\label{teoq}
If $\V$ is isomorphic to a quotient of a symmetric algebra, then $\V$ is either symmetric or isomorphic to $4\B(-1,\tfrac{1}{2}; \nu)^\times$.
\end{teo}

 Let 
\begin{equation}\label{defX}
    X:= \{a_i : i\in \Z\}=a_0^{\Miy(\V)}\cup a_1^{\Miy(\V)}.
\end{equation}
The action of $\Miy(\V)$ on the set $X$ determines an axet $(\Miy(\V), X, \tau)$ (see~\cite{axet}) which we call {\em the axet of $\V$}\index{axet of $\V$}. 
 We say that $\V$ has \index{axet!regular}{\em regular axet}\index{regular axet}
 if the orbits of the two generating axes $a_0$ and $a_1$ under the Miyamoto group $\Miy(\V)$ have the same length, otherwise we say that $\V$ has \index{axet!skew}{\em skew axet}\index{skew axet}. The case when $\V$ has skew axet has been classified in~\cite{Turner1, Turner2} and gives the following:

\begin{teo} \cite[Theorem 1.3]{Turner2}
\label{skew}
Assume $\V$ has skew axet. Then $\V$ is isomorphic to one of the following algebras:
\begin{enumerate}
    \item  $3\C^\prime(\al,1-\al)$, for $\al \neq \tfrac{1}{2}$;
    \item  $3\QQ^\prime(\tfrac{1}{3},\tfrac{2}{3})$.
\end{enumerate}
\end{teo}

Note that, in~\cite{Turner2}, the algebra $3\QQ^\prime(\tfrac{1}{3},\tfrac{2}{3})$ is denoted by $\QQ_2(\tfrac{1}{3},\tfrac{2}{3})$ and is defined only if $\ch(
\F)\neq 5$, while the algebra $\QQ_2(\tfrac{1}{3})^\times \oplus \langle \mathbbm 1 \rangle $ is constructed when $\ch(\F)=5$. Since the two algebras have the same multiplication table, we use the symbol $3\QQ^\prime(\tfrac{1}{3},\tfrac{2}{3})$ for both of them (see Section~\ref{sec:ns}).

In Chapter \ref{chaplast}, we consider the case where $V=V_e$ (respectively $V=V_o$). Then, the symmetric subalgebra  $\V_e$ contains also $a_1$ (respectively $a_0$), which is still an axis of Monster type $(\al,\bt)$. In these cases, we prove the following:

\begin{teo} 
\label{teonsa}
Suppose that $\al\neq 2\bt$, $V\in \{V_e, V_o\}$, and $\V$ has a regular axet. Then $\V$ is symmetric. 
\end{teo}

Now assume that $\V$ satisfies none of the hypotheses of Theorems~\ref{al=2bt},~\ref{HWthm}, \ref{teoq}, \ref{skew}, and \ref{teonsa}.  Set
\[
n:=|a_0^{\Miy(\V)}|, \mbox{ with } n\in \N\cup \{\infty\}.
\]
Since we are assuming that $V\not \in \{V_e, V_o\}$ and $\V$ has regular axet, it follows that
\[
a_0^{\Miy(\V)}\neq a_1^{\Miy(\V)}\mbox{ and } |a_0^{\Miy(\V)}|=|a_1^{\Miy(\V)}|=n,
\]
whence
\[
|X|=|a_0^{\Miy(\V)} \cupdot a_1^{\Miy(\V)}|=
|a_0^{\Miy(\V)}| + |a_1^{\Miy(\V)}| \in \{2n, \infty\}.
\]
In this case we say that $\V$ has axet $X(n+n)$. If $n=1$, an easy argument proves that $\V$ is symmetric (see Lemma~\ref{lhyp1}). 
In Chapter~\ref{chap:regular}, we deal with the cases when $n\geq 2$. If $n\in \{2,3\}$, then we prove the following:

\begin{teo}
\label{teo2}
Assume  $\al\neq 2\bt$ and that $\V$ has a regular axet with $n=2$. Then $\V$ satisfies the hypothesis of either Theorem~\ref{teoq} or Theorem~\ref{teonsa}.
\end{teo}

\begin{teo}
\label{teo3}
Assume  $\al\neq 2\bt$ and that $\V$ has a regular axet with $n=3$. Then $\V$ satisfies the hypothesis of either Theorem~\ref{teoq} or Theorem~\ref{teonsa}.
\end{teo}

Finally, suppose that $n\geq 4$. The dimension of the linear span of the set $X$ is called the {\em axial dimension}\index{axial dimension} of $\V$ and will be denoted by $\adim(\V)$\nomenclature{$\adim(\V)$}{\pageref{teo3}}. 
Chapter~\ref{ch:infinite} deals with the case when $\V_e$ and $\V_o$ are isomorphic to a quotient of either an algebra of Jordan type $\bt=\tfrac{1}{2}$ or of $\IY_3(\al, \tfrac{1}{2}; \mu)$.  We show that in these cases, $\V$ satisfies the hypothesis of either Theorem~\ref{teoq}, or Theorem~\ref{teonsa}. Using this result, in Section~\ref{chap:larger}, we  prove 

\begin{teo}
\label{teolarger}
Assume  $\al\neq 2\bt$ and that $\V$ has regular axet with $n\geq 4$. Then $\V$ satisfies the hypothesis of either Theorem~\ref{teoq} or Theorem~\ref{teonsa}.
\end{teo}

In Chapter~\ref{proof}, we prove the Main Theorem and  its corollaries.

\chapter{General setup}\label{ch1}

\section{Basic definitions and preliminary results}\label{basicdef}
In this section and in the remainder of this paper, we assume that $\F$\nomenclature{$\F$}{\pageref{basicdef}} is a field of characteristic other than $2$ and $R_\F$\nomenclature{$R_\F$}{\pageref{basicdef}} is a commutative associative ring containing $\F$ (in this paper we shall almost always take $\F=R_\F$, except when defining the universal algebra in Section~\ref{universal}). Let $\al$ and $\bt$ be two distinct  elements in $\F\setminus \{0,1\}$, and  $V$ a commutative $R_\F$-algebra. 

A {\em fusion law}\index{fusion law} is a pair $\mathcal F:=(F, \star)$ where $F$ is a subset of $\F$ and $\star$ is a map
$$\star\colon F \times F \rightarrow 2^F$$
where, as usual, $2^F$ denotes the power set of $F$. For the purpose of this paper we focus on the fusion laws $\mathcal {J}(\eta)$\nomenclature{$\mathcal {J}(\eta)$}{\pageref{Jeta}}\label{Jeta} (for $\eta\in \F\setminus\{0,1\}$) and $\Mab$\nomenclature{$\Mab$}{\pageref{Mab}}\label{Mab}  described in Table~\ref{fusion}, which are called \index{fusion law!Jordan type}{\em Jordan type}\index{Jordan type fusion law} and \index{fusion law!Monster type}{\em Monster type}\index{Monster type fusion law} fusion  law, respectively. 
{
\renewcommand{\arraystretch}{1.5}
\begin{table}[H]
\[ 
\begin{array}{ccc}
\begin{array}{|c||c|c|c|c|}
\hline
\star & 1 & 0 & \eta \\
\hline
\hline
1 & 1 &\emptyset & \eta \\
\hline
0 & \emptyset& 0 & \eta \\
\hline
\eta & \eta & \eta & 1,0 \\
\hline
\end{array}
&\hspace{1cm} &
\begin{array}{|c||c|c|c|c|}
\hline
\star & 1 & 0 & \al & \bt\\
\hline
\hline
1 & 1 & \emptyset& \al & \bt\\
\hline
0 & \emptyset& 0 & \al & \bt\\
\hline
\al & \al & \al & 1,0 & \bt\\
\hline
\bt & \bt & \bt & \bt & 1,0,\al\\
\hline
\end{array}\\
\mathcal J(\eta) & & \Mab
\end{array}
\]
\caption{Fusion laws $\mathcal J(\eta)$ and $\Mab$} \label{fusion}
\end{table}
Given an element $a$ of $V$, denote by $\ad_a$ the {\em adjoint map}\index{adjoint map} associated with $a$, that is 
\[
\begin{aligned}
\ad_a\colon V& \longrightarrow V\\
v&\longmapsto av.
\end{aligned}
\]
For $\lambda\in F$, $S\subseteq F$,  and $a\in V$, let
\[
V_\lambda^a:=\{v\in V : av=\lambda v\}\quad \mbox{and}\quad V_S^a:=\bigoplus_{\lm \in S}V_\lm^a.
\]\nomenclature{$V_\lambda^a$}{\pageref{auto}}\label{auto} \nomenclature{$V_S^a$}{\pageref{auto2}}\label{auto2}
An {\em $\mathcal{F}$-axis}\index{F-axis@$\mathcal F$-axis} of $V$ is an element $a\in V$ such that 
\begin{enumerate}
    \item[A0)] $a$ is idempotent,
    \item[A1)] $V=\bigoplus_{\lambda \in F}V_\lambda^a$,
    \item[A2)] $V_\lambda^a V_\mu^a \subseteq V_{\lambda\star \mu}^a$, for every $\lambda, \mu \in F$. 
\end{enumerate}
Further, $a$ is {\em primitive}\index{primitive axis} if, 
\begin{enumerate}
\item[A3)]   $V_1^a=R_\F a$.
\end{enumerate}
If $a$ is a primitive axis of $V$, then, by~\cite[Lemma~3.4]{FMS3}, there is a linear function $\lambda_{a} \colon V\to R_\F$\nomenclature{$\lambda_{a}$}{\pageref{la}}\label{la} such that every $x\in V$ can be written in a unique way as 
\begin{equation}\label{xlambda}
 x=\lambda_{a}(x)a+\sum_{\mu\in F\setminus \{1\}}x_\mu,   
\end{equation}
where $x_\mu\in V_\mu^a$ for every $\mu \in F\setminus \{1\}$. The scalar 
$\lambda_a(x)$ is called the {\em projection}\index{projection}
 of $x$ onto $a$.
 
Note that, for every axis $a$ of Monster type $(\al, \bt)$ in $V$, the fusion law $\Mab$ induces a $\Z_2$-grading on $V$, by setting 
\[
V_+^a:=V_1^a+V_0^a+V_{\alpha}^a \quad \mbox{ and } \quad  V_-^a=V_{ \beta}^a.
\]\nomenclature{$V_+^a$}{\pageref{V+}}\label{V+}\nomenclature{$V_-^a$}{\pageref{V-}}\label{V-}
It follows that the linear map 
\[
\tau_a\colon V\to V
\]\nomenclature{$\tau_a$}{\pageref{tau}}\label{tau} that fixes the elements of $V_+^a $ and negates the elements of $V_-^a$ is an involutory automorphism of $V$ called {\em Miyamoto involution}\index{Miyamoto involution} associated to $a$ (see~\cite[Proposition~3.4]{HRS})\footnote{In this paper, Miyamoto involutions and Miyamoto groups (see below) are defined only for axes of Monster type, since  this suffices for our purposes. For the general definition see~\cite[Section~2.3]{survey}}.

If the $R_\F$-algebra $V$ is generated (as an algebra) by a set $\mathcal A$ of primitive $\mathcal F$-axes, then  we say that the pair $\V:=(V, \mathcal A)$  is an \index{axial algebra!F-axial algebra@$\mathcal{F}$-axial algebra}{\em $\mathcal F$-axial algebra}\index{F-axial algebra@$\mathcal F$-axial algebra}  over $R_\F$. We call the $R_\F$-algebra $V$ the {\em support}\index{support of $\V$} of $\V$ and will be also denoted by $\supp (\V)$\nomenclature{$\supp (\V)$}{\pageref{supp}}\label{supp}. If $\mathcal F=\mathcal J(\eta)$ (respectively $\mathcal F=\Mab$), then we say that $\V$ is of {\em Jordan type}\index{axial algebra!of Jordan type} $\eta$ (respectively {\em Monster type $(\al, \bt)$})\index{axial algebra!of Monster type}. We remark that, as a difference to the definition of $\mathcal F$-algebra given in this paper, in the existing literature the condition of primitivity is not always required.

A {\it Frobenius form}\index{Frobenius form} on an $\mathcal F$-axial algebra $\V$ over $R_\F$ is a non-zero bilinear form $\kappa \colon V\times V\to R_\F$ which is {\it invariant}\index{invariant form} with respect to the algebra product, that is, for every $u,v,w\in V$, 
\[
\kappa(uv,w)=\kappa(u,vw).
\]
The invariance of the Frobenius form $\kappa$ implies that for every axis $a$ of $V$, the eigenspaces of $\ad_a$ are mutually orthogonal. In particular it follows that the Miyamoto involution $\tau_a$ is an isometry.
Moreover, since $V$ is commutative and generated by idempotents, invariance also  implies  that $\kappa$ is symmetric.

Let $\V:=(V,\mathcal{A})$ be an $\mathcal{M}(\al,\bt)$-axial algebra over $R_\F$. The {\it Miyamoto group}\index{Miyamoto group} of $\V$ is the group generated by the Miyamoto involutions associated to the axes in $\mathcal{A}$ and will be denoted by $\Miy(\V)$\nomenclature{$\Miy(\V)$}{\pageref{Miya}}\label{Miya}. Note that, if $\eta\in \{\al, \bt\}$, then, by Table~\ref{fusion}, every axial algebra $\V$ of Jordan type $\eta$ is also an $\mathcal{M}(\al,\bt)$-axial algebra. Precisely, if $\eta=\al$, then $V_\bt^a=\{0\}$ for every $a\in \mathcal{A}$ and $\Miy(\V)$ is trivial. If $\eta=\bt$, then $V_\al^a=\{0\}$ for every $a\in \mathcal A$ and $\Miy(\V)$ is trivial if and only if also $V_\bt^a=\{0\}$ for every $a\in \mathcal{A}$, whence, by~\cite[Lemma~3.3]{FMS3}, $V$ is associative.

Given two $\mathcal F$-axial algebras $\V:=(V,\mathcal{A})$ and $\mathcal W:=(W,\mathcal{B})$ over $R_\F$, we say that an $R_\F$-algebra homomorphism
$$\phi\colon V \to W$$
is an 
{\em axial homomorphism}\index{axial homomorphism}
if 
\[
\mathcal{A}^\phi \subseteq \mathcal{B}. 
\]
If $\phi$ is also bijective, we say that  $\V$ and $\mathcal{W}$ are {\em isomorphic}\index{isomorphic algebras}. Note that 
the $\Mab$-axial algebras $3\C(\al)$ and $3\C^\prime(\al,1-\al)$ have isomorphic supports but are not isomorphic as axial algebras (see Note \ref{table3Cskew axes} to Table~\ref{table3Cskew}).

\medskip

For the remainder of this section we assume that $\V:=(V,\mathcal A)$ is an $\mathcal{M}(\al,\bt)$-axial algebra over $R_\F$ such that $\mathcal A$ contains two elements (not necessarily distinct), which we shall denote by $a_0$ and  $a_1$. In this case we say that $\V$ is \index{axial algebra!$2$-generated}{\it $2$-generated}\index{$2$-generated axial algebra}. For $k\in \{0,1\}$,  denote by $\tau_k$ be the Miyamoto involution associated to $a_k$. Set $\rho:=\tau_0\tau_1$\nomenclature{$\rho$}{\pageref{rho}}\label{rho}, and, for $i\in \Z$, set
\begin{equation}\label{aroi}
    a_{2i}:=a_0^{\rho^i} \quad \mbox{  and }\quad  a_{2i+1}:=a_1^{\rho^i}.
\end{equation}
Note that, since $\rho$ is an automorphism of $V$, for every $j\in \Z$,  $a_j$ is an axis. Denote by $\tau_j:=\tau_{a_j}$ \nomenclature{$\tau_j$}{\pageref{tauj}}\label{tauj} the corresponding Miyamoto involution. 

\begin{lemma}\cite[Lemmas~2.1 and~2.2]{FMS2}\label{invariant}
For every $ n\in \Z_+ $ and $r,t\in \Z$, 
\begin{enumerate}
    \item  $
a_ra_{r+n}-\bt(a_r+a_{r+n})$ is invariant under $\langle \tau_r, \tau_{r+n}\rangle$;\label{invariant_1}
\item  if $r\equiv_n t$, then 
$
a_ra_{r+n}-\bt(a_r+a_{r+n})=a_ta_{t+n}-\bt(a_t+a_{t+n}).
$  
\end{enumerate}
\end{lemma}

Let $n$ be a non negative integer, denote by $\bar r$ the congruence class $r+n\Z$ and define 
\begin{equation}\label{defs}
s_{\ti,n}:=a_ia_{i+n}-\beta (a_i+a_{i+n}). 
\end{equation}\nomenclature{$s_{\ti,n}$}{\pageref{sin}}\label{sin}
For $\bt=\tfrac{1}{2}$, the following formula, which is an immediate consequence of Equation~\eqref{defs}, will be used without reference:
\[
 \mbox{when }\quad \bt=\tfrac{1}{2}, \quad s_{\ti,n}=-\tfrac{1}{2}(a_i-a_{i+n})^2. 
\]
Let $i\in \Z_{\geq 1}$. Since $a_j$ is a primitive axis, for every $j\in \Z$, by Equation~\eqref{xlambda}, we can define  
\begin{equation}\label{deflmlmf}
    \lm_i:=\lambda_{a_0}(a_i) \quad \mbox{and} \quad \lm_i^f:=\lambda_{a_1}(a_{1-i}).
\end{equation}\nomenclature{$\lm_i$}{\pageref{nomlm}}\label{nomlm}\nomenclature{$\lm_i^f$}{\pageref{nomlm}}
Let 
\nomenclature{$u_i$}{\pageref{nomlm}}\nomenclature{$v_i$}{\pageref{nomlm}}\nomenclature{$w_i$}{\pageref{nomlm}}
\nomenclature{$\tilde u_i$}{\pageref{nomlm}}\nomenclature{$\tilde v_i$}{\pageref{nomlm}}\nomenclature{$\tilde w_i$}{\pageref{nomlm}}
\begin{equation}
\begin{aligned}\label{defui}
u_i &:= \tfrac{1}{\alpha} ((\lambda_i  - \beta - \alpha \lambda_i ) a_0 + \tfrac{1}{2}(\alpha - \beta) (a_i + a_{-i})-s_{\tz,i} ); \\
v_i &:= \tfrac{1}{\alpha} ((\bt-\lambda_i )a_0+\tfrac{\bt}{2}(a_i+a_{-i})+s_{\tz,i});\\
w_i &:= \tfrac{1}{2}(a_i-a_{-i});\\
\tilde u_i &:= \tfrac{1}{\alpha} ((\lambda_i^f  - \beta - \alpha \lambda_i^f ) a_1 + \tfrac{1}{2}(\alpha - \beta) (a_{1-i} + a_{1+i})-s_{\tz,i} );\\
\tilde v_i &:= \tfrac{1}{\alpha} ((\bt-\lambda_i^f )a_1+\tfrac{\bt}{2}(a_{1-i}+a_{1+i})+s_{\tz,i});\\
\tilde w_i &:= \tfrac{1}{2}(a_{1-i}-a_{1+i}).
\end{aligned}
\end{equation}

\begin{lemma}\cite[Lemma~6.4]{FMS3}\label{ui}
Let $i\in \Z_{\geq 1}$, then, for $\ad_{a_0}$ \textup{(}respectively $\ad_{a_1}$\textup{)},
\begin{enumerate}
    \item  $u_i$ (respectively $\tilde u_i$) is a $0$-eigenvector, 
    \item  $v_i$ (respectively $\tilde v_i$) is an $\al$-eigenvector, 
    \item $w_i$ (respectively $\tilde w_i$) is a $\bt$-eigenvector.
\end{enumerate} 
Moreover,  
$ 
a_i=\lambda_i  a_0+u_i+v_i +w_i \quad \mbox{ and }\quad a_{1-i}=\lambda_i^f  a_1+{\tilde u}_i+{\tilde v}_i +{\tilde w}_i.$
 \end{lemma}
 \begin{lemma}\cite[Lemma 6]{FM}\label{s}
 For every $n\in \Z_+$, $i,r\in \Z_{\geq 1}$, and $j\in \{0,1\}$, the following hold:
\begin{enumerate}
\item  for every $x\in V$, $\lambda_{a_j}(x)=\lambda_{a_j}(a_jx)$; \label{s_1}
\item  $\lambda_{a_0}(s_{\tz,n})=\lm_n-\bt-\bt\lm_n$ and $\lambda_{a_1}(s_{\tu,n})=\lm_n^f-\bt-\bt\lm_n^f$; \label{s_2}
\item  $\lambda_{a_0}(a_{-i})=\lambda_i$ and $\lambda_{a_1}(a_{i+1})=\lambda_i^f$; \label{s_3}
\item  $\lambda_{a_0}(s_{\tr,n+1})=\lambda_{a_0}(s_{-\tr,n+1})$. \label{s_4}
\end{enumerate}
\end{lemma}

\begin{lemma}\label{quotient}
 Let $I$ be an ideal of $V$ not containing $a_i$ for $i\in \{0,1\}$. Then $(V/I, \{a_0+I, a_1+I\})$ is an $\mathcal{M}(\al,\bt)$-axial algebra such that  
$$
\lm_{a_{0+I}}(a_1+I)=\lmu \mbox{ and } \lm_{a_{1+I}}(a_0+I)=\lmf
.$$
Moreover, the canonical projection $\pi:V\to V/I$ induces a homomorphism of $\V$ onto $(V/I, \{a_0+I, a_1+I\})$.
\end{lemma}
\begin{proof}
    The proof is immediate.
\end{proof}

We denote the algebra $(V/I, \{a_0+I, a_1+I\})$ of Lemma~\ref{quotient} by $\V/I$ \nomenclature{$\V/I$}{\pageref{V/I}}\label{V/I} and call it a {\em  quotient of $\V$}\index{quotient of $\V$}.

\begin{lemma}\label{action}
Let $h,m\in \Z$, with $m$ positive. With the above notation, the following assertions hold:
\begin{enumerate}
\item  $\rho^m$ maps $s_{\tih , 2m-1}$ to $s_{\tih+\tu, 2m-1}$;  \label{action_1}
\item  $\Miy(\V)$ acts transitively on each one of the following sets 
\[
\{s_{2\tr,2m} \mid \tr\in \Z_{2m}\}, \quad
\{s_{2\tr+\tu,2m} \mid \tr\in \Z_{2m}\},\quad
\{s_{\tr,2m-1} \mid \tr\in \Z_{2m-1}\}.
\]
\end{enumerate}
  \end{lemma}
\begin{proof}
 The result is immediate since, for every $j\in \Z$, $\rho$ maps $a_j$ to $a_{j+2}$ and so it maps $s_{\tj, n}$ to $s_{\tj+\bar 2, n}$.   
\end{proof}

Let $\sigma$ be an involutory automorphism of $R_\F$ that fixes pointwise the elements of $\F$.
A {\em $\sigma$-flip}\index{$\sigma$-flip} of $\V$ is a map
$
f:V\to V
$
 satisfying the following conditions
\begin{enumerate}
    \item[F1)] $f$ is a $\sigma$-semi automorphism of the $R_\F$-module $V$, that is, for all $v_1, v_2\in V$, $r_1, r_2\in R_\F$, $(r_1v_1+r_2v_2)^f=r_1^\sigma v_1^f+r_2^\sigma v_2^f$,
    \item[F2)] $f$ preserves the algebra multiplication of $V$,
    \item[F3)] $f$ swaps $a_0$ and $a_1$.
\end{enumerate}
 Clearly, $f$ has order two.


\begin{lemma}\label{translation}
Let $f$ be a $\sigma$-flip of $\V$. For every $i\in \Z$, 
$$
a_i^f=a_{1-i} \quad \mbox{ and }\quad a_i^{\tau_0 f}=a_{i+1}.
$$
\end{lemma}
\begin{proof}
Let $\mu\in \{1,0,\al, \bt\}$. Then $\mu\in \F$, whence $\mu^\sigma=\mu$. By F1) and F3), 
\[
(V_\mu^{a_0})^f=V_\mu^{a_1}\quad \mbox{ and } (V_\mu^{a_1})^f=V_\mu^{a_0}.
\]      
By the definition of Miyamoto involution, it follows that conjugation by $f$ swaps $\tau_0$ and $\tau_1$, and so it inverts $\rho$. Assume first that $i$ is even, say $i=2j$ with $j \in \Z$, then 
\[
(a_{i})^f=((a_0)^{\rho^j})^f=(a_0)^{\rho^j f}=(a_{0})^{f\rho^{-j}}=((a_{0})^f)^{\rho^{-j}}=a_1^{\rho^{-j}}=a_{i-1}.
\]
A similar computation gives the first assertion if $i$ is odd. 
This second one follows immediately.
\end{proof}

Note that, if $\sigma$ is the identity map on $R_\F$, then $f$ is an $R_\F$-algebra automorphism of $V$. In this case we say that $f$ is a {\em flip}\index{flip}. We call $\V$ \index{axial algebra!symmetric $2$-generated}\emph{symmetric}\index{symmetric $2$-generated axial algebra} if it has a flip. 

For a symmetric algebra $\V$ let
\begin{equation}
  V^\ast:=\lla \,(a_0-a_1)^\tau  :  \tau\in \langle f, \Miy(\mathcal V)\rangle \,\rra  
\end{equation}\nomenclature{$V^{\ast}$}{\pageref{Vast}}\label{Vast}
and 
\begin{equation}
  V^{\ast \ast}:=\lla\, (a_0-a_2)^\tau  :  \tau\in \langle f, \Miy(\mathcal V)\rangle\, \rra .  \end{equation}\nomenclature{$V^{\ast\ast}$}{\pageref{Vastast}}\label{Vastast}
Note that $V^{\ast \ast} \subseteq V^\ast$, since $a_0-a_2=(a_0-a_1)+(a_1-a_2)=(a_0-a_1)-(a_0-a_1)^{\tau_1}$.

Let $V_e$\nomenclature{$V_e$}{\pageref{Ve}}\label{Ve} be the subalgebra of $V$ generated by $a_0$ and $a_2$ and let $V_o$\nomenclature{$V_o$}{\pageref{Vo}}\label{Vo} be the subalgebra of $V$ generated by $a_{-1}$ and $a_1$. Set
\[
\V_e:=(V_e, \{a_0, a_2\})\quad \mbox{and} \quad \V_o:=(V_o, \{a_{-1}, a_1\}).
\]

\begin{lemma}\label{lem:VeVo}
$\V_e$ and $\V_o$ are symmetric $\Mab$-axial algebras. Moreover
$$
V_e^\ast = \lla  (a_0-a_2)^\tau : \tau\in  {\Miy(\V)} \rra,\quad V_e^{\ast\ast} = \lla  (a_{-2}-a_2)^\tau :  \tau\in  \Miy(\V) \rra
$$
and 
$$ V_o^{\ast\ast}=\lla  (a_{-1}-a_{1})^\tau :  \tau\in  {\Miy(\V)}\rra, \quad V_o^{\ast\ast}=\lla  (a_3-a_{-1})^\tau :  \tau\in  \Miy(\V)\rra.$$ 
    \end{lemma}
    \begin{proof}
 By definition $\V_e$ and $\V_o$ are generated by $\Mab$-axes, hence they are $\Mab$-axial algebras.  Moreover, by the definition of $a_2$ and $a_{-1}$, $\tau_0$ swaps $a_0$ and $a_2$ and $\tau_1$ swaps $a_1$ and $a_{-1}$. Thus $\V_e$ and $\V_o$ are symmetric. 
 Moreover,
$$
\langle \tau_0, \Miy(\V_o)\rangle= \Miy(\V)=\langle \tau_1, \Miy(\V_e)\rangle,  
$$
whence, by the definition of $V_e^\ast$ and $V_o^\ast$
the result follows.
    \end{proof}

Recall the definition of axial dimension $\adim(\V)$ of $\V$, given in the Overview. 

\begin{lemma}\cite[Lemma~2.1]{Yabe}\label{lemma:adim}
Assume $\V$ is symmetric. Then, for every $i\in \Z$, $(a_i, \ldots , a_{i+\adim(\V)-1})$  is a basis of the linear span of $A^{\Miy(\V)}$.    
\end{lemma}


A \index{axet!Z2-axet@$\mathbb{Z}_2$-axet}{\em $\Z_2$-axet}%
\index{Z2-axet@$\mathbb{Z}_2$-axet} (see~\cite{axet}) is a triple $(G,Y,\tau)$ where $G$ is a group, $Y$ is a $G$-set, and 
\[
\tau \colon Y \to G
\]
is a map (written $\tau_y = \tau(y)$), such that, for all $y\in Y$ and $g \in G$,
\begin{enumerate}[enum_note]
\item $\tau_y \in G_y$,
\item $\tau_y^2 = 1$, 
\item $\tau_{yg} = {\tau_y}^g$,
\end{enumerate}
and we set $\Miy Y := \langle \tau_y : y\in Y\rangle \leq G$.  The axet $(G,Y,\tau)$ is \index{axet!$2$-generated}{\em $2$-generated}\index{$2$-generated axet} if there exist $y_0, y_1\in Y$ such that 
\[
Y=y_0^G\cup y_1^G.
\]
In particular, if $X$ is defined as in Equation~\eqref{defX} on page~\pageref{defX} and $\tau\colon X \to \Miy(\V)$ is the map that associates to each $a \in X$ the Miyamoto involution $\tau_a$, then the triple $(\Miy(\V), X, \tau )$ is a $2$-generated $\Z_2$-axet, which we call \index{axet!of $\V$}{\em the axet of $\V$}\index{axet of $\V$}.  

The  $2$-generated $\Z_2$-axets have been classified in~\cite{axet}. Let $n\in \N\cup \{\infty \}$. For $n\in \N$,  let $\Pi_n:=\{P_0, \ldots , P_{n-1}\}$ be the vertex set of a regular $n$-gon and, for every $i\in \{0, \dots n-1\}$, let $\tau_{P_i}$ be the reflection in $P_i$. For $n=\infty$  let   $\Pi_n=\Z$,  let $P_i:=i$, and, for every $i\in \Z$, let $\tau_{P_i}$ be the map that sends, for every $j\in \Z$, $i+j$ to $i-j$. Let $M_{n}:=\langle \tau_{P_i} : P_i \in\Pi_n \rangle$ and, as above, let
\[
\begin{aligned}
    \tau \colon \Pi_n &\longrightarrow M_{n}\\
    P_i & \longmapsto \tau_{P_i}.
  \end{aligned}
\]
Then  $(M_{n}, \Pi_n, \tau )$ is a $2$-generated $\Z_2$-axet denoted by $X(n)$\nomenclature{$X(n)$}{\pageref{X(n)}}\label{X(n)}. Note that, if $n$ is odd, $M_n$ is the dihedral group of order $2n$ and $\Pi_n=P_0^{M_n}=P_1^{M_n} $; if $n$ is even or $n=\infty$, $M_n$ is a subgroup of order $2$ in the dihedral group of order $2n$, $\Pi_n=P_0^{M_n}\cupdot P_1^{M_n} $, and $|P_0^{M_n}|=|P_1^{M_n}|$. 
Now let $k\in \N$ and $n=4k$. Let 
\[
P_i^\prime:=
\begin{cases}
    P_i & \mbox{ if } i\equiv 0 \bmod 2\\
    \{ P_i, P_{i+2k} \} & \mbox{ if } i\equiv 1 \bmod 2
\end{cases}
\]
and $\Pi^\prime:=\{P_0^\prime, \ldots , P_{n-1}^\prime\}$. 
Since $\tau_{P_i}=\tau_{P_{i+2k}}$, for every $i\in \{0,\ldots , n-1\}$, and $\tau_{P_i}$ permutes the pairs of opposite vertices in $\Pi_n$,  the map $\tau$ induces in a natural way a map 
\[
\begin{aligned}
\tau^\prime\colon \Pi^\prime & \longrightarrow M_{n} \\
 P_i^\prime &\longmapsto \tau_{P_i^\prime}
\end{aligned}
\]
and $(M_{n}, \Pi_n^\prime, \tau^\prime )$ is a $2$-generated $\Z_2$-axet denoted by $X^\prime(3k)$\nomenclature{$X^\prime(n)$}{\pageref{X'}}\label{X'}. In this case, $$\Pi_n^\prime={P_0^\prime}^{M_n}\cupdot {P_1^\prime}^{M_n}\quad  \mbox{ and }\quad  |{P_0^\prime}^{M_n}|=2|{P_1^\prime}^{M_n}|.$$

\begin{theorem}\cite[Theorem 1.1]{axet}\label{axetthm}
Let $(G, Y, \tau )$ be a $2$-generated $\Z_2$-axet with $|Y|=n$, where $n\in \N\cup \{\infty\}$. Then $(G, Y, \tau )$ is isomorphic to one of the following: 
\begin{enumerate}
  \item   $X(n)$, 
\item  $X^\prime(n)$, where $n = 3k$ and $k \in \N$.
\end{enumerate}
\end{theorem} 
By Theorem~\ref{axetthm} and the above discussion, we get the  the following.
\begin{cor}\label{axetcor}
 Let $\V$, $a_0$, $a_1$, and $X$ be as above and set
 \[
 n_i:=|a_i^{\Miy(\V)}|
 \]
 for $i\in \{0,1\}$. Then 
 \begin{enumerate}
 \item  either $|X|=\infty$ and $n_0=n_1$; or
 \item  $|X|\in \N$ and $n_0n_1^{-1}\in \{\tfrac{1}{2}, 1,2\}$. \label{axetcor_2}   
 \end{enumerate}
 \end{cor}

 \begin{lemma}\label{lhyp1}
   If $a_0=a_2$ or $a_1=a_{-1}$, then either $\V$ is an algebra of Jordan type (and in particular it is symmetric) or $\V$ has a skew axet (and Theorem~\ref{skew} holds).  
 \end{lemma}
 \begin{proof}
   Suppose $a_0=a_2$. Then, by Corollary~\ref{axetcor}~\ref{axetcor_2} (and using the same notation), $n_0=1$ and $n_1\in \{1,2\}$. If $n_1=2$, then $\V$ has a skew axet and Theorem~\ref{skew} holds. If $n_1=1$, then $a_1=a_{-1}$, and so $\tau_0$ and $\tau_1$ are the identity on $V$. This implies that  the $\bt$-eigenspaces of $\ad_{a_0}$ and $\ad_{a_1}$ are trivial. Hence $\V$ is a $2$-generated  algebra of Jordan type and, by~\cite{HRS1}, it  is symmetric.  
 \end{proof}


\section{The universal $\Mab$-axial algebra $\mathfrak {V}$}\label{universal}

As in the previous one, in this section $\F$ is a field of characteristic other than $2$. 
By~\cite[Theorem~5.8 and Corollary~5.10]{FMS3} there exist a ring $\tilde \F$\nomenclature{$\tilde \F$}{\pageref{universal}} containing $\F$ as a subring and  a $2$-generated $\Mab$-axial algebra (the {\it universal $\Mab$-axial algebra})\index{universal $\Mab$-axial algebra} $\mathfrak{V}:=(\Vo,\{\ba_0, \ba_1\})$\nomenclature{$\mathfrak{V}$}{\pageref{universal}} over the ring $\tilde \F$, such that, for every $2$-generated $\Mab$-axial algebra $\V:=(V, \{a_0, a_1\})$ over the field $\F$, there is a  unique surjective evaluation map (depending on $\V$)  
\begin{equation}\label{uni}
    \varphi \colon \Vo\cupdot\tilde \F\to V\cupdot\F
\end{equation}
with the following properties 
\begin{enumerate}
\item[U1)] $\varphi|_{\tilde{\F}}$ is a ring homomorphism such that  $\tilde{\F}^\varphi=\F$ and $\varphi|_{\F}=\mathrm{id}_{\F}$,
\item[U2)] for every $i, j\in \{0,1\}$,
$(\lambda_{\ba_i}(\ba_j))^{\varphi}=\lambda_{a_i}(a_j)$,
\item[U3)] $\varphi|_{\Vo}$ is an $\F$-algebra epimorphism when $\Vo$ is given the natural structure of $\F$-algebra by restricting the ring of scalars, 
\item[U4)]  $\ba_0^{\varphi}=a_0$ and $\ba_1^{\varphi}=a_1$,
\item[U5)] for every $\mu\in \tilde{\F}$ and $\bv\in \Vo$, $(\mu\bv)^\varphi=\mu^{\varphi}\bv^{\varphi}$.
\end{enumerate}

By~\cite[Corollary 5.9]{FMS3}, the algebra $(\Vo,\{\ba_0, \ba_1\})$ has a $\sigma$-flip $\bf f$, where $\sigma\colon \tilde{\F}\to \tilde{\F}$ is such that, for every $i\in \Z$,
\begin{equation}\label{lambdai}
  (\lambda_{\ba_0}(\ba_i))^\sigma =\lambda_{\ba_1}(\ba_{1-i}).  
\end{equation}

As above, it is convenient to denote the map $\sigma$ also by the symbol $\bf f$. In this case we understand that the domain of $\bf f$ is $\Vo\cupdot \tilde{\F}$ and $\bf f$ acts as a flip on $\Vo$ whilst it acts as $\sigma$ on $\tilde{\F}$.
Thus, in the case of the universal algebra $\Vo$, $\lm_i^f$ defined in Equation~\eqref{deflmlmf} is exactly the image of $\lm_i$ under $\bf f$. 

 We recall from~\cite{FMS3} some formulas which hold in the algebra $\bf V$ and are needed for this paper. We use bold letters for the elements of $\bf V$.

 For $i\in \N$, set 
\begin{equation}\label{gammai}
    \gamma_i:=(1- \alpha)\lambda_i +\beta( \alpha- \beta- 1). 
\end{equation}
Further let  \nomenclature{$H$}{\pageref{primo}}\nomenclature{$I$}{\pageref{primo}}\nomenclature{$J$}{\pageref{primo}}\nomenclature{$K$}{\pageref{primo}}
\nomenclature{$L$}{\pageref{primo}}\nomenclature{$P$}{\pageref{primo}}\nomenclature{$Q$}{\pageref{primo}}\nomenclature{$R$}{\pageref{primo}}
\nomenclature{$S$}{\pageref{primo}}\nomenclature{$T$}{\pageref{primo}}\nomenclature{$U$}{\pageref{primo}}
\begin{align*}
H&:=\beta^2( \alpha- \beta)\\
I&:=-2   \alpha   \beta   \lmu+2\beta(1- \alpha) \lmf+\tfrac{\beta}{2}(4   \alpha^2  -2   \alpha   \beta- \alpha   +4   \beta^2-2   \beta)\\
J&:= \tfrac{1}{(\alpha-\beta)}\left ((6   \alpha^2-8   \alpha   \beta-2   \alpha+4   \beta)   \lmu^2+(2   \alpha^2-2   \alpha)   \lmu   \lmf \right .\\
&\phantom{{}={}\tfrac{1}{(\alpha-\beta)}(} \quad +2(-2   \alpha^2-2   \alpha   \beta+ \alpha)(\alpha-\beta) \lmu-4\beta(   \alpha -1)(\alpha-\beta)   \lmf  \\
&\phantom{{}={}\tfrac{1}{(\alpha-\beta)}(} \quad \left. \vphantom{\tfrac{1}{(\alpha-\beta)}} - \alpha   \beta(\alpha-\beta)   \lmd+ (4   \alpha^2   \beta-2   \alpha   \beta+2   \beta^3)(\alpha-\beta)\right )\\
K&:= \tfrac{2}{\bt}I \\
L&:= 2\beta ( \alpha -  \beta)\\
P&:=\beta(\alpha-\beta)^2(\alpha-4\beta)\\
Q&:= 4\alpha\beta(\alpha-\beta)\lmu+2(-\alpha^3+5\alpha^2\beta+\alpha^2-4\alpha\beta^2-5\alpha\beta+4\beta^2)\lmf \\
&\phantom{{}={}} \quad + \beta(-10\alpha^2\beta-\alpha^2+14\alpha\beta^2+7\alpha\beta-4\beta^3-6\beta^2) \\
R &:= 2\left ( 2(-3\alpha^2+4 \alpha \beta+\alpha-2 \beta) \lmu^2+2\alpha(1-\alpha) \lmu \lmf \right .\\
&\phantom{{}={}2(} \quad \left . +2(\alpha^3+4 \alpha^2 \beta-6 \alpha \beta^2-3 \alpha \beta+4 \beta^2)\lmu+2\alpha\beta(\alpha -1) \lmf \right .\\
&\phantom{{}={}2(} \quad \left . \vphantom{2(-3\alpha^2+4 \alpha \beta+\alpha-2 \beta) \lmu^2+2\alpha(1-\alpha) \lmu \lmf} +\alpha\beta(\alpha- \beta)\lmd+\beta(-\alpha^3 -8 \alpha^2 \beta+13 \alpha \beta^2+4 \alpha \beta-4 \beta^3-4 \beta^2)\right )\\
S &:= 4\left (2   \alpha(\alpha-  \beta)   \lmu+\alpha( \alpha- 1)  \lmf+(-6   \alpha^2   \beta+10   \alpha   \beta^2+ \alpha   \beta-4   \beta^3)\right )\\
T &:= -2\alpha\beta(\alpha -    \beta)\\
U &:= 2\beta( \alpha - \beta)(\alpha-2\beta).
\end{align*}
These polynomials and the ones defined in the next section are available in~\cite{codenightmare}. 
The following result recalls some structure constants of the algebra $\Vo$ which will be needed later.

 \begin{lemma}\label{primo}\textup{\cite[Lemma 3.1]{FelixPaper}\cite[Lemmas~6.3, 6.8]{FMS3}}
Let $i\in \N$, the following equalities hold in the algebra $\Vo$:
 \begin{enumerate}[itemsep=8pt, topsep=8pt]
 \item \label{primo_1} $\ba_0\bs_{\tz,i}=( \alpha- \beta) \bs_{\tz,i}+\gamma_i \ba_0+\tfrac{\bt}{2}( \alpha - \beta)(\ba_i+\ba_{-i})$;
 \item \label{primo_2} $\ba_1s_{\tz,1}=( \alpha- \beta) \bs_{\tz,1}+\gamma_1^{\bf f} \ba_1+\tfrac{\bt}{2}(\al-\bt)(\ba_0+\ba_{2}))$;
 \item \label{primo_3} $\ba_{-1}\bs_{\tz,1}=
    ( \alpha- \beta) \bs_{\tz,1}+\gamma_1^{\bf f} \ba_{-1}+\tfrac{\bt}{2}( \alpha - \beta)(\ba_0+\ba_{-2})$;
 \item \label{primo_4} if $i>1$, then $\ba_1s_{\tu,i}=( \alpha- \beta) \bs_{\tu,i}+\gamma_i^{\bf f} \ba_1+\tfrac{\bt}{2}( \alpha - \beta)(\ba_{-i+1}+\ba_{i+1})$;
 \item \label{primo_5} $( \alpha-2   \beta)\ba_0\bs_{\tu,2}=H(\ba_{-2}+\ba_2)+I(\ba_1+\ba_{-1})+J\ba_0+K\bs_{\tz,1}+L\bs_{\tz,2}$;
\item \label{primo_6}
$4(\alpha-2\beta)\bs_{\tz,1} \cdot \bs_{\tz,1}=P(\ba_{-2}+\ba_2)+Q(\ba_{-1}+\ba_1)+R\ba_0+S\bs_{\tz,1}+T\bs_{\tz,2}+U\bs_{\tu,2}$.
\end{enumerate}
\end{lemma}

\begin{cor}\label{cor:lm2f}~
\begin{enumerate}
\item   $\lm_{\ba_0}(\bs_{\tu,2})=\tfrac{2(\al-1)}{\al-\bt}\lmu(\lmu-\lmf)+(1-2\bt)\lmu+\bt\lmd-\bt$ \label{cor:lm2f_1}
\item   $\lm_{\ba_1}(\bs_{\tz,2})=\frac{2(\al-1)}{\al-\bt}\lmf(\lmf-\lmu)+(1-2\bt)\lmf+\bt\lmdf-\bt$. \label{cor:lm2f_2}
\end{enumerate}
\end{cor}
\begin{proof}
Let $\bu_1$ and $\bv_1$ be as in Equation~\eqref{defui} on page~\pageref{defui}.  By the fusion law, $\bu_1\bu_1+\bu_1\bv_1$ is a $0$-eigenvector for $\ad_{\ba_0}$ and so 
\begin{equation}\label{uu=0}
\lm_{\ba_0}(\bu_1\bu_1+\bu_1\bv_1)=0.
\end{equation}
 Using Lemma~\ref{primo} parts \ref{primo_1}-\ref{primo_4},  we can express $\bu_1\bu_1+\bu_1\bv_1$ as a linear combination of $ \ba_{-2}$, $\ba _{-1}$, $\ba_0$, $ \ba_{1}$, $\ba _{2}$, $\bs_{\tz,1}$, $\bs_{\tz,2}$, and $\bs_{\tu,2}$. Since, by Equation~\eqref{deflmlmf} on page~\pageref{deflmlmf} and Lemma~\ref{s}, the values of $\lm_{a_0}$ on $ \ba_{-2}$, $\ba _{-1}$, $\ba_0$, $ \ba_{1}$, $\ba _{2}$, $\bs_{\tz,1}$, and $\bs_{\tz,2}$ are known, from Equation~\eqref{uu=0} we deduce the expression for $\lm_{\ba_0}(\bs_{\tu,2})$, as in \ref{cor:lm2f_1}. By applying the flip $\bf f$, we get \ref{cor:lm2f_2}.
\end{proof}

\begin{cor}\label{a1xa2} 
With the notation of Lemma~$\ref{primo}$,
\begin{enumerate}
\item  if $\ba_1-\ba_{-1}=x(\ba_2-\ba_{-2})$ for some $x\in \tilde{\F}\setminus \{0\}$, then: 
\[
\left ((1-\al)(\lmf-\lmu)+\tfrac{\bt}{2}(\al-\bt)(\tfrac{1}{x}-x)\right )(\ba_1-\ba_{-1})-\tfrac{\bt}{2}(\al-\bt)x(\ba_3-\ba_{-3})=0,
\]\label{a1xa2_1}
 \item  if $ \ba_0-\ba_{2}=x(\ba_{-1}-\ba_{3}) $ for  some $ x\in \tilde {\F}\setminus \{0\}$, then 
\[
\left ((1-\al)(\lmu-\lmf)+\tfrac{\bt}{2}(\al-\bt)(\tfrac{1}{x}-x)\right )(\ba_0-\ba_{2})-\tfrac{\bt}{2}(\al-\bt)x(\ba_{-2}-\ba_{4})=0.
\]\label{a1xa2_2}
\end{enumerate}
\end{cor}
\begin{proof}
By parts~\ref{primo_2} and~\ref{primo_3} of Lemma~\ref{primo} we have
\[
 (\ba_1-\ba_{-1})s_{\tz,1}=\gamma_1^{\bf f}(\ba_1-\ba_{-1})+\tfrac{\bt}{2}(\al-\bt)(\ba_2-\ba_{-2})   
\]
and similarly, by applying $\tau_1$ and $\tau_1\tau_0$ to part~\ref{primo_1} of Lemma~\ref{primo}, we get
\[
     x(\ba_2-\ba_{-2})s_{\tz,1}=\gamma_1 x(\ba_2-\ba_{-2})
  +\tfrac{1}{2}\bt(\al-\bt)x(\ba_1+\ba_3-\ba_{-1}-\ba_{-3})
\]
Since $\ba_1-\ba_{-1}=x(\ba_2-\ba_{-2})$, taking the difference of these two equations we get 
\begin{align*}
  0 &=  (\ba_1-\ba_{-1})\bs_{\tz,1}-x(\ba_2-\ba_{-2})\bs_{\tz,1}\\
  &= (\gamma_1^{\bf f}-\gamma_1)(\ba_1-\ba_{-1})+\tfrac{\bt}{2}(\al-\bt)\left ((\tfrac{1}{x}-x)(\ba_1-\ba_{-1})-x(\ba_3-\ba_{-3})\right )\\
  &= (1-\al)(\lmf-\lmu)(\ba_1-\ba_{-1})+\tfrac{\bt}{2}(\al-\bt)\left ((\tfrac{1}{x}-x)(\ba_1-\ba_{-1})-x(\ba_3-\ba_{-3})\right )
\end{align*}
 proving \ref{a1xa2_1}. Part \ref{a1xa2_2} is obtained by applying $\bf f$ to part \ref{a1xa2_1}. 
\end{proof}


\section{Some relations between $\mathfrak{V}_e$ and $\mathfrak{V}_o$}\label{relVeVo}

We can now prove some of the key relations between $\mathfrak{V}_e$ and $\mathfrak{V}_o$ \nomenclature{$\mathfrak{V}_e$, $\mathfrak{V}_o$}{\pageref{relVeVo}} mentioned in the Overview.
\begin{lemma}
In the algebra  $\Vo$ the following equalities hold:
\begin{align}
    P(\ba_{-2}-\ba_4)+(P-R)(\ba_2-\ba_0) &= Q (\ba_3-\ba_{-1}),\label{equa1} \\
    P(\ba_{3}-\ba_{-3})+(P-R^{\bf f})(\ba_{-1}-\ba_1) &= Q^{\bf f}(\ba_{-2}-\ba_2),\label{equa2}  
\end{align}
\begin{equation}
\begin{aligned}\label{equa3}
0 &= (S-S^{\bf f})\bs_{\tz,1}+(T-U)(\bs_{\tz,2}-\bs_{\tu,2})+P(\ba_{-2}-\ba_3)\\
&\phantom{{}={}} \quad +(Q-P)\ba_{-1}+(R-Q^{\bf f})\ba_0+(Q-R^{\bf f})\ba_1+(P-Q^{\bf f})\ba_2.
\end{aligned}    
\end{equation}

\end{lemma}
\begin{proof}
By applying $\bf f$ to parts~\ref{primo_5} and~\ref{primo_6}
of Lemma~\ref{primo} we get, respectively, 
\begin{equation} \label{4.7.1f}
( \alpha-2   \beta)\ba_1\bs_{\tz,2}=H(\ba_{3}+\ba_{-1})+I^{\bf f}(\ba_0+\ba_2)+J^{\bf f}\ba_1+K^{\bf f}\bs_{\tz,1}+L\bs_{\tu,2},
\end{equation}
and 
\begin{equation}\label{D}
4(\alpha-2\beta)\bs_{\tz,1} \cdot \bs_{\tz,1}=P(\ba_{-1}+\ba_3)+Q^{\bf f}(\ba_{0}+\ba_2)+R^{\bf f}\ba_1+S^{\bf f}\bs_{\tz,1}+
T\bs_{\tu,2}+U\bs_{\tz,2}.
\end{equation}
Equation~\eqref{equa1} now follows by taking the difference between Lemma~\ref{primo}\ref{primo_6} and its image under $\tau_1$. Equation~\eqref{equa2} is obtained by applying $\bf f$ to the previous one.  Equation~\eqref{equa3} follows by taking the difference between Lemma~\ref{primo}\ref{primo_6} and Equation~\eqref{D}. 
  \end{proof}

Set \nomenclature{$A$}{\pageref{A}}\nomenclature{$B$}{\pageref{B}}
\begin{equation}
\begin{aligned}
A &:= 2\Big ((\al-4\bt)(\al^2-4\al\bt+\al+2\bt)\lmu  \\
 &\phantom{{}:={}2(}\  +(3\al^3-12\al^2\bt-\al^2+16\al\bt^2+6\al\bt-8\bt^2)\lmf \\
 &\phantom{{}:={}2(}\ -(3\al^4-14\al^3\bt+18\al^2\bt^2+2\al^2\bt-4\al\bt^3-12\al\bt^2+8\bt^3)\Big)(\lmu-\lmf)  \\
 &\phantom{{}:={}} \ -\al\bt(\al-\bt)(\al-4\bt)(\lmd-\lmdf) 
 \end{aligned}     
 \label{A} 
\end{equation}
%
and
\begin{equation}
\begin{aligned}
B &:= \Big(2(\al^2-2\al\bt-\al+4\bt)(-3\al^2+4\al\bt+\al-2\bt)\lmu^2  \\
&\phantom{{}:={}\Big(} \quad -2\al(\al-1)(\al^2-2\al\bt-\al+4\bt)\lmu\lmf \\
&\phantom{{}:={}\Big(} \quad +(6\al^5-24\al^4\bt-6\al^4+24\al^3\bt^2+36\al^3\bt+8\al^2\bt^3-60\al^2\bt^2 \\
&\phantom{{}:={}\Big(+(} \quad -4\al^2\bt+16\al\bt^3+20\al\bt^2-16\bt^3)\lmu \\
&\phantom{{}:={}\Big(} \quad  +2\bt(3\al^4-16\al^3\bt+24\al^2\bt^2-\al^3+8\al^2\bt-24\al\bt^2-2\al\bt+8\bt^2)\lmf\\
&\phantom{{}:={}\Big(} \quad +\al\bt(\al-\bt)(\al^2-2\al\bt-\al+4\bt)\lmd \\
&\phantom{{}:={}\Big(} \quad  -\bt(6\al^5-28\al^4\bt+32\al^3\bt^2+8\al^2\bt^3-2\al^4+17\al^3\bt-42\al^2\bt^2 \\
&\phantom{{}:={}\Big(-\bt(} \quad  -5\al^2\bt+22\al\bt^2-8\bt^3)\Big)(\lmu-\lmf) \\
&\phantom{{}:={}} \quad +\bt^2(\al-\bt)(\al-4\bt)(2\al^2-2\al\bt-\al+2\bt)(\lmdf-\lmd). \label{B} 
\end{aligned}      
\end{equation}
%

\begin{lemma}\label{s1s2}
Let  $\al\not \in \{2\bt, 4\bt\}$. With the above notation, in the algebra $\Vo$ the following relations hold:
 \begin{equation}\label{s1s2eq}
A(\ba_{-2}-\ba_2)=\tfrac{2}{\bt(\al-\bt)}B(\ba_{-1}-\ba_1)
 \end{equation}
 and 
  \begin{equation}\label{s1s2eq2}
 A^{\bf f}(\ba_{3}-\ba_{-1})=\tfrac{2}{\bt(\al-\bt)}B^{\bf f}(\ba_{2}-\ba_0).
 \end{equation}
\end{lemma}
\begin{proof}
Let $\al\not \in \{2\bt, 4\bt\}$. By~\cite[Proposition 6.10]{FMS3}, $\Vo$ is linearly spanned by the set 
$$\mathfrak B:=\{\ba_{-2}, \ba_{-1}, \ba_0, \ba_1, \ba_2, \bs_{\tz,1}, \bs_{\tz,2},  \bs_{\tu,2}\}.$$  
By~\cite[Remark 6.11]{FMS3}, the structure constants of $\Vo$ relative to the set $\mathfrak B$ have been computed using {\sc Singular} in~\cite{codenightmare}. In particular we get 
$$
\bs_{0,1}\bs_{1,2}-(\bs_{0,1}\bs_{1,2})^{\tau_0}=-\tfrac{A}{2(\al-4\bt)^2}(\ba_{2}-\ba_{-2})+\tfrac{B}{\bt(\al-\bt)(\al-4\bt)^2}(\ba_{-1}-\ba_1).
$$
Equation~\eqref{s1s2eq} now follows, since both $\bs_{0,1}$ and $\bs_{0,2}$ are $\tau_0$-invariant.  Equation~\eqref{s1s2eq2} follows by applying $\bf f$ to Equation~\eqref{s1s2eq}.   
\end{proof}

Let 
\begin{align*}
    C&:= 2(4\bt-1)\lmu\lmf+(16\bt-1)\lmu^2-2\bt(4\bt-1)\lmf\\ 
    &\phantom{{}={}} \qquad -4\bt(13\bt-1)\lmu  -3\bt^2\lmd+3\bt^2(12\bt-1). 
\end{align*} 

\begin{lemma}\label{l7.3}\cite[Lemma~7.3]{FMS3}
Let $\al=4\bt$, then in the algebra $\Vo$ the following relations hold:
\begin{equation}
\label{eqs2-4bt}
\tfrac{1}{2\bt}Q^{\bf f}(\ba_{-2} -\ba_2)=4C^{\bf f}(\ba_{-1}-\ba_{1})
\end{equation}
and 
\begin{equation}
\label{eqs3-4bt}
\tfrac{1}{2\bt}Q(\ba_{3} -\ba_{-1})=4C(\ba_2-\ba_{0}).
\end{equation}
\end{lemma}

\section{Evaluation in $\V$}\label{evaluation}

Let $\V:=(V,\{a_0,a_1\})$ be a $2$-generated $\Mab$-axial algebra over $\F$.  Since, by Equation~\eqref{uni} on page~\pageref{uni}, $V$ is a homomorphic image of $\Vo$ and $\F$ is a homomorphic image of $\tilde{\F}$, Equations~\eqref{gammai}-\eqref{eqs3-4bt} still hold in the algebra $V$, once $\ba_i$ is replaced by $a_i$, $\bs_{\bar{r},j}$ by $s_{\bar{r},j}$, and every coefficient in $\tilde{\F}$ by its image under $\varphi$. 

It is important to note that the map
\[
\begin{aligned}
\F & \longrightarrow \F\\
 Z^\varphi &\longmapsto (Z^{\bf f})^\varphi
\end{aligned}
\]
where $Z \in \tilde{\F}$, is not necessarily an automorphism of $\F$.\footnote{For example, take the algebra $A = 3\C(\al)$.  If $\al \neq -1$, then $A$ has an an identity $\1$.  Set $X = \{ a_0, \1-a_1, \1-a_2\}$.  One can check that $\lla a_0, \1-a_1 \rra = A$ and so $A$ is a $2$-generated (non-symmetric) axial algebra of Monster type $\mathcal{M}(\al, 1-\al)$ (see Table~\ref{table3Cskew}).  In this algebra, $\lmu = \lm_{a_0}(\1-a_1) = 1-\frac{\al}{2}$ and $\lmf = \lm_{\1-a_1}(a_0) = \frac{1}{2} + \frac{\al}{2}$, which are not equal as $\al \neq \frac{1}{2}$.}

In order to simplify notation, when it is clear from the context, we shall omit the superscript $\varphi$ and use the same symbol $Z$ to denote both an element  $Z$ in $\tilde{ \F}$ and its image $Z^\varphi$ in $\F$. However, to avoid confusion,  given\nomenclature{$I^f$}{\pageref{Zf}}
\[
Z\in \{\gamma_i,H,I,J,K,L,P,Q,R,S,T,U, A,B, C\},\nomenclature{$J^f$}{\pageref{Zf}} 
\]
 we shall denote $(Z^{\bf f})^\varphi$ by $Z^f$
\nomenclature{$Z^f$}{\pageref{Zf}}\label{Zf} instead of $Z^{\bf f}$. \nomenclature{$B^f$}{\pageref{Zf}}  
We again stress that although ${\bf f}$ is a semi-automorphism of $\tilde{\F}$, the map $f$ is not an automorphism of $\F$.\nomenclature{$Q^f$}{\pageref{Zf}}
  Note also that, by Equation~\eqref{lambdai} on page~\pageref{lambdai}, this notation for $\lmf$ and $\lmdf$ is consistent with the definition given in Equation~~\eqref{deflmlmf} on page~\pageref{deflmlmf}.
\nomenclature{$R^f$}{\pageref{Zf}}

If the coefficients $Q$, $Q^f$, $A$, and $A^f$,\nomenclature{$A^f$}{\pageref{Zf}}
\nomenclature{$S^f$}{\pageref{Zf}} defined on page~\pageref{primo} and page~\pageref{A}, are non zero, then Equations~\eqref{equa1}, \eqref{equa2}, \eqref{s1s2eq}, \eqref{s1s2eq2}, \eqref{eqs2-4bt}, or \eqref{eqs3-4bt} generally give relations between the even and the odd subalgebras $\V_e$ and $\V_o$ as mentioned in the Overview. On the other hand, if the coefficients of those equations are zero, one obtains algebraic relations on $\al$, $\bt$, $\lmu$, $\lmf$, $\lmd$, and $\lmdf$. Both situations occur in the following lemmas. 

\begin{lemma}\label{Ve*Vo*}
Assume $a_0\neq a_2$ and $a_{-1}\neq a_1$. Then
\begin{enumerate}
    \item  if $\al\not \in \{2\bt, 4\bt\}$ and $B\neq 0$, or if $\al=4\bt$ and $C^f\neq 0$, then  $V_o^\ast=V_e^{\ast\ast}$;
    \item  if $\al\not \in \{2\bt, 4\bt\}$ and $B^f\neq 0$, or if $\al=4\bt$ and $C\neq 0$, then $V_e^\ast=V_o^{\ast\ast}$.\label{Ve*Vo*_2}
\end{enumerate}
\end{lemma}
\begin{proof}
 This follows immediately from Lemmas~\ref{s1s2},~\ref{l7.3}, and~\ref{lem:VeVo}.   
\end{proof}

\begin{lemma}\label{newQ-Q^f}
With the above notation, the following hold in the field $\F$.
\begin{enumerate}
    \item \label{newQ-Q^f_1} If $Q-Q^f=0$, then either $\lm_1=\lmf$ or $\al\neq 2$ and $\bt=\frac{\al(\al-1)}{2(\al-2)}$.
    \item \label{eqQ=0} If $Q=0$, then 
\[
    \lm_1=\frac{2(\al-1)(\al-4\bt)}{4\al\bt}\lmf+\frac{\bt(10\al\bt-4\bt^2+\al-6\bt)}{4\al\bt}.
\]
\item \label{newQ-Q^f_3}
$\begin{aligned}[t]
R -R^f &= 4(-3\al^2+4\al\bt+\al-2\bt)(\lmu+\lmf)(\lmu-\lmf) \\
 &\phantom{{}={}} \quad +4(\al^3+3\al^2\bt-6\al\bt^2-2\al\bt+4\bt^2)(\lm_1-\lmf) \\
&\phantom{{}={}} \quad +2\al\bt(\al-\bt)(\lmd-\lmdf).
\end{aligned}$
\end{enumerate}
\end{lemma}
\begin{proof}
Parts \ref{eqQ=0} and \ref{newQ-Q^f_3} are immediate by the definitions of $Q$ and $R$. Part~\ref{newQ-Q^f_1} follows since  we have 
\[
Q-Q^f=2(\al-\bt)(\al^2-2\al\bt-\al+4\bt)(\lm_1-\lmf).\qedhere
\]
\end{proof}

\begin{lemma}\label{S}
With the above notation, assume $\bt=\tfrac{1}{2}$. Then, either $\lmu=\lmf$ or $s_{\tz,1}\in \langle V_e, V_o\rangle$.   
\end{lemma}
\begin{proof}
Assume $\bt=\tfrac{1}{2}$ and $\lmu\neq \lmf$. Then, by the  definition of $S$, $$S-S^f=4\al^2(\lm-\lmf)\neq 0.$$ Equation~\eqref{equa3} then  implies 
\[
s_{\tz, 1}\in \langle a_{-2}, a_{-1}, a_0, a_1, a_2, a_3, s_{\tz,2}, s_{\tu,2}\rangle \leq \langle V_e, V_o\rangle. \qedhere
\]
\end{proof}

\begin{lemma}\label{supernew}
With the above notation, in the algebra $\V$ the followings hold.
  \begin{enumerate}
    \item  Assume $\adim(V_e)\geq 4$. \label{supernew_1}
    \begin{enumerate}
        \item  If $\al\not \in \{2\bt, 4\bt\}$, then either $a_3=a_{-1}$ or $A^f=B^f=0$.
        \item  If $\al=4\bt$, then either $a_3=a_{-1}$ or $Q=C=0$, whence 
        $$\lmu=\tfrac{18\bt-1}{8} \quad \mbox{ and } \quad \lmd=\frac{40\bt^2-14\bt+1}{12\bt^2}\lmf-\frac{(10\bt-1)^2}{192\bt^2}.$$
    \end{enumerate}
    \item  Assume $\adim(V_o)\geq 4$. \label{supernew_2}
    \begin{enumerate}
        \item  If $\al\not \in \{2\bt, 4\bt\}$, then either $a_2=a_{-2}$ or $A=B=0$.
        \item  If $\al=4\bt$, then either $a_2=a_{-2}$ or $Q^f=C^f=0$, whence
        $$\lmf=\tfrac{18\bt-1}{8}\quad \mbox{ and }\quad \lmdf=\frac{40\bt^2-14\bt+1}{12\bt^2}\lmu-\frac{(10\bt-1)^2}{192\bt^2}.$$
    \end{enumerate}
  \end{enumerate}
\end{lemma}
\begin{proof}
 Assume $\adim(V_e)\geq 4$. Since $\V_e$ is symmetric, by Lemma~\ref{lemma:adim}, the vectors $a_{-2}$, $a_0$, $a_2$, and $a_4$ are linearly independent. Thus, as $\bt(\al-\bt)\neq 0$, Corollary~\ref{a1xa2}\ref{a1xa2_2} implies that $a_0-a_2$ is not a non-trivial multiple of $a_{-1}-a_3$. Therefore, if $\al\not \in \{2\bt, 4\bt\}$, then Lemma~\ref{s1s2}  yields that
 \[
 \mbox{either}\quad a_3=a_{-1},\quad \mbox{or} \quad A^f=B^f=0.
 \]
 If $\al=4\bt$, then  Lemma~\ref{l7.3} gives that 
 \[
 \mbox{either}\quad a_3=a_{-1},\quad \mbox{or}\quad Q=C=0.
 \]
This proves \ref{supernew_1}. The proof of \ref{supernew_2} is similar.
\end{proof}


\begin{proposition}\label{reg}
 Assume $\lmu=\lmf$, $\lmd=\lmdf$. Suppose further that either 
 \begin{enumerate}
     \item  $(\al,\bt)\neq (2, \tfrac{1}{2})$ or \label{reg_1}
     \item  $\lm_3=\lm_3^f$ and $V$ is spanned by the set $a_{-2}, a_{-1}, a_0, a_1, a_2, a_3, s_{\tz,1}, s_{\tz,2}$.\label{reg_2}
 \end{enumerate}
 Then $\V$ is isomorphic to a quotient of a symmetric algebra.
\end{proposition} 
\begin{proof}
 If $\al\neq 4\bt$, then the result follows by~\cite[Proposition~6.15(ii)]{FMS3}. Assume $\al=4\bt$ and $\al\neq 2$. If $\lmu=\lmf\neq \tfrac{18\bt-1}{2}$, then by~\cite[Theorem~7.12]{FMS3},  $V=\langle a_{-1}, a_0, a_1, a_2, s_{\tz,1}, s_{\tz,2}\rangle$. If $\lmu=\lmf=\tfrac{18\bt-1}{2}$, then by the proof of Claim 5 in~\cite[Theorem~7.12]{FMS3} we have that either $V=\langle a_{-1}, a_0, a_1, a_2, s_{\tz,1}\rangle$ or $\lm_3=\lm_3^f$ and $V=\langle a_{-2}, a_{-1}, a_0, a_1, a_2, a_3, s_{\tz,1}, s_{\tz,2}\rangle$. With an argument similar to that used in the proof of~\cite[Proposition~6.15]{FMS3}, in all cases the hypotheses imply that $\ker(\varphi|_{\tilde \F})$ is invariant under the flip $f$. 
Hence the result follows from~\cite[Corollary~5.10]{FMS3}.   
\end{proof}

\begin{lemma}
    \label{lemma:A+Af}
Assume $A=A^f=B=B^f=0$, then $\lmu=\lmf$. Moreover, if also $\al\neq 4\bt$, then $\lmd=\lmdf$ and $\V$ is isomorphic to a quotient of a symmetric algebra. 
    \end{lemma}
\begin{proof}
By Equation~\eqref{A} we get 
     $$
0=A+A^f=-4\al(\lmu-\lmf)^2(\al^2-2\al\bt-\al+4\bt),   
    $$
    whence either $\lmu=\lmf$
    or $\al\neq 2$ and $\bt=\frac{\al(\al-1)}{2(\al-2)}$. In the latter case, by Equation~\eqref{B}, we get 
$$
0= B+B^f=-\tfrac{4\al^4(\al-1)}{(\al-2)^3}
(\lmu-\lmf)^2
    $$
    whence, again, $\lmu=\lmf$. 
 Thus, in both cases, $\lmu=\lmf$, whence
 $$
 0=A=-\al\bt(\al-\bt)(\al-4\bt)(\lmd-\lmdf).
 $$
If $\al\neq 4\bt$, then $\lmdf=\lmd$ and the result follows by Proposition~\ref{reg}.
\end{proof}
\bigskip

\chapter{Known algebras}\label{known}

In this chapter we describe the algebras appearing in the Main Theorem. Moreover, we obtain some properties of the symmetric algebras that are needed in the proof of the Main Theorem. More information on these algebras can be found in~\cite{HWQ, HW, HRS1,  MM, split, axet}. As throughout this paper, $\F$ is a field of characteristic other than $2$ and $\al$ and $\bt$ are distinct elements in $\F\setminus\{0,1\}$.

\section{The $2$-generated axial algebras of Jordan type}\label{Jordan}

The $2$-generated algebras of Jordan type are, up to isomorphism, the algebras $1\A$, $2\B$, $3\C(\eta)$, its quotient $3\C(-1)^\times$ for $\eta=-1$, $\J(\delta)$ and its quotient $\J(0)^\times$ for $\delta=0$ (see~\cite{HRS1}).  There are listed in Table~\ref{table2} in Section~\ref{thetables}, where  a basis, the structure constants and the 
relevant values of the Frobenius form are given (recall that all algebras considered in this paper are commutative  and the bilinear form $(\, ,\,)$ is always symmetric).  Note that, except for the cases $1\A$ and $2\B$, each isomorphism class of $2$-generated $\mathcal J(\eta)$-axial algebras splits into two distinct isomorphism classes of $\Mab$-axial algebras, according whether $\eta=\al$ or $\eta=\bt$. All these algebras are symmetric and afford an invariant Frobenius form.

In the next lemmas we describe some features of the   algebras of Table~\ref{table2}, when considered as $\Mab$-axial algebras with $\beta=\eta$.  Note that in this case the $\bt$-eigenspaces of the adjoint maps $\ad_{a_0}$ and $\ad_{a_1}$ are non-trivial, whence the  Miyamoto groups of the algebras are non-trivial.

\begin{lemma}\label{sub2B}
Let $\V=2\B$. The following assertions hold:
\begin{enumerate}
    \item  $\lmu=0$; \label{sub2B_1}
\item  $a_2=a_0$; \label{sub2B_2}
\item  $V=V^{\ast}$; \label{sub2B_3}
\item  $V^{\ast\ast}=\{0\}$.  \label{sub2B_4}
\end{enumerate}
\end{lemma}
\begin{proof}
As $a_1$ is a $0$-eigenvector of $\ad_{a_0}$, \ref{sub2B_1} is true. Notice $\ad_{a_0}$ and $\ad_{a_1}$ only have eigenvalues of $1$ and $0$; thus $\tau_0$ and $\tau_1$ are trivial and \ref{sub2B_2} follows. Part \ref{sub2B_4} follows immediately from \ref{sub2B_2}. 

For \ref{sub2B_3}, notice that $(a_0-a_1)^2=a_0+a_1$ and so $a_0,a_1\in V^{\ast}$.
\end{proof}

\begin{lemma}\label{sub3C}
Let $\V\in \{3\C(\bt), 3\C(-1)^\times\}$.  The  following assertions hold:
\begin{enumerate}
\item  $\lmu=\tfrac{\bt}{2}$; \label{sub3C_1}
\item $a_2=a_{-1}$; \label{sub3C_2}
\item  $V^\ast=V^{\ast\ast}$; \label{sub3C_3}
\item  $V=V^\ast$ if $\bt \neq 2$, and $V^\ast=\langle a_0-a_1, a_0-a_2\rangle$ if $\bt=2$. \label{sub3C_4}
\end{enumerate}
In particular, if $\V=3\C(-1)^\times$ and $\ch \F = 3$, then $\bt=-1 = 2$ and $V^\ast= \langle a_0-a_1\rangle$.
\end{lemma}
\begin{proof}
Let 
\begin{equation}\label{uv3C}
 u:=\bt a_1-(a_0+a_{-1}) \quad \mbox{ and }\quad  w:=a_0-a_{-1}   
\end{equation}
Using the basis and the multiplication given in Table~\ref{table2}, we see that $u$ (respectively $v$) is a $0$-eigenvector (respectively $\tfrac{1}{2}$-eigenvector) for $\ad_{a_1}$, and 
\[
a_0=\tfrac{\bt}{2}a_1-\tfrac{1}{2}u+\tfrac{1}{2}v.
\]
By the uniqueness of the above decomposition, it follows that $\lmf=\tfrac{\bt}{2}$. Since the algebra is symmetric, $\lmu=\lmf=\tfrac{\bt}{2}$. By the definitions of $a_2$,  $\tau_1$, and Equation~\eqref{uv3C},
$$
a_2=a_0^{\tau_1}=\tfrac{\bt}{2}a_1-\tfrac{1}{2}u-\tfrac{1}{2}v=a_{-1},
$$
proving \ref{sub3C_2}. Hence $
a_{-2}=a_2^{\tau_0}=a_{-1}^{\tau_0}=a_1$ and 
$
a_{2}-a_{-2}=a_{-1}-a_1$.
This implies $V^{\ast\ast}=V^\ast$. 

By Table~\ref{table2}, $(a_0-a_1)^2=(1-\bt)(a_0+a_1)+\bt a_{-1}$. If $\bt\neq 2$, then the three vectors $a_0-a_1$, $a_{-1}-a_0$, and $(a_0-a_1)^2$ are linearly independent. Thus $V^\ast$ has the same dimension as $V$, giving $V=V^\ast$. If $\bt=2$, then by the multiplication table, $\langle a_0-a_1, a_0-a_{-1}\rangle$ is a subalgebra of $V$, and by \ref{sub3C_2} it is invariant under $\tau_0$ and the flip $f$. Thus, by Lemma~\ref{translation}, it is invariant under the Miyamoto group of $V$, whence $V^\ast=\langle a_0-a_1, a_0-a_{-1}\rangle$. 

When $\V=3\C(-1)^\times$, the result follows since $3\C(-1)^\times$ is the quotient of $3\C(-1)$ over the ideal $\F(a_0+a_1+a_2)$.
\end{proof}

\begin{lemma}\label{subJ}
 Let $\V\in \{ \J(\delta), \J(0)^\times\}$. Then the following assertions hold:
 \begin{enumerate}
     \item  $\lmu=2\delta+1$; \label{subJ_1}
     \item for every $i\in \Z$, $a_{i-1}=-a_{i+1}+(2+8\delta) a_i-4s_{0,1}$; \label{subJ_2}
    \item  $V=V^\ast$, if $\delta\neq 0$, and  $V^\ast= \langle  a_0-a_1, s_{\tz,1}\rangle$, if $\delta= 0$; \label{subJ_3}
    \item  $V^\ast=V^{\ast\ast}$, if $\delta\neq -\tfrac{1}{2}$, and  $V^{\ast\ast}=\langle a_2-a_0\rangle=\langle a_0+a_1+2s_{\tz , 1}\rangle$, if $\delta= -\tfrac{1}{2}$.  \label{subJ_4}
 \end{enumerate}
 In particular, if $V=\J(0)^\times$, then the flip acts on $V^{\ast\ast}$ as the multiplication by $-1$.
\end{lemma}
\begin{proof}
Let 
\begin{equation}\label{eig}
u:=-\delta a_0+s_{0,1}\quad  \mbox{ and } \quad w:=-(1+4\delta )a_0+a_1+2s_{0,1}.
\end{equation}
Using the basis and the multiplication table given in Table~\ref{table2}, we see that $u$ (respectively $w$) is a $0$-eigenvector (respectively $\tfrac{1}{2}$-eigenvector) for $\ad_{a_0}$, and
\begin{equation*}
a_1=(1+2\delta)a_0-2u+w.
\end{equation*}
By the uniqueness of the above decomposition, it follows that $\lmu=1+2\delta$. By the definitions of $a_{-1}$ and $\tau_0$, and Equation~\eqref{eig},
$$
a_{-1}=a_1^{\tau_0}=(1+2\delta)a_0-2u-w=-a_1+(2+8\delta) a_0-4s_{0,1}.
$$ 
Since $V$ is symmetric, part \ref{subJ_2} follows by Lemma~\ref{translation}.

By the multiplication table, we get $(a_0-a_1)^2=-2s_{0,1}$. Assume first  $\delta\neq 0$,  then the three vectors $a_0-a_1$, $a_{-1}-a_0$, and $(a_0-a_1)^2$ are linearly independent. Thus $V^\ast$ has the same dimension as $V$, giving $V=V^\ast$. Assume now  $\delta=0$ and let $U^\ast:=\langle  a_0-a_1, s_{0,1}\rangle$. Then, by the multiplication table,  $U$ is a subalgebra of $V$ and, since, by \ref{subJ_2}, $s_{0,1}=\tfrac{1}{4}(a_{0}-a_1+a_0-a_{-1})$, it is contained in $V^\ast$. Again by \ref{subJ_2}, 
$$
(a_0-a_1)^{\tau_0}=a_0-a_{-1}=a_1-a_0-4s_{0,1}\in U^\ast,
$$
whence, by Lemma~\ref{invariant}\ref{invariant_1},
\begin{equation}\label{Utauo}
(U^\ast)^{\tau_0}=U^\ast.
\end{equation}
 Furthermore, since the flip $f$ negates $a_0-a_1$ and fixes $s_{0,1}$,  
 \begin{equation}\label{Uflip}
 (U^\ast)^f=U^\ast. 
 \end{equation} 
 By Equations~\eqref{Utauo} and~\eqref{Uflip}, it follows that $(U^\ast)^{\tau_0 f}=U^\ast$. Thus, by Lemma~\ref{translation}, $U^\ast$ contains $a_i-a_{i-1}$ for every $i\in \Z$, whence $V^\ast= U^\ast$, proving \ref{subJ_3}.

 By \ref{subJ_2}, for $i\in\{0,1\}$, we get 
\[
a_2=-a_0+(2+8\delta)a_1-4s_{0,1} \quad \mbox{and} \quad a_{-1}=-a_1+(2+8\delta)a_0-4s_{0,1}.
\]
Substituting $a_2$ and $a_{-1}$ by the above values, we get 
$$
a_2-a_0+a_1-a_{-1}=(4+8\delta)(a_1-a_0)
.$$
Thus, if $\delta\neq -\tfrac{1}{2}$, then $a_0-a_1$ is contained in $V^{\ast\ast}$, whence $V^\ast=V^{\ast\ast}$. Assume $\delta=-\tfrac{1}{2}$ and let $U^{\ast\ast}:=\langle a_2-a_0\rangle$. By \ref{subJ_2}, for every $i\in \Z$, we have $a_{i+1}=-a_{i-1}-2a_i-4s_{0,1}$, whence
\begin{align*}
(a_{i}+a_{i-1}+2s_{0,1})^{\tau_0 f} &= a_{i+1}+ a_i+2s_{0,1}\\
&=(-a_{i-1}-2a_i-4s_{0,1})+a_i+2s_{0,1}\\
&=-(a_{i}+a_{i-1}+2s_{0,1}).   
\end{align*}
By applying ${\tau_0 f}$ twice, we get 
\[
a_{i+1}+ a_i+2s_{0,1}=a_{i-1}+a_{i-2}+2s_{0,1}, \quad \mbox{ for every }  i\in \Z,
\]
whence
\[
  a_{i+1}-a_{i-1}=a_{i-2}-a_i, \quad \mbox{ for every }  i\in \Z.
\]
So, $U^{\ast \ast}$ is invariant under the Miyamoto group of $\V$. By \ref{subJ_2}, 
$$
a_2-a_0=-a_0-2a_1-4s_{0,1}-a_0=-2(a_0+a_1+2s_{0,1}),
$$
whence, by the multiplication table,
$(a_2-a_0)^2=0$. It follows that $U^{\ast\ast}$ is a subalgebra of $V$ invariant under the Miyamoto group and on which the flip acts trivially, whence it coincides with $V^{\ast\ast}$, 
giving \ref{subJ_4}.
Finally, since $\J(0)^\times=\J(0)/\langle s_{0,1}\rangle$, the result for $\J(0)^\times$  follows.  
\end{proof}

\section{Non Jordan type $2$-generated symmetric $\Mab$-axial algebras}
\label{sec:symmetric}

The $2$-generated symmetric algebras of Monster type, which are not of Jordan type, are, up to isomorphism, 
\begin{description}
\item [Table~\ref{table3A}] $3\A (\al,\bt )$ and its quotient $3\A(\al, \tfrac{1-3\al^2}{3\al-1})^\times$ for $\bt=\tfrac{1-3\al^2}{3\al-1}$;
\item [Table~\ref{table4A}] $4\A(\tfrac{1}{4},\bt)$ and its quotient $4\A(\tfrac{1}{4},\tfrac{1}{2})^\times$ for $\bt=\tfrac{1}{2}$; 
\item [Table~\ref{table4B}] $4\B(\al, \tfrac{\al^2}{2} )$ and its quotient $4\B(-1, \tfrac{1}{2} )$ for $\al=-1$; 
\item [Table~\ref{table4J}] $4\J(2\bt ,\bt)$ and its quotient $4\J(-\tfrac{1}{2} ,-\tfrac{1}{4})$ for $\bt=-\tfrac{1}{4}$;
\item [Table~\ref{table4Y}] $4\Y(\tfrac{1}{2},\bt)$; 
\item [Table~\ref{table4Yal}] $4\Y(\al, \tfrac{1-\al^2}{2} )$; 
\item [Table~\ref{table5A}] $5\A(\al, \tfrac{5\al-1}{8})$; 
\item [Table~\ref{table6A}] $6\A\left (\al,-\tfrac{\al^2}{4(2\al-1)} \right )$ and its three quotients $6\A\left (\al,-\tfrac{\al^2}{4(2\al-1)} \right )^\times$ for $\al\in \{\tfrac{2}{3}, \tfrac{1\pm\sqrt{97}}{24}\}$;
\item [Table~\ref{table6J}] $6\J(2\bt,\bt)$ and its quotient $6\J(-\tfrac{2}{7},-\tfrac{1}{7})$ for $\bt=-\tfrac{1}{7}$;
\item [Table~\ref{table6Y}] $6\Y(\tfrac{1}{2},2)$ and its quotient $6\Y(\tfrac{1}{2},2)^\times$;
\item [Table~\ref{tableIY3bis}] $\IY_3(\al, \tfrac{1}{2} ; \mu)$ and its quotients $\IY_3(\al, \tfrac{1}{2} ; 1)^\times$
and  $\IY_3(-1, \tfrac{1}{2} ; \mu)^\times$
;
\item [Table~\ref{tableIY5}] $\IY_5(\al, \tfrac{1}{2} )$ and its quotient $\IY_5(\al, \tfrac{1}{2} )^\times$;
\item [Table~\ref{tablehatH}] the Highwater algebra $\mathcal H$, its cover in characteristic $5$ $\hatH$ and their quotients.
    \end{description}

    In Tables~\ref{table3A}-\ref{tablehatH} in Section~\ref{thetables}, for each algebra, 
 a basis, the structure constants and the 
relevant values of the Frobenius form are given (for the quotients of $\mathcal H$ and $\hatH$ refer to~\cite{HWQ}).
  
A straightforward computation shows that these are indeed symmetric axial algebras of Monster type. Alternatively one can use the Magma package~\cite{MP}.
Since the bases given for the algebras described in this section differ slightly from those used in~\cite{MP} (see also~\cite{MM}),  we give for each algebra the rule for the change of basis. In the tables appearing in this section the indices $i$, used for the elements $a_i$ of the basis, are chosen in $\Z_n$ (if $n\in \N$) or $\Z$ (if $n=\infty$). This will be  specified in the heading of each table. It  turns out that $n$ is such that the corresponding algebra $\V:=(V, \{a_0,a_{1}\})$ has axet isomorphic to $X(n)$ (see Theorem~\ref{axetthm}) and the notation $a_i$  is consistent with Equation~\eqref{aroi}. 

As in the previous section, in the next lemmas we describe some features of the above  algebras.

\begin{lemma}\label{sub3A}
Let $\V\in \{3\A(\al, \bt), 3\A(\al, \tfrac{1-3\al^2}{3\al-1})^\times\}$. Then, the following assertions hold:
\begin{enumerate}
\item  $\lmu=\tfrac{(3\al^2+3\al\bt-\al-2\bt)}{4(2\al-1)}$;\label{sub3A_1}
\item  $s_{\tz,1}=z+\tfrac{\al-\bt}{2}(a_0+a_1+a_2)$;\label{sub3A_2}
\item  $V^{\ast\ast}=V^\ast$;\label{sub3A_3}
\item   $V=V^\ast$, unless $(3\al^2+3\al\bt-9\al-2\bt+4)(3\al+\bt-2)=0$ and $\V=3\A(\al, \bt)$. In this case  $V^\ast=\langle a_0-a_1, a_0-a_2, (2\bt-1)a_1+s_{\tz,1}\rangle$.\label{sub3A_4}
\end{enumerate}
\end{lemma}
\begin{proof}
Suppose first $\V= 3\A(\al, \bt)$.
Since $\lmu=\tfrac{(a_0, a_1)}{(a_0,a_0)}$, part \ref{sub3A_1} follows by Table~\ref{table3A}. Part \ref{sub3A_2} follows immediately by Equation~\eqref{defs} on page~\pageref{defs} and Table~\ref{table3A}. Since $a_0=a_3$, part \ref{sub3A_3} follows as in the proof of Lemma~\ref{sub3C}.

Set
\[
c:=(a_0-a_1)(a_1-a_2)=(2\bt-1)a_1+s_{\tz,1}= \tfrac{\al-\bt}{2}(a_0+a_2)+\tfrac{\al+3\bt-2}{2}a_1+z
\]
By Table~\ref{table3A}, 
\begin{align*}
   4(2\al-1) c^2&=(\al-\bt)(3\al^3+9\al^2(\bt-1)-2\al\bt^2-7\al\bt+\bt(\bt-1)+4\al)(a_0+a_2)\\
    &\phantom{{}={}} \quad +(3\al^4-6\al^3\bt-7\al^2\bt^2-14\al\bt^3-3\al^3+16\al^2\bt+28\al\bt^2+7\bt^3\\
    &\phantom{{}={}} \quad -4\al^2-29\al\bt-13\bt^2+10\al+12\bt-4)a_1\\
    &\phantom{{}={}} \quad +(9\al^3+5\al^2\bt-20\al\bt^2-22\al^2+13\al\bt+10\bt^2+9\al-8\bt)z
\end{align*}
A direct check shows that the three vectors $a_0-a_1$, $a_1-a_2$, and  $c$ are linearly independent and 
$$
V^\ast=\langle a_0-a_1, a_1-a_2, c, c^2\rangle.
$$
Now part \ref{sub3A_4} follows, since the determinant of the matrix, whose row entries are the coefficients of $a_0-a_1$, $a_1-a_2$, $c$, $c^2$ with respect to the basis given in Table~\ref{table3A}, is  
$$
(\al-1)(3\al^2+3\al\bt-9\al-2\bt+4)(3\al+\bt-2).
$$
 
  Now let $\V=3\A(\al, \tfrac{1-3\al^2}{3\al-1})^\times$, then 
$$
c=\tfrac{6\al^2-\al-1}{6\al-2}(a_0+a_2)- \tfrac{6\al^2+7\al-5}{6\al-2}a_1
$$ 
and the three vectors  $a_0-a_1$, $a_1-a_2$, and  $c$ are linearly independent. Since they are contained in $V^\ast$ and $V$ has dimension $3$, we get $V=V^\ast$, proving part \ref{sub3A_4}.  Parts \ref{sub3A_1}-\ref{sub3A_3} follow since $3\A(\al, \tfrac{1-3\al^2}{3\al-1})^\times$ is a quotient of $3\A(\al, \tfrac{1-3\al^2}{3\al-1})$.
\end{proof}

\begin{lemma}\label{sub4A}
Let $\ch(\F)\neq 3$ and let $\V\in \{4\A(\tfrac{1}{4}, \bt), 4\A(\tfrac{1}{4}, \tfrac{1}{2})^\times\}$. Then, the following assertions hold:
\begin{enumerate}
    \item $\lmu=\bt$;\label{sub4A_1}
    \item  $s_{\tz,1}=e+\tfrac{1-4\bt}{8}(a_{-1}+a_0+a_1+a_2)$;\label{sub4A_2}
    \item  $V=V^\ast=V^{\ast\ast}$.\label{sub4A_3}
\end{enumerate} 
\end{lemma}
\begin{proof}
Since (in both cases) $\lmu=\tfrac{(a_0, a_1)}{(a_0, a_0)}$, by Table~\ref{table4A}, we get \ref{sub4A_1}. Part \ref{sub4A_2} follows immediately by Equation~\eqref{defs} on page~\pageref{defs} and Table~\ref{table4A}. 

 To prove \ref{sub4A_3}, notice that $V^{\ast\ast}=\lla V_e^\ast,V_o^\ast\rra$. By Note~\ref{table4A evensub} in Table~\ref{table4A},  $\V_e\cong \V_o\cong 2\B$, whence, by Lemma \ref{sub2B}\ref{sub2B_3}, $V_e^\ast=V_e$ and $V_o^\ast=V_o$. Hence $V^{\ast\ast}=\lla V_e,V_o\rra=V$. As $V^{\ast\ast}\leq V^\ast$, we get \ref{sub4A_3}. 
  \end{proof}


\begin{lemma}\label{sub4B}
Let $\V\in \{ 4\B(\al, \tfrac{\al^2}{2}), 4\B(-1, \tfrac{1}{2})^\times\}$. Then, the following assertions hold:
\begin{enumerate}
    \item  $\lmu=\tfrac{\al^2}{4}$;\label{sub4B_1}
    \item  $s_{\tz,1}=-\tfrac{\al^2}{4}(a_{-1}+a_0+a_1+a_2-c)$;\label{sub4B_2}
    \item  $V=V^\ast=V^{\ast\ast}$.\label{sub4B_3}
\end{enumerate}
\end{lemma}
\begin{proof}
Assume $\V=4\B(\al, \tfrac{\al^2}{2})$. Since $\lmu=\tfrac{(a_0, a_1)}{(a_0, a_0)}$, by Table~\ref{table4B}, we get \ref{sub4B_1}. Part \ref{sub4B_2} follows immediately by Equation~\eqref{defs} on page~\pageref{defs} and Table~\ref{table4B}. 

To prove \ref{sub4B_3}, note that, by Table~\ref{table4B},
\begin{align*}
(a_{-1}-a_1)^2 &= (1-\al)(a_{-1}+a_1)+\al c,\\
(a_{0}-a_2)^2 &= (1-\al)(a_{0}+a_2)+\al c,
\end{align*}
and
\[
(a_{-1}-a_1)^2(a_{0}-a_2)^2 = \al^2(\al-2)^2c.
\]
The five vectors $a_{-1}-a_1$, $a_0-a_2$, $(a_{-1}-a_1)^2$, $(a_0-a_2)^2$, and $(a_{-1}-a_1)^2(a_0-a_2)^2$ are linearly independent if and only if $ 16(\al-1)^2(\al-2)^2\neq 0$. Since $\al=2$ implies $\bt=2$, it follows that $ 16(\al-1)^2(\al-2)^2\neq 0$, whence $V^{\ast\ast}$ has dimension $5$ and so $V=V^{\ast\ast}$. Since $V^\ast\geq V^{\ast\ast}$, we get \ref{sub4B_3}. If $\V=4\B(-1,\tfrac{1}{2})^\times$, then the proof is similar.
%
\end{proof}

\begin{lemma}\label{sub4Y1/2}
Let $\V=4\Y(\tfrac{1}{2}, \bt)$. Then, the following assertions hold:
\begin{enumerate}
    \item  $\lmu=4\bt^2$;\label{sub4Y1/2_1}
    \item  $s_{\tz,1}=4\bt^2 z-\tfrac{\bt}{2}(a_{-1}+a_0+a_1+a_2)$;\label{sub4Y1/2_2}
    \item  $V=V^\ast=V^{\ast\ast}$.\label{sub4Y1/2_3}
\end{enumerate}
\end{lemma}
\begin{proof}
Since $\lmu=\tfrac{(a_0,a_1)}{(a_0,a_0)}$, part \ref{sub4Y1/2_1} follows by Table~\ref{table4Y}. Part \ref{sub4Y1/2_2} follows by Equation~\eqref{defs} on page~\pageref{defs} and Table~\ref{table4Y}.

By Table~\ref{table4Y}, we have
\begin{align*}
(a_{-1}-a_1)^2 &= a_{-1}+a_1-(4\bt-1)(a_0+a_2)-\tfrac{2}{\bt}(4\bt-1)s_{0,1},\\
(a_{0}-a_2)^2 &= a_{0}+a_2-(4\bt-1)(a_{-1}+a_1)-\tfrac{2}{\bt}(4\bt-1)s_{0,1},
\end{align*}
and
\[
(a_{-1}-a_1)^2\cdot (a_{0}-a_2)^2 = 8\bt(2\bt-1)^2(a_{-1}+a_0+a_1+a_2)+16(2\bt-1)^2s_{0,1}.
\]
It is straightforward to see that the five vectors 
\[
a_{-1}-a_1, a_{0}-a_2, (a_{-1}-a_1)^2, (a_{0}-a_2)^2, (a_{-1}-a_1)^2\cdot (a_{0}-a_2)^2
\]
of $V^{\ast\ast}$ are linearly independent if and only if 
\[
1024\bt^2(2\bt-1)^2\neq 0,
\]
which is always the case, since $\bt\not \in \{0,\tfrac{1}{2}\}=\{0,\al\}$.
Thus 
\[
V=V^{\ast\ast}\leq V^\ast\leq V,
\]
whence \ref{sub4Y1/2_3}. 
\end{proof}

\begin{lemma}\label{sub4Y}
Let $\V=4\Y(\al, \tfrac{1-\al^2}{2})$. Then, the following assertions hold:
\begin{enumerate}
    \item  $\lmu=\tfrac{1}{4}(2-\al)(1+\al)$; \label{sub4Y_1}
    \item  $s_{\tz,1}=\tfrac{(\al+1)^2}{4}c+ \tfrac{\al^2-1}{4}(a_{-1}+a_0+a_1+a_2) $; \label{sub4Y_2}
    \item  $V=V^\ast$; \label{sub4Y_3}
    \item  if $\al\neq 2$, then $V=V^{\ast\ast}$; \label{sub4Y_4}
    \item  if $\al= 2$, then $\ch(\F)\neq 3$ and
    \[
    V^{\ast\ast}=\langle a_{-1}-a_1, a_0-a_2, -3(a_{-1}+a_0+a_1+a_2)+4s_{0,1}\rangle= \la a_{-1}-a_1, a_0-a_2, c\ra.
    \]\label{sub4Y_5}
\end{enumerate}
\end{lemma}
\begin{proof}
Since $\lmu=\tfrac{(a_0,a_1)}{(a_0,a_0)}$ and by Equation~\eqref{defs} on page~\pageref{defs}, parts \ref{sub4Y_1} and \ref{sub4Y_2} follows immediately from Table~\ref{table4Yal}.


By Table~\ref{table4Yal} and part \ref{sub4Y_2},  
\[
(a_0-a_1)^2=\al^2(a_{0}+a_1)-2s_{0,1},
\]
and
\[
(a_0-a_1)( a_1-a_2)= \tfrac{1}{2}(\al-1)a_{-1}-\tfrac{1}{2}(\al^2-1)(a_0+a_2)-\tfrac{1}{2}(2\al^2-\al+1)a_1 +\tfrac{2\al}{\al+1} s_{0,1}.
\]
A direct check shows that, since $\al\not \in \{0, \pm 1\}$, the five vectors  $a_{-1}-a_0$, $a_0-a_1$, $a_2-a_1$, $(a_0-a_1)^2$, and $(a_0-a_1)( a_1-a_2)$ of $V^\ast$ are linearly independent. Hence $V=V^\ast$. 

Similarly,
\begin{align*}
(a_{-1}-a_1)^2 &= a_{-1}+a_1+(\al-1)(a_0+a_2)-\tfrac{4}{\al+1}s_{0,1},\\
(a_{0}-a_2)^2 &= a_{0}+a_2+(\al-1)(a_{-1}+a_1)-\tfrac{4}{\al+1}s_{0,1},
\end{align*}
and 
\[
(a_{-1}-a_1)^2\cdot (a_{0}-a_2)^2=\tfrac{3}{16}(a_{-1}+a_0+a_1+a_2)+s_{0,1}.
\]
It is straightforward to see that the five vectors 
\[
a_{-1}-a_1, a_{0}-a_2, (a_{-1}-a_1)^2, (a_{0}-a_2)^2, (a_{-1}-a_1)^2\cdot (a_{0}-a_2)^2
\]
of $V^{\ast\ast}$ are linearly independent if and only if $
16(\al-2)^2(\al-1)^2\neq 0$.
Thus we get \ref{sub4Y_4}. Finally, if $\al=2$, we get 
\[
(a_{-1}-a_1)^2 (a_{0}-a_2)^2=-3(a_{0}-a_2)^2=-3(a_{-1}-a_1)^2=-3(a_{-1}+a_0+a_1+a_2)+4s_{0,1}
\]
and a direct check shows that \ref{sub4Y_5} holds. 
\end{proof}

\begin{lemma}\label{sub5A}
  Let $\V=5\A(\al, \tfrac{5\al-1}{8})$. Then, the following assertions hold:
  \begin{enumerate}
      \item  $\lmu=\tfrac{3(5\al-1)}{32}$; \label{sub5A_1}
      \item  $s_{\tz,1}=\tfrac{5\al-1}{32}w-(a_{-2}+a_{-1}+a_0+a_1+a_2)$; \label{sub5A_2}
      \item  $V^\ast=V^{\ast\ast}$; \label{sub5A_3}
      \item  if $\ch(\F)\neq 5$ and $\al\neq \tfrac{7}{3}$, then $V=V^\ast$; \label{sub5A_4}
      \item  if $\ch(\F)= 5$ or $\al = \tfrac{7}{3}$, then $V^\ast$ is the radical of the Frobenius form and 
      $V^\ast=\langle a_{-2}-a_{-1}, a_{-1}-a_0, a_0-a_1, a_1-a_2, -\tfrac{5}{4}(a_0+a_1)-2s_{0,1}\rangle $. \label{sub5A_5}
       \end{enumerate} 
\end{lemma}
\begin{proof}
Since $\lmu=\tfrac{(a_0,a_1)}{(a_0,a_0)}$ and by Equation~\eqref{defs} on page~\pageref{defs}, by Table~\ref{table5A} we get \ref{sub5A_1} and \ref{sub5A_2}.  Since $a_0-a_1=a_5-a_1=a_5-a_3+a_3-a_1$, claim \ref{sub5A_3} holds.

By Table~\ref{table5A} and part \ref{sub5A_2}, 
$$
(a_0-a_1)^2=-\tfrac{5}{4}(a_0+a_1)-2s_{0,1}
$$
and 
\[
(a_1-a_0)(a_{-2}-a_2)=-\tfrac{1}{16}(5\al-1)(a_{-2}+a_{-1}+a_0+a_1+a_2)-2s_{0,1}.
\]
It follows that the six vectors 
\[
a_{-2}-a_{-1}, \quad a_{-1}-a_0,\quad  a_0-a_1, \quad a_1-a_2, \quad (a_0-a_1)^2, (a_1-a_0)(a_{-2}-a_2)
\]
are linearly independent unless either $\F$ has characteristic $5$ or $\al=\tfrac{7}{3}$.
In these two cases, the subspace they generate is of dimension $5$, coincides with the radical of the Frobenius form and is therefore equal to $V^\ast$. 
\end{proof}

 \begin{lemma}
     \label{subA6}
 Let $\varepsilon\in\{-1,1\}$ and
 \[
 \V\in \left \{6\A\left (\al, -\tfrac{\al^2}{4(2\al-1)}\right ),  6\A\left (\tfrac{2}{3}, -\tfrac{1}{3}\right )^\times, 6\A\left (\tfrac{1+\varepsilon\sqrt{97}}{24}, \tfrac{53+\varepsilon5\sqrt{97}}{192}\right )^\times \right \}.
 \]
 Then, the following assertions hold:
 \begin{enumerate}
     \item $\lmu=\tfrac{\al^2(2-3\al)}{16(2\al-1)^2}$; \label{sub6A_1}
     \item  $s_{\tz,1}=\tfrac{-\al^2}{8(2\al-1)}((c+z)-(a_{-2}+a_{-1}+a_0+a_1+a_2+a_3))$; \label{sub6A_2}
     \item  $\dim(V^\ast)\geq 6$; \label{sub6A_3}
     \item  $\dim(V^{\ast\ast})\geq 5$. \label{sub6A_4}
 \end{enumerate}
 \end{lemma}
 \begin{proof}
  Assume $\V=6\A\left (\al, -\tfrac{\al^2}{4(2\al-1)}\right )$.  Parts \ref{sub6A_1} and \ref{sub6A_2} follow by Table~\ref{table6A}, similarly as in the previous lemmas. 
By Table~\ref{table6A}, 
\[
(a_0-a_2)^2=- \tfrac{2\al(3\al-1)}{4(2\al-1)}a_{-2}-  \tfrac{1}{2}(\al-2)(a_0+a_2)+ \tfrac{2\al(5\al-2)}{8(2\al-1)}z.
\]
Since the six vectors 
\[
a_{-2}-a_{0},\;\; a_{-1}-a_{1},\;\; a_{0}-a_2,\;\; a_1-a_3,\;\; a_0-a_1,\;\; (a_0-a_2)^2
\]
are linearly independent and all but $a_0-a_1$ lie in $V^{\ast\ast}$, we get \ref{sub6A_3} and \ref{sub6A_4}.
If $\V\in\{ 6\A\left (\tfrac{2}{3}, -\tfrac{1}{3}\right )^\times, \;\; 6\A\left (\tfrac{1+\varepsilon\sqrt{97}}{24}, \tfrac{53+\varepsilon 5\sqrt{97}}{192}\right )^\times\}$, the proof is similar. 
 \end{proof}

\begin{lemma}\label{sub6Y}
Let $\V \in \{ 6\Y(\frac{1}{2}, 2), 6\Y(\frac{1}{2}, 2)^\times \}$.  Then, $\V^{\ast\ast} = \la a_0-a_2, a_0-a_{-2} \ra$.
\end{lemma}
\begin{proof}
First suppose that $\V = 6\Y(\frac{1}{2}, 2)$.  By definition, $\V^{\ast\ast} = \lla a_0-a_2, a_0-a_{-2}, a_1-a_3, a_1 - a_{-1} \rra$, however $a_1-a_3 = (a_{-2}+d) - (a_0+d) = -(a_0-a_{-2})$ and similarly, $a_1-a_{-1} = a_{-2}-a_2 \in \la a_0-a_2, a_0-a_{-2} \ra$.  By \cite[Table 25]{MM}, $\la a_0-a_2, a_0-a_{-2} \ra$ is an ideal of $\V$ and so $\V^{\ast\ast} = \la a_0-a_2, a_0-a_{-2} \ra$, as claimed.  Since $6\Y(\frac{1}{2}, 2)^\times$ is a quotient by $\la z \ra$, the above argument applies here too.

\end{proof}

 \begin{lemma}\label{subIY3}
Let $\V\in \{\IY_3(\al, \tfrac{1}{2}; \mu), \quad  \IY_3(-1, \tfrac{1}{2}; \mu)^\times, \quad \IY_3(\al, \tfrac{1}{2}; 1)^\times\}$. Then, the following assertions hold:
\begin{enumerate}
    \item $\lmu=\tfrac{1}{4}\al(1-\mu)+\tfrac{1}{2}(\mu+1)$; \label{subIY3_1}
    \item  for every $i\in \Z$, $a_{i+2}=a_{i-1}-(2\mu+1)( a_i-a_{i+1})$; \label{subIY3_2}
\item  $ V^{\ast\ast}=V^\ast$ if and only if $\mu\neq  -1$; \label{subIY3_3}
\item  if $\mu=-1$, then $ V^{\ast\ast}=\langle a_{-1}-a_1\rangle $; \label{subIY3_4}
\item   $V=V^\ast$ if and only if $\al\neq 2$ and $\mu\neq 1$; \label{subIY3_5}
\item    if $ \al=2$, then $V^\ast $ has basis $ (a_0-a_{-1}, a_0-a_1, s_{0,1})$; \label{subIY3_6}
\item  if  $\mu=1$, then $V^\ast$ has basis either  $(a_0-a_{-1}, a_1-a_0, s_{\tz,1})$ or $(a_0-a_{-1}, a_1-a_0)$, according whether $\V=\IY_3(\al, \tfrac{1}{2}; 1)$ or  $\V= \IY_3(\al, \tfrac{1}{2}; 1)^\times$. \label{subIY3_7}
   \end{enumerate}
\end{lemma}
\begin{proof}
Since $\lmu=\tfrac{(a_0, a_1)}{(a_0, a_0)}$, part \ref{subIY3_1} follows by Table~\ref{tableIY3bis}.
Suppose first $\V= \IY_3(\al, \tfrac{1}{2}; \mu)$.  With the notation of Table~\ref{tableIY3bis},  let 
 \begin{align*}
u_1 &:= -2(\al-1)a_{-1}+4(2\al-1)\mu a_0+\left (2\al(\al+1)(1-\mu)-2(2\al-1) \right )a_1+8s_{\tz,1},\\
v_1 &:= 2 a_{-1}-4\mu a_0+2(\al\mu-\al+1) a_1+8s_{\tz,1},\\
w_1 &:= a_{-1}-2(\mu+1)a_0+(2\mu+1)a_1.
\end{align*}
Using Table~\ref{tableIY3bis}, we get that $u_1$, $v_1$ and $w_1$, are eigenvectors for $\ad_{a_1}$ relative to the eigenvalues $0$, $\al$, $\tfrac{1}{2}$, respectively. Moreover, %
\[
a_0=-\tfrac{1}{8\al}u_1+\tfrac{1}{2}w_1+\tfrac{1}{8\al}v_1+\left (\tfrac{1}{4}\al(1-\mu)+\tfrac{1}{2}(\mu+1)\right )a_1.
\]
So applying $\tau_1$ to the above, we get
\[
a_2=a_0^{\tau_1}=-\tfrac{1}{8\al}u_1-\tfrac{1}{2}w_1+\tfrac{1}{\al}v_1+\left (\tfrac{1}{4}\al(1-\mu)+\tfrac{1}{2}(\mu+1)\right )a_1=a_{-1}-(2\mu+1)(a_0-a_1).
\]
Since $\V$ is symmetric, part \ref{subIY3_2} follows by Lemma~\ref{translation}. 

By \ref{subIY3_2}, for every $i\in \Z$, the following equality holds:
\begin{align*}
(2\mu+2)(a_i-a_{i+1}) &= a_i-(a_{i-1}-(2\mu+1)( a_i-a_{i+1}))-(a_{i+1}-a_{i-1})\\
&= a_{i}-a_{i+2}-(a_{i+1}-a_{i-1})\in V^{\ast\ast} ,
\end{align*}
giving \ref{subIY3_3}.
Suppose $\mu=-1$. A straightforward computation gives $(a_1-a_{-1})^2=0$. Since by part \ref{subIY3_2}, $a_{i+2}-a_{i}=a_{i-1}-a_{i+1}$ for every $i\in \Z$, 
$\langle a_{-1}-a_1\rangle$ is a subalgebra of $V$ and it is invariant under the Miyamoto group, whence \ref{subIY3_4} follows.

Let
$
U^\ast:=\lla a_0-a_{-1}, a_0-a_1 \rra .
$
Clearly $U^\ast \subseteq V^\ast$. We claim that, for every $i\in \Z$, 
\begin{equation}\label{indN}
a_{i+2}-a_{i+1}\in U^\ast ,
\end{equation}
whence $U^\ast=V^\ast$.
Since $\langle a_0-a_{-1}, a_0-a_1\rangle$ is invariant under $\tau_0$, we only need to prove Equation~\eqref{indN} for every $i\in \N$. 
 This follows by induction on \ref{subIY3_1}, since, by \ref{subIY3_2}, 
$$
a_{i+2}-a_{i+1}=-2\mu(a_i-a_{i+1})-(a_i-a_{i-1}).
$$
A direct check shows that the three vectors $a_0-a_{-1}$, $a_0-a_1$, and $s_{\tz,1}$ are linearly independent and
\[
V^\ast=\langle a_0-a_{-1}, a_0-a_1, s_{0,1}, (s_{0,1})^2 \rangle .
\]
On the other hand, since the four vectors $a_0-a_{-1}$, $a_0-a_1$, $s_{0,1}$, and $(s_{0,1})^2$, are linearly independent if and only if
\[
(\mu-1)^2(\al-1)(2\al-1)(\al-2)\neq0,  
\]
 we get \ref{subIY3_5}, 
and
 \[
 \mbox{ if } \al=2 \mbox{ or  } \mu=1, \quad  \mbox{ then } \quad V^\ast \mbox{ has basis } (a_0-a_{-1}, a_0-a_1, s_{0,1}).
 \]
 This gives \ref{subIY3_6} and \ref{subIY3_7} in the case $\V=\IY_3(\al, \tfrac{1}{2}; \mu)$. The proofs in the cases $\V= \IY_3(-1, \tfrac{1}{2}; \mu)^\times$ and $\V= \IY_3(\al, \tfrac{1}{2}; 1)^\times $ are similar. 
\end{proof}

 \begin{lemma}
     \label{subIY5}
    Let $\V\in \{\IY_5(\al, \tfrac{1}{2}), \IY_5(\al, \tfrac{1}{2})^\times \}$. Then, the following assertions hold:
    \begin{enumerate}
        \item  $\lmu=1$; \label{subIY5_1}
        \item  for every $i\in \Z$, $a_{i+5}=a_i+5(a_{i+4}-a_{i+1})-10(a_{i+3}-a_{i+2})$; \label{subIY5_2}
        \item $V^{\ast\ast}=V^\ast=\langle a_{-2}-a_{-1}, a_{-1}-a_0, a_0-a_1, a_1-a_2, s_{0,1}\rangle$. \label{subIY5_3}
    \end{enumerate}
 \end{lemma}
 \begin{proof}
 Part \ref{subIY5_1} follows by Table~\ref{tableIY5}, since $\lmu=\tfrac{(a_0,a_1)}{(a_0,a_0)}=1$.
 
Let $\V=\IY_5(\al, \tfrac{1}{2})$ and let 
\begin{align*}
u &:= \tfrac{2\al-1}{4\al}\left( a_{-2} -5a_{-1} + 10 a_0 -9 a_1 +3 a_2 \right) -s_{0,1},\\
v &:= \tfrac{1}{4\al}(a_{-2}-5a_{-1}+10a_0-9a_1+3a_2)+s_{0,1},\\
w &:= -a_{-2}+5(a_{-1}-a_2)-10a_0+11a_1.
\end{align*}
A direct check shows that $u$, $v$, and $w$ are, respectively, $0$-, $\al$, and $\bt$-eigenvectors for $\ad_{a_2}$ and 
$$
a_1=a_2+u+v+\tfrac{1}{2}w.
$$
By the definition of $\tau_2$, we get
$$
a_3=a_1^{\tau_2}=a_2+u+v-\tfrac{1}{2}w=a_{-2}+5(a_2-a_{-1})-10(a_1-a_0).
$$
Since $V$ is symmetric, \ref{subIY5_2} follows by Lemma~\ref{translation}.

By Table~\ref{tableIY5}, a direct check shows that $V^\ast$ is equal to the radical of the form $\langle a_{-2}-a_{-1}, a_{-1}-a_0, a_0-a_1, a_1-a_2, s_{0,1}\rangle$. Moreover, 
\[
(a_0-a_2)^2=\tfrac{1}{2}(a_{-2}+a_2)-2(a_{-1}+a_1)+3a_0-8s_{0,1},
\]
whence the five vectors of $V^{\ast\ast}$, $a_{-2}-a_0$, $a_{-1}-a_1$, $a_0-a_2$, $a_1-a_3$, $(a_0-a_2)^2$ are linearly independent. Hence $\dim(V^{\ast\ast})=\dim(V^\ast)$ and \ref{subIY5_3} follows.

If $\V=\IY_5(\al, \tfrac{1}{2})^\times$, the result follows by the definition of $\IY_5(\al, \tfrac{1}{2})^\times$. 
 \end{proof}


 \begin{lemma}\label{lambdasut}
 Let $\V=(V, \{a_0, a_1\})$ be a symmetric $2$-generated $\mathcal{M}(2,\tfrac{1}{2})$-axial algebra. Then, $\lmu=1$ if and only if either $\V$ is isomorphic to a quotient of $\mathcal H$, or $\F$ has characteristic $5$ and $\V$ is isomorphic to a quotient of $\hatH$. In particular, $\lm_i=1$ for every $i\in \Z$.    
\end{lemma}
\begin{proof}
Assume first that $\V$ is isomorphic to a quotient of $\hatH$. Then, 
by~\cite[Proposition 3.10]{HW} and~Lemma~\ref{quotient}, $\lm_i=1$ for every $i\in \Z$.

Conversely, let $\V=(V, \{a_0, a_1\})$ be a symmetric $2$-generated $\mathcal{M}(2,\tfrac{1}{2})$-axial algebra and suppose $\lmu=1$. By the Classification Theorem of Symmetric algebras, since $(\al,\bt)=(2,\tfrac{1}{2})$, either $\V$ is isomorphic to a quotient of $\hatH$ and we are done, or one of the following holds:
  \begin{enumerate}
      \item $\V$ is isomorphic to a quotient of $3\C(2)$, $\J(\delta)$, $\IY_3(2, \tfrac{1}{2}; \mu)$, or $\IY_5(2, \tfrac{1}{2})$;
      \item $\F$ has characteristic $5$ and $\V$ is isomorphic to $6\A(2, \tfrac{1}{2})$;
      \item $\F$ has characteristic $7$ and $\V$ is isomorphic to a quotient of $4\A(2, \tfrac{1}{2})$.
  \end{enumerate}
  By~\cite[Theorem~11.2]{HW}, all the above algebras are quotients of $\hatH$, except $\J(\delta)$, if $\delta\neq 0$, and $4\A(2, \tfrac{1}{2})$.
  
If $\V\cong \J(\delta)$, then by Lemma~\ref{subJ}\ref{subJ_1}, $\lmu=2\delta+1\neq 1$ provided $\delta\neq 0$. If $\ch(\F)=7$ and $\V\cong 4\A(2, \tfrac{1}{2})$, then by~Lemma~\ref{sub4A}\ref{sub4A_1}, $\lmu=\bt=\tfrac{1}{2}\neq 1$. 
\end{proof}

In view of Lemma~\ref{supernew}, of particular relevance in this paper are the quotients of $\hatH$ with axial dimension at most $3$. Let $L$ be the ideal of $\hat{\mathcal  H}$ generated by $(a_0-a_1-a_2+a_3)$. Set  $I\mathcal {H}_3:=\hatH/L$, whence 
\begin{equation}\label{modL}
    a_3\equiv -a_0+a_1+a_2\quad \bmod (L).
\end{equation}

\begin{lemma}
    \label{subIH3}
    Let  $\ch(\F)\neq 3$ and let $\V\cong I\mathcal H_3$. Then, the following assertions hold:
    \begin{enumerate}
        \item  $\lmu=1$; \label{subIH3_1}
        \item  for every $i\in \Z$, $a_{i+2}=-a_{i-1}+a_i+a_{i+1}$; \label{subIH3_2}
        \item  $V$ has basis $(a_0, a_1, a_2, s_{\tz,1})$; \label{subIH3_3}
        \item  $V^\ast=\langle a_0-a_1, a_1-a_2, s_{\tz,1}\rangle$; \label{subIH3_4}
        \item  $V^{\ast\ast}=\langle a_0-a_2\rangle$; \label{subIH3_5}
        \item  $\V$ has axet $X(\infty)$ if $\ch(\F)=0$, and $X(2p)$ if $\ch(\F)=p$; \label{subIH3_6}
        \item  suppose $\V/I$ is a proper quotient of $\V$, then $\V/I$ is isomorphic to an algebra in the set $\{1\A, \J(0)^\times, 3\C(2)\}$.  \label{subIH3_7}
    \end{enumerate}
\end{lemma}
\begin{proof}
  Part \ref{subIH3_1} follows by Lemma~\ref{lambdasut}.  Since $\V\cong I\mathcal H_3$, by Equation~\eqref{modL},
  \[
  a_3=-a_0+a_1+a_2.
  \]
  By Note~\ref{tablehatH quotients} in Table~\ref{tablehatH},  $I\mathcal H_3$ is symmetric, whence \ref{subIH3_2} follows by Lemma~\ref{translation}. Part \ref{subIH3_3} follows from~\cite[Theorem~9.6]{HWQ}. A direct check gives \ref{subIH3_4} and \ref{subIH3_5}. In order to prove \ref{subIH3_6}, note that, with respect to the basis in \ref{subIH3_3}, $\tau_0\tau_1$ has matrix
   \[
\renewcommand\arraystretch{1.2}
 \begin{pmatrix}
 0 & 0 & 1 & 0\\
 -1 & 1 & 1 & 0\\
 -1 & 0 & 2 & 0\\
 0 & 0 & 0 & 1
 \end{pmatrix}
 , \quad \mbox{whose Jordan form is} \quad 
\begin{pmatrix}
 1 & 1 & 0 & 0\\
 0 & 1 & 0 & 0\\
 0 & 0 & 1 & 0\\
 0 & 0 & 0 & 1
 \end{pmatrix}
\]

  Hence, $|\tau_0\tau_1|=p$, where  $p=\ch(\F)$ and $p=\infty$ if $\ch(\F)=0$, giving \ref{subIH3_6}.

  Let $I$ be a non-trivial ideal of $V$ and let $u=x_0a_0+x_1a_1+x_2a_2+ts_{\tz,1}$ be a non-zero vector in $I$. Suppose first that $x_0=x_1=x_2=0$. Then $s_{\tz,1}\in I$, whence $I$ contains $a_0s_{\ttz, 1}=-\frac{3}{4} a_0 + \frac{3}{8}( a_{-1}+ a_{1}) +\frac{3}{2} s_{\ttr ,1}$. We may therefore assume without loss of generality that $x_0=1$.  Then $I$ contains $u-u^{\tau_1}=(1-x_2)(a_0-a_2)$. If $x_2\neq 1$, then $a_0-a_2\in I$, whence, by~\cite[Lemma~11.3]{HWQ}, $\V/(u)\cong 3\C(2)$ and, since $3\C(2)$ is simple (here  $\ch(\F)\neq 3$), we get \ref{subIH3_7}. Assume $x_2=1$. Since \ref{subIH3_2} implies $a_1-a_{-1}=a_2-a_0$ and $a_2-a_{-2}=2(a_2-a_0)$, we get $u-u^{\tau_0}=x_1(a_1-a_{-1})+x_2(a_2-a_{-2})=(2+x_1)(a_2-a_0)\in I$. As above, \ref{subIH3_7} follows, unless $x_1=-2$. Finally, let $x_1=-2$. If $t=0$, then, by~\cite[Lemma~11.3]{HWQ}, $\V/(u) \cong \J(0)^\times$. If $t\neq 0$, then $I$ contains
  $$
s_{\tz,1}=\tfrac{1}{3t}( us_{\tz,1}+ \tfrac{3}{2}u)
  $$
  and 
  $$
a_0-2a_1+a_2=-\tfrac{8}{3}a_0s_{\tz,1}+4s_{\tz,1},  
  $$
  whence, again, $\V/I $ is isomorphic to a quotient of $\J(0)^\times$.
\end{proof}

\begin{lemma}
    \label{subH}
Let $\ch(\F)\neq 3$ and let $\V$ be a quotient of $\hatH$ with axial dimension $2$ or $3$. Then, one of the following occurs:
 \begin{enumerate}
     \item $\V\cong 3\C(2)$; \label{subH_1}
     \item $\V\cong \J(0)$ or $\J(0)^\times$; \label{subH_2}
     \item $\V\cong \IY_3(2, \tfrac{1}{2}; \mu)$ or $\IY_3(2, \tfrac{1}{2}; 1)^\times$; \label{subH_3}
     \item $\V\cong I\mathcal H_3$. \label{subH_4}
 \end{enumerate}
\end{lemma}
\begin{proof}
Assume first that $\V$ is maximal with respect to being of axial dimension $d\in \{2,3\}$. Then, by~\cite[Theorem~9.5]{HWQ}, $\V=\hatH/(u)$, where $u=\sum_{i=0}^d \al_i a_i$ with $\al_0\neq 0\neq \al_d$, $\sum_{i=0}^d\al_i=0$, and $\al_i=\varepsilon \al_{d-i}$ for every $i\in \{0,\ldots , d\} $, $\varepsilon=\pm 1$. If $d=2$, then, up to scaling, either $u=a_0-a_2$ or $u=a_0-2a_1+a_2$. By~\cite[Lemma~11.3]{HWQ}, in the former case $\V\cong 3\C(2)$, while in the latter case $\V\cong \J(0)^\times$. Let $d=3$. If $\varepsilon=-1$, then, up to scaling, $u=a_0+\delta a_1-
\delta a_2-a_3$, for some $\delta\in \F$. Then,  by~\cite[Lemma~11.4]{HWQ}, $\V\cong \IY_3(2, \tfrac{1}{2}; \mu)$, with $\delta=-2\mu-1$. If $\varepsilon=1$, then up to scaling $u=a_0-a_1-a_2+a_3$ and \ref{subH_4} holds. 

Suppose now that $\V$ is not maximal with respect to being of axial dimension $d$. Then $\V$ is isomorphic to a quotient of the above four algebras. The algebras $3\C(2)$ and $\J(0)^\times$ are simple. The quotients of $\IY_3(2, \tfrac{1}{2}; \mu)$ with axial dimension $2$ or $3$ are isomorphic to one of the following $\J(0)$, $\J(0)^\times$,  $\IY_3(2, \tfrac{1}{2}; 1)^\times$, or $3\C(2)$ if $\mu=-1$ (see~\cite[Table~28]{HWQ}). 
Finally, by Lemma~\ref{subIH3}, the proper quotients of $I\mathcal H_3$ with  axial dimension $2$ or $3$ are isomorphic to either $\J(0)^\times$ or $3\C(2)$.
   \end{proof}

\section{$2$-generated non-symmetric $\Mab$-axial algebras}
\label{sec:ns}

The known $2$-generated non-symmetric $\Mab$-axial algebras are
\begin{description}
    \item [Table~\ref{tableQ2}] $4\QQ(2\bt,\bt)$ and its quotient $4\QQ(-1,-\tfrac{1}{2})^\times$ for $\bt=-\tfrac{1}{2}$;
    \item [Table~\ref{tableQ2bis}] $3\QQ^\prime(\tfrac{1}{3}, \tfrac{2}{3})$; 
    \item [Table~\ref{table3Cskew}] $3\C^\prime(\eta, 1-\eta)$ for $\eta\not \in \{ 0,1,\tfrac{1}{2}\}$;
    \item [Table~\ref{4Bq}] $4\B(-1, \tfrac{1}{2}; \nu)^\times$ for $\nu\neq \tfrac{1}{2}$.
\end{description}
In Tables~\ref{tableQ2}-\ref{4Bq} in Section~\ref{thetables}, for each algebra,  a basis, the structure constants and the relevant values of the Frobenius form are given. A straightforward computation shows that these are indeed axial algebras of Monster type. They all admit a Frobenius form $(\:, \:)$.

\section{The tables}\label{thetables}

\begin{table}[H] \label{table2}
\nomenclature{$1\A$}{\pageref{table2}}
\nomenclature{$2\B$}{\pageref{table2}}
\nomenclature{$3\C(\eta)$}{\pageref{table2}} 
\nomenclature{$3\C(-1)^\times$}{\pageref{table2}}
\nomenclature{$\J(\delta)$}{\pageref{table2}}
\nomenclature{$\J(0)^\times$}{\pageref{table2}}
\nomenclature{$S(8\delta+2)$}{\pageref{table2}}
\nomenclature{$\hat S(2)^\circ$}{\pageref{table2}}
\nomenclature{$S(2)^\circ$}{\pageref{table2}}

\caption{\Large The algebras $1\A$, $2\B$, $3\C(\eta)$ and $\J(\delta)$}

\bigskip

{\sc Other names}
\begin{enumerate}[enum_note]
\item $\J(\delta)$, for $\delta\neq 0$, is $\mathrm{Cl}^J(\F^2,b)$ in~\cite{HRS1}, and $S(8\delta+2)$ in~\cite{axet, survey}
\item $\J(0)$ is $\mathrm{Cl}^{00}(\F^2,b)$ in~\cite{HRS1}, and $\hat S(2)^\circ$ in ~\cite{axet, survey}
\item $\J(0)^\times$ is $\mathrm{Cl}^0(\F^2,b)$ in~\cite{HRS1}, and $S(2)^\circ$ in ~\cite{axet, survey}
\end{enumerate}

\bigskip

{\tiny \makegapedcells
\[
\begin{array}{|c|c|A|A|}
\hline
\text{Type} & \text{Basis} & \multicolumn{2}{l|}{\text{Products}} & \multicolumn{2}{l|}{\text{Form}}\\
\hline
1\A & a_0 & a_0\cdot a_0 &= a_0 & (a_0,a_0) &=1 \\
 \hline
2\B &\begin{aligned} &a_0,\\ &a_1 \end{aligned} & a_0\cdot a_1 &= 0 & (a_0, a_1) &= 0 \\
 \hline
\multirow{5}{*}{$3\C(\eta)$}
&\multirow{5}{*}{$\begin{aligned} &a_0,\\ &a_1,\\ &a^\ast \end{aligned}$}
& a_0\cdot a_1 &=\tfrac{\eta}{2}(a_0+a_1-a^\ast) & (a_0, a_1) &=\tfrac{\eta}{2}\\
&& a_0\cdot a^\ast &=\tfrac{\eta}{2}(a_0+a^\ast-a_1) & (a_0, a^\ast) &=\tfrac{\eta}{2} \\
&& a_1\cdot a^\ast &=\tfrac{\eta}{2}(a_1+a^\ast-a_0) & (a_1, a^\ast) &=\tfrac{\eta}{2} \\
&& a^\ast\cdot a^\ast &=a^\ast & (a^\ast, a^\ast) &=1 \\
\hline
 3\C(-1)^\times & \begin{aligned} &a_0,\\ &a_1 \end{aligned} & a_0\cdot a_1 &=-a_0-a_1 & (a_0,a_1) &=-\frac{1}{2} \\
 \hline
\multirow{5}{*}{$\J(\delta)$} & \multirow{5}{*}{$\begin{aligned} &a_0,\\ &a_1,\\ &s_{\tz,1} \end{aligned}$}
& & & (a_0,a_1) &= 2\delta+1\\
&& a_0\cdot a_1 &= \tfrac{1}{2}(a_0+a_1)+s_{\tz,1} & (a_0,s_{\tz,1})&=\delta\\
&& u\cdot s_{\tz,1} &= \delta  u, \mbox{ for every } u\in \supp(\J(\delta)) & (a_1,s_{\tz,1})&=\delta\\
&&&& (s_{\tz,1},s_{\tz,1}) &= 2\delta^2 \\
 \hline
 \J(0)^\times & \begin{aligned}
 &a_0,\\
 &a_1 \end{aligned}
& a_0\cdot a_1 &= \frac{1}{2}(a_0+a_1)
& (a_0,a_1) &=1 \\
\hline
\end{array} 
\]
}
 \bigskip
 \bigskip

 \begin{center} {\sc Notes } 
 \end{center}
 \bigskip
 
 \begin{enumerate}[enum_note]
    \item  In each one of the above algebras, for every $i\in \{1,0\}$, $a_i a_i=a_i$ and $(a_i, a_i)=1$.
     \item  $3\C(\eta)$ is of Jordan type $\eta$. The basis vector $a^\ast$ is the image of $a_1$ via the automorphism $\sigma_0$ that fixes the $1$- and  $0$-eigenspaces for $\ad_{a_0}$ and negates the $\eta$-eigenspace. In particular, $a^\ast=a_{-1}$ when $3\C(\eta)$ is considered as an $\mathcal{M}(\al,\eta)$-axial algebra.
     \item  $3\C(-1)^\times$ is the quotient of $3\C(-1)$ modulo the ideal $\F(a_0+a_1+a^\ast)$.
     \item  $\J(\delta)$ is of Jordan type $\tfrac{1}{2}$. 
     \item  $\J(0)^\times$ is the quotient of $\J(0)$ modulo the ideal $\F s_{\tz,1}$.
     \item  $2\B$ is isomorphic to the quotient of $\J(-\tfrac{1}{2})$ modulo the ideal $\F(a_0+a_1+2s_{\tz,1})$.
     \item  $3\C(\tfrac{1}{2})\cong \J(-\tfrac{3}{8})$. \label{table2 3C J}
     \item  In $\ch(\F)=3$, $3\C(-1)\cong \J(0)$ and $3\C(-1)^\times\cong \J(0)^\times$.
     \item $1\A$ has axet $X(1)$. $2\B$ has axet $X(2)$. The algebras $3\C(\eta)$ and $3\C(-1)^\times$ considered as $\Mab$-algebras with $\bt=\eta$ or $\bt=-1$ respectively, have axet $X(3)$; with $\al=\eta$ or $\al=-1$ respectively, they have axet $X(2)$. $\J(\delta)$ and $ \J(0)^\times$ have axet depending on the characteristic of $\F$ (see~\cite[\S 5.2]{axet}).
     \item In the setting of the Conway-Norton-Griess algebra and Majorana Theory, the algebra $3\C(\tfrac{1}{4})$ is denoted by $2\A$, the algebra $3\C(\tfrac{1}{32})$ is denoted by $3\C$.
 \end{enumerate}
  
\end{table}

 \begin{table}[H]\label{table3A}
\nomenclature{$3\A(\al,\bt)$}{\pageref{table3A}}
\nomenclature{$3\A^\prime_{\al,\bt}$}{\pageref{table3A}}
\nomenclature{$\mathrm{III}(\al, \bt, 0)$}{\pageref{table3A}}
\nomenclature{$3\A(\al, \tfrac{1-3\al^2}{3\al-1})^\times$}{\pageref{table3A}}

\caption{\Large $3\A(\al,\bt)$}

\bigskip

{\sc Other names}

$3\A^\prime_{\al,\bt}$ in~\cite{FelixPaper}, \quad $\mathrm{III}(\al, \bt, 0)$ in~\cite{Yabe}

\bigskip

{\tiny \makegapedcells
\[
\begin{array}{|c|A|}
\hline
 \text{Basis} &\multicolumn{2}{l|}{\text{Products and form ($i\in \Z_3$)}} \\
\hline
\multirow{12.5}{*}{$\begin{aligned} & a_{0}, \\ & a_1,\\ & a_2,\\ & z \end{aligned}$}
&
a_i\cdot a_i&= a_i\\
& a_{i}\cdot a_{i+1} &= \bt(a_i+a_{i+1})+\tfrac{\al-\bt}{2}(a_0+a_1+a_2)+z\\
& a_{i}z &= -\tfrac{\al(3\al^2+3\al\bt-\bt-1)}{4(2\al-1)}a_i\\
& z^2 &= -\tfrac{\al(3\al^2+3\al\bt-\bt-1)}{4(2\al-1)}z\\
\cline{2-3}
& (a_i, a_i) &= 1\\
& (a_{i} , a_{i+1}) &= \tfrac{3\al^2+3\al\bt-\al-2\bt}{4(2\al-1)}\\ 
& (a_i,z) &= -\tfrac{\al(3\al^2+3\al\bt-\bt-1)}{4(2\al-1)} \\
& (z,z) &= \tfrac{\al^2(9\al+\bt-5)(3\al^2+3\al\bt-\bt-1)}{16(2\al-1)^2}\\
 \hline
\end{array}
 \]}

 \bigskip
 \bigskip

\begin{center} {\sc Notes }
 \end{center}
 \bigskip
 
 \begin{enumerate}[enum_note]
     \item  $3\A(\al,\beta)$ is defined only for $\al\neq \tfrac{1}{2}$. 
     \item  If $\bt\neq \tfrac{1-3\al^2}{3\al-1}$, then $\mathbbm 1 := -\tfrac{4(2\al-1)}{\al(3\al^2+3\al\bt-\bt-1)}z$ is the identity element of the support of $3\A(\al,\beta)$.  If $\bt=\tfrac{1-3\al^2}{3\al-1}$, then there is no identity element and $z$ is an annihilating element.\label{table3A id ann}
     \item  If $\bt=\tfrac{1-3\al^2}{3\al-1}$, then the radical of the Frobenius form is the one dimensional ideal $\F z$. The factor algebra modulo $\F z$ is $3\A (\al, \tfrac{1-3\al^2}{3\al-1})^\times$.\label{table3A quotient}
     \item  Since  $-1\equiv_3 2$, the basis is the one used in~\cite{MP}.
     \item $3\A(\al,\beta)$ and its quotient $3\A (\al, \tfrac{1-3\al^2}{3\al-1})^\times$ have axet $X(3)$.
 \item In the setting of the Conway-Norton-Griess algebra and Majorana Theory, the algebra $3\A(\tfrac{1}{4}, \tfrac{1}{32})$ is denoted by $3\A$.
 \end{enumerate}
 
\end{table}

\newpage

 \begin{table}[H]\label{table4A}
\nomenclature{$4\A(\tfrac{1}{4},\bt)$}{\pageref{table4A}}
\nomenclature{$4\A_\bt$}{\pageref{table4A}}
\nomenclature{$\mathrm{IV}_1(\tfrac{1}{4}$}{\pageref{table4A}}

 \caption{\Large $4\A(\tfrac{1}{4},\bt)$}

\bigskip

{\sc Other names:}
$4\A_\bt$ in~\cite{FelixPaper},  $\mathrm{IV}_1(\tfrac{1}{4}, \bt)$ in~\cite{Yabe}

\bigskip

{\tiny  \makegapedcells
\[
\begin{array}{|c|A|}
\hline
 \text{Basis} & \multicolumn{2}{l|}{\text{Products and form ($i\in \Z_4$)}} \\
\hline
\multirow{13}{*}{$\begin{aligned} & a_{-1},\\ & a_{0}, \\ & a_1,\\ & a_2,\\ & e \end{aligned}$}
&
a_i\cdot a_i &= a_i\\
& a_{i}\cdot a_{i+1} &= \tfrac{1+4\bt}{8}(a_{i}+a_{i+1})+\tfrac{1-4\bt}{8}(a_{i-1}+a_{i+2})+e\\
& a_{i}\cdot a_{i+2} &=  0\\
& a_{i}e &= \tfrac{2\bt-1}{8}a_i\\
& e^2 &= \tfrac{2\bt-1}{8}e \\
\cline{2-3}
& (a_{i}, a_{i}) &= 1\\ 
& (a_{i}, a_{i+1}) &= \bt \\ 
& (a_{i}, a_{i+2}) &= 0\\ 
& (a_i,e) &= \tfrac{2\bt-1}{8}\\
& (e,e) &= \tfrac{(2\bt-1)^2}{16} \\
 \hline
\end{array} 
 \]
 }

\bigskip

\bigskip

\begin{center} {\sc Notes}
 \end{center}
 \bigskip
 
 \begin{enumerate}[enum_note]
     \item If $\bt\neq \tfrac{1}{2}$, then $\mathbbm 1 := \tfrac{8}{2\bt-1}e$ is the identity element of the support of $4\A(\tfrac{1}{4},\beta)$.  If $\bt=\tfrac{1}{2}$, then there is no identity element and $e$ is an annihilating element.
     \item If $\bt=\tfrac{1}{2}$, then $\F e$ is an ideal of $4\A(\tfrac{1}{4}, \tfrac{1}{2})$ contained in the radical of the Frobenius form. The factor algebra modulo this ideal is $4\A(\tfrac{1}{4}, \tfrac{1}{2})^\times$. 
     \item  \label{table4A evensub} The odd and the even subalgebras of $4\A(\tfrac{1}{4}, \tfrac{1}{2})$ and their images in the quotient $4\A(\tfrac{1}{4}, \tfrac{1}{2})^\times$ are all isomorphic to $2\B$.
     \item  The basis is the one used in~\cite{MP}.
     \item $4\A(\tfrac{1}{4}, \bt)$ and its quotient $4\A(\tfrac{1}{4}, \tfrac{1}{2})^\times$ have axet $X(4)$.
     \item In the setting of the Conway-Norton-Griess algebra and Majorana Theory, the algebra $4\A(\tfrac{1}{4}, \tfrac{1}{32})$ is denoted by $4\A$.
 \end{enumerate}
 
\end{table}

\newpage

 \begin{table}[H]\label{table4B}
 \nomenclature{$4\B(\al,\tfrac{\al^2}{2})$}{\pageref{table4B}}
\nomenclature{$4\B_\al$}{\pageref{table4B}}
\nomenclature{$\mathrm{IV}_2(\al,\tfrac{\al^2}{2})$}{\pageref{table4B}} 
 
 \caption{\Large $4\B(\al,\tfrac{\al^2}{2})$}

 \bigskip

{\sc Other names:}
$
4\B_\al\mbox{ in~\cite{FelixPaper}, }\;\;\; \mathrm{IV}_2(\al,\tfrac{\al^2}{2}) \mbox{ in~\cite{Yabe}} 
$

\bigskip

{\tiny
$$
\begin{array}{|c|A|}
\hline
 \text{Basis} & \multicolumn{2}{l|}{\text{Products and form ($i\in \Z_4)$}} \\
\hline
\multirow{9}{*}{$\begin{aligned} & a_{-1},\\ & a_{0}, \\ & a_1,\\ & a_2,\\ & c \end{aligned}$}
&
a_{i}\cdot a_{i} &= a_{i}\\
& a_{i}\cdot a_{i+1} &= \tfrac{\al^2}{4}\left(a_{i}+a_{i+1} +c-(a_{i+2}+a_{i-1})\right)\\
& a_{i}\cdot a_{i+2} &=  \tfrac{\al}{2}(a_0+a_2-c)\\
&  a_{i}c &= \tfrac{\al}{2}(a_{i}+c-a_{i+2}) \\
&  c^2 &= c \\
\cline{2-3}
& (a_i, a_i) &=  1\\
& (c,c) &= 1\\
& (a_{i}, a_{j}) &= \tfrac{\al^2}{4},\quad i\neq j\\ 
& (a_i,c) &= \tfrac{\al}{2} \\
 \hline
\end{array} 
 $$}

\bigskip

\bigskip

\begin{center} {\sc Notes }
 \end{center}
 \bigskip
 
 \begin{enumerate}[enum_note]
     \item From the formula for the product $a_0a_1$ we get 
     $$c=\tfrac{4}{a^2}a_0a_1+(a_{-1}+a_2)-(a_0+a_1).$$
     In particular $c^f=c$.  Moreover, $c$ is a $\mathcal J(\al)$-axis in the whole algebra (see~\cite[Section 5]{MM}).
     \item   
     The odd and the even subalgebras of $4\B(\al,\tfrac{\al^2}{2})$ are both isomorphic to $ 3\C(\al)$.
     Moreover, for every $i\in  \Z_4$, $\lla a_i, a_{i+2}\rra=\la a_i, a_{i+2},c\ra$, where $(a_i,a_{i+2},c)$ is the natural basis given in Table \ref{table2}.
     \item If $\al=-1$, then the radical of the form coincides with the algebra annihilator and is  a $2$-dimensional ideal with basis $(a_0+a_2+c,\quad a_{-1}+a_1+c)$. Moreover $\F \left (a_{-1}+a_0+a_1+a_2+2c\right )$ is an ideal of $4\B(-1, \tfrac{1}{2})$ which is invariant under the flip; the factor algebra modulo this ideal is the symmetric algebra denoted by $4\B(-1,\tfrac{1}{2})^\times$\nomenclature{$4\B(-1,\tfrac{1}{2})^\times$}{\pageref{table4B}} (see~\cite[Proposition~5.11]{MM}). There are also non-symmetric quotients which we describe in Section~\ref{sec:ns}.
     \item  The odd and the even subalgebras of $4\B(-1,\tfrac{1}{2})^\times$ are both isomorphic to  $3\C(-1)^\times$.
     \item The basis is the one used in~\cite{MP}. 
     \item $4\B(\al, \tfrac{\al^2}{2})$ and its quotient $4\B(-1,\tfrac{1}{2})^\times$ has axet $X(4)$.
     \item In the setting of the Conway-Norton-Griess algebra and Majorana Theory, the algebra $4\B(\tfrac{1}{4}, \tfrac{1}{32})$ is denoted by $4\B$.
      \end{enumerate}

\end{table}

\newpage

\begin{table}[H]\label{table4J}
\nomenclature{$4\J(2\bt,\bt)$}{\pageref{table4J}}

\caption{{\Large $4\J(2\bt,\bt)$}}

\bigskip 
 
{\sc Other names:}
$
\mathrm{IV}_1(\al, \tfrac{\al}{2}) \mbox{ in~\cite{Yabe}}
$
\nomenclature{$\mathrm{IV}_1(\al, \tfrac{\al}{2})$}{\pageref{table4J}}

\bigskip

{\tiny \makegapedcells 
\[
\begin{array}{|c|A|}
\hline
\text{Basis} &\multicolumn{2}{l|}{\text{Products and form ($i\in\Z_4$)}} \\
\hline
 \multirow{12}{*}{$\begin{aligned} & a_{-1},\\ & a_{0}, \\ & a_1,\\ & a_2,\\ & w \end{aligned}$}
& a_{i}\cdot a_{i}  &= a_{i}\\
&  a_{i}\cdot a_{i+1} &= \tfrac{\bt}{2}(2a_{i}+2a_{i+1}-w)\\
& a_{i}\cdot a_{i+2} &=  0\\
&  a_{i}\cdot w &= \bt(2a_i-(a_{i-1}+a_{i+1})+w)\\
&  w^2 &= w \\
 \cline{2-3}
&  (a_i, a_i) &= 1\\
&  (a_{i}, a_{i+1}) &= \bt\\
&  (a_{i}, a_{i+2}) &= 0\\
& (a_i,w) &= 2\bt\\
& (w,w) &= 2
\\
\hline
\end{array} 
 \]
 }
\bigskip

\bigskip

\begin{center} {\sc Notes} 
\bigskip

 \end{center}
 \begin{enumerate}[enum_note]
     \item  By~\cite{FMS2, MM}, $4\J(\tfrac{1}{2},\tfrac{1}{4})\cong 4\Y(\tfrac{1}{2},\tfrac{1}{4})$.
     \item  Let $q:=a_{-1}+a_0+a_1+a_2+w$. If $\bt\neq -\tfrac{1}{4}$, then $\mathbbm 1 := \tfrac{1}{4\bt+1}q$ is the identity element of the support of $4\J(2\bt,\bt)$.  If $\bt=-\tfrac{1}{4}$, then there is no identity element and $q$ is an annihilating element.
     \item  If $\bt=-\tfrac{1}{4}$, then the radical of the Frobenius form is the one dimensional ideal $\F q$. The factor algebra modulo the radical is $4\J(-\tfrac{1}{2}, -\tfrac{1}{4})^\times$\nomenclature{$4\J(-\tfrac{1}{2}, -\tfrac{1}{4})^\times$}{\pageref{table4J}} (see~\cite[Lemma 5.2]{FMS2}).
     \item  The odd and the even subalgebras of $4\J(2\bt, \bt)$ and $4\J(-\tfrac{1}{2}, -\tfrac{1}{4})^\times$ are all isomorphic to $2\B$.
     \item  The basis is the one used in~\cite{MP}.
     \item $4\J(2\bt, \bt)$ and it quotient $4\J(-\tfrac{1}{2}, -\tfrac{1}{4})^\times$ has axet $X(4)$.
 \end{enumerate} 
 
\end{table}

\newpage

 \begin{table}[H]
 \caption{{\Large $4\Y(\tfrac{1}{2},\bt)$}}
 \label{table4Y}
 \nomenclature{$4\Y(\tfrac{1}{2},\bt)$}{\pageref{table4Y}}
 
 \bigskip

{\sc Other names:}
$ 
\mathrm{IV}_2(\tfrac{1}{2},\bt) 
\mbox{ in~\cite{Yabe}} 
$
\nomenclature{$\mathrm{IV}_2(\tfrac{1}{2},\bt)$}{\pageref{table4Y}}

\bigskip

{\tiny \makegapedcells 
$$
\begin{array}{|c|A|}
\hline
 \text{Basis}& \multicolumn{2}{l|}{\text{Products and form ($i\in \Z_4$)}}\\
\hline
\multirow{12.5}{*}{$\begin{aligned} & a_{-1},\\ & a_{0}, \\ & a_1,\\ & a_2,\\ & z \end{aligned}$}
& a_{i}\cdot a_{i} &= a_{i}\\
& a_{i}\cdot a_{i+1} &= \tfrac{\bt}{2}(a_{i}+a_{i+1})-\tfrac{\bt}{2}(a_{i+2}+a_{i-1})+4\bt^2 z\\
& a_{i}\cdot a_{i+2} &=  \tfrac{1-4\bt}{2}(a_{i}+a_{i+2})+4\bt(4\bt-1) z\\
& a_{i}z &= \tfrac{1}{4}(a_i-a_{i+2})+2\bt z\\
& z^2 &= z \\
 \cline{2-3}
& (a_i, a_i) &= 1,\\
& (a_{i}, a_{i+1}) &= 4\bt^2\\ 
& (a_{i}, a_{i+2}) &= (4\bt-1)^2\\ 
& (a_i,z) &= 2\bt\\
&  (z,z) &=  1 \\
 \hline
\end{array} 
 $$}

\bigskip

\bigskip

\begin{center} {\sc Notes }
 \end{center}
 \bigskip
 
 \begin{enumerate}[enum_note]
      \item $\1=\tfrac{1}{1-2\bt}( \tfrac{1}{2}(a_{-1}+a_0+a_1+a_2)+(1-6\bt)z)$ is the identity of the algebra.
     \item   The odd and the even subalgebras of
     $4\Y(\tfrac{1}{2},\bt)$ are both isomorphic to $\J(4\bt(2\bt-1))$. Moreover, for every $i\in \Z_4$,
     $\lla a_i,a_{i+2}\rra=\la a_i,a_{i+2},z\ra$, so that  $z$ is an idempotent contained in both the odd and the even subalgebras. Notice that 
     the basis of $\lla a_i,a_{i+2}\rra$ corresponding to that given in Table~\ref{table2} for $\J(4\bt(2\bt-1))$ is  $( a_i, a_{i+2}, -2\bt(a_i+a_{i+2})+4\bt(4\bt-1)z)$. 
     \item  The basis is the one used in~\cite{MP}.
     \item $4\Y(\tfrac{1}{2},\bt)$ has axet $X(4)$.
 \end{enumerate} 
 
\end{table}

\newpage

 \begin{table}[H]
 \caption{{\Large $4\Y(\al, \tfrac{1-\al^2}{2})$}}\label{table4Yal}
\nomenclature{$4\Y(\al, \tfrac{1-\al^2}{2})$}{\pageref{table4Yal}} 

\bigskip

{\sc Other names:}
$
\mathrm{IV}_2(\al, \tfrac{1-\al^2}{2}) 
\mbox{ in~\cite{Yabe}}
$
\nomenclature{$\mathrm{IV}_2(\al, \tfrac{1-\al^2}{2})$}{\pageref{table4Yal}}

\bigskip
 
{\tiny \makegapedcells 
$$
\begin{array}{|c|A|}
\hline
\text{Basis} & \multicolumn{2}{l|}{\text{Products and form ($i\in\Z_4$)}} \\
\hline
 \multirow{13}{*}{$\begin{aligned} & a_{-1},\\ & a_{0}, \\ & a_1,\\ & a_2,\\ & c \end{aligned}$}
& a_{i}\cdot a_{i} &= a_{i}\\
& a_{i}\cdot a_{i+1} &= \tfrac{\bt}{2}(a_{i}+a_{i+1})-\tfrac{\bt}{2}(a_{i+2}+a_{i-1}) + \tfrac{(\al+1)^2}{4}c \\
& a_{i}\cdot a_{i+2} &=  \tfrac{\al-1}{2}(a_{i}+a_{i+2})+\tfrac{\al+1}{2}c\\
&  a_{i}c &= \tfrac{\al-1}{2}(a_{i+2}-a_{i})+\tfrac{\al+1}{2}c\\
&  c^2 &= c \\
 \cline{2-3}
&  (a_i, a_i) &= 1\\
&  (a_{i}, a_{i+1}) &= \tfrac{(2-\al)(\al+1)}{4}\\
&  (a_{i}, a_{i+2}) &= \tfrac{\al}{2}\\
& (a_i,c) &= \tfrac{2-\al}{2}\\
& (c,c) &= \tfrac{2-\al}{\al+1} \\
\hline
\end{array} 
 $$
 }

\bigskip

\bigskip

\begin{center} {\sc Notes}.
 \end{center}
 \bigskip
 
 \begin{enumerate}[enum_note]
     \item  Note that $\al+1\neq 0$, since $\bt\neq 0$.
     \item The element $c$ is a $\mathcal J(1-\al)$-axis in $4\Y(\al, \tfrac{1-\al^2}{2})$ (see~\cite[Table~14]{MM}).
     \item  
     The odd and the even subalgebras are both isomorphic to $3\C(\al)$. Moreover, 
     for every $i\in \Z_4$, $\lla  a_i,a_{i+2}\rra=\la a_1, a_{i+2}, c\ra$. The basis of $\lla a_i,a_{i+2}\rra$ corresponding to the one given in Table~\ref{table2} is $(a_i,a_{i+2},c_i)$, where, for $i\in \Z_4$, $c_i:=\tfrac{1}{2}(a_i+a_{i+2}-3c)$.
     \item  The basis is the same as the one used in~\cite{MP}.
     \item $4\Y(\al, \tfrac{1-\al^2}{2})$ has axet $X(4)$.
 \end{enumerate}  

\end{table}

\newpage 

 \begin{table}[H]
 \caption{{\Large $5\A(\al, \tfrac{5\al-1}{8})$}}\label{table5A}
 \nomenclature{$5\A(\al, \tfrac{5\al-1}{8})$}{\pageref{table5A}}

 \bigskip 
 
{\sc Other names:}
$
5\A_\al\mbox{ in~\cite{FelixPaper},}\;\;\; \mathrm{V}_1(\al, \tfrac{5\al-1}{8})\mbox{ in~\cite{Yabe}}
$
\nomenclature{$5\A_\al$}{\pageref{table5A}}
\nomenclature{$\mathrm{V}_1(\al, \tfrac{5\al-1}{8})$}{\pageref{table5A}}

\bigskip

{\tiny \makegapedcells 
$$
\begin{array}{|c|A|}
\hline
 \text{Basis} & \multicolumn{2}{l|}{\text{Products and form ($i\in \Z_5$)}} \\
\hline
\multirow{12}{*}{$\begin{aligned} & a_{-2}, \\& a_{-1},\\ & a_{0}, \\ & a_1,\\ & a_2,\\ & w \end{aligned}$}
& a_{i}\cdot a_{i} &= a_{i}\\
& a_{i}\cdot a_{i+1} &= \tfrac{5\al-1}{32} w + \tfrac{3(5\al-1)}{32}(a_{i}+a_{i+1})-\tfrac{5\al-1}{32}(a_{i+2}+a_{i-1}+a_{i-2})\\
& a_{i}\cdot a_{i+2} &= \tfrac{-5\al+1}{32} w + \tfrac{3(5\al-1)}{32}(a_{i}+a_{i+2})-\tfrac{5\al-1}{32}{4}(a_{i+1}+a_{i-1}+a_{i-2})\\
& a_{i}w &= \tfrac{3\al+1}{8} w + \tfrac{3\al+1}{8}(a_{i-1}+a_{i+1})-\tfrac{3\al+1}{8}(a_{i+2}+a_{i-2})\\
& w^2 &= \tfrac{(3\al+1)(7-3\al)}{8(5\al-1)}(a_{-2}+a_{-1}+a_0+a_1+a_2)\\
 \cline{2-3}
& (a_i,a_i) &= 1\\
& (a_{i}, a_{j}) &= \tfrac{3(5\al-1)}{32},\quad  i\neq j\\ 
& (a_i,w) &= 0\\
& (w,w) &= \tfrac{5(3\al+1)(7-3\al)}{8(5\al-1)} \\
 \hline
\end{array} 
 $$}

\bigskip

\bigskip 

\begin{center} {\sc Notes }
 \end{center}
 \bigskip
 
 \begin{enumerate}[enum_note]
     \item  If $\ch(\F)\neq 5$, then $\1=\tfrac{1}{5(\al-\bt)}(a_{-2}+a_{-1}+a_0+a_1+a_2)$ is the identity of $V$ (see~\cite[Section 6.1]{MM}).\label{table5A id}
     \item If $\ch(\F)= 5$, then $I:=\langle a_{-2}+a_{-1}+a_0+a_1+a_2\rangle $ is the annihilator of the algebra and the factor algebra $\V/I$ is denoted by $5A(\al, \tfrac{1}{2})^\times$ (see also Note~\ref{tableIY5 char5} in Table~\ref{table5A} and~\cite[Corollary 6.6]{MM}).  
     \item  The basis is the one used in~\cite{MP}.
     \item $5\A(\al, \tfrac{5\al-1}{8})$ has axet $X(5)$.
     \item In the setting of the Conway-Norton-Griess algebra and Majorana Theory, the algebra $5\A(\tfrac{1}{4}, \tfrac{1}{32})$ is denoted by $5\A$.
 \end{enumerate} 
\end{table}

\newpage

 \begin{table}[H]
 \caption{{\Large $6\A\left (\al, -\tfrac{\al^2}{4(2\al-1)}\right ) $}}\label{table6A}
\nomenclature{$6\A\left (\al, -\tfrac{\al^2}{4(2\al-1)}\right ) $}{\pageref{table6A}} 
 \bigskip 

{\sc  Other names:}
$
 6\A_\al 
 \mbox{ in~\cite{FelixPaper},}\;\;\;
 \mathrm{VI}_2\left (\al, -\tfrac{\al^2}{4(2\al-1)}\right )
 \mbox{ in~\cite{Yabe}}
$
\nomenclature{$6\A_\al$}{\pageref{table6A}}
\nomenclature{$\mathrm{VI}_2\left (\al, -\tfrac{\al^2}{4(2\al-1)}\right )$}{\pageref{table6A}}

\bigskip

{\tiny \makegapedcells 
\[
\begin{array}{|c|A|}
\hline
 \text{Basis} & \multicolumn{2}{l|}{\text{Products and form ($i\in \Z_6$)}}  \\
\hline
\multirow{24.5}{*}{$\begin{aligned} & a_{-2}, \\& a_{-1},\\ & a_{0}, \\ & a_1,\\ & a_2,\\ & a_3, \\& c,\\ & z \end{aligned}$}
& a_{i}\cdot a_{i} &= a_{i}\\
&  a_{i} \cdot a_{i+1} &= \tfrac{-\al^2}{8(2\al-1)}\left(c+z+a_{i}+a_{i+1}-(a_{i+2}+a_{i+3}+a_{i-1}+a_{i-2})\right)\\
&  a_{i}\cdot  a_{i+2} &= \tfrac{\al}{4}(a_i+a_{i+2})+\tfrac{\al(3\al-1)}{4(2\al-1)}a_{i+4}-\tfrac{\al(5\al-2)}{8(2\al-1)}z\\
& a_{i}\cdot a_{i+3} &= \tfrac{\al}{2}(a_i+a_{i+3}-c)\\
& a_{i}\cdot c &= \tfrac{\al}{2}(a_i+c-a_{i+3})\\
& a_i\cdot z &= \tfrac{\al(3\al-2)}{4(2\al-1)}(2a_i - a_{i-2}-a_{i+2} + z)\\
& c^2 &= c \\
& c\cdot z &= 0 \\
& z^2 &=  \tfrac{(\al+2)(3\al-2)}{4(2\al-1)}z\\
\cline{2-3}
& (a_i,a_i) &= 1\\
& (a_{i}, a_{i+1}) &=  -\tfrac{\al^2(3\al-2)}{(4(2\al-1))^2}\\
& (a_{i}, a_{i+2}) &= \tfrac{\al(21\al^2-18\al+4)}{(4(2\al-1))^2}\\
& (a_{i}, a_{i+3}) &= \tfrac{\al}{2}\\
& (a_i,c) &= \tfrac{\al}{2}\\
& (a_i, z) &= \tfrac{\al(7\al-4)(3\al-2)}{8(2\al-1)^2}\\
& (c,c) &= 1\\
& (c,z) &= 0\\
& (z,z) &= \tfrac{(\al+2)(7\al-4)(3\al-2)}{8(2\al-1)^2} \\
 \hline
\end{array} 
 \]}
  \end{table}

\bigskip

\begin{center} {\sc Notes } 
 \end{center}
 
 \bigskip
 
 \begin{enumerate}[enum_note]
    \item  Let 
    $q:= 2(2\al-1)(3\al-2)\sum_{i=-2}^3a_i+(5\al-2)(3\al-2)c+4(2\al-1)(3\al-1)z.$
    If $\al \notin \{\tfrac{2}{3},\tfrac{1-\sqrt{97}}{24},\tfrac{1+\sqrt{97}}{24}\}$, then $\mathbbm 1 := \tfrac{1}{(12\al^2-\al-2)(3\al-2)}q$ is the identity element of the support of $6\A\left (\al, -\tfrac{\al^2}{4(2\al-1)}\right )$.  If $\al \in \{\tfrac{2}{3},\tfrac{1-\sqrt{97}}{24},\tfrac{1+\sqrt{97}}{24}\}$, then there is no identity element and $q$ is an annihilating element.\label{table6A id ann}
    \item  If $\al=\tfrac{2}{3}$, then $z\in \la q\ra$ and $\F z$ is an ideal. The quotient modulo this ideal is 
    $6\A\left (\tfrac{2}{3}, -\tfrac{1}{3}\right )^\times$.\label{table6A quotient 2/3}
    \nomenclature{$6\A\left (\tfrac{2}{3}, -\tfrac{1}{3}\right )^\times$}{\pageref{table6A}}
    \item  If $\al\in\{\tfrac{1-\sqrt{97}}{24},\tfrac{1+\sqrt{97}}{24}\} $, then $\F q$ is an ideal and the quotient modulo $\F q$ is 
    $6\A\left (\al, -\tfrac{\al+2}{48(2\al-1)}\right )^\times$.\label{table6A quotient sqrt97}
    \nomenclature{$6\A\left (\al, -\tfrac{\al+2}{48(2\al-1)}\right )^\times$}{\pageref{table6A}}
    \item  Let $\al\neq \tfrac{2}{5}$. The odd and the even subalgebras of 
    $6\A\left (\al, -\tfrac{\al^2}{4(2\al-1)}\right )$ are  isomorphic to $3\A\left (\al,-\tfrac{\al^2}{4(2\al-1)}\right )$ (see~\cite[Table~20]{MM}). For every $i\in \Z_6$, 
     $\lla a_i,a_{i+2}\rra$ has basis $( a_i,a_{i+2},a_{i+4},z)$. The basis corresponding to that given in Table~\ref{table3A} is
    \[
    \left(a_i,\: a_{i+2},\: a_{i+4},\: \tfrac{-\al}{8(2\al-1)}((3\al-2)(a_i+a_{i+2}+a_{i+4})+(5\al-2)z)\right).
    \]
    \item  The odd and the even subalgebras of 
    $6\A\left (\tfrac{2}{5}, \tfrac{1}{5}\right )$ are both  isomorphic to $3\C\left (\tfrac{1}{5}\right )$ (see~\cite[Table~20]{MM}). For every $i\in \Z_6$,  
     $\lla a_i,a_{i+2}\rra$ has basis $( a_i,a_{i+2},a_{i+4})$ which is the same basis corresponding to that given in Table~\ref{table2}.
    \item   For $\al=\tfrac{2}{3}$,  the odd and the even subalgebras of the quotient $6\A\left (\tfrac{2}{3}, -\tfrac{1}{3}\right )^\times$ are both isomorphic to 
     $3\A(\tfrac{2}{3},-\tfrac{1}{3})^\times$. 
    \item For $\al \in \{\tfrac{1-\sqrt{97}}{24},\tfrac{1+\sqrt{97}}{24}\}$, the odd and the even subalgebras of the quotient $6\A\left (\al, -\tfrac{\al^2}{4(2\al-1)}\right )^\times$ are both isomorphic to $3\A\left (\al,-\tfrac{\al^2}{4(2\al-1)}\right )$.
    \item For every $i\in \Z_6$, the algebras $(\lla a_i, a_{i+3}\rra, \{ a_i, a_{i+3}\})$ are all isomorphic to $ 3\C(\al)$ (see~\cite[Table~20]{MM}). $\lla a_i, a_{i+3}\rra$ has basis $( a_i, a_{i+3},c)$, which is the basis corresponding to that given in Table~\ref{table2}.
    \item For $\al\in \{\tfrac{2}{3}, \tfrac{1-\sqrt{97}}{24},\tfrac{1+\sqrt{97}}{24}\}$ and for every $i\in \Z_6$, the images of the algebras $(\lla a_i, a_{i+3}\rra, \{ a_i, a_{i+3}\})$ in the quotient  $6\A\left (\al, -\tfrac{\al^2}{4(2\al-1)}\right )^\times$  are all isomorphic to $ 3\C(\al)$. 
     \item The basis is the one used in~\cite{MP}.
    \item  $6\A\left (\al, -\tfrac{\al^2}{4(2\al-1)}\right )$ and its quotients $6\A\left (\al, -\tfrac{\al^2}{4(2\al-1)}\right )^\times$  for $\al\in \{\tfrac{2}{3},\tfrac{1\pm\sqrt{97}}{24}\}$ have axet $X(6)$.
    \item In the setting of the Conway-Norton-Griess algebra and Majorana Theory, the algebra $6\A(\tfrac{1}{4}, \tfrac{1}{32})$ is denoted by $6\A$.
 \end{enumerate} 
 

 \newpage

 \begin{table}[H]
 \caption{{\Large $6\J(2\bt,\bt)$}}\label{table6J}
 \nomenclature{$6\J(2\bt,\bt)$}{\pageref{table6J}}
 
 \bigskip

{\tiny \makegapedcells 
\[
\begin{array}{|c|A|}
\hline
 \text{Basis} & \multicolumn{2}{l|}{\text{Products and form ($i\in \Z_6$)}}  \\
\hline
\multirow{21}{*}{$\begin{aligned} & a_{-2}, \\& a_{-1},\\ & a_{0}, \\ & a_1,\\ & a_2,\\ & a_3, \\& u,\\ & w \end{aligned}$}
& a_{i}\cdot a_{i} &= a_{i}\\
&  a_{i} \cdot a_{i+1} &= \tfrac{\bt}{2}(2(a_{i}+a_{i+1})-w)\\
&  a_{i}\cdot  a_{i+2} &= \tfrac{\bt}{2}(a_i+a_{i+2}-a_{i+4})\\
& a_{i}\cdot a_{i+3} &= \bt(a_i+a_{i+3}-u)\\
& a_{i}\cdot u &= \bt(a_i+u-a_{i+3})\\
& a_i\cdot w &= \bt(w+2a_i-a_{i-1}-a_{i+1})\\
& u^2 &= u \\
& u\cdot w &= \bt u \\
& w^2 &=  (\bt+1)w-\bt u\\
\cline{2-3}
& (a_i,a_i) &= 1\\
& (a_{i}, a_{i+1}) &=  \bt\\
& (a_{i}, a_{i+2}) &= \tfrac{\bt}{2}\\
& (a_{i}, a_{i+3}) &= \bt\\
& (a_i,u) &= \bt \\
& (a_i, w) &= 2\bt \\
& (u,u) &= 1\\
& (u,w) &= \bt\\
& (w,w) &= \bt+2 \\
 \hline
\end{array} 
 \]}
\bigskip

\bigskip

\begin{center} {\sc Notes} 
 \end{center}
 \bigskip
 
 \begin{enumerate}[enum_note]
 \item  By~\cite{FMS2, MM}, $6\J(\tfrac{2}{5},\tfrac{1}{5})\cong 6\A(\tfrac{2}{5},\tfrac{1}{5})$.
 \item  Let $q:= a_{-2}+a_{-1}+a_0+a_1+a_2+a_3+u+w$. If $\bt \neq -\tfrac{1}{7}$, then $\mathbbm 1 := \tfrac{1}{7\bt+1}q$ is the identity element of the support of $6\J(2\bt, \bt)$.  If $\bt = -\tfrac{1}{7}$, then there is no identity element and $q$ is an annihilating element (see~\cite[\S 7.2]{MM}).
     \item  If $\bt=-\tfrac{1}{7}$, then  $\F q$ is an ideal of $6\J(-\tfrac{2}{7}, -\tfrac{1}{7})$. The quotient modulo this ideal is 
     $6\J(-\tfrac{2}{7}, -\tfrac{1}{7})^\times$ 
     \nomenclature{$6\J(-\tfrac{2}{7}, -\tfrac{1}{7})^\times$}{\pageref{table6J}}  (see~\cite[Proposition~7.17]{MM}).
     \item  The odd and the even subalgebras of $6\J(2\bt, \bt )$ and $6\J(-\tfrac{2}{7}, -\tfrac{1}{7})^\times$ are all isomorphic to $3\C(\bt)$. For every $i\in \Z_6$, $\lla a_i, a_{i+2}\rra=\la a_i, a_{i+2},a_{i+4}\ra$ and $ (a_i, a_{i+2},a_{i+4})$ is the basis corresponding to that given in Table~\ref{table2}.
     \item  For every $i\in \Z_6$, $(\lla a_i, a_{i+3}\rra,\{a_i,a_{i+3}\})\cong 3\C(2\bt)$. Moreover,  $\lla a_i, a_{i+3}\rra=\la a_i, a_{i+3},u\ra$ and $ (a_i, a_{i+3},u)$ is the basis corresponding to that given in Table~\ref{table2}.
     \item  The basis is the one used in~\cite{MP}.
     \item $6\J(2\bt,\bt)$ and its quotient $6\J(-\tfrac{2}{7}, -\tfrac{1}{7})^\times$ have axet $X(6)$.
 \end{enumerate}
\end{table}

\newpage

 \begin{table}[H]
 \caption{{\Large $6\Y(\tfrac{1}{2},2)$}}\label{table6Y}
 \nomenclature{$6\Y\left (\tfrac{1}{2}, 2\right )$}{\pageref{table6Y}}

 \bigskip
 
{\tiny \makegapedcells
$$
\begin{array}{|c|A|}
\hline
\text{Basis} & \multicolumn{2}{l|}{\text{Products and form ($i\in 2\Z_6$)}}  \\
\hline
\multirow{14.5}{*}{$\begin{aligned} & a_{-2}, \\& a_{0},\\ & a_{2}, \\ & d,\\ & z \end{aligned}$}
& a_{i}\cdot a_{i} &= a_{i}\\
& a_{i} \cdot a_{i+2} &= a_{i}+a_{i+2}-a_{i+4}\\
& a_{i}\cdot d &= \tfrac{1}{2}d+z\\
& a_{i}\cdot z &= 0\\
& d\cdot d &= -2z\\
& z\cdot z &= 0\\
& z\cdot d &= 0 \\
\cline{2-3}
& (a_i,a_{j}) &= 1 \\
& (a_i,d) &= 0 \\ 
& (a_i,z) &= 0 \\
& (d,d) &=  0\\
& (d,z) &= 0 \\
& (z,z) &= 0 \\
 \hline
\end{array} 
 $$}
\bigskip

\bigskip

\begin{center} {\sc Notes}
 \end{center}
 
 \bigskip
 
 \begin{enumerate}[enum_note]
     \item  $a_1:=a_{-2}+d$. \label{table6Y axis relation}
     \item   $6\Y\left (\tfrac{1}{2}, 2\right )^\times$
     \nomenclature{$6\Y\left (\tfrac{1}{2}, 2\right )^\times$}{\pageref{table6Y}} is the quotient of $6\Y\left (\tfrac{1}{2}, 2\right )$ modulo the ideal $\F z$. \label{table6Y quotient}
     \item  $6\Y\left (\tfrac{1}{2}, 2\right )$ is isomorphic (but not axially isomorphic as a $\mathcal M(\frac{1}{2},2)$-algebra due to the change of fusion law from $\mathcal{M}(\frac{1}{2}, 2)$ to $\mathcal{M}(2, \frac{1}{2})$) to a quotient of the Highwater algebra $\mathcal H$, hence is also of Monster type $(2,\tfrac{1}{2})$ (see the Notes to Table~\ref{tablehatH}).
     \item  By definition $\lmu=\tfrac{(a_0,a_1)}{(a_0, a_0)}=1$.
     \item  By definition, $s_{\tz, 1}=z-(a_0+a_2+a_4+\tfrac{3}{2}d)$.
     \item   The odd and the even subalgebras of $6\Y\left (\tfrac{1}{2}, 2\right )$ and $6\Y\left (\tfrac{1}{2}, 2\right )^\times$ are isomorphic to $3\C(\al)$ (see~\cite[Table~20]{MM}).
     \item  For every $i\in \Z_6$, $(\lla a_i, a_{i+3}\rra,\{a_i,a_{i+3}\})\cong \J(0)$ (\cite[Table~20]{MM}). Moreover  $\lla a_i, a_{i+3}\rra=\la a_i, d,z\ra$ and $ (a_i, a_i+d,z)$ is the basis corresponding to that given in Table~\ref{table2}.
     \item   The image of $(\lla a_i, a_{i+3}\rra,\{a_i,a_{i+3}\})$ in $6\Y\left (\tfrac{1}{2}, 2\right )^\times$ is isomorphic to $\J(0)^\times$ (\cite[Table~20]{MM}).
     \item  The basis used in~\cite{MP, MM} is $(a_0,a_2, a_4, d,z)$, where $a_4=a_{-2}$.
     \item $6\Y\left (\tfrac{1}{2}, 2\right )$ and its quotient $6\Y\left (\tfrac{1}{2}, 2\right )^\times$ have axet $X(6)$.
 \end{enumerate} 
\end{table}

 \newpage


 \begin{table}[H]
 \caption{{\Large $\IY_3(\al, \tfrac{1}{2}; \mu)$}}\label{tableIY3bis}
 \nomenclature{$\IY_3(\al, \tfrac{1}{2}; \mu)$}{\pageref{tableIY3bis}}
 
  \bigskip

{\sc Other names:}
$
\mathrm{III}(\al, \tfrac{1}{2}, -2\mu-1) \mbox{ in~\cite{Yabe}}, \;\; S(b, \al), \mbox{ if $\al\neq -1$}, \mbox{ and } \widehat S(b,-1)^\circ \mbox{ in~\cite{split}}  
$
\nomenclature{$\mathrm{III}(\al, \tfrac{1}{2}, -2\mu-1)$}{\pageref{tableIY3bis}}
\nomenclature{$S(b, \al)$}{\pageref{tableIY3bis}}
\nomenclature{$\widehat S(b,-1)^\circ$}{\pageref{tableIY3bis}}

\bigskip
  
{\tiny \makegapedcells 
\[
\begin{array}{|c|A|}
\hline
 \text{Basis} & \multicolumn{2}{l|}{\text{Products and form ($i\in \Z$)}}  \\
\hline
\multirow{13}{*}{$\begin{aligned} & a_{-1}, \\ & a_0,\\ & a_1,\\ & s_{\tz,1} \end{aligned}$}
& a_{i}\cdot a_{i} &= a_i, \quad  i\in \{-1,0,1\}\\
& a_{i}\cdot a_{i+1} &= \tfrac{1}{2}(a_{i}+a_{i+1})+s_{\tz,1}, \quad  i\in \{-1,0\}\\
& a_{-1}\cdot a_{1} &= \tfrac{1}{2}(a_{-1}+a_{1})+2(\mu+1)s_{\tz,1}\\
& a_{i}s_{\tz,1} &= \tfrac{\al^2-\al+1}{4}(\mu-1)a_i+\tfrac{2\al-1}{8}(a_{-1}+a_{1}-2\mu a_0)+\tfrac{2\al-1}{2}s_{\tz,1}, \quad i\in \{-1,0, 1\}\\
& (s_{\tz,1})^2 &= \tfrac{(2\al-1)(\al^2-3\al+2)}{16}(\mu-1)(a_{-1}+a_1-2\mu a_0) -\tfrac{3\al^2-5\al+1}{4}(\mu-1)s_{\tz,1}\\
 \cline{2-3}
& (a_i, a_i) &= 1, \quad  i\in \{-1,0, 1\}\\
& (a_i , a_{i+1}) &= \tfrac{1}{4}\al(1-\mu)+\tfrac{1}{2}(\mu+1), \quad  i\in \{-1,0\}\\ 
& (a_{-1}, a_{1}) &= -\tfrac{1}{2}(\al\mu^2-2\mu^2-\al)\\
& (a_i,s_{\tz,1}) &= -\tfrac{1}{8}(\al-2)(\mu-1), \quad  i\in \{-1,0, 1\} \\ 
& (s_{\tz,1},s_{\tz,1}) &= -\tfrac{1}{16}(\al-2)(\al^2-\al+1)(\mu-1)^2 \\ 
 \hline
\end{array} 
 \]}

\bigskip

\bigskip

\begin{center} {\sc Notes} 
 \end{center}
 \bigskip
 
 \begin{enumerate}[label=$\arabic*.$, ref=$\arabic*$]
     \item There is a slight deviation to the notation in \cite{survey}, in which we use a semicolon instead of a comma to split the extra parameter that is not related to the fusion law.
     \item  $\IY_3(\al, \tfrac{1}{2}; \mu)$ is of Monster  type $(\al, \tfrac{1}{2})$. For $\mu\neq 1$, a detailed description of this algebra can be found in~\cite{split}, where it is denoted by $S(b, \al)$ when $\al\neq -1$ and  by $\widehat{S}(b,-1)^\circ$ when $\al=-1$. 
     \item  Let $q:=4s_{\tz,1}+(1-2\al)(a_{-1}+a_1-2\mu a_0)$. If $\al\neq -1$ and $\mu\neq 1$, then $\mathbbm 1 :=\tfrac{1}{\al(\al+1)(\mu-1)}q$ is the identity element for the support of $\IY_3(\al, \tfrac{1}{2}; \mu)$. If $\al=-1$ or $\mu=1$, then there is no identity element and $\langle q\rangle$ is the annihilator of the algebra. \label{qann}
     \item  If $\al=-1$ or $\mu=1$, then $\F q$ is an ideal of $\IY_3(\al, \tfrac{1}{2}; \mu)$ contained in the radical of the Frobenius form. The quotient modulo this ideal is 
     $\IY_3(-1, \tfrac{1}{2}; \mu)^\times$ if $\al=-1$
     \nomenclature{$\IY_3(-1, \tfrac{1}{2}; \mu)^\times$}{\pageref{tableIY3bis}} or  
     $\IY_3(\al, \tfrac{1}{2}; 1)^\times$ if $\mu=1$.
     \nomenclature{$\IY_3(\al, \tfrac{1}{2}; 1)^\times$}{\pageref{tableIY3bis}}
     \item  $\IY_3(\al, \tfrac{1}{2}; -\tfrac{1}{2})\cong 3\A(\al, \tfrac{1}{2})$ (see~\cite{MM}).\label{tableIY3 3A}
     \item \label{tableIY3basis1} Let $\al\neq -1$ and $\mu\neq 1$. In~\cite{MP, MM} the basis used is $(e,f, z_1, z_2)$, where 
     $$e:= \tfrac{1}{\mu-1}(a_{-1}-2a_0+a_1), \quad f:=\tfrac{1}{\mu-1}( a_{-1}-2\mu a_0+(2\mu-1)a_1),$$
     $$z_1:= \tfrac{1}{\al(\mu-1)}( (1-\al)(a_{-1}+a_1)+ 2\mu(\al-1)a_0+4s_{\tz,1}), \mbox{ and }$$
     $$z_2:= \tfrac{1}{ (\al+1)(\mu-1) }( (\al-2)(a_{-1}+a_1)-2\mu(\al-2)a_0-4s_{\tz,1} ).$$
     Moreover, the form is rescaled by $(\al+1)$.  Notice that $\mathbbm1=z_1+z_2$.
   \item \label{tableIY3basis2} Let $\al=-1$ and $\mu\neq 1$. The basis used in~\cite{MP} is $(e,f, z_1, n)$, where 
   $$e:= \tfrac{1}{\mu-1}(a_{-1}+a_1-2a_0 ), \quad  
f:= \tfrac{1}{\mu-1} ( a_{-1}-2\mu a_0+ (2\mu-1)a_1),$$
$$z_1:= -\tfrac{2}{\mu-1} ( a_{-1}+a_1-2\mu a_0+2s_{\tz,1}),  \mbox{ and } n:=-\tfrac{1}{\mu-1}q.$$
\item \label{tableIY3basis3} When  $\mu=1$, then  the basis used in~\cite{MP} is $(a_0,a_1, z, n)$, where 
$$z:=\tfrac{1}{\al}s_{\tz,1}+\tfrac{1}{4\al}(a_1+a_{-1}-2a_0) \quad \mbox{ and }\quad 
 n:=\tfrac{1}{8\al}q.$$
\item $\IY_3(\al, \tfrac{1}{2}; \mu)$ has axet depending on the characteristic of $\F$ (see~\cite[\S 5]{split}). 
 \end{enumerate}
 \end{table}

\newpage 

 \begin{table}[H]
 \caption{{\Large $\IY_5(\al,\tfrac{1}{2})$}}\label{tableIY5}
\nomenclature{$\IY_5(\al,\tfrac{1}{2})$}{\pageref{tableIY5}} 

 \bigskip
 
{\sc Other names:}
$
\mathrm{V}_2(\al, \tfrac{1}{2}) \mbox{ in~\cite{Yabe}}  
$
\nomenclature{$\mathrm{V}_2(\al, \tfrac{1}{2})$}{\pageref{tableIY5}} 

\bigskip
 
{\tiny \makegapedcells 
$$
\begin{array}{|c|A|}
\hline
\mbox{Basis} & \multicolumn{2}{l|}{\text{Products and form ($\{i,j\}\subseteq \{-2, -1, 0, 1, 2\}\subseteq \Z$)}}  \\
\hline
\multirow{13.5}{*}{$\begin{aligned} & a_{-2}, \\& a_{-1}, \\ & a_0,\\ & a_1,\\ & a_2,\\ & s_{\tz,1} \end{aligned}$}
& a_{i}\cdot a_{i} &= a_i\\
& a_{i}\cdot a_{j} &= \tfrac{1}{2}(a_{i}+a_{j})+s_{\tz,1}, \quad |i-j|=1\\
& a_{i}\cdot a_{j} &= \tfrac{1}{2}(a_{i}+a_{j})+4s_{\tz,1}-\tfrac{1}{4}(6a_0+a_{-2}+a_2-4(a_1+a_{-1})), \quad |i-j|=2\\
& a_{i}\cdot a_{j} &= \tfrac{1}{2}(a_i+a_j)-\tfrac{3}{2}(6a_0+a_2+a_{-2}-4(a_1+a_{-1}))+9s_{\tz,1},\\
 &&(i,j)\in \{ (-2,1), (-1,2)\}\\
& a_{-2}\cdot a_2 &= 16s_{0,1}+20(a_1+a_{-1})-(\tfrac{9}{2}(a_2+a_{-2})+30a_0)\\
& a_{i}s_{\tz,1} &= \tfrac{2\al-1}{8}(4s_{\tz,1}+a_{i-1}+a_{i+1}-2a_i)
 \\
& s_{\tz,1}\cdot s_{\tz,1} &= \tfrac{(2\al-1)(2\al-3)}{32}(6a_0+a_2+a_{-2}-4(a_1+a_{-1}))  \\
 \cline{2-3}
& (a_{i}, a_{j}) &= 1\\ 
& (a_{i}, s_{\tz,1}) &= 0\\
& (s_{\tz,1}, s_{\tz,1}) &= 0\\
 \hline
\end{array} 
 $$}

\bigskip

\bigskip

\begin{center} {\sc Notes}
 \end{center}
 \bigskip
 
 \begin{enumerate}[enum_note]
     \item  Let $n:=a_2+a_{-2}-4(a_1+a_{-1})+6a_0$. Then $n$ is an annihilating element in the algebra. $\F n$ is an ideal of $\IY_5(\al, \tfrac{1}{2})$ contained in the radical of the Frobenius form. The quotient modulo this ideal is $\IY_5(\al, \tfrac{1}{2})^\times$.
     \nomenclature{$\IY_5(\al,\tfrac{1}{2})^\times$}{\pageref{tableIY5}} 
     \item  If $\ch(\F)=5$, then $\IY_5(\al,\tfrac{1}{2})\cong 5\A(\al, \tfrac{5\al-1}{8})$ (see~\cite[Theorem on p.3]{MM}). If further $\al=2$, then $\IY_5(2,\tfrac{1}{2})$ is isomorphic to a quotient of $\hatH$ (see~\cite[Theorem~11.2]{HWQ}). \label{tableIY5 char5} 
     \item  In~\cite{MP, MM} the basis used is the same as here, with $z:=s_{\tz,1}$.
     \item $\IY_5(\al, \tfrac{1}{2})$ has axet depending on the characteristic of $\F$ (see~\cite[\S 7.3]{axet}).
 \end{enumerate}
 \end{table}

\newpage

 \begin{table}[H]
 \caption{{\Large $\hatH$}}\label{tablehatH}
\nomenclature{$\hatH$}{\pageref{tablehatH}}

 \bigskip

 For $r\in \Z$, denote by $\ttr$ the congruence class $r+3\Z$ and define $\delta \colon \Z_3 \to \F$ by $\delta(\ttz) = 0$, $\delta(\ttu) = 1$ and $\delta(\ttd) = -1$.
 
 \bigskip
 
{\tiny \makegapedcells 
\[
\begin{array}{|c|A|}
\hline
 \text{Basis} & \multicolumn{2}{l|}{\text{Products and form ($i\in \Z$) }}  \\
\hline
\multirow{11.5}{*}{$\begin{aligned}  &a_{i}, i\in \Z, \\ &s_{\ttz,j},
j\in \N, 
\\ 
 &s_{\ttu,j},
j\in 3\N, \\
 &s_{\ttd,j},
j\in 3\N, 
\end{aligned}$}
& a_{i}\cdot a_i &= a_i\\
& a_{i}\cdot a_{j} &= \tfrac{1}{2}(a_{i}+a_{j})+s_{\ttz,|i-j|}\\
& a_{i}s_{\ttr,j} &= -\frac{3}{4} a_i + \frac{3}{8}( a_{i-j}+ a_{i+j}) +\frac{3}{2} s_{\ttr ,j} + \delta(\tilde{\imath} - \ttr)( s_{\ttr -\ttu,j}- s_{\ttr  + \ttu,j})\\
& s_{\ttr, j}  s_{\ttT, k}&= \frac{3}{4}( s_{\ttr, j}+ s_{\ttT, k}) - \frac{1}{8} \sum_{\ttx \in \Z_3} (s_{\ttx, |j-k|} + s_{\ttx, j+k}),
 \quad \{i,j\}\not \subseteq 3\N \\
& s_{\tta, 3j}  s_{\ttb, 3k}  &= \frac{3}{4}\sum_{h = j,k} (s_{\tta, 3h}+ s_{\ttb, 3h}- s_{-(\tta+\ttb), 3h}) \\
 &&\phantom{{}={}} \quad - \frac{3}{8} \sum_{h = |j-k|, j+k} (s_{\tta, 3h}+ s_{\ttb, 3h}- s_{-(\tta+\ttb), 3h})\\
\cline{2-3} 
& (a_{i} , a_{j}) &= 1\\ 
& (a_{i} , s_{\ttz,j}) &= 0\\
& (s_{\ttr,j}, s_{\ttT,h}) &= 0\\
 \hline
\end{array} 
 \]}
\bigskip

\bigskip

\begin{center} {\sc Notes} 
 \end{center}
 \bigskip
 
 \begin{enumerate}[enum_note]
    \item  In the table the following notation is used: $s_{\ttz, 0}:=0$, and if $j\not \in 3\N$, then $s_{\ttu, j}:=s_{\ttd, j}:=s_{\ttz, j}$. 
    \item The subspace $J$, spanned by  $\{s_{\ttz,j}-s_{\ttd,j}, s_{\ttu,j}-s_{\ttz,j} :  j\in 3\N\}$ is an ideal of $\hatH$
     \item If $\ch(\F)=5$, then  $\hatH$ is of Monster  type $(2, \tfrac{1}{2})$ (see~\cite{FM, HWQ}). $\hatH$ is called the {\it \textup{(}characteristic $5$\textup{)}  cover of the Highwater algebra}\index{cover of the Highwater algebra}\index{Highwater algebra!cover of}.
     \item If $\ch(\F)\neq 3$, then $\mathcal H: = \hatH/J$\nomenclature{$\mathcal H$}{\pageref{tablehatH}} is of Monster  type $(2, \tfrac{1}{2})$ (see~\cite{HW}). $\mathcal H$ is called the {\it Highwater algebra}\index{Highwater algebra}.
     \item If $\ch(\F)= 3$, then $\hatH$ is not primitive and $\hatH/J$ is an infinite dimensional baric Jordan algebra~\cite[Theorem 3.4]{HW}.
     \item  With respect to the same pair of generating axes $a_0$ and $a_1$, $6\Y\left (\tfrac{1}{2}, 2\right )$ is also of Monster type $(2,\tfrac{1}{2})$ since, by~\cite[Theorem~11.9]{HWQ}, $6\Y(\tfrac{1}{2},2)\cong \hatH/\F q$, where $q:=2s_{\ttz,2}+a_{-1}+a_2-(a_0+a_1)$. 
     \item \label{tablehatH quotients} By~\cite[Theorem~1.4]{HWQ}, every quotient of ${\mathcal H}$ is a $2$-generated symmetric $\mathcal M(2,\tfrac{1}{2})$-axial algebra. 
     \item Both $\hat{ \mathcal H}$ (in characteristic 5) and $\mathcal H$ have axet $X(\infty)$.
 \end{enumerate}
 \end{table}
 
\newpage

\begin{table}[H]
\caption{{\Large $4\QQ(2\bt,\bt)$}}\label{tableQ2}
\nomenclature{$4\QQ(2\bt,\bt)$}{\pageref{tableQ2}}

 \bigskip
 
{\sc Other names:}
$
\QQ_2(\bt) \mbox{ in~\cite{Joshi, FMS2}}  
$
\nomenclature{$\QQ_2(\bt)$}{\pageref{tableQ2}} 

\bigskip

{\tiny \makegapedcells 
$$
\begin{array}{|c|A|}
\hline
 \mbox{Basis} & \multicolumn{2}{l|}{\text{Products and form ($i\in \Z_4$)}}  \\
\hline
\multirow{10.5}{*}{$\begin{aligned}  &a_{-1}, \\
& a_0, \\
& a_1, \\ 
& a_2
\end{aligned}$}
& a_i\cdot a_i &= a_i \\
& a_{i}\cdot a_{j} &= \tfrac{\bt}{2}(2a_i+a_{j}-a_{-j}),  \quad i\in \{0,2\}, j\in \{1,-1\}\\
& a_{-1}\cdot a_1&= \bt(a_1+a_{-1}-a_0-a_2)\\
& a_0\cdot a_2&= 0\\
\cline{2-3}
& (a_{i} , a_{i}) &= 1, \quad i\in \{0,2\}\\ 
& (a_{i} , a_{i}) &= 2, \quad i\in \{-1,1\}\\
& (a_i, a_j) &= \bt, \quad |i-j|\in \{1,3\}\\
& (a_1, a_{-1}) &= 2\bt \\
& (a_0, a_2) &= 0 \\
 \hline
\end{array} 
 $$}
\bigskip

\bigskip

\begin{center} {\sc Notes}
 \end{center}
 \bigskip 
 \begin{enumerate}[enum_note]
     \item  $a_0$ is a $\mathcal J(\bt)$-axis, $a_1$ is an $\mathcal M(2\bt, \bt)$-axis, whence $4\QQ(2\bt,\bt)$ is of Monster  type $(2\bt,\bt)$. 
     \item  If $\bt=-\tfrac{1}{2}$, then the form has non-trivial radical $\F (a_{-1}+a_0+a_1+a_2)$. $4\QQ(-1,-\tfrac{1}{2})^\times$
     \nomenclature{$4\QQ(-1,-\tfrac{1}{2})^\times$}{\pageref{tableQ2}} is the quotient of $4\QQ(-1,-\tfrac{1}{2})$ modulo the radical.
     \item  The Frobenius form is induced by the Frobenius form of the Matsuo algebra associated to the symmetric group $S_4$ (see~\cite[\S 5.3]{Joshi}).
     \item 
     The even subalgebra of $4\QQ(2\bt,\bt)$ is isomorphic to $2\B$.
     \item  
     The odd subalgebra of $4\QQ(2\bt,\bt)$ is isomorphic to $3\C(2\bt)$ and has basis $( a_{-1}, a_1, a_0+a_2)$ which is the basis corresponding to that given in Table \ref{table2}.
     \item   In $4\QQ(-1,-\tfrac{1}{2})^\times$,  
     the even subalgebra is isomorphic to $2\B$; the odd subalgebra is isomorphic to $3\C(-1)^\times$ and the elements of the basis corresponding to the one given in Table~\ref{table2} are the images of $a_{-1}$ and $a_1$.
  \item  The basis used in~\cite[Table~6]{Joshi} is $(a_0,a_2,a_1,a_{-1})$. 
     \item  $4\QQ(2\bt,\bt)$ and $4\QQ(-1,-\tfrac{1}{2})^\times$ have axet $X(4)$.
 \end{enumerate}

 \end{table}

 \newpage 

\begin{table}[H]
\caption{{\Large $3\QQ^\prime(\tfrac{1}{3}, \tfrac{2}{3})$}}\label{tableQ2bis}
\nomenclature{$3\QQ^\prime(\tfrac{1}{3}, \tfrac{2}{3})$}{\pageref{tableQ2bis}}

\bigskip 

{\sc  Other names:}
$
 \QQ_2(\tfrac{1}{3},\tfrac{2}{3}), \mbox{ if $\ch(\F)\neq 5$},  \mbox{ and } \QQ_2(\tfrac{1}{3})^\times\oplus\la \1 \ra, \mbox{ if $\ch(\F)= 5$}, \mbox{ in~\cite{Turner1}}
$
\nomenclature{$\QQ_2(\tfrac{1}{3},\tfrac{2}{3})$}{\pageref{tableQ2bis}}
\nomenclature{$\QQ_2(\tfrac{1}{3})\oplus\la \1 \ra$}{\pageref{tableQ2bis}}

\bigskip

{\tiny \makegapedcells 
$$
\begin{array}{|c|A|}
\hline
 \mbox{Basis} & \multicolumn{2}{l|}{\text{Products and form}} \\
\hline
\multirow{14}{*}{$\begin{aligned}   \\
& a_0, \\
& a_1,\\ 
& a_2, \\
& s
\end{aligned}$}
& a_i\cdot a_i &= a_i, \quad i\in \{0,1,2\} \\
& a_i\cdot a_1 &=  \tfrac{2}{3}a_i-\tfrac{1}{6}(s-a_1), \quad i\in \{0,2\} \\
& a_0\cdot a_2&= 0\\
& a_i\cdot s &= \tfrac{2}{3}a_i+\tfrac{1}{6}(s-a_1), \quad i\in \{0,2\}\\
& a_1\cdot s&= \tfrac{2}{3}(a_0+a_2)-\tfrac{1}{3}(a_1+s)\\
\cline{2-3}
& (a_{i} , a_{i}) &= \tfrac{5}{8},  \quad i\in \{0,2\}\\ 
& (a_{1} , a_{1}) &= 1 \\
& (a_i, a_j) &= \tfrac{5}{12},  \quad |i-j|=1\\
& (a_1, s) &= \tfrac{1}{6} \\
& (a_i, s) &= \tfrac{5}{12},  \quad i\in \{0,2\}\\
& (a_0, a_2) &= 0 \\
& (s,s) &= 1 \\
 \hline
\end{array} 
 $$}
\bigskip

\bigskip

\begin{center} {\sc Notes }
 \end{center}
 \bigskip
 \begin{enumerate}[enum_note]
     \item  $a_0$ is a $\mathcal J(\tfrac{1}{3})$-axis, $a_1$ is an $\mathcal M(\tfrac{1}{3}, \tfrac{2}{3})$-axis, whence $3\QQ^\prime(\tfrac{1}{3}, \tfrac{2}{3})$ is of Monster  type $(\tfrac{1}{3}, \tfrac{2}{3})$ (see~\cite{Turner1}).
     \item   If $\ch(\F)\neq 5$, then $\mathbbm 1:=\tfrac{3}{5}(a_{-1}+a_0+a_1+a_2)$ is the identity element of $V$. Moreover, $(V, \{a_0, \mathbbm 1-a_1\})\cong 4\QQ(\tfrac{2}{3},\tfrac{1}{3})$ (see~\cite[\S 3.2.1 and Lemma~6.3]{Turner1}). 
     \item  If $\ch (\F)= 5$, then $3\QQ^\prime(\tfrac{1}{3}, \tfrac{2}{3})$ is isomorphic to the algebra $\QQ_2(\tfrac{1}{3})^\times \oplus \la\mathbbm 1\ra$ defined in~\cite[\S 3.2.2]{Turner1}. With respect to the basis $(x,y,z,\mathbbm 1)$ used in~\cite[Table~4]{Turner1}, here $a_0=y$, $a_1=\mathbbm 1-z$, $a_2=x$, and $s=x+y+z+\mathbbm 1$ (see also~\cite[Lemma 6.4]{Turner1}).
     \item If $\ch(\F)\neq 5$, then the Frobenius form coincides with the form defined on the algebra $4\QQ(\tfrac{2}{3},\tfrac{1}{3})$ in~\cite{Turner1}, rescaled by $\tfrac{5}{8}$. Note that, if $\ch(\F)=5$, then $\{a_0, a_2\}$ is contained in the radical of the Frobenius form (see~\cite[\S 3.1]{survey}).  
     \item   The even subalgebra  is isomorphic to $
     2\B$. The odd subalgebra is isomorphic to $1\A$.
     \item  The algebra $3\QQ^\prime(\tfrac{1}{3}, \tfrac{2}{3})$ has axet $X'(3)$ (see~\cite{Turner1}).
       \end{enumerate}
\end{table}

\newpage

\begin{table}[H]
\caption{{\Large $3\C^\prime(\eta, 1-\eta)$, $\eta\not \in \{0,1, \tfrac{1}{2}\}$}}\label{table3Cskew}
\nomenclature{$3\C^\prime(\eta, 1-\eta)$}{\pageref{table3Cskew}}

 \bigskip
 
{\sc Other names:}
$
3\C(\eta,1-\eta) \mbox{ in~\cite{FMS2, Turner1}}  
$
\nomenclature{$3\C(\eta,1-\eta)$}{\pageref{table3Cskew}} 

\bigskip

{\tiny \makegapedcells 
$$
\begin{array}{|l|A|}
\hline
 \mbox{Basis} & \multicolumn{2}{l|}{\text{Products and form}}  \\
\hline
\multirow{13}{*}{$\begin{aligned}  
& a_0, \\
& a_1, \\
& a^\ast
\end{aligned}$}
& a_i\cdot a_i &= a_i \quad  i\in \{0,1\}\\
& a^\ast\cdot a^\ast &=  a^\ast \\
& a_{0}\cdot a^\ast &= \tfrac{1+\eta}{2}a_0+\tfrac{1-\eta}{2}(a^\ast-a_1)\\
& a^\ast\cdot a_1&= \tfrac{1+\eta}{2}a_0-\tfrac{1-\eta}{2}(a^\ast+a_1)\\
& a_0\cdot a_1 &=  \tfrac{1+\eta}{2}a_0+\tfrac{1-\eta}{2}(a_1-a^\ast) \\
 \cline{2-3}
& (a_{0} , a_{0}) &= 2-\eta, \\
& (a_1 , a_1) &= \eta+1\\
& (a^\ast , a^\ast) &= \eta+1\\
& (a_1, a^\ast) &= \tfrac{1}{2}\eta(\eta+1) \\
& (a_0, a_1) &= \tfrac{1}{2}(2-\eta)(\eta+1)\\
& (a_0, a^\ast) &= \tfrac{1}{2}(2-\eta)(\eta+1)\\
 \hline
\end{array} 
 $$}
\bigskip

\bigskip

\begin{center} {\sc Notes }
 \end{center}
 \bigskip
 
 \begin{enumerate}[enum_note]
    \item The values $\eta\in\{0,1,\tfrac{1}{2}\}$ are excluded,  since the above structure constants define, for $\eta\in\{0,1\}$, an algebra which is not primitive and, for $\eta=\tfrac{1}{2}$, a  $\mathcal J(\tfrac{1}{2})$-algebra isomorphic to $ \J(-\tfrac{1}{8})$ where $(a_0,a_1,\tfrac{1}{4}(a_0-a_1-a^\ast))$ is the corresponding basis in Table~\ref{table2}.
     \item  $a_0$ is a $\mathcal J(1-\eta)$-axis while $a_1$ and $a^\ast$ are $\mathcal J(\eta)$-axes. Hence $3\C^\prime(\eta,1-\eta)$ can be regarded as an algebra of Monster type $(\eta, 1-\eta)$ or of Monster type $(1-\eta, \eta)$. \label{table3Cskew axes}
     \item   
     $\mathbbm 1:=\tfrac{1}{\eta}(a_1+a^\ast-a_0)$ is the identity of the algebra. 
     \item  If $3\C^\prime(\eta, 1-\eta)$ is regarded as an  algebra of Monster type $(\eta, 1-\eta)$, then $a_1^{\tau_0}=a^\ast$.  
      \item  If $3\C^\prime(\eta, 1-\eta)$ is regarded as an algebra of Monster type $(1-\eta, \eta)$, then  $a_0^{\tau_1}=\mathbbm 1-a^\ast$.
    \item  Let $\V_1=(\lla a_1,a^\ast\rra,\{a_1,a^*\})$. If $\eta\neq -1$, then $\V_1\cong 3\C(\eta)$, where $(a_1,a^\ast,\mathbbm 1 -a_0)$ is the corresponding basis in Table \ref{table2}. If $\eta=-1$, then $\V_1\cong 3\C(-1)^\times$, where $(a_1,a^\ast)$ is the corresponding basis in Table \ref{table2}.
    \item Let $\V_0=(\lla a_0,\mathbbm1-a^\ast\rra,\{a_0,\mathbbm 1-a^*\})$. If $\eta\neq 2$, then $\V_0\cong 3\C(1-\eta)$, where $(a_0,\mathbbm 1 -a^\ast,\mathbbm 1-a_1)$ is the corresponding basis in Table \ref{table2}. If $\eta=2$, then $\V_0\cong 3\C(-1)^\times$, where $(a_0,\mathbbm 1- a^\ast)$ is the corresponding basis in Table \ref{table2}.
    \item  The Frobenius form is uniquely determined up to non-zero scalars. Note that, if $\eta=-1$, then $a_1$ and $a^\ast$ belong to the radical of the form; if $\eta=2$, then $a_0$ belongs to the radical of the form.  
    \item   $3\C^\prime(\eta,1-\eta)$ was first constructed in \cite[Theorem 4.1.1]{FelixThesis}. However, it was not observed to be a skew axial algebra until \cite{Turner1}.
    \item  The algebra $3\C^\prime(\eta, 1-\eta)$ has axet $X'(3)$ (see~\cite{Turner1}).
           \end{enumerate}
\end{table}

\newpage 

 \begin{table}[H]
 \caption{{\Large $4\B(-1, \tfrac{1}{2}; \nu)^{\times }$}}\label{4Bq}
\nomenclature{$4\B(-1, \tfrac{1}{2}; \nu)^{\times }$}{\pageref{4Bq}}

 \bigskip
 
{\tiny \makegapedcells 
\[
\begin{array}{|l|A|}
\hline
 \mbox{Basis} & \multicolumn{2}{l|}{\text{Products and form}}  \\
\hline
\multirow{10.5}{*}{$\begin{aligned}  
& a_{-1}, \\
& a_0, \\
& a_1,\\
& u
\end{aligned}$}
& a_i\cdot a_i &= a_i \quad i\in \{-1,0,1\}\\
& a_{0}\cdot a_i &= \tfrac{1}{2}a_0-\tfrac{1}{4}a_i-\tfrac{3}{4}a_{-i}+\tfrac{1+\nu}{4}u,  \quad i\in \{1,-1\}\\
& a_{-1}\cdot a_1 &= -a_{-1}-a_1+\tfrac{\nu}{2} u\\
& a_i\cdot u &=  0, \quad i\in \{-1,0,1\}\\
& u^2 &= 0\\
\cline{2-3}
& (a_i , a_i) &= 1, \\
& (a_i, a_j) &= \tfrac{\al^2}{4},  \quad |i-j|=1\\
& (a_{-1}, a_1) &= \tfrac{\al}{2}\\
& (a_i,u) &= 0 \\
 \hline
\end{array} 
\]
}
\bigskip

\bigskip

\begin{center} {\sc Notes} 
\bigskip

 \end{center}
 \begin{enumerate}[enum_note]
     \item  The algebra $4\B(-1,\tfrac{1}{2}; \nu)^{\times}$ is isomorphic to the quotient of $4\B(-1,\tfrac{1}{2})$ over the ideal $\F(\nu(a_0+a_2)+(1-\nu)(a_1+a_{-1})+c)$. If $\nu\neq \tfrac{1}{2}$, it is non-symmetric. If $\nu=\tfrac{1}{2}$, then $4\B(-1,\tfrac{1}{2}; \tfrac{1}{2})^{\times}$ is isomorphic to $4\B(-1, \tfrac{1}{2})^\times$ (see~\cite[Proposition~5.21]{MM}).
     \item    $a_2=a_0^{\tau_1}=-a_0+a_1+a_{-1}-u$.
     \item   If $\nu\neq 0$, then the odd subalgebra is isomorphic to $3\C(-1)$ and has basis $(a_1,a_{-1},-a_1-a_{-1}+\nu u)$ which is the basis corresponding to that given in Table \ref{table2}. If $\nu=0$, then the odd subalgebra is isomorphic to $3\C(-1)^\times$ and has basis $(a_1,a_{-1})$ which is the basis corresponding to that given in Table \ref{table2}.
     \item    If $\nu\neq 1$, then the even subalgebra is isomorphic to $3\C(-1)$ and has basis $(a_0,a_2,-a_1-a_{-1}+\nu u)$ which is the basis corresponding to that given in Table \ref{table2}. If $\nu=1$, then the even subalgebra is isomorphic to $3\C(-1)^\times$ and has basis $(a_0,a_2)$ which is the basis corresponding to that given in Table \ref{table2}.
     \item $4\B(-1,\tfrac{1}{2}; \nu)^{\times}$ has axet $X(4)$.
\end{enumerate}
\end{table}

\chapter{Algebras of $\mathcal H$-type}
\label{H}

In this chapter we prove Theorem~\ref{HWthm}, therefore we assume, for the remainder of this chapter, that $\V=(V, \{a_0, a_1\})$ is a $\mathcal{M}(2,\tfrac{1}{2})$-axial algebra such that $\{\lmu,\lmf,\lmd,\lmdf\}=\{1\}$. 
\bigskip



\begin{lemma}\label{s=0}
 For every $r\in \Z$, $s_{\tr, 0}=0$.   
\end{lemma}
\begin{proof}
 Since $\bt=\tfrac{1}{2}$,  this follows immediately by Equation~\eqref{defs} on page~\pageref{defs}.  
\end{proof}
\begin{lemma}\label{prodottifacili}
For every $i,j\in \Z$,
 $a_ia_j=\tfrac{1}{2}(a_i+a_j)+s_{\ti , |i-j|}$.
\end{lemma}
\begin{proof}
 This follows by Equation~\eqref{defs} on page~\pageref{defs}. 
\end{proof}

\begin{lemma}\label{evenlevel}
For every  $i,r,t\in 2\Z$, $j,k\in 2\Z_{\geq 0}$ the following hold:
\begin{enumerate}
    \item $\lm_{i}=\lm_{i}^f=1$;\label{evenlevel i}
    \item  if $\ch (\F)\neq 5$ or $j\not \equiv_6 0$, then $s_{\tz,j}=s_{\tr,j}$ and $s_{\tu,j}=s_{\tr+\tu,j}$;\label{evenlevel ii}
    \item  if $\ch (\F)= 5$,  $j\equiv_6 0$, and $r\equiv_6 t$, then $s_{\tr,j}=s_{\tT,j}$ and $s_{\tr+\tu,j}=s_{\tT+\tu,j}$.\label{evenlevel iii}
\end{enumerate}
Moreover, if $\{j,k\}\not \subseteq 6\N$, then
\begin{enumerate}\setcounter{enumi}{3}
    \item  $s_{\tr, j}  s_{\tT, k}= \frac{3}{4}( s_{\tr, j}+ s_{\tT, k}) - \frac{1}{8} \sum_{l\in\{0,2,4\}} (s_{\bar{l}, |j-k|} + s_{\bar{l}, j+k})$;\label{evenlevel iv}
 \item 
 $s_{\tu+\tr, j}  s_{\tu+\tT, k}= \frac{3}{4}( s_{\tu+\tr, j}+ s_{\tu+\tT, k}) - \frac{1}{8} \sum_{l\in\{1,3,5\}} (s_{\bar{l}, |j-k|} + s_{\bar{l}, j+k})$.\label{evenlevel v}
\end{enumerate} 
\end{lemma}
\begin{proof}
Let $x\in \{e,o\}$. By hypothesis, $\lm_{a_0}(a_2)=\lmd=1$, if $x=e$, or $\lm_{a_{1}}(a_{-1})=\lmdf=1$, if $x=o$.  Since, by Lemma~\ref{lem:VeVo}, $\V_x$ is symmetric,   by Lemma~\ref{lambdasut}, either $\V_x$  is isomorphic to a quotient of $\mathcal H$, or  $
\ch(\F)=5$ and $\V_x$ is isomorphic to a quotient of $\hatH$. So the result follows by~\cite[Definition~3.1 and Lemma 3.4]{HWQ}. 
\end{proof}

For $r\in \Z$ denote by $\tilde r$ the congruence class $r+3\Z$. 
Define

\begin{equation*}
\begin{aligned}
\dl \colon \Z/3\Z &\longrightarrow \F \\
\tilde r & \longmapsto 0 \\
\tilde 0 & \longmapsto 0 \\
\tilde 1 & \longmapsto 1 \\
\tilde 2 & \longmapsto -1
\end{aligned}
\begin{aligned}
&\left.\vphantom{
    \begin{aligned}
    \dl \colon \Z/3\Z &\longrightarrow \F
    \end{aligned}
}\right. \\
&\left.\vphantom{
    \begin{aligned}
    \tilde r & \longmapsto 0
    \end{aligned}
} \right. \phantom{\rbrace}\quad \mbox{for every $\tilde r\in \Z/3\Z$ if $\ch( \F)\neq 5$,}\\
&\left.\vphantom{
\begin{aligned}
\tilde 0 & \longmapsto 0 \\
\tilde 1 & \longmapsto 1 \\
\tilde 2 & \longmapsto -1
\end{aligned}
}
\right\rbrace\quad \mbox{if $\ch( \F)= 5$.}
\end{aligned}
\end{equation*}

\begin{lemma}\label{newformulas}
Let $i\geq 2$. Suppose that for every  $l,r\in \Z$, $1\leq j\leq i$ the following hold: 
\begin{enumerate}[enum_arabic]
\item  if $\ch (\F)\neq 5$ or $j\not \equiv_3 0$, then  $s_{\tr,j}=s_{\tz,j}$;  
\item  if $\ch (\F)= 5$,  $j\equiv_3 0$, and $r\equiv_3 t$, then $s_{\tr,j}=s_{\tT,j}$; 
\item  $a_l s_{\tr,j}= -\tfrac{3}{4} a_l+\tfrac{3}{8}( a_{l-j}+ a_{l+j})+\tfrac{3}{2} s_{\tr,j}+\delta(\tilde \jmath-\tilde r)(s_{\tr-\tu,j}-s_{\tr+\tu,j}).$\label{newformulas c}
\end{enumerate}
 Then, for every $1\leq h\leq k\leq i$, 
\begin{enumerate}
\item 
$s_{\tz, h}(a_k+a_{-k})=-\tfrac{3}{4}(a_k+a_{-k})+\tfrac{3}{8} (a_{k-h}+ a_{k+h}+a_{-k-h}+ a_{-k+h})+3s_{\tz,h}$;\label{newformulas i}
\item $
\begin{aligned}[t]
u_hu_k &= -\tfrac{9}{16}a_0+\tfrac{9}{32}(a_h+a_{-h}+a_k+a_{-k}) \\
&\phantom{{}={}} \quad -\tfrac{9}{64} (a_{k-h}+ a_{k+h}+a_{-k-h}+ a_{-k+h})\\
&\phantom{{}={}} \quad +\tfrac{9}{64}(s_{\tih,|h-k|}+s_{\tih,h+k}+s_{-\tih,h+k}+s_{-\tih,|h-k|}) \\
&\phantom{{}={}} \quad -\tfrac{9}{16}s_{\tz,h} -\tfrac{9}{16}s_{\tz,k} +\tfrac{1}{4}s_{\tz,h}s_{\tz,k};
\end{aligned}
$\label{newformulas ii}
\item
$
\begin{aligned}[t]
u_hv_k &= \tfrac{3}{16}a_0-\tfrac{3}{32}(a_k+a_{-k}+a_h+a_{-h})\\
&\phantom{{}={}} \quad +\tfrac{3}{64}  (a_{k-h}+ a_{k+h}+a_{-k-h}+ a_{-k+h})\\
&\phantom{{}={}} \quad +\tfrac{3}{64} (s_{\tih,|h-k|}+s_{\tih,h+k}+s_{-\tih,h+k}+s_{-\tih,|h-k|}) \\
&\phantom{{}={}} \quad -\tfrac{3}{16}s_{\tz, k}-\tfrac{3}{16}s_{\tz,h} -\tfrac{1}{4}s_{\tz,h}s_{\tz,k};
\end{aligned}
$\label{newformulas iii}
\item 
$
\begin{aligned}[t]
v_hv_k &= \tfrac{3}{16}a_0-\tfrac{3}{32}(a_k+a_{-k}+a_h+a_{-h}) \\
&\phantom{{}={}} \quad +\tfrac{3}{64}  (a_{k-h}+ a_{k+h}+a_{-k-h}+ a_{-k+h})\\
&\phantom{{}={}} \quad +\tfrac{1}{64}(s_{\tih,|h-k|}+s_{\tih,h+k}+s_{-\tih,h+k}+s_{-\tih,|h-k|}) \\
&\phantom{{}={}} \quad -\tfrac{1}{16}s_{\tz,h}-\tfrac{1}{16}s_{\tz, k}+\tfrac{1}{4}s_{\tz,h}s_{\tz,k}.
\end{aligned}
$\label{newformulas iv}
\end{enumerate}
Similarly, 
\begin{enumerate}[resume*]
\item
$
\begin{aligned}[t]
s_{\tu, h}(a_{1+k}+a_{1-k}) &= -\tfrac{3}{4}(a_{1+k}+a_{1-k})+\tfrac{3}{8} (a_{1+k-h}+ a_{1+k+h}+a_{1-k-h} \\
&\phantom{{}={}-\tfrac{3}{4}(a_{1+k}+a_{1-k})+\tfrac{3}{8} (} \quad + a_{1-k+h})+3s_{\tu,h};
\end{aligned}
$\label{newformulas v}
\item
$
\begin{aligned}[t]
\tilde u_h\tilde u_k &= -\tfrac{9}{16}a_1+\tfrac{9}{32}(a_{1-h}+a_{1+h}+a_{1-k}+a_{1+k}) \\
&\phantom{{}={}} \quad -\tfrac{9}{64} (a_{1+k-h}+ a_{1+k+h} +a_{1-k-h}+ a_{1-k+h}) \\
&\phantom{{}={}} \quad +\tfrac{9}{64}(s_{\tih+\tu,|h-k|}+s_{\tih+\tu,h+k}+s_{\tu-\tih,h+k}+s_{\tu-\tih,|h-k|}) \\
&\phantom{{}={}} \quad -\tfrac{9}{16}s_{\tu,h} -\tfrac{9}{16}s_{\tu,k} +\tfrac{1}{4}s_{\tu,h}s_{\tu,k};
\end{aligned}
$\label{newformulas vi}
\item
$
\begin{aligned}[t]
\tilde u_h\tilde v_k &= \tfrac{3}{16}a_1-\tfrac{3}{32}(a_{1+k}+a_{1-k}+a_{1+h}+a_{1-h}) \\
&\phantom{{}={}} \quad +\tfrac{3}{64}(a_{1+k-h}+ a_{1+k+h} +a_{1-k-h}+ a_{1-k+h}) \\
&\phantom{{}={}} \quad +\tfrac{3}{64} (s_{\tih+\tu,|h-k|}+s_{\tih+\tu,h+k}+s_{\tu-\tih,h+k}+s_{\tu-\tih,|h-k|})\\
&\phantom{{}={}} \quad -\tfrac{3}{16}s_{\tu, k}-\tfrac{3}{16}s_{\tu,h} -\tfrac{1}{4}s_{\tu,h}s_{\tu,k};
\end{aligned}
$\label{newformulas vii}

\item
$
\begin{aligned}[t]
\tilde v_h\tilde v_k &= \tfrac{3}{16}a_1-\tfrac{3}{32}(a_{1+k}+a_{1-k}+a_{1+h}+a_{1-h}) \\
&\phantom{{}={}} \quad +\tfrac{3}{64}  (a_{1+k-h}+ a_{1+k+h}+a_{1-k-h}+ a_{1-k+h}) \\
&\phantom{{}={}} \quad +\tfrac{1}{64}(s_{\tih+\tu,|h-k|}+s_{\tih+\tu,h+k}+s_{\tu-\tih,h+k}+s_{\tu-\tih,|h-k|})\\
&\phantom{{}={}} \quad -\tfrac{1}{16}s_{\tu,h}-\tfrac{1}{16}s_{\tu, k}+\tfrac{1}{4}s_{\tu,h}s_{\tu,k}.
\end{aligned}
$\label{newformulas viii}
\end{enumerate}
\end{lemma}
\begin{proof}
We prove~\ref{newformulas i}-\ref{newformulas iv}, the proof of~\ref{newformulas v}-\ref{newformulas viii} is similar.
If $\ch (\F)=5$, the result follows by~\cite[Lemma~11]{FM}. Let $\ch (\F)\neq 5$.
By~\ref{newformulas c},
\begin{align*}\label{newformula1}
s_{\tz, h}(a_k+a_{-k})&=-\tfrac{3}{4} a_k+\tfrac{3}{8}( a_{k-h}+ a_{h+k}) +\tfrac{3}{2}s_{\tz,h}  \nonumber\\
&\phantom{{}={}} \quad -\tfrac{3}{4} a_{-k}+\tfrac{3}{8}( a_{-k-h}+ a_{h-k}) +\tfrac{3}{2}s_{\tz,h} \nonumber\\
&=-\tfrac{3}{4}(a_k+a_{-k})+\tfrac{3}{8} (a_{k-h}+ a_{k+h}+a_{-k-h}+ a_{-k+h})+3s_{\tz,h}, \nonumber
\end{align*}
giving~\ref{newformulas i}.
Using~\ref{newformulas i} and the explicit expressions of $u_k$ and $u_h$ given in Lemma~\ref{ui}, we obtain
\begin{align*}
4u_hu_k&=(-\tfrac{3}{2}a_0+\tfrac{3}{4}a_h+\tfrac{3}{4}a_{-h}-s_{\tz, h})(-\tfrac{3}{2}a_0+\tfrac{3}{4}a_k+\tfrac{3}{4}a_{-k}-s_{\tz, k})\\
&=\tfrac{9}{4}a_0-\tfrac{9}{8}a_0(a_h+a_{-h}+a_k+a_{-k})+\tfrac{3}{2}a_0s_{\tz,k}+\tfrac{3}{2}a_0s_{\tz,h} \\
&\phantom{{}={}} \quad +\tfrac{9}{16}(a_h+a_{-h})(a_k+a_{-k})-\tfrac{3}{4}s_{\tz,h}(a_k+a_{-k})-\tfrac{3}{4}s_{\tz, k}(a_h+a_{-h})\\
&\phantom{{}={}} \quad +s_{\tz,h}s_{\tz,k} \\
&=\tfrac{9}{4}a_0-\tfrac{9}{4}a_0-\tfrac{9}{16}(a_h+a_{-h}+a_k+a_{-k})-\tfrac{9}{4}s_{\tz,h}-\tfrac{9}{4}s_{\tz, k}\\
&\phantom{{}={}} \quad -\tfrac{9}{8}a_0+\tfrac{9}{16}(a_k+a_{-k})+\tfrac{9}{4}s_{\tz,k}-\tfrac{9}{8}a_0+\tfrac{9}{16}(a_h+a_{-h})+\tfrac{9}{4}s_{\tz,h}\\
&\phantom{{}={}} \quad +\tfrac{9}{16}(a_h+a_{-h}+a_k+a_{-k})\\
&\phantom{{}={}} \quad +\tfrac{9}{16}(s_{\tih,|h-k|}+s_{\tih,h+k}+s_{-\tih,h+k}+s_{-\tih,|h-k|}) \\
&\phantom{{}={}} \quad +\tfrac{9}{16}(a_k+a_{-k})-\tfrac{9}{32} (a_{k-h}+ a_{k+h}+a_{-k-h}+ a_{-k+h})-\tfrac{9}{4}s_{\tz,h} \\
&\phantom{{}={}} \quad +\tfrac{9}{16}(a_h+a_{-h})-\tfrac{9}{32} (a_{h-k}+ a_{h+k}+a_{-h-k}+ a_{-h+k})-\tfrac{9}{4}s_{\tz,k} \\
&\phantom{{}={}} \quad +s_{\tz,h}s_{\tz,k}.
\end{align*}
Rearranging the summands we get claim~\ref{newformulas ii}. Claims~\ref{newformulas iii} and~\ref{newformulas iv} are obtained in a similar way.
\begin{align*}
4u_hv_k&=(-\tfrac{3}{2}a_0+\tfrac{3}{4}a_h+\tfrac{3}{4}a_{-h}-s_{\tz, h})(-\tfrac{1}{2}a_0+\tfrac{1}{4}a_k+\tfrac{1}{4}a_{-k}+s_{\tz, k})\\
&=\tfrac{3}{4}a_0-\tfrac{3}{8}a_0(a_h+a_{-h}+a_k+a_{-k})-\tfrac{3}{2}a_0s_{\tz,k}+\tfrac{1}{2}a_0s_{\tz,h} \\
&\phantom{{}={}} \quad +\tfrac{3}{16}(a_h+a_{-h})(a_k+a_{-k})-\tfrac{1}{4}s_{\tz,h}(a_k+a_{-k})+\tfrac{3}{4}s_{\tz, k}(a_h+a_{-h})\\
&\phantom{{}={}} \quad -s_{\tz,h}s_{\tz,k} \\
&= \tfrac{3}{4}a_0-\tfrac{3}{4}a_0-\tfrac{3}{16}(a_h+a_{-h}+a_k+a_{-k})-\tfrac{3}{4}s_{\tz,h}-\tfrac{3}{4}s_{\tz, k}\\
&\phantom{{}={}} \quad +\tfrac{9}{8}a_0-\tfrac{9}{16}(a_k+a_{-k})-\tfrac{9}{4}s_{\tz,k}-\tfrac{3}{8}a_0+\tfrac{3}{16}(a_h+a_{-h})+\tfrac{3}{4}s_{\tz,h} \\
&\phantom{{}={}} \quad +\tfrac{3}{16}(a_h+a_{-h}+a_k+a_{-k})\\
&\phantom{{}={}} \quad +\tfrac{3}{16}(s_{\tih,|h-k|}+s_{\tih,h+k}+s_{-\tih,h+k}+s_{-\tih,|h-k|}) \\
&\phantom{{}={}} \quad +\tfrac{3}{16}(a_k+a_{-k})-\tfrac{3}{32} (a_{k-h}+ a_{k+h}+a_{-k-h}+ a_{-k+h})-\tfrac{3}{4}s_{\tz,h} \\
&\phantom{{}={}} \quad -\tfrac{9}{16}(a_h+a_{-h})+\tfrac{9}{32} (a_{h-k}+ a_{h+k}+a_{-h-k}+ a_{-h+k})+\tfrac{9}{4}s_{\tz,k} \\
&\phantom{{}={}} \quad -s_{\tz,h}s_{\tz,k}.
\end{align*}

\begin{align*}
4v_hv_k&= (-\tfrac{1}{2}a_0+\tfrac{1}{4}a_h+\tfrac{1}{4}a_{-h}+s_{\tz, h})(-\tfrac{1}{2}a_0+\tfrac{1}{4}a_k+\tfrac{1}{4}a_{-k}+s_{\tz, k})\\
&= \tfrac{1}{4}a_0-\tfrac{1}{8}a_0(a_h+a_{-h}+a_k+a_{-k})-\tfrac{1}{2}a_0s_{\tz,k}-\tfrac{1}{2}a_0s_{\tz,h} \\
&\phantom{{}={}} \quad +\tfrac{1}{16}(a_h+a_{-h})(a_k+a_{-k})+\tfrac{1}{4}s_{\tz,h}(a_k+a_{-k})+\tfrac{1}{4}s_{\tz, k}(a_h+a_{-h})\\
&\phantom{{}={}} \quad +s_{\tz,h}s_{\tz,k} \\
&= \tfrac{1}{4}a_0-\tfrac{1}{4}a_0-\tfrac{1}{16}(a_h+a_{-h}+a_k+a_{-k})-\tfrac{1}{4}s_{\tz,h}-\tfrac{1}{4}s_{\tz, k}\\
&\phantom{{}={}} \quad +\tfrac{3}{8}a_0-\tfrac{3}{16}(a_k+a_{-k})-\tfrac{3}{4}s_{\tz,k}+\tfrac{3}{8}a_0-\tfrac{3}{16}(a_h+a_{-h})-\tfrac{3}{4}s_{\tz,h} \\
&\phantom{{}={}} \quad +\tfrac{1}{16}(a_h+a_{-h}+a_k+a_{-k})\\
&\phantom{{}={}} \quad +\tfrac{1}{16}(s_{\tih,|h-k|}+s_{\tih,h+k}+s_{-\tih,h+k}+s_{-\tih,|h-k|}) \\
&\phantom{{}={}} \quad -\tfrac{3}{16}(a_k+a_{-k})+\tfrac{3}{32} (a_{k-h}+ a_{k+h}+a_{-k-h}+ a_{-k+h})+\tfrac{3}{4}s_{\tz,h} \\
&\phantom{{}={}} \quad -\tfrac{3}{16}(a_h+a_{-h})+\tfrac{3}{32} (a_{h-k}+ a_{h+k}+a_{-h-k}+ a_{-h+k})+\tfrac{3}{4}s_{\tz,k} \\
&\phantom{{}={}} \quad +s_{\tz,h}s_{\tz,k}.\qedhere
\end{align*}
\end{proof}

\begin{proposition}\label{exclaim}
In the hypotheses of Theorem~\ref{HWthm}, for every $i,j\in \Z_+$, $j< i$, $r,l\in \Z$, and $t\in \{0,1,2\}$  the following assertions hold:
\begin{enumerate}
\item $\lm_i=\lm_i^f=1$;\label{exclaim i} 
\item if $\ch (\F)\neq 5$, then $s_{\tr,i}=s_{\tz,i}$;\label{exclaim ii} 
\item if $\ch (\F)= 5$ and $i\not \equiv_3 0$, then $s_{\tr,i}=s_{\tz,i}$;\label{exclaim iii} 
\item if $\ch (\F)= 5$, $i\equiv_3 0$, and $t\equiv_3 r$,  then $s_{\tr,i}=s_{\tT,i}$;\label{exclaim iv} 
\item $\lambda_{a_0}(s_{\tr,i})=0$; \label{exclaim v} 
\item $
 a_l s_{\tr,i}= -\tfrac{3}{4} a_l+\tfrac{3}{8}( a_{l-i}+ a_{l+i})+\tfrac{3}{2}s_{\tr,i}+\delta(\tilde \imath-\tilde r)(s_{\tr-\tu,i}-s_{\tr+\tu,i});
 $\label{exclaim vi} 
 \item  $s_{\tr, j}s_{\tr, i-j}=\tfrac{3}{4}(s_{\tr, j}+s_{\tr, i-j})-\tfrac{3}{8}s_{\tr, |i-2j|}-\tfrac{3}{4}s_{\tr, i}+\tfrac{3}{16}(s_{\tr+\tu, i}+s_{\tr+\bar 2, i}).$\label{exclaim vii} 
\end{enumerate}
\end{proposition}
\begin{proof}
We proceed by induction on $i$.
\medskip

\noindent{\bf Step 1.} 
{\it Assume $i\in\{1,2\}$, then parts \ref{exclaim i}-\ref{exclaim vii} hold.}
\medskip

Note that, since $i\not \equiv_3 0$, we do not need to prove \ref{exclaim iv}. By hypothesis $$\lm_i=\lm_i^f=1,$$ giving \ref{exclaim i}.  By Equation~\eqref{defs} on page~\pageref{defs}, for every $r\in \Z$, 
\begin{equation} \label{s01}
s_{\tr,1}=s_{\tz,1},
\end{equation}
and, since,  by~\cite[Corollary~7.2]{FMS3}, $s_{\tz,2}=s_{\tu,2}$, for every $r\in \Z$,
\begin{equation}\label{s12x}
s_{\tr,2}=s_{\tz,2},
\end{equation}
giving \ref{exclaim ii} and \ref{exclaim iii}. Since $\bt=\tfrac{1}{2}$, by Lemma~\ref{s}\ref{s_2},
\[
\lm_{a_0}(s_{\tz,i})=0,
\]
giving \ref{exclaim v}. Since $\al=2$ and $\bt=\tfrac{1}{2}$, by  Lemma~\ref{primo}\ref{primo_1},
\begin{equation}\label{a0s1x}
a_0 s_{\tz,i}= -\tfrac{3}{4} a_0+\tfrac{3}{8}( a_{-i}+ a_{i})+\tfrac{3}{2}s_{\tz,i}.
\end{equation}
By Equations~\eqref{s01} and~\eqref{s12x}, taking the orbits of both members of Equation~\eqref{a0s1x} under $\Miy(\V)$, we get \ref{exclaim vi}. Finally, since, by hypothesis  $0<j<i<2$, the only possibility is $j=1$ and $i=2$. By Equation~\eqref{defs} on page~\pageref{defs} 
$s_{\tr,1}=s_{\tz, 1}$ and, by Lemma~\ref{s=0},  $s_{\tr,0}=0$, 
thus, part \ref{exclaim vii} reduces to proving the equality  
$$s_{\tz, 1}s_{\tz, 1}=\tfrac{3}{4}(s_{\tz, 1}+s_{\tz, 1})-\tfrac{3}{4}s_{\tz, 2}+\tfrac{3}{16}(s_{\tu, 2}+s_{\bar 2, 2}),$$
 which follows by the second formula in~\cite[Lemma~6.8]{FMS3} (computed for the parameters $\al=2$ and  $\bt=\tfrac{1}{2}$)  and Equation~\eqref{s12x}.
\medskip

Next assume  $i\geq 3$ and the result true for every $m\leq i$. Let $1\leq h,k\leq i$ be such that $h+k=i+1$. In particular
\begin{equation}\label{hk}
h\equiv_{i+1} -k
\end{equation}
\medskip


\noindent{\bf Step 2.} {\it The following identities hold:}
\begin{enumerate}[enum_arabic]
\item  $\lambda_{a_0}(u_hu_{k})=\lambda_{a_0}(u_hv_{k})=0$;\label{step2 a}
\item  $\lambda_{a_0}(s_{\tih,h+k})=\tfrac{1}{2}(\lm_{h+k}-1)$; \label{step2 b} 
\item $\lm_{a_0}(s_{\tz,h}s_{\tz,k})=\tfrac{9}{16}(\lm_{h+k}-1)$; \label{step2 c}
\item  $\lambda_{a_0}(v_hv_k)=\tfrac{1}{4}(\lm_{h+k}-1)$. \label{step2 d}
\end{enumerate}
\medskip

 Claim \ref{step2 a} follows since, by the fusion law, $u_hu_{k}$ and $u_hv_{k}$ are a $0$- and a $2$-eigenvector for $\ad_{a_0}$. Adding equations  in parts \ref{newformulas ii} and  \ref{newformulas iii} of Lemma~\ref{newformulas}, we get 
\begin{align}\label{u+v}
u_hu_{k}+u_hv_{k}&=-\tfrac{3}{8}a_0+\tfrac{3}{16}(a_h+a_{-h}+a_k+a_{-k}) \nonumber \\
&\phantom{{}={}} \quad -\tfrac{3}{32} (a_{k-h}+ a_{k+h}+a_{-k-h}+ a_{-k+h})\\
&\phantom{{}={}} \quad +\tfrac{3}{16}(s_{\tih,|h-k|}+s_{\tih,h+k}+s_{-\tih,h+k}+s_{-\tih,|h-k|}) -\tfrac{3}{4}s_{\tz,k} -\tfrac{3}{4}s_{\tz,h},\nonumber 
\end{align}
Since, by claim~\ref{step2 a}, $\lm_{a_0}(u_hu_k)=\lm_{a_0}(u_hv_k)=0$, applying $\lambda_{a_0}$ to both members of Equation~\eqref{u+v}, by the linearity of  $\lambda_{a_0}$, we get 
\begin{align} \label{new14}
\nonumber 0 &= -\tfrac{3}{8}\lm_{a_0}(a_0)\\
\nonumber &\phantom{{}={}} \quad +\tfrac{3}{16}(\lm_{a_0}(a_k)+\lm_{a_0}(a_{-k})+\lm_{a_0}(a_h)+\lm_{a_0}(a_{-h}))\\
&\phantom{{}={}} \quad -\tfrac{3}{32} (\lm_{a_0}(a_{k-h})+ \lm_{a_0}(a_{h+k})+\lm_{a_0}(a_{-h-k})+ \lm_{a_0}(a_{-k+h}))\\
\nonumber&\phantom{{}={}} \quad +\tfrac{3}{16}(\lm_{a_0}(s_{\tih,|k-h|})+\lm_{a_0}(s_{\tih,h+k})+\lm_{a_0}(s_{-\tih,h+k})+\lm_{a_0}(s_{-\tih,|k-h|}))\\
\nonumber &\phantom{{}={}} \quad -\tfrac{3}{4}(\lm_{a_0}(s_{\tz,h})+\lm_{a_0}(s_{\tz,k})).
\end{align}
By definition $\lm_{a_0}(a_0)= 1$, by the inductive hypothesis and parts \ref{s_3} and \ref{s_4} of Lemma~\ref{s}, 
$$\lm_{a_0}(a_k) = \lm_{a_0}(a_{-k}) = \lm_{a_0}(a_h) =\lm_{a_0}(a_{-h}) = \lm_{a_0}(a_{k-h})=\lm_{a_0}(a_{-k+h})=1,$$
and 
$$\lm_{a_0}(s_{\tih,|k-h|})= \lm_{a_0}(s_{-\tih,|k-h|})=\lm_{a_0}(s_{\tz,h})=\lm_{a_0}(s_{\tz,k}) =0,$$ hence, again by parts \ref{s_3} and \ref{s_4} of Lemma~\ref{s},   Equation~\eqref{new14} becomes  
$$0=\tfrac{3}{16}-\tfrac{3}{16}\lm_{h+k}+\tfrac{3}{8}\lm_{a_0}(s_{\tih,h+k}),$$
giving \ref{step2 b}. 
Similarly,  by claims \ref{step2 a} and \ref{step2 b}, Lemma~\ref{newformulas}\ref{newformulas ii}, and the inductive hypothesis, we get
\begin{align*}
0&=\lm_{a_0}(u_hu_k)\\
&=\tfrac{9}{32}-\tfrac{9}{32}\lm_{h+k}+\tfrac{9}{32}\lm_{a_0}(s_{\tih,h+k})+\tfrac{1}{4}\lm_{a_0}(s_{\tz,h}s_{\tz,k})\\
&=\tfrac{9}{32}-\tfrac{9}{32}\lm_{h+k}+\tfrac{9}{64}(\lm_{h+k}-1)+\tfrac{1}{4}\lm_{a_0}(s_{\tz,h}s_{\tz,k})\\
&= -\tfrac{9}{64}(\lm_{h+k}-1)+\tfrac{1}{4}\lm_{a_0}(s_{\tz,h}s_{\tz,k}),
\end{align*}
giving \ref{step2 c}. 
By Lemma~\ref{newformulas}\ref{newformulas iv}, the inductive hypothesis, Lemma~\ref{s}, and claims \ref{step2 b} and \ref{step2 c}, 
\begin{align*}
\lambda_{a_0}(v_hv_k)&=\tfrac{3}{16}-\tfrac{3}{8}+\tfrac{3}{32}+\tfrac{3}{32}\lm_{h+k}+\tfrac{1}{32}\lm_{a_0}(s_{\tih,h+k})+\tfrac{1}{4}\lm_{a_0}(s_{\tz,h}s_{\tz,k})\\
&=\tfrac{3}{32}(\lm_{h+k}-1)+\tfrac{1}{64}(\lm_{h+k}-1)+\tfrac{9}{64}(\lm_{h+k}-1)\nonumber \\
&=\tfrac{1}{4}(\lm_{h+k}-1),\nonumber 
 \end{align*}
 giving \ref{step2 d}.
 
 \medskip
\noindent{\bf Step 3.} {\it The following identity holds:}
\begin{enumerate}[resume*]
 \item $a_0(s_{\tih,|h-k|}+s_{-\tih,|h-k|})=-\tfrac{3}{2}a_0+\tfrac{3}{4}(a_{-|h-k|}+a_{|h-k|})+3s_{\tz,|h-k|}$.\label{step3 e}
\end{enumerate}

\medskip

If $\ch (\F)\neq 5$ or $|h-k|\not \equiv_3 0$, by the inductive hypotheses and parts \ref{exclaim ii}, \ref{exclaim iii}, and \ref{exclaim v}, 
\[
s_{\tih, |h-k|}=s_{-\tih, |h-k|}=s_{\tz, |h-k|}
\]
and  
\[
a_0s_{\tz,|h-k|}=-\tfrac{3}{4}a_0+\tfrac{3}{8}(a_{-|h-k|}+a_{|h-k|})+\tfrac{3}{2}s_{\tz, |h-k|},
\]
and \ref{step3 e} follows. Assume $\ch (\F)=5$ and $|h-k|\equiv_3 0$. Then, by the inductive hypotheses and parts \ref{exclaim iii} and \ref{exclaim v}, 
\begin{align*}
a_0(s_{\tih,|h-k|}+s_{-\tih,|h-k|}) &= -\tfrac{3}{4}a_0+\tfrac{3}{8}(a_{-|h-k|}+a_{|h-k|})+\tfrac{3}{2}s_{\tih, |h-k|}\\
&\phantom{{}={}} \quad +\delta(-\tilde h)(s_{\tilde h-\tilde 1, |h-k|}-s_{\tilde h+\tilde 1, |h-k|})\\
&\phantom{{}={}} \quad -\tfrac{3}{4}a_0+\tfrac{3}{8}(a_{-|h-k|}+a_{|h-k|})+\tfrac{3}{2}s_{-\tih, |h-k|}\\
&\phantom{{}={}} \quad +\delta(\tilde h)(s_{-\tilde h-\tilde 1, |h-k|}-s_{-\tilde h+\tilde 1, |h-k|})\\
&= -\tfrac{3}{2}a_0+\tfrac{3}{4}(a_{-|h-k|}+a_{|h-k|}) + c
\end{align*} 
where
\begin{align*}
c &:=  \delta(-\tilde h)(s_{\tilde h-\tilde 1, |h-k|}-s_{\tilde h+\tilde 1, |h-k|})+\delta(\tilde h)(s_{-\tilde h-\tilde 1, |h-k|}-s_{-\tilde h+\tilde 1, |h-k|})\\
&\phantom{{}={}} \quad +\tfrac{3}{2}(s_{\tilde h, |h-k|}+s_{-\tilde h, |h-k|}).
\end{align*}
Thus the result follows if we show that 
$$
c=3s_{\tz, |h-k|}.
$$
This is immediate  if  $h\equiv_3 0$, since by definition, $\delta(\tilde 0)=0$.  If $h\not \equiv_3 0$, then $\{\tilde h, -\tilde h\}=\{\tilde 1, \tilde 2\}$ and we may without loss of generality assume that $\tilde h=\tilde1$ and $-\tilde h=\tilde2$. Since $\ch(\F)=5$, $-2\equiv_5 3$, whence  
\begin{align*}
c&= \delta(-\tilde 1)(s_{\tilde 1-\tilde 1, |h-k|}-s_{\tilde 1 +\tilde 1, |h-k|})+\delta(\tilde 1)(s_{-\tilde 1-\tilde 1, |h-k|}-s_{-\tilde 1 +\tilde 1, |h-k|}) \\
&\phantom{{}={}} \quad +\tfrac{3}{2}(s_{\tilde1, |h-k|}+s_{\tilde 2, |h-k|})\\
&= -(s_{\tilde 0, |h-k|}-s_{\tilde 2, |h-k|})+(s_{\tilde 1, |h-k|}-s_{\tilde 0, |h-k|})-(s_{\tilde1, |h-k|}+s_{\tilde 2, |h-k|})\\
&= -2s_{\tilde 0, |h-k|}=3s_{\tilde 0, |h-k|}.
\end{align*}

\medskip

\noindent{\bf Step 4.} {\it The following identities hold:} 
\begin{enumerate}[resume*]
    \item $ \begin{aligned}[t]
a_0s_{\tih,h+k}&=(\tfrac{1}{4}-\lambda_{h+k})a_0+\tfrac{3}{8}( a_{k+h}+a_{-k-h})+\tfrac{3}{2}s_{0,h+k}\\
&\phantom{{}={}} \quad +\tfrac{1}{4}s_{-\tih,h+k}-\tfrac{1}{4}s_{\tih,h+k};
\end{aligned}$\label{step4 f}
    \item $ \begin{aligned}[t]
a_0s_{-\tih,h+k}&=(\tfrac{1}{4}-\lambda_{h+k})a_0+\tfrac{3}{8}( a_{k+h}+a_{-k-h})+\tfrac{3}{2}s_{0,h+k}\\
&\phantom{{}={}} \quad -\tfrac{1}{4}s_{-\tih,h+k}+\tfrac{1}{4}s_{\tih,h+k};
\end{aligned}$\label{step4 g}
 \item $ \begin{aligned}[t]
s_{\tz,h}s_{\tz,k}&=-\tfrac{3}{4}(1-\lm_{i+1})a_0+\tfrac{3}{4}(s_{\tz,k}+s_{\tz,h}-s_{\tz, |k-h|}-s_{\tz, i+1})\\
&\phantom{{}={}} \quad +\tfrac{3}{16}(s_{\tih,|h-k|}+s_{\tik,|h-k|}+s_{\tih,i+1}+s_{\tik,i+1}); 
\end{aligned}$\label{step4 h}
\item $ \begin{aligned}[t]
s_{\tu,h}s_{\tu,k}&=-\tfrac{3}{4}(1-\lm_{i+1}^f)a_1+\tfrac{3}{4}(s_{\tu,k}+s_{\tu,h}-s_{\tu, |k-h|}-s_{\tu, i+1})\\
&\phantom{{}={}} \quad +\tfrac{3}{16}(s_{\tu+\tih, |k-h|}+s_{\tu-\tih, |k-h|}+s_{\tu-\tih,i+1}+s_{\tu-\tik,i+1}). 
\end{aligned}$\label{step4 j}
\end{enumerate}

\medskip

By parts \ref{newformulas ii} and \ref{newformulas iv} of Lemma~\ref{newformulas}, and claim \ref{step2 d}, 
\begin{align*}
u_h u_k-v_h v_k+\lambda_{a_0}(v_hv_k)a_0 &= \tfrac{3}{8}(a_h+a_{-h}+a_k+a_{-k})\\
&\phantom{{}={}} \quad -\tfrac{3}{16} (a_{k-h}+ a_{k+h}+a_{-k-h}+ a_{-k+h})\\
&\phantom{{}={}} \quad +\tfrac{1}{8}(s_{\tih,|h-k|}+s_{\tih,h+k}+s_{-\tih,h+k}+s_{-\tih,|h-k|})\\
&\phantom{{}={}} \quad -\tfrac{1}{2}s_{\tz,h} -\tfrac{1}{2}s_{\tz,k}+(\tfrac{1}{4}\lm_{h+k}-1)a_0.
\end{align*}
Since, by the fusion law, $u_h u_k-v_h v_k+\lambda_{a_0}(v_hv_k)a_0$ is a $0$-eigenvector for $\ad_{a_0}$ and $a_0a_0=a_0$, multiplying by $a_0$ the above equation, we get 
\begin{align*}
0 &= \tfrac{3}{8}a_0(a_h+a_{-h}+a_k+a_{-k})\\
&\phantom{{}={}} \quad -\tfrac{3}{16} a_0(a_{k-h}+ a_{k+h}+a_{-k-h}+ a_{-k+h})\\
&\phantom{{}={}} \quad +\tfrac{1}{8}a_0(s_{\tih,|h-k|}+s_{\tih,h+k}+s_{-\tih,h+k}+s_{-\tih,|h-k|})\\
&\phantom{{}={}} \quad -\tfrac{1}{2}a_0s_{\tz,h} -\tfrac{1}{2}a_0s_{\tz,k}+(\tfrac{1}{4}\lm_{h+k}-1)a_0,
\end{align*}
whence 
\begin{align} \label{a0*sum}
a_0(s_{\tih,h+k}+s_{-\tih,h+k}) &= -8(\tfrac{1}{4}\lm_{h+k}-1)a_0-3a_0(a_h+a_{-h}+a_k+a_{-k})\nonumber  \\
&\phantom{{}={}} \quad +\tfrac{3}{2} a_0(a_{k-h}+ a_{k+h}+a_{-k-h}+ a_{-k+h})\\
&\phantom{{}={}} \quad -a_0(s_{\tih,|h-k|}+s_{-\tih,|h-k|})+4a_0s_{\tz,h}+4a_0s_{\tz,k}.\nonumber 
\end{align}
By Lemma~\ref{prodottifacili}, the inductive hypothesis, and claim \ref{step3 e},  the last term of the above equation is 
\begin{multline*} 
(8-2\lm_{h+k})a_0 -3(2s_{\tz,h}+2s_{\tz,k}+2a_0+\tfrac{1}{2}(a_h+a_{-h}+a_k+a_{-k})) \nonumber\\
\shoveleft{ +\tfrac{3}{4}(a_{k-h}+ a_{k+h}+a_{-k-h}+ a_{-k+h})+3(s_{\tz,h+k}+s_{\tz,|h-k|})+3a_0} \\
\shoveleft{ +\tfrac{3}{2}a_0-\tfrac{3}{4}(a_{-|h-k|}+a_{|h-k|})-3s_{\tz,|h-k|}} \\
 \shoveleft{ -6a_0+\tfrac{3}{2}(a_{-h}+a_{h}+a_{-k}+a_{k})+6s_{\tz,h} +6s_{\tz,k}} \\
= (\tfrac{1}{2}-2\lambda_{h+k})a_0+\tfrac{3}{4}( a_{k+h}+a_{-k-h})+3s_{0,h+k},\nonumber 
\end{multline*} 
whence
\[
a_0(s_{\tih,h+k}+s_{-\tih,h+k})=(\tfrac{1}{2}-2\lambda_{h+k})a_0+\tfrac{3}{4}( a_{k+h}+a_{-k-h})+3s_{0,h+k}.
\]
On the other hand, since $s_{\tih,h+k}-s_{-\tih,h+k}$ is a $\tfrac{1}{2}$-eigenvector for $\ad_{a_0}$, 
\begin{equation}\label{a0*diff}
a_0(s_{\tih,h+k}-s_{-\tih,h+k})=\tfrac{1}{2}(s_{\tih,h+k}-s_{-\tih,h+k}).
\end{equation} 
Taking the sum and the difference of both members of Equations~\eqref{a0*sum} and~\eqref{a0*diff}, we get \ref{step4 f} and \ref{step4 g}.
By the fusion law, $u_hu_{k}$ (respectively $u_hv_{k}$) is a $0$-eigenvector (respectively $2$-eigenvector) for $\ad_{a_0}$, thus
$$0=a_0(u_hu_{k}+u_hv_{k})-2u_hv_k.$$  Substituting in the second term of  the above equation the expression of $(u_hu_{k}+u_hv_{k})$ given in Equation~\eqref{u+v} and  the expression of $u_hv_k$ given in  Lemma~\ref{newformulas}\ref{newformulas iii}, we get  
\begin{align*}
0&= a_0\big (-\tfrac{3}{8}a_0+\tfrac{3}{16}(a_h+a_{-h}+a_k+a_{-k})-\tfrac{3}{32} (a_{k-h}+ a_{k+h}+a_{-k-h}+ a_{-k+h}) \\
&\phantom{{}={}\big(} \quad +\tfrac{3}{16}(s_{\tih,|h-k|}+s_{\tih,h+k}+s_{-\tih,h+k}+s_{-\tih,|h-k|}) -\tfrac{3}{4}s_{\tz,k} -\tfrac{3}{4}s_{\tz,h} \big )
\\
&\phantom{{}={}} \  -2\big (\tfrac{3}{16}a_0-\tfrac{3}{32}(a_k+a_{-k}+a_h+a_{-h})+\tfrac{3}{64}  (a_{k-h}+ a_{k+h}+a_{-k-h}+ a_{-k+h}) \\
&\phantom{{}={}\ -2\big(}  +\tfrac{3}{64} (s_{\tih,|h-k|}+s_{\tih,h+k}+s_{-\tih,h+k}+s_{-\tih,|h-k|})\\
&\phantom{{}={}\ -2\big(} -\tfrac{3}{16}(s_{\tz, k}+s_{\tz,h} +\tfrac{4}{3}s_{\tz,h}s_{\tz,k})     \big).
\end{align*}
Thus, by Lemma~\ref{prodottifacili},
\begin{align*}
0&= -\tfrac{3}{8}a_0+\tfrac{3}{16}(2a_0+\tfrac{1}{2}a_k+\tfrac{1}{2}a_{-k}+\tfrac{1}{2}a_{h}+\tfrac{1}{2}a_{-h}+2s_{\tz,k}+2s_{\tz,h})\\
&\phantom{{}={}} \quad -\tfrac{3}{32}(2a_0+\tfrac{1}{2}a_{k-h}+\tfrac{1}{2} a_{h+k}+\tfrac{1}{2}a_{-h-k}+\tfrac{1}{2} a_{-k+h}+2s_{\tz, |k-h|}+2s_{\tz, h+k})\\
&\phantom{{}={}} \quad +\tfrac{3}{16}(a_0s_{\tih,h+k}+a_0s_{\tih,|k-h|}+a_0s_{-\tih,h+k}+a_0s_{-\tih,|k-h|})\\
&\phantom{{}={}} \quad -\tfrac{3}{4}a_0s_{\tz,h}-\tfrac{3}{4}a_0s_{\tz,k}\\
&\phantom{{}={}} \quad -\tfrac{3}{8}a_0+\tfrac{3}{16}(a_k+a_{-k}+a_h+a_{-h})-\tfrac{3}{32} (a_{k-h}+ a_{k+h}+a_{-k-h}+ a_{-k+h})\\
&\phantom{{}={}} \quad -\tfrac{3}{32}(s_{\tih,|h-k|}+s_{\tih,h+k}+s_{-\tih,h+k}+s_{-\tih,|h-k|})+\tfrac{3}{8}s_{\tz,h}+\tfrac{3}{8}s_{\tz, k}+\tfrac{1}{2}s_{\tz,h}s_{\tz,k}.
\end{align*}
Using the inductive hypothesis to compute the products 
$$a_0s_{\tz,h},\; a_0s_{\tz,k}, \; a_0s_{\tih,|k-h|},  \mbox{ and } \;a_0s_{-\tih,|k-h|}$$
and Equation~\eqref{a0*sum} to compute the products 
$$ 
a_0s_{\tih,h+k} \;\mbox{ and  } \;a_0s_{-\tih,h+k},
$$
the above equation becomes 
\begin{align*}
0&= -\tfrac{3}{16}a_0+\tfrac{3}{32}(a_k+a_{-k}+a_{h}+a_{-h})+\tfrac{3}{8}(s_{\tz,k}+s_{\tz,h})\\
&\phantom{{}={}} \quad -\tfrac{3}{64}(a_{k-h}+ a_{h+k}+a_{-h-k}+ a_{-k+h})-\tfrac{3}{16}(s_{\tz, |k-h|}+s_{\tz, h+k})\\
&\phantom{{}={}} \quad +\tfrac{3}{16}\left ((\tfrac{1}{2}-2\lm_{h+k})a_0+\tfrac{3}{4}(a_{h+k}+a_{-h-k})+3s_{\tz,h+k}-\tfrac{3}{2}a_0 \right .\\
&\phantom{{}={}} \quad \left .+\tfrac{3}{4}(a_{k-h}+a_{-k+h})+3s_{\tz,|k-h|}\right )\\
&\phantom{{}={}} \quad -\tfrac{3}{4}\left (-\tfrac{3}{2}a_0+\tfrac{3}{8}(a_h+a_{-h}+a_k+a_{-k})+\tfrac{3}{2}s_{\tz,h}+\tfrac{3}{2}s_{\tz,k}  \right )\\
&\phantom{{}={}} \quad -\tfrac{3}{8}a_0+\tfrac{3}{16}(a_k+a_{-k}+a_h+a_{-h})-\tfrac{3}{32} (a_{k-h}+ a_{k+h}+a_{-k-h}+ a_{-k+h})\\
&\phantom{{}={}} \quad -\tfrac{3}{32}(s_{\tih,|h-k|}+s_{\tih,h+k}+s_{-\tih,h+k}+s_{-\tih,|h-k|})+\tfrac{3}{8}s_{\tz,h}+\tfrac{3}{8}s_{\tz, k}+\tfrac{1}{2}s_{\tz,h}s_{\tz,k}\\
&= \tfrac{3}{8}(1-\lm_{h+k})a_0-\tfrac{3}{8}(s_{\tz,k}+s_{\tz,h})+\tfrac{3}{8}s_{\tz, |k-h|}+\tfrac{3}{8}s_{\tz, h+k} \\
&\phantom{{}={}} \quad -\tfrac{3}{32}(s_{\tih,|h-k|}+s_{-\tih,|h-k|}+s_{\tih,h+k}+s_{-\tih,h+k})+\tfrac{1}{2}s_{\tz,h}s_{\tz,k}.
\end{align*}
Making explicit $s_{\tz,h}s_{\tz,k}$, we get
\begin{align*}
s_{\tz,h}s_{\tz,k}&=-\tfrac{3}{4}(1-\lm_{i+1})a_0+\tfrac{3}{4}(s_{\tz,k}+s_{\tz,h}-s_{\tz, |k-h|}-s_{\tz, i+1})\\
&\phantom{{}={}} \quad +\tfrac{3}{16}(s_{\tih,|h-k|}+s_{-\tih,|h-k|}+s_{\tih,i+1}+s_{-\tih,i+1}), 
\end{align*}     
which implies \ref{step4 h}, since, by Equation~\eqref{hk}, $s_{-\tih, i+1}=s_{\tik, i+1}$. 
By repeating the same argument, replacing $a_0$ by $a_1$ and the eigenvectors $u_h$ and $v_h$ by $\tilde u_h$ and $\tilde v_h$ respectively, we obtain \ref{step4 j}.
\medskip

\noindent{\bf Step 5.} {\it If $i$ is odd, then parts \ref{exclaim i} and \ref{exclaim v} hold.}
\medskip

Since $i+1$ is even, part \ref{exclaim i} follows by Lemma~\ref{evenlevel}\ref{evenlevel i} while part \ref{exclaim v} follows by \ref{exclaim i} and claim \ref{step2 b}.
\medskip

\bigskip

\noindent{\bf Step 6.} {\it If $i$ is even and $\ch(\F)\neq 5$, then \ref{exclaim i} and \ref{exclaim v} hold.}
\medskip

By the inductive hypothesis, $s_{\tz,h}=s_{\tu, h}$, $s_{\tz,k}=s_{\tu, k}$, and $s_{\tr, |k-h|}=s_{\tz, |k-h|}$ for every $r\in \Z$. Since $s_{\tz,h}$, $s_{\tz,k}$, and $s_{\tr, |k-h|}$  are invariant under the map $\rho^{(i+2)/2}$,
$$
0=(s_{\tz,h}s_{\tz,k})^{\rho^{(i+2)/2}}-s_{\tu,h}s_{\tu,k}.
$$
By \ref{step4 h}, \ref{step4 j}, and Lemma~\ref{action}\ref{action_1} (note that $i+1$ is odd), the second term of the above equation is 
\begin{multline*}
 -\tfrac{3}{4}(1-\lm_{i+1})a_{i+2}-\tfrac{3}{4}s_{\tu, i+1} +\tfrac{3}{16}(s_{\tih+\tu,i+1}+s_{\tik+\tu,i+1})\\
+\tfrac{3}{4}(1-\lm_{i+1}^f)a_1+\tfrac{3}{4}s_{\tu, i+1} -\tfrac{3}{16}(s_{\tu-\tih,i+1}+s_{\tu-\tik,i+1})
\end{multline*}
Since $\tih=-\tik$, we get
$0=-\tfrac{3}{4}(1-\lm_{i+1})a_{i+2}+\tfrac{3}{4}(1-\lm_{i+1}^f)a_1$, whence
\begin{equation}\label{stheta}
(1-\lm_{i+1})a_{i+2}=(1-\lm_{i+1}^f)a_1.     
\end{equation}
Since both $a_1$ and $a_{i+2}$ are idempotents, Equation~\eqref{stheta} implies that 
$$
\mbox{either}\quad
 \lm_{i+1}=\lm_{i+1}^f=1 \quad\mbox{or } \quad  a_1=a_{i+2}.
 $$
In the former case, we are done. In the latter case, by Lemma~\ref{s}\ref{s_3}, 
\[
\lambda_{i+1}^f=\lambda_{a_1}(a_{i+2})=\lambda_{a_1}(a_{1})=1,
\]
whence Equation~\eqref{stheta} becomes $(1-\lm_{i+1})a_{i+2}=0$, which implies $\lm_{i+1} =1$, proving \ref{exclaim i}. 
By \ref{step2 b}, $\lambda_{a_0}(s_{\tz, i+1})=0$, proving \ref{exclaim v}.
\medskip

\noindent{\bf Step 7.} {\it If $i$ is even and $\ch(\F)\neq 5$, then part \ref{exclaim ii} holds.}
\medskip

By Steps 5 and 6,  claims ~\ref{step4 h} and ~\ref{step4 j} become respectively, for every $1\leq h\leq i$ and $k=i+1-h$,
\begin{equation}\label{shsk}
s_{\tz,h}s_{\tz,k}=\tfrac{3}{4}(s_{\tz,k}+s_{\tz,h})
-\tfrac{3}{8}s_{\tz, |k-h|}-\tfrac{3}{4}s_{\tz, i+1} +\tfrac{3}{16}(s_{\tih,i+1}+s_{\tik,i+1}) 
\end{equation}
and
\begin{equation}\label{shskf}
s_{\tu,h}s_{\tu,k}=\tfrac{3}{4}(s_{\tu,k}+s_{\tu,h})-\tfrac{3}{8}s_{\tu, |k-h|}-\tfrac{3}{4}s_{\tu, i+1} +\tfrac{3}{16}(s_{\tu-\tih,i+1}+s_{\tu-\tik,i+1}).
\end{equation} 
Since $\ch(\F)\neq 5$, by the inductive hypothesis, $s_{\tz,h}s_{\tz,k}=s_{\tu,h}s_{\tu,k}$ and so, subtracting Equation~\eqref{shskf} from Equation~\eqref{shsk} and rescaling by $\tfrac{16}{3}$, we get, for every $1\leq h\leq i$ and $k=i+1-h$,
\begin{equation}\label{eqhk}
-4s_{\tz, i+1} +s_{\tih,i+1}+s_{\tik,i+1}+4s_{\tu, i+1} -s_{\tu-\tih,i+1}-s_{\tu-\tik,i+1}=0. 
\end{equation}
If $h=1$, then $$ s_{-\tu, i+1}-5(s_{\tz, i+1} -s_{\tu, i+1}) -s_{-\bar 2,i+1}=0,$$
whence taking the orbits of both members of the above equation under $\langle \rho^{(i+2)/2}\rangle$ and making explicit $s_{\tT, i+1}$, we get, for every $\tT \in \Z/(i+1)\Z$
\begin{equation}\label{formulast}
s_{\tT, i+1}=5(s_{\tT+\tu, i+1} -s_{\tT +\tilde 2, i+1}) +s_{\tT+\tilde 3,i+1}.
\end{equation}
If $h=2$, then Equation~\eqref{eqhk} becomes
\begin{equation*}\label{h=2}
0=-4s_{\tz, i+1} +s_{\bar 2,i+1}+s_{-\bar 2,i+1}+4s_{\tu, i+1} -s_{-\tu,i+1}-s_{\bar 3,i+1} 
\end{equation*}
whence, using Equation~\eqref{formulast} to substitute $s_{\tT, i+1}$ for $\tT\in \{-\bar 2, \bar 3\}$, 
we get
\begin{equation*}\label{pass}
0=-10(s_{\tz, i+1}-s_{\tu, i+1})+4(s_{-\tu, i+1}-s_{\bar 2, i+1}).
\end{equation*}
Since, by Equation~\eqref{formulast} with $\tT=-\tu$,  $s_{-\tu, i+1}-s_{\bar 2, i+1}=5(s_{\tz, i+1}-s_{\tu, i+1})$, it follows that
\[
0 = 10(s_{\tz, i+1} -s_{\tu, i+1}).
\]
Since $\ch (\F)\neq 5$, we get $s_{\tz, i+1} =s_{\tu, i+1}$. Taking the orbits of both members under $\langle \rho^{(i+2)/2}\rangle $, we get $s_{\tz, i+1} =s_{\tih, i+1}$ for every $\tih \in \Z/(i+1)\Z$.
\medskip

\noindent{\bf Step 8.} {\it If $i$ is even and $\ch(\F)=5$, then parts \ref{exclaim i} and \ref{exclaim v} hold.}
\medskip

 Choose $h$ and $k$ such that none of them is a multiple of $3$. Then, by the inductive hypothesis, $s_{\tz,h}=s_{\tu, h}$, $s_{\tz,k}=s_{\tu, k}$, whence, as in Step 6, 
\[
0=(s_{\tz,h}s_{\tz,k})^{\rho^{(i+2)/2}}-s_{\tu,h}s_{\tu,k} 
\]
and so by Equations~\ref{step4 h} and \ref{step4 j}, and Lemma~\ref{action}\ref{action_1}, 
\begin{align*}
0&= 3(1-\lm_{i+1})a_{i+2}-3(1-\lm_{i+1}^f)a_1\\
&\phantom{{}={}} \quad +3(s_{\ti+\bar 2, |h-k|} +s_{\tih+\ti+\bar 2, |h-k|}+s_{-\tih+\ti+\bar 2, |h-k|})\nonumber \\
&\phantom{{}={}} \quad -3(s_{\tu, |k-h|} +s_{\tu+\tih,|k-h|}+s_{\tu-\tih,|k-h|}).
\end{align*}
Since $h$ is not a multiple of $3$,   $\Z/3\Z=\{\tilde r, \tilde r+\tilde h, \tilde r-\tilde h\}$, for every $r\in \Z$. By the inductive hypothesis, this yields that, for every $r\in \Z$,
$$\{s_{\tr, |h-k|} ,s_{\tih+\tr, |h-k|},s_{-\tih+\tr, |h-k|}\}=\{s_{\tz, |k-h|}, s_{\tu, |k-h|}, s_{\bar 2, |k-h|}\},
$$
the two sets collapsing into singletons if $|k-h|\not \equiv_3 0$, whence 
$$0=3(1-\lm_{i+1})a_{i+2}-3(1-\lm_{i+1}^f)a_1.
$$
As in the proof of Step 6, we get  $\lm_{i+1}=\lm_{i+1}^f=1$ 
and $\lambda_{a_0}(s_{\tih, i+1})=0$. 
\bigskip

\noindent{\bf Step 9.} {\it Assume $\ch(\F)=5$ and $i$ is even, then parts \ref{exclaim iii} and \ref{exclaim iv} hold.}
 \medskip

By Step 8, claim~\ref{step4 h} becomes 
\begin{align}\label{shsk5}
s_{\tz,h}s_{\tz,k} &= 2(s_{\tz,k}+s_{\tz,h})\\
&\phantom{{}={}} \quad -2(s_{\tz, |k-h|}+s_{\tih,|h-k|}+s_{-\tih,|h-k|}+s_{\tz, i+1}+s_{\tih,i+1}+s_{-\tih,i+1}). \nonumber
\end{align}
Since $i\geq 3$, without loss of generality we may choose $k\in 2\Z$ and none of $h$ and $k$ lying in $3\Z$. Then, as in Step 6, 
$$s_{\tz,h}s_{\tz,k}=(s_{\tz,h}s_{\tz,k})^{\tau_{k/2}}, $$
whence, by Equation~\eqref{shsk5},
\begin{align*}
0&=  s_{\tz,h}s_{\tz,k}-(s_{\tz,h}s_{\tz,k})^{\tau_{k/2}} \nonumber\\
&\phantom{{}={}} \quad -2(s_{\tz, |k-h|}+s_{\tih,|h-k|}+s_{-\tih,|h-k|}+s_{\tz, i+1}+s_{\tih,i+1}+s_{-\tih,i+1}) \nonumber\\
&\phantom{{}={}} \quad +2(s_{\tik, |k-h|}+s_{\tik-\tih,|h-k|}+s_{\tik+\tih,|h-k|}+s_{\tik, i+1}+s_{\tik-\tih,i+1}+s_{\tik+\tih,i+1})\\
&= 2(s_{\tik-\tih,i+1}-s_{\tih,i+1}).
\end{align*}
Thus, by Equation~\eqref{hk}, 
\begin{equation}\label{s5}
 s_{\tih,i+1} =s_{\tik-\tih,i+1}= s_{-2\tih,i+1}. 
\end{equation}
If $i\not \equiv_3 0$, then, for $h=1$ and $k=i$, Equation~\eqref{s5} gives $s_{\tu,i+1} = s_{-\bar 2,i+1}$, whence, by Lemma~\ref{action}\ref{action_1}, for every $r\in \Z$,
\begin{equation}\label{mod3}
s_{\tr, i+1}=(s_{-\bar 2,i+1})^{\rho^{(i+2)(r+2)/2}}=(s_{\tu,i+1})^{\rho^{(i+2)(r+2)/2}} =s_{\tr+\bar 3, i+1},
\end{equation}
which, for $i\equiv _3 2$, is equivalent to \ref{exclaim iii}. 

Assume $i\equiv_3 1$. Then $ 3$ and $i+1$ are coprime and there exists $a\in \Z$ such that $1\equiv_{i+1} 3a$. Thus, by Equation~\eqref{mod3}, 
$$
s_{\tz, i+1}=s_{3\bar a, i+1}=s_{\tu, i+1},
$$
and \ref{exclaim iii} follows by Lemma~\ref{action}\ref{action_1}. 

Finally, assume $i\equiv_3 0$. Then, for $h=i-1$ and $k=2$, Equation~\eqref{s5} gives $s_{-\bar 2,i+1} = s_{\bar 4,i+1}$. As above, by Lemma~\ref{action}\ref{action_1}, for every $r\in \Z$,
\begin{equation*}\label{mod6}
s_{\tr, i+1}=(s_{-\bar 2,i+1})^{\rho^{(i+2)(r+2)/2}}=(s_{\bar 4,i+1})^{\rho^{(i+2)(r+2)/2}} =s_{\tr+\bar 6, i+1}.
\end{equation*} 
Since $i+1$ is odd and $i\equiv_3 0$, $6$ is coprime to $i+1$. Taking $a\in \Z$ such that $1\equiv_{i+1}6a$ we get 
$$
s_{\tz, i+1}=s_{6\bar a, i+1}=s_{\tu, i+1},
$$ 
and \ref{exclaim iv} follows by Lemma~\ref{action}. 
\medskip

\noindent{\bf Step 10.} {\it If $i$ is odd, then the following hold:
\begin{align}
s_{\tz,2}s_{\tz,i-1} &= \tfrac{3}{4}(s_{\tz,2}+s_{\tz,i-1})-\tfrac{1}{8}\sum_{x\in \{0,2,4\}}(s_{\overline{x},i-3}+s_{\overline{x},i+1})\label{evens}
\intertext{and}
s_{\tu,2}s_{\tu,i-1} &= \tfrac{3}{4}(s_{\tu,2}+s_{\tu,i-1})-\tfrac{1}{8}\sum_{x\in\{0,2,4\}}(s_{\tu+\overline{x},i-3}+s_{\tu+\overline{x},i+1}).\label{odds}
\end{align}
}
\medskip
Since $i+1$ is even, the formulas are particular cases of Lemma~\ref{evenlevel}\ref{evenlevel iv}.
\medskip

\noindent{\bf Step 11.} {\it Assume $\ch(\F)\neq 5$ and $i$ is odd, then part \ref{exclaim ii} holds.}
\medskip

Since $i+1$ is even and $\ch(\F)\neq 5$, by Lemma~\ref{evenlevel},  
$$s_{\tz,i+1}=s_{2\tr, i+1}, \mbox{ and } s_{\tu,i+1}=s_{2\tr+\tu,i+1}$$ for every $r\in \Z$. By the inductive hypothesis, 
$$s_{0,2}=s_{1,2}, \;\;s_{0,i-3}=s_{1,i-3} \mbox{ and } s_{0,i-1}=s_{1,i-1},$$ 
whence 
$$s_{\tz,2}s_{\tz,i-1}=s_{\tu,2}s_{\tu,i-1}.$$ 
Taking the difference between Equations~\eqref{evens} and~\eqref{odds} we get $s_{\tz,i+1}=s_{\tu,i+1}$ and part \ref{exclaim ii} follows.
\medskip

\noindent{\bf Step 12.} {\it Assume $\ch(\F)= 5$ and $i$ is odd, then parts \ref{exclaim iii} and \ref{exclaim iv} hold.}
\medskip

We consider separately the cases $i\equiv_3 0$, $i\equiv_3 1$, and $i\equiv_3 2$. Recall that we are assuming $i\geq 3$.
\medskip

Assume $i\equiv_3 0$. By the inductive hypothesis and, for $i=3$, by Lemma~\ref{s=0}, $$s_{\tz, 2}=s_{\tu,2},\;\;s_{\tz, i-1}=s_{\tu,i-1},\mbox{ and } \sum_{x\in \{0,2,4\}}s_{\bar x,i-3}=\sum_{x\in \{0,2,4\}}s_{\tu+\bar x,i-3}.$$ Since $i+1\not \equiv_3 0$, by Lemma~\ref{evenlevel}\ref{evenlevel ii}, 
\begin{equation}\label{pign}
s_{\tz,i+1}=s_{2\tr, i+1} \mbox{ and } s_{\tu,i+1}=s_{2\tr+\tu,i+1} \mbox{  for every } r\in \Z .
\end{equation}  Therefore, as in Step 11, taking the difference between  Equations~\eqref{evens} and~\eqref{odds} we get 
$$
3s_{\tz,i+1}=3s_{\tu,i+1}. 
$$
Since $\ch(\F) \neq 3$, part \ref{exclaim iii} follows by Equation~\eqref{pign}.

Assume $i\equiv_3 1$. By the inductive hypothesis,   $s_{\tz, 2}=s_{\tu, 2}$, 
$$
\{ s_{\tz, i-1}, s_{\tu, i-1}, s_{\bar 2, i-1}\}=\{ s_{\tz, i-1}, s_{\bar 2, i-1}, s_{\bar 4, i-1}\}=\{ s_{\tu, i-1}, s_{\bar 3, i-1}, s_{\bar 5, i-1}\}.
$$
It follows that  
$$
0=s_{\tz, 2}\left (\sum_{x\in \{0,2,4\}} s_{\bar x, i-1}\right )-s_{\tu, 2}\left (\sum_{x\in \{1,3,5\}} s_{\bar x, i-1}\right ).
$$
Since the index $i-1$ is even, rewriting the second term of the above equality using Lemma~\ref{evenlevel}\ref{evenlevel v}, \ref{evenlevel iv}, and \ref{evenlevel ii}, we get 
\begin{align*}
   0&= s_{\tz, i+1}+s_{\bar 2,i+1}+s_{\bar 4,i+1}-s_{\tu, i+1}-s_{\bar 3,i+1}-s_{\bar 5,i+1}\\
   &= 3(s_{\tz, i+1}-s_{\tu, i+1}),
\end{align*}
whence, as in the previous case, $s_{\tz, i+1}=s_{\tu, i+1}$ and part \ref{exclaim iii} follows.

Finally, assume $i \equiv_3 2$. Then $i\geq 5$. Since  $\tfrac{3}{4}=2$ and $\tfrac{3}{16}=-2$ in characteristic $5$, by evaluating the equations in~\ref{step4 h} and~\ref{step4 j} for $h=3$ and $k=i-2$, and using Equation~\eqref{hk}, we get, respectively,
\begin{equation}\label{fin1}
    s_{\tz, 3}s_{\tz, i-2}=2(s_{\tz,3}+s_{\tz,i-2}-s_{\tz,i-5}-s_{\bar 3,i-5}-s_{-\bar 3,i-5}-s_{\tz,i+1}-s_{\bar 3,i+1}-s_{-\bar 3,i+1})
\end{equation}
and
\begin{equation}\label{fin2}
    s_{\tu, 3}s_{\tu, i-2}=2(s_{\tu,3}+s_{\tu,i-2}-s_{\tu,i-5}-s_{\bar 4,i-5}-s_{-\bar 2,i-5}-s_{\tu,i+1}-s_{\bar 4,i+1}-s_{-\bar 2,i+1}).
\end{equation}
Since by the inductive hypothesis, 
$$s_{\bar 4,i-2}=s_{\tu, i-2}, \mbox{ and } s_{\bar 4,i-5}=s_{\tu, i-5}=s_{\bar 7, i-5}=s_{-\bar 2, i-5},$$
applying $\rho^2$ to Equation~\eqref{fin1}, we get 
\begin{equation}\label{fin3}
    s_{\tu, 3}s_{\tu, i-2}=2(s_{\tu,3}+s_{\tu,i-2}-s_{\tu,i-5}-s_{\bar 7,i-5}-s_{\tu,i-5}-s_{\bar 4,i+1}-s_{\bar 7 ,i+1}-s_{\tu,i+1}).
\end{equation}
By Lemma~\ref{evenlevel}\ref{evenlevel iii}, $s_{\bar 7 ,i+1}=s_{\tu ,i+1}$, thus, taking the difference between Equations~\eqref{fin2} and~\eqref{fin3}, we get
\[
s_{-\bar 2,i+1}=s_{\tu ,i+1}.
\]
Part \ref{exclaim iii} now follows by Lemma~\ref{evenlevel}\ref{evenlevel iii} and applying $\rho$ and $\rho^2$ to the above equation.
\medskip

\noindent{\bf Step 13.} {\it Part \ref{exclaim vi} holds.}
\medskip

If $l$ is even, the result follows by claim~\ref{step4 f} in Step 4 and the action of $\Miy(\V)$. If $l$ is odd, by proceeding as in Step 4, replacing $a_0$ by $a_1$ and the eigenvectors $u_h$ and $v_h$ by $\tilde u_h$ and $\tilde v_h$, we get 
$$
a_1s_{\tih, h+k}=-\tfrac{3}{4}a_1+\tfrac{3}{8}(a_{1-h-k}+a_{1+h+k})+\tfrac{3}{2}s_{\tu, h+k}-\tfrac{1}{4}s_{\tih, h+k}+\tfrac{1}{4}s_{\bar 2-\tih, h+k},
$$
whence \ref{exclaim vi} again follows by the action of $\Miy(\V)$.
\medskip

\noindent{\bf Step 14.} {\it Part \ref{exclaim vii} holds.}
\medskip

This follows by~\ref{step4 h} and ~\ref{step4 j} in Step 4, parts \ref{exclaim ii}-\ref{exclaim iv} and the action of $\Miy(\V)$.
\end{proof}

In order to complete the proof of Theorem~\ref{HWthm}, we need only to show that if $\ch(\F)=5$, then the product $s_{\tr, 3i}s_{\tT, 3j}$ satisfies the same formula as in the algebra $\hatH$, for every $i,j\in \N$, $r,t\in \Z$. Note that, by Proposition~\ref{exclaim}, we may take $\{r,t\}\subseteq \{0,1,2\}$. This is the goal of the following lemma, whose proof is essentially the same as the one of~\cite[Lemma 13]{FM}.

\begin{lemma}\label{prodss3} 
Let $\ch(\F)=5$. For every $i,j\in 3\N$ and $\{r,t\}\subseteq \{0,1,2\}$,
 $$
s_{\tr, i}s_{\tT, j}=2\sum_{h\in \{i,j\}}(s_{\tr, h}+s_{\tT, h}-s_{-(\tr+\tT), h})- \sum_{h\in \{|i-j|,i+j\}}(s_{\tr, h}+s_{\tT, h}-s_{-(\tr+\tT), h}).
 $$
\end{lemma}
\begin{proof}
Fix $i,j\in 3\N$. Set 
\begin{equation*}\label{u3}
\overline{ u}_{i}:= a_0+2(a_{-i}+ a_{i}) -2( s_{\tz,i}+ s_{\tu,i}+ s_{\td,i}).
\end{equation*}
Using Proposition~\ref{exclaim}\ref{exclaim vi} it is straightforward to see that  $\overline{ u}_{i}$ is a $0$-eigenvector for $\ad_{a_0}$. Hence, 
 by the fusion law, we have 
 $$a_0(\overline{u}_{i}u_{j}+\overline{u}_{i}v_{j})=2 \overline{u}_{i}v_{j}
 $$ 
 (where $u_j$ and $v_j$ are the eigenvectors defined in Equation~\eqref{defui} on page~\pageref{defui}). 
 Substituting the expressions for $\overline u_i$ and $u_j$ in the above equation and using Proposition~\ref{exclaim} to expand the multiplications, we get 
\begin{equation}\label{s3*sum}
s_{\tz,j}(s_{\tu,i}+s_{\td,i})=-s_{\tz,j}-s_{\tz,i}-2(s_{\tz,|j-i|}+s_{\tz,i+j}). 
\end{equation}
Similarly, define
\begin{equation*}\label{u3f}
\hat u_i:=a_1+2(a_{1-i}+a_{1+i})-2(s_{\tilde 0,i}+ s_{\tu,i}+ s_{\td,i}).    
\end{equation*}
It is straightforward to see that $\hat u_i$ is a $0$-eigenvector for $\ad_{a_1}$. Hence, 
 by the fusion law, we have 
 $$a_1(\hat{u}_{i}u_{j}+\hat{u}_{i}v_{j})=2 \hat {u}_{i}v_{j}
 $$  
and, as above, we get
\[
s_{\tu,j}(s_{\tz,i}+s_{\td,i})=-s_{\tu,j}-s_{\tu,i}-2(s_{\tu ,|j-i|}+s_{\tu,i+j}). 
\]
Taking the difference between Equation~\eqref{s3*sum} and this, we obtain
\begin{align*}
s_{\td,i}(s_{\tz,j}-s_{\tu,j}) &= -(s_{\tz,j}-s_{\tu ,j}+s_{\tz,i}-s_{\tu,i})\\
&\phantom{{}={}} \quad -2(s_{\tz,|j-i|}-s_{\tu,|j-i|}+s_{\tz,i+j}-s_{\tu,i+j}). \nonumber
\end{align*}
Furthermore, swapping $i$ and $j$ in Equation~\eqref{s3*sum}, we have
\begin{equation*}\label{s3f*sum2}
s_{\tz,i}(s_{\tu,j}+s_{\td,j})=-s_{\tz,i}-s_{\tz,j}-2(s_{\tz,|j-i|}+s_{\tz,i+j}), 
\end{equation*}
whence
\begin{align*}
s_{\tz,i}s_{\tu,j} &= 3\left ( s_{\tz,i}(s_{\tu ,j}+s_{\td ,j})-(s_{\td,i}(s_{\tz,j}-s_{\tu,j}))^{\tau_1}\right )\\
&= 3\Big(-s_{\tz,i}-s_{\tz,j}-2(s_{\tz,|j-i|}+s_{\tz ,i+j})\\
&\phantom{{}={}} \quad -\big( -s_{\tz,j}+s_{\tu,j}-s_{\tz,i}+s_{\tu,i}-2(s_{\tz,|j-i|}-s_{\tu,|j-i|}+s_{\tz,i+j}-s_{\tu,i+j})       \big)^{\tau_1} \Big)\\
&= 2(s_{\tz,i}+s_{\tu,i}-s_{\td,i}+s_{\tz,j}+s_{\tu,j}-s_{\td,j})\\
&\phantom{{}={}} \quad -(s_{\tz,|i-j|}+s_{\tu,|i-j|}-s_{\td,|i-j|}+s_{\tz ,i+j}+s_{\tu,i+j}-s_{\td, i+j}).
\end{align*}
By applying $\rho$ and $\rho^2$ to the above equation, by Proposition~\ref{exclaim}\ref{exclaim iii}, we get  the formulas for the products $s_{\tu,h}s_{\td,k}$ and $s_{\tz,h}s_{\td,k}$, respectively. 
\end{proof}

\begin{proof}[Proof of Theorem~\ref{HWthm}]
By Lemma~\ref{prodottifacili}, Proposition~\ref{exclaim}, and Lemma~\ref{prodss3}, the subspace of $V$ spanned by all axes $a_i$ and all vectors $s_{\tr, j}$, for $i,r\in \Z$, $j\in \N$, is closed under the algebra multiplication. Hence it coincides with $V$. It follows that there exists a surjective algebra homomorphism from $\mathcal H$ (or $\hat{ \mathcal H}$ when $\ch (\F)=5$) to $\V$, whence the result follows. 
\end{proof}

\begin{cor}\label{reduction}
 Let $\al=4\bt$ and let $\V=(V, \{a_0, a_1\})$ be a $2$-generated $\Mab$-axial algebra. If $Q=Q^f=0$, then 
  $\V$ is either isomorphic to a quotient of a symmetric algebra or has a skew axet.
\end{cor}
\begin{proof}
If $a_0= a_2$ or $a_1= a_{-1}$, then the result follows by Lemma~\ref{lhyp1}. Assume $a_0\neq a_2$ and $a_1\neq a_{-1}$.
Since $\al=4\bt$, $P=0$. Hence, by the hypothesis $Q=Q^f=0$ and Equations~\eqref{equa1} and~\eqref{equa2} on page~\pageref{equa1}, we get $R=R^f=0$. By Lemma~\ref{newQ-Q^f} and the definition of $R$ and $R^f$, it follows that 
  \[
 \lmu=\lmf=\tfrac{18\bt-1}{8} \quad \mbox{ and } \quad \lmd=\lmdf=\tfrac{480\bt^3-228\bt^2+28\bt-1}{64\bt^2}.
  \]
Thus, if $\al\neq 2$, then the result follows by Proposition~\ref{reg}\ref{reg_1}. If $\al=2$, then $\bt=\tfrac{1}{2}$ and  $\lmu=\lmf=\lmd=\lmdf=1$. Thus claim holds by Theorem~\ref{HWthm}. 
\end{proof}

\chapter{The case $V\in \{V_e,V_o\}$}\label{chaplast}

In this chapter, we prove Theorem~\ref{teonsa}. By Theorems~\ref{al=2bt}, \ref{skew}, and Lemma~\ref{lhyp1}, we may assume that $\V$ is a $2$-generated $\Mab$-axial algebra satisfying the following properties:
\begin{hyp}\label{Hypns}~
\begin{enumerate}[enum_arabic]
    \item $\al\neq 2\bt$, \label{Hypns_1}
    \item $ a_0\neq a_2$ and  $ a_1\neq a_{-1}$, \label{Hypns_2}
    \item $\V$ has regular axet, \label{Hypns_3}
    \item  $V=V_e$. \label{Hypns_4}
\end{enumerate}
\end{hyp}

Note that conditions \ref{Hypns_2} and \ref{Hypns_3} in Hypothesis~\ref{Hypns} are equivalent to saying that $\V$ has axet $X(n+n)$ with $n\geq 2$. In particular, by Lemma~\ref{ui}, $a_0-a_2$ (respectively $a_{-1}-a_1$) is a non-zero $\bt$-eigenvector for $\ad_{a_1}$ (respectively $\ad_{a_0}$), whence
\begin{equation}\label{nontrivial}
   V_\bt^{a_i}\neq \{0\}\quad \mbox{ for every }\quad  i\in \Z. 
\end{equation}



If $I$ is an ideal of an algebra $A$, then clearly every idempotent $e \in A$ induces an idempotent $\bar{e}$ (possibly trivial) of $A/I$.  The following examines when we can lift idempotents of $A/I$ to idempotents of $A$.

\begin{lemma}\label{annihilating idempotents}
Let $A$ be an algebra, $I \unlhd A$ such that $I \leq \Ann(A)$ and $e\in \A$. If $\bar{e}$ is an idempotent of $A/I$, then $\bar{e}$ lifts to a unique idempotent in $A$.  
\end{lemma}
\begin{proof}
Since $\bar{e}$ is an idempotent of $A/I$, there exists $r \in I$ such that $e^2 = e+r$.  Then,
\[
(e+r)^2 = e^2 = e+r
\]
and so the lift $e+r$ of $\bar{e}$ is an idempotent of $A$.  Now, assume that $e+s$ is another idempotent in $A$, where $s \in I$.  Then 
\[
e+s = (e+s)^2 = e^2 = e+r
\]
and so $s = r$.
\end{proof}

 If $e \in A$ is an idempotent and $\ad_e$ is semisimple, then clearly if $ev = \lm v$, then $\bar{e} \bar{v} = \lm \bar{v}$ and so the eigenvalues of $\ad_{\bar{e}}$ are a subset of the eigenvalues of $\ad_e$.

%

\section{Jordan type}

\begin{lemma}\label{lem:JT}
Assume $\V$ satisfies Hypothesis~$\ref{Hypns}$. If $\V_e$ is an algebra of Jordan type, then $\V$ is symmetric.
\end{lemma}
\begin{proof}
Assume $\V$ satisfies Hypothesis~\ref{Hypns} and let $\V_e$ be an algebra of Jordan type. Then,  by Hypothesis~\ref{Hypns}\ref{Hypns_4} and Equation~\eqref{nontrivial}, $a_0$ is a $\mathcal J(\bt)$-axis, whence,  by Table~\ref{table2} and up to isomorphism,
\[
\V_e\in \{ 3\C(\bt), 3\C(-1)^\times, \J(\delta), \J(0)^\times\}.
\]

If $a_1\in a_0^{\Miy(\V_e)}$, then $\V$ is symmetric and we are done. So, from now on, suppose that $a_1\notin a_0^{\Miy(\V_e)}$.

If $\V_e=3\C(-1)^\times$, 
then an easy check (see~\cite[Proposition~2.3.3]{FelixThesis}) shows that the only idempotents in $V_e\setminus \{0, \1\}$ belong to $a_0^{\Miy(\V_e)}$, contradicting $a_1\notin a_0^{\Miy(\V_e)}$.

 If $\V_e= 3\C(\beta)$, then 
 by~\cite[Theorem~2.3.4]{FelixThesis} (or by a straightforward computation), the idempotents in $V_e\setminus \{0,\1\}$ are
\[
a_0, \;\; a_2,\;\; a_4,\;\;\1-a_0,\;\; \1-a_2,\;\; \1-a_4.
\] 
Since  $a_1\notin a_0^{\Miy(\V_e)}$,  $a_1\in \{\mathbbm 1-a_0, \mathbbm 1-a_2, \mathbbm 1-a_4 \}$ and so $a_1$ is a $\mathcal J(1-\beta)$-axis.  Since $a_1$ must also be a $\Mab$ axis with a non-trivial $\bt$-eigenspace, $1-\bt = \bt$ and so $\bt = \frac{1}{2}$. By Note~\ref{table2 3C J} in Table \ref{table2}, $3\C(\frac{1}{2}) \cong \J(-\tfrac{3}{8})$.


Suppose that $\V_e$ is isomorphic to either $\J(\delta)$, or $\J(0)^\times$, and $\bt=\tfrac{1}{2}$.  An easy calculation (or see, for example, \cite[p. 82]{axet}) shows that the only non-trivial idempotents in $V\setminus \{0,\1\}$ are $\mathcal{J}(\frac{1}{2})$-axes. 
So $a_1$ is a $\mathcal J(\tfrac{1}{2})$-axis, thus $\V$ is of Jordan type whence symmetric.
%
%
\end{proof}


\section{Quotients of $\IY$-algebras that are not of Jordan type}

The algebras $\IY_3(\al, \frac{1}{2}; \mu)$ and $\IY_5(\al, \frac{1}{2})$ and their quotients generically have infinitely many axes of Monster type $(\al, \frac{1}{2})$.

\begin{lemma}
 \label{Ve=IY3}
If $\V_e$ is isomorphic to either $\IY_3(\al, \frac{1}{2}; \mu)$, $\IY_3(-1, \frac{1}{2}; \mu)^\times$, or $\IY_3(\al, \frac{1}{2};1)^\times$, then $\V$ is symmetric.
\end{lemma}
\begin{proof} We show that in all cases $V$ admits an involutory automorphism that swaps $a_0$ and $a_1$, so, by definition, $\V$ is symmetric.

First assume that $\al \neq -1$ and $\mu \neq 1$ and so $\V_e \cong \IY_3(\al, \frac{1}{2}; \mu)$. By Note~\ref{tableIY3basis1} in Table~\ref{tableIY3bis}, $V_e$ has basis $(e,f,z_1,z_2)$. As in~\cite[Definition 1 and Theorem 4]{split}, let $E:=\la e,f\ra$ and let  $$b\colon E\times E\to \F$$ be  the unique symmetric bilinear form such that 
\[
  b(e,e)=1=b(f,f)\quad  \mbox{ and } \quad  b(e,f)=\mu. 
\]
By \cite[Theorem~2]{split}, the non-trivial idempotents of $V = V_e$ are $\1$ and
\begin{enumerate}
    \item $z_1$, which is a primitive $\mathcal J(\al)$-axis,
    \item  $z_2$, which is a primitive $\mathcal J(1-\al)$-axis,
    \item $x_a(u):=\tfrac{1}{2}(u+\al z_1+(\al+1)z_2)$,  for $u\in \la e,f\ra$ with $b(u,u)=1$, which are primitive $\mathcal M(\al, \tfrac{1}{2})$-axes,
    \item  $x_b(u):=\tfrac{1}{2}(u+(2-\al) z_1+(1-\al)z_2)$,  for $u\in \la e,f\ra$ with $b(u,u)=1$, which are primitive $\mathcal M(1-\al, \tfrac{1}{2})$-axes.
\end{enumerate}
By Note~\ref{tableIY3basis1} in Table~\ref{tableIY3bis}, $a_0 = x_a(e)$.  Since $\bt\neq 1-\al$ (otherwise $\al=\bt$) and $a_1$ is, by hypothesis, a $\mathcal M(\al, \frac{1}{2})$-axis with a non-trivial $\frac{1}{2}$-eigenspace, 
\[
a_1\not \in \{\1, z_1, z_2,x_b(u)\}.
\]
Hence $a_1 = x_a(u)$, for some $u \in E $, with $b(u,u)=1$.  By \cite[Theorem~3.11]{split}, $\Aut (V)$ fixes $z_1$ and $z_2$ and acts on $E$ as the full orthogonal group $\mathrm{O}(E, b)$.  Note that $u \neq e$, otherwise $a_1=x_a(e)=a_0$.  So we see that the reflection $t_{e-u}$  in $\mathrm{O}(E, b)$ of centre $\langle e-u\rangle$ (see~\cite[p. 93]{Asch})
   swaps $e$ and $u$.  This reflection lifts to an involutory automorphism of $V$ that swaps $a_0 = x_a(e)$ and $a_1 = x_a(u)$.

Next assume $\al = -1$ and so $\V_e\cong \IY_3(-1, \frac{1}{2}; \mu)$; we proceed similarly to the previous case. By Note~\ref{tableIY3basis2} in Table~\ref{tableIY3bis}, $V_e$ has basis $(e,f,z_1,n)$. Let $E$ and $b$ be as in the previous case.
By \cite[Proposition~6.2]{split} the idempotents of $V=V_e$ are 
\begin{enumerate}
    \item $z_1$, which is a $\mathcal J(-1)$-axis, and 
    \item $x_{-1}(u) := \frac{1}{2}(u -z_1+n)$ (where $u$ is an element of $E$ such that $b(u,u) = 1$) which is a  $\mathcal M(-1, \frac{1}{2})$-axis.
    \end{enumerate}  
So $a_1 = x_{-1}(u)$ for some $u \in E$ where $b(u,u) = 1$. As above, by \cite[Theorem 6.6]{split}, $\Aut (V)$ fixes $z_1$ and $n$ and acts on $E$ as the full  orthogonal group $\mathrm{O}(E, b)$. Similarly to above, the reflection $t_{e-u}$ lifts to an involutory automorphism of $V$ that swaps  $a_0$ and $a_1$.

If $\V_e\cong \IY_3(-1, \tfrac{1}{2};\mu)^\times$, then, by Note~\ref{qann} in Table~\ref{tableIY3bis}, we may assume that $\V_e=\IY_3(-1, \frac{1}{2}; \mu)/I$, where $I$ is the annihilator of the algebra $ U :=\supp( \IY_3(-1, \frac{1}{2}; \mu))$ and, by Lemma~\ref{quotient}, $a_0=x_{-1}(e)+I$.  By Lemma~\ref{annihilating idempotents}, $a_1 \in V_e$ lifts to a unique idempotent $a^{\uparrow}$ of $U$. As above, by~\cite[Proposition~6.2]{split}, $\ad_{a^\uparrow}$ is semisimple.  Since by assumption, $a_1$ has a non-trivial $\frac{1}{2}$-eigenspace and $I \subseteq A_0^{a^\uparrow}$, $a^{\uparrow}$ also has a non-trivial $\frac{1}{2}$-eigenspace and hence $a^\uparrow = x_{-1}(u)$, for some $u \in E$ with $b(u,u)=1$, and $u\neq e$ since $a_1\neq a_0$.  As above, the reflection $t_{e-u}$ lifts to an automorphism of $U$ that induces on $U/I$ an automorphism that swaps $a_0$ and $a_1$.


Assume $\mu = 1$ and $\V_e= \IY_3(\al, \frac{1}{2}; 1)$. By Note~\ref{tableIY3basis3} in Table~\ref{tableIY3bis}, $V_e$ has basis $(a_0, a_2, z,n)$. By \cite[Proposition 8.2]{MM}, the idempotents of $V=V_e$ are $x_1(\zeta) := \zeta a_0 + (1-\zeta)a_2 -\zeta(1-\zeta) (z - 2n)$, where $\zeta \in \F$ and these are all $\mathcal M(\al, \frac{1}{2})$-axes. In particular there is $\zeta' \in \F$ such that $a_1=x_1(\zeta^\prime)$. Let $\phi$ be the linear map $V\to V$ that fixes $z$ and $n$ and maps $a_0 \mapsto x_1(\zeta')$ and $a_2 \mapsto x_1(1+\zeta')$. A direct computation shows that $\phi$ is an involutory  automorphism of the algebra $V$ that swaps $a_0$ and $a_1$. Finally, if $\V_e=\IY_3(\al, \frac{1}{2};1)^\times$, then an analogous argument to the previous quotient produces an automorphism swapping $a_0$ and $a_1$.
\end{proof}




\begin{lemma}
    \label{Ve=IY5}
 If $\V_e$ is isomorphic to $\IY_5(\al, \tfrac{1}{2})^\times$, or $\IY_5(\al, \tfrac{1}{2})$, then $\V$ is symmetric.
\end{lemma}
\begin{proof}
By Hypothesis~\ref{Hypns}\ref{Hypns_4}, $V=V_e$.
By \cite[Proposition 8.8 and Table 30]{MM}, every idempotent in $V$ is of the form 
\[
x(\zeta)= \zeta a_0 + (1-\zeta) a_2 +u
\]
for some $\zeta\in \F$ and $u$ in $V^\perp$. By the bilinearity the Frobenius form, for every $\zeta_1,\zeta_2\in \F$,
\[\begin{split}
(x(\zeta_1),x(\zeta_2))&=\zeta_1\zeta_2(a_0,a_0)+\zeta_1(1-\zeta_2)(a_0,a_2)\\
&\phantom{{}+{}}  +\zeta_2(1-\zeta_1)(a_2,a_0)+(1-\zeta_1)(1-\zeta_2)(a_0,a_2)\\
&=1. 
\end{split}\]
Since $a_{-2}, a_{-1}, a_0,a_1,a_2$ are also of the form $x(\zeta)$, we get
\[
\lmu=\tfrac{(a_0,a_1)}{(a_0,a_0)}=1=\tfrac{(a_0,a_1)}{(a_1,a_1)}=\lmf,
\]
\[
\lmd=\tfrac{(a_0,a_2)}{(a_0,a_0)}=1=\tfrac{(a_{-1},a_1)}{(a_1,a_1)}=\lmdf,
\]
and
\[
\lm_3=\tfrac{(a_0,a_3)}{(a_0,a_0)}=1=\tfrac{(a_{-2},a_1)}{(a_1,a_1)}=\lm_3^f.
\]
Thus, 
by Lemma~\ref{reg}\ref{reg_2}, $\V$ is a quotient of a symmetric algebra, whence by Theorem~\ref{teoq}, $\V$ is either symmetric or isomorphic to $4\B(-1, \tfrac{1}{2}; \nu)^{\times}$. The latter case cannot occur, since $\dim V \geq 5$ while $4\B(-1, \tfrac{1}{2}; \nu)^{\times}$ has dimension $4$.  
\end{proof}

\section{Quotients of the Highwater algebra and cover}

In this section we assume that $\V$ satisfies Hypothesis~\ref{Hypns}.

\begin{lemma}\label{lem:H1}
If $\V_e$ is isomorphic to a quotient of $\hat{\mathcal  H}$, then either $\V$ is symmetric or $\adim(\V_e)\leq 3$.
\end{lemma}
\begin{proof}
Let $\V_e$ be isomorphic to a quotient of $\hatH$. Then $(\al,\bt)=(2,\tfrac{1}{2})$ and, by Lemma~\ref{lambdasut}, $\lmd=1$.
Assume, for a contradiction, that $\V$ is non-symmetric and $\adim(\V_e)> 3$, so $a_{-2} \neq a_2$.  From Lemma~\ref{supernew}\ref{supernew_1}, it follows that either $a_3=a_{-1}$, or $\lmu=\lmf=1$.  However, since we assume that $\V$ has a regular axet and $a_{-2} \neq a_2$, we must have $\lmu = \lmf = 1$. Since $\V$ is non-symmetric, Theorem \ref{HWthm} implies 
$$
\lmdf\neq 1.
$$


If $\adim(\V_o)$ is also greater than $3$, then Lemma~\ref{supernew}\ref{supernew_2} implies that $\lmdf=1$, a contradiction.    Hence  $\adim(\V_o)\leq 3$ and a direct check (or see~\cite[p. 470]{HWQ}) shows that 
\[
\V_o\in \{ 3\C(2), \J(\delta), \IY_3(2,\tfrac{1}{2}; \mu), \IY_3(2,\tfrac{1}{2}; 1)^\times\}.
\]
By Lemmas~\ref{sub3C}, \ref{subJ}, and~\ref{subIY3}, it follows that $\lmdf=1$, again a contradiction.  
\end{proof}

Recall the definition of the algebra $I\mathcal H_3 = \hatH/(a_0-a_1-a_2+a_3)$ from page~\pageref{subIH3}.

\begin{lemma}\label{Ve<>IH3}
If $\V$ is non-symmetric, then  $\V_e\not \cong I\mathcal H_3$.
\end{lemma}
\begin{proof}
Assume for a contradiction that $\V$ is non-symmetric and $\V_e\cong I\mathcal H_3$. Since $V=V_e$, by Lemma~\ref{subIH3}, $V$ has basis $(a_{-2}, a_0, a_2, s_{\tz,2})$. Let $\gm_{-2}, \gm_0, \gm_2, t\in \F$ be such that 
\[
a_1=\gm_{-2}a_{-2}+\gm_0a_0+\gm_2a_2+ts_{\tz,2}.
\]
By Lemma~\ref{subIH3}\ref{subIH3_2}, 
\[
(a_{-2})^{\tau_1} = a_4 = -a_{-2}+a_0+a_2.
\]
Since $\tau_1$ swaps $a_0$ and $a_2$ and fixes $a_1$ and $s_{\tz,2}$, it follows that $\gm_{-2}=0$ and  $\gm_2=\gm_0$, hence 
\[
a_1=\gm_0(a_0+a_2)+ts_{\tz,2}.
\]
As $a_0-a_2$ is a $\tfrac{1}{2}$-eigenvector for $\ad_{a_1}$, by Table~\ref{tablehatH}, we have
\begin{align*}
0 &= a_1(a_0-a_2)-\tfrac{1}{2}(a_0-a_2)\\
&= (\gm_0(a_0+a_2)+ts_{\tz,2})(a_0-a_2)-\tfrac{1}{2}(a_0-a_2)\\
&= \tfrac{3}{4}ta_{-2}+ (\gm_0-\tfrac{3}{2}t-\tfrac{1}{2})a_0 -(\gm_0-\tfrac{3}{4}t-\tfrac{1}{2})a_2, 
\end{align*}
whence $t=0$ and $\gm_0=\tfrac{1}{2}$. It follows that $0=a_1^2-a_1=\tfrac{1}{2}s_{\tz,2}$, a contradiction.
\end{proof}

\begin{cor}
    \label{cor:H}
    If $\V$ satisfies Hypothesis~$\ref{Hypns}$ and is non-symmetric, then $\V_e$ is not isomorphic to a quotient of $\hatH$.
\end{cor}
\begin{proof}
Suppose for a contradiction that $\V$ is isomorphic to    a quotient of $\hatH$. By Lemmas~\ref{lem:H1}, ~\ref{subH}, and~\ref{Ve<>IH3}, $\V_e$ is isomorphic to one of the following algebras
$$
3\C(2), \quad \J(0), \quad \J(0)^\times, \quad \IY_3(2, \tfrac{1}{2}; \mu), \quad \IY_3(2, \tfrac{1}{2}; 1)^\times. 
$$
Thus we get a contradiction, using Lemma~\ref{lem:JT} and  
 Corollary~\ref{Ve=IY3}.
\end{proof}

\section{Algebras with finite axet}

In this section, we will consider the remaining algebras and show that if $\V = \V_e$, then they are all symmetric.  For the majority of these algebras, we show that there does not exists a Monster type $(\al, \bt)$ axis $a_1$.

assume that $\V$ is a non-symmetric algebra satisfying Hypothesis~\ref{Hypns} and $\V_e$ is isomorphic to a symmetric  algebra with finite axet. Since $\V$ is non-symmetric, the following condition holds 
\begin{equation*}\label{a1noteven}
    a_1^{\Miy(\V)}\cap a_0^{\Miy(\V)}=\emptyset .
\end{equation*}

\begin{lemma}\label{lem:Ve 3A}
The algebra $\V_e$ is not isomorphic to $3\A(\al, \bt)$.
\end{lemma}
\begin{proof}
If $\bt = \tfrac{1}{2}$, then, by Note~\ref{tableIY3 3A} in Table~\ref{tableIY3bis}, $3\A(\al, \tfrac{1}{2}) \cong \IY_3(\al, \frac{1}{2}; -\frac{1}{2})$ and the result follows from Lemma~\ref{Ve=IY3}. Assume for a contradiction that  
\[
\bt\neq \tfrac{1}{2} \quad \mbox{ and }\quad  \V_e\cong 3\A(\al, \bt).
\]
By Table~\ref{table3A}, $V=V_e$ has basis $a_0$, $a_2$, $a_4$, $z$, and, by~\cite[\S 8]{FelixPaper}, every eigenspace of $\ad_{a_0}$ has dimension one.
By Lemma~\ref{sub3A}\ref{sub3A_2}, $\V_e$ has axet $X(3)$, whence, as we assume $\V$ has a regular axet, $\V_o$ has axet $X(3)$ too. Thus 
\[
a_0=a_6, \quad  a_{-3}=a_3, \quad  a_{-2}\neq a_2, \quad a_{-1}\neq a_3.
\]
In particular, $\tau_0$ and $\tau_3$ act as the identity on $\lla a_0, a_3\rra$, whence the $\bt$-eigenspaces for $\ad_{a_0}$ and $\ad_{a_3}$ in $\lla a_0, a_3\rra$ are trivial. Define 
\begin{equation*}
\V_3=(\lla a_0, a_3\rra, \{a_0,a_3\}) \mbox{ and } V_3=\supp(\V_3).
\end{equation*}
Thus  $\V_3$ is a $\mathcal J(\al)$-axial algebra. 

By the definition of $3\A(\al, \bt)$, $\al \neq \frac{1}{2}$, so the classification of $2$-generated axial algebras of Jordan type (see Table~\ref{table2}) implies that $\V_3$ is isomorphic to $2\B$, $3\C(\al)$, or $3\C(-1)^\times$.  Let
\begin{equation*}
    \label{pol}
   p(\al, \bt) := 3\al^2 + 3\al\bt -\bt-1. 
\end{equation*}
We split into two cases: $p(\al,\bt)=0$ and $p(\al,\bt)\neq 0$.

Suppose first that $p(\al,\bt)=0$. Then, by Note ~\ref{table3A id ann} in Table~\ref{table3A}, $\Ann(V)=\la z\ra$, whence $\Ann(V)$ is the $0$-eigenspace of $\ad_{a_0}$. Either $V_3$ contains the $0$-eigenspace or it does not. If it does, $V_3$ has a non-trivial annihilator, whence $\V_3\cong 3\C(-1)$ and $\al=-1$. If not, $\V_3\cong 3\C(-1)^\times$ and $\al=-1$. Either way, $\al=-1$. Notice that $0=p(-1, \bt) = 3 - 3\bt - \bt -1 = 2(1-2\bt)$ and so $\bt = \frac{1}{2}$, contradicting our assumptions.

Therefore $p(\al, \bt)\neq 0$. By Note ~\ref{table3A id ann} in Table~\ref{table3A}, $V_e$ has an identity 
\[
\1 := -\tfrac{4(2\al-1)}{\al p(\al,\bt)}z.
\] 
Clearly, $\1 - a_0$ is contained in the $0$-eigenspace of $\ad_{a_0}$.  Since every eigenspace of $\ad_{a_0}$ is $1$-dimensional, $V_0^{a_0}=\langle \1-a_0 \ra$.  
%
We claim that $\V_3 \not \cong 2\B$.  Suppose for a contradiction that $\V_3 \cong 2\B$. Then, as $a_3$ is primitive, $a_3 = \1-a_0$ and so $a_1=(a_3)^{\tau_2}=\1-a_{-2}$.  So $a_1$ is simultaneously a $\mathcal M(\al, \bt)$-axis and a $\mathcal M(1-\al, 1-\bt)$-axis. Since $\al\not \in \{0,1,\bt\}$, this implies $\bt = 1-\al$.  Let $A_\al(a_{-2}) = \la v \ra$, then as $a_1 = \1-a_{-2}$, $v \in A_{1-\al}(a_1)$.  However, we know that $a_0-a_2$ is a $\bt=1-\al$ eigenvector for $a_1$.  So, $0 \neq a_0-a_2 \in A_\al(a_{-2}) \cap A_\bt(a_{-2})$, a contradiction.  Hence $\V_3 \not \cong 2\B$.

Therefore $\V_3 \cong 3\C(\al)$, or $\V_3 \cong 3\C(-1)^\times$. In either case, $a_3\in V_{\{0,1,\al\}}^{a_0}$. 
By Lemma~\ref{sub3C}\ref{sub3C_1}, $\lm_{a_0}(a_3)= \tfrac{\al}{2}$.  Let
\[
v_2 := \tfrac{(\al+\bt-1)}{4(2\al-1)}a_0+\tfrac{1}{2}(a_2+a_{-2})+\tfrac{1}{\al}z.
\]
be the $\al$-eigenvector of $\ad_{a_0}$ defined in Equation~\eqref{defui} on page \pageref{defui}. So as $A_0(a_0) = \la \1-a_0 \ra$, $A_\al(a_0) = \la v_2 \ra$, there exist $\gamma, \mu\in \F$ such that    
\begin{equation}
    \label{trya3}
    a_3 = \tfrac{\al}{2} a_0 + \gamma (\1-a_0) + \mu v_2.
\end{equation}
Then, with $q(\al, \bt):=3\al^2+3\al\bt-9\al-2\bt+4$,
\[
(v_2)^2=-\tfrac{3\al-\bt-1}{16\al(2\al-1)^2}\left (\al^2q(\al, \bt)a_0+\al(\al-1)p(\al,\bt)(\1-a_0)\right ),
\]
whence
\begin{align*}
a_3^2 &= \tfrac{\al^2}{4} a_0 + \gamma^2 (\1-a_0) + \mu^2 v_2^2+\al^2\mu v_2+2\gamma \mu (1-\al)v_2 
\end{align*}
and
\begin{align*}
 0 &= a_3^2-a_3\\
 &=  \tfrac{\al^2}{4} a_0 + \gamma^2 (\1-a_0) -\mu^2\tfrac{3\al-\bt-1}{16\al(2\al-1)^2}\left (\al^2q a_0+\al(\al-1)p(\1-a_0)\right )\\
 &\phantom{{}={}} \quad +\al^2\mu v_2+2\gamma \mu (1-\al)v_2-\tfrac{\al}{2} a_0 - \gamma (\1-a_0) - \mu v_2\\
 &= \left ( \tfrac{\al^2}{4}-\tfrac{\al}{2} -\mu^2\tfrac{\al^2(3\al-\bt-1)}{16\al(2\al-1)^2}q\right )a_0 + \left (\gamma^2-\gamma -\mu^2\tfrac{\al(\al-1)(3\al-\bt-1)}{16\al(2\al-1)^2} p\right )(\1-a_0)\\
 &\phantom{{}={}} \quad +(\al^2\mu +2\gamma \mu (1-\al)  - \mu )v_2
\end{align*}
Since $a_0$, $\1-a_0$ and $v_2$ are distinct eigenvectors, they are linearly independent and so
\begin{align}
      0 &= \tfrac{\al^2}{4}-\tfrac{\al}{2} -\mu^2\tfrac{\al^2(3\al-\bt-1)}{16\al(2\al-1)^2}q \label{sy1} \\
      0 &= \gamma^2-\gamma -\mu^2\tfrac{\al(\al-1)(3\al-\bt-1)}{16\al(2\al-1)^2} p           \label{sy2} \\
    0 &= \mu(\al-1)(\al+1-2\gamma).    \label{sy3}
\end{align}
If $\mu=0$, from Equations~\eqref{sy1} and~\eqref{sy2} we get $\al=2$ and $\gamma\in \{0,1\}$. From Equation~\eqref{trya3}, it follows that if $\gamma=0$, then $a_3 = a_0$, contradicting ${a_0}^{\Miy(\V)} \cap {a_1}^{\Miy(\V)} = \emptyset$, and if $\gamma=1$, then  $a_3 = \1$, another contradiction as $\1$ is not primitive. Hence $\mu\neq 0$.   From Equation~\eqref{sy3}, as $\al \neq 1$, we get 
\[
\gamma=\tfrac{\al+1}{2}.
\]
Taking the difference between $(\al-1) p$ times Equation~\eqref{sy1} and $\al q$ times Equation~\eqref{sy2} and 
substituting the above value of $\gm$, we get 
\[
0=\left (\tfrac{\al^2}{4}-\tfrac{\al}{2} \right )(\al-1)p-\tfrac{1}{4}\al(\al^2-1)q=\tfrac{1}{4}\al(\al-1)\big ( (\al-2)p-(\al+1)q \big ),
\]
whence, 
since $\al\not \in \{0,1\}$,  
\[
0=(\al-2)p-(\al+1)q= -2(2\al-1)(2\bt-1).
\]
This is a contradiction, since $\al\neq \tfrac{1}{2}$ for $\V_e = 3\A(\al, \bt)$ and $\bt\neq \tfrac{1}{2}$ by the initial assumption.
\end{proof}

\begin{lemma}\label{Ve=3Ax}
The algebra $\V_e$ is not isomorphic to $3\A(\al, \tfrac{1-3\al^2}{3\al-1})^\times$.
\end{lemma}
\begin{proof}
Let $\bt=\tfrac{1-3\al^2}{3\al-1}$ and assume, for a contradiction, that $\V_e \cong 3\A(\al, \bt)^\times$.  As before, the axet of $\V$ is regular and is $X(6)$; hence $(\la\la a_0, a_3\ra\ra, \{a_0, a_3\})$ is a $\mathcal{J}(\al)$-axial algebra. By Note~\ref{table3A quotient} to Table~\ref{table3A}, $3\A(\al, \bt)^\times$ is the quotient of $3\A(\al, \bt)$ by the annihilator of the algebra. So $\ad_{a_0}$ has eigenvalues $1$, $\al$ and $\bt$ and every eigenspace has dimension one.  In particular, $0$ is not an eigenvalue of $\ad_{a_0}$, so $(\lla a_0, a_3\rra, \{a_0, a_3\})$ is isomorphic to $3\C(-1)^\times$. Thus $\al=-1$ and so $\bt = \tfrac{1-3\al^2}{3\al-1} = \tfrac{1}{2}$. However, by \cite[Proposition 4.1]{MM}, $V_e\cong \IY_3(-1,\tfrac{1}{2};-\tfrac{1}{2})^\times$ and so $\V$ is symmetric by Corollary \ref{Ve=IY3}.
\end{proof}

\begin{lemma}\label{Ve=4}
The algebra $\V_e$ is not isomorphic to one of the following algebras: $4\A(\tfrac{1}{4},\bt)$, $4\A(\tfrac{1}{4},\tfrac{1}{2})^\times$,  $4\B(\al,\tfrac{\al^2}{2})$, $4\B(-1,\tfrac{1}{2})^\times$, $4\Y(\tfrac{1}{2},\bt)$, or $4\Y(\al,\tfrac{1-\al^2}{2})$. 
\end{lemma}
\begin{proof}\label{4}
By Lemmas~\ref{sub4A}, \ref{sub4B}, \ref{sub4Y}, \ref{sub4Y1/2}, and~\ref{sub4Y}, $\V_e$ has axet $X(4)$ and so, as the axet of $\V$ is regular, $\V$ has axet $X(8)$. By Tables~\ref{table4A}, \ref{table4B}, \ref{table4Y}, and~\ref{table4Yal}, 
$$V=V_e=\la a_{-2}, a_0,a_2, a_4, s_{\tz,2}\ra,
$$ using a change of indices to coincide with the even subalgebra.  As $a_1$ is invariant under its own Miyamoto map, we get
\[ a_1=\mu_0(a_0+a_2)+\mu_1(a_4+a_{-2})+\mu_2 s_{\tz,2},\]
where $\mu_0, \mu_1, \mu_2\in \F$.
Applying $\tau_3$, we get
\[ a_5=\mu_1(a_0+a_2)+\mu_0(a_4+a_{-2})+\mu_2 s_{\tz,2}.\]
Notice that $a_1(a_0-a_2)=a_5(a_0-a_2)=\bt(a_0-a_2)$. Hence,
\begin{align*}
0&=(a_1-a_5)(a_0-a_2)=(\mu_0-\mu_1)\big((a_0+a_2)-(a_4+a_{-2})\big)(a_0-a_2)\\
&=(\mu_0-\mu_1)\big( (a_0-a_2) - s_{\tz,4}-\bt(a_0+a_4) - s_{\tz,2}-\bt(a_0+a_{-2})\\
&\phantom{{}={}} \quad +s_{\td,2}+\bt(a_2+a_4)+s_{\td,4}+\bt(a_2+a_{-2})\big)\\
&=(\mu_0-\mu_1)\big((1-2\bt)(a_0-a_2)+s_{\td,4}-s_{\tz,4}\big)
\end{align*}
by Lemma \ref{invariant}, as $2 \equiv 0 \mod 2$.  Note that $\mu_0 \neq \mu_1$, otherwise $a_1=a_5$, a contradiction. Hence,
\[
    0=(1-2\bt)(a_0-a_2)+s_{\td,4}-s_{\tz,4}.
\]
Since $\rho=\tau_0\tau_1$ swaps $s_{\tz,4}$ and $s_{\td,4}$, taking the sum of the above and its image under $\rho$ we get
\[
    0=(1-2\bt)(a_0-a_2+a_2-a_4)=(1-2\bt)(a_0-a_4),
\]
whence, as $a_0\neq a_4$, $\bt=\tfrac{1}{2}$ and so, from above, $s_{2,4}=s_{0,4}$.  In particular, the only algebras with $\bt=\frac{1}{2}$ are
\[
4\A(\tfrac{1}{4},\tfrac{1}{2}),\quad  4\A(\tfrac{1}{4},\tfrac{1}{2})^\times, \quad  4\B(-1,\tfrac{1}{2}), \quad  4\B(-1,\tfrac{1}{2})^\times .
\]
If $\V_e$ is isomorphic either to $4\A(\tfrac{1}{4},\frac{1}{2})$, or $4\A(\tfrac{1}{4},\tfrac{1}{2})^\times$, by Table~\ref{table4A}, $a_0a_4=a_2a_{-2}=0$, whence
\[
-\bt(a_2+a_{-2})=s_{\td,4}=s_{\tz,4}=-\bt(a_0+a_4),
\]
a contradiction to the linearly independence of the basis. 
Hence $\V_e$ is isomorphic to $4\B(-1,\frac{1}{2})$ or to $4\B(-1,\tfrac{1}{2})^\times$. By Table~\ref{table4B}, we get
\[
\tfrac{\al}{2}(a_0+a_4-c)-\bt(a_0+a_4)=s_{\tz,4}=s_{\td,4}=\tfrac{\al}{2}(a_2+a_{-2}-c)-\bt(a_2+a_{-2}),
\]
where $c$ is the third $\mathcal{J}(-1)$ axis shared in the subalgebras of $\lla a_0,a_4\rra$ and $\lla a_2,a_{-2}\rra$. Hence
\[
(\tfrac{\al}{2}-\bt)(a_0+a_4)=(\tfrac{\al}{2}-\bt)(a_2+a_{-2})
\]
As $\tfrac{\al}{2}-\bt = -1 \neq 0$, this is a contradiction with the linear independence of the basis. 
\end{proof}

\begin{lemma}\label{Ve=5}
The algebra $\V_e$ is not isomorphic to $5\A\left (\al,\frac{5\al-1}{8} \right )$.
\end{lemma}
\begin{proof}
If $\ch(\F)=5$, then by Note~\ref{tableIY5 char5} in Table \ref{tableIY5}, $5\A\left (\al,\frac{5\al-1}{8}\right )\cong \IY_5(\al, \tfrac{1}{2})$ and the result follows from Lemma~\ref{Ve=IY5}. 
So we may suppose that $\ch(\F)\neq 5$ and, for a contradiction, assume that $\V_e\cong 5\A\left (\al,\tfrac{5\al-1}{8}\right )$. By Table~\ref{table5A}, $V=V_e$ has basis 
\[
a_{-4}, \: a_{-2},\: a_0, \: a_2, \: a_4,\: w .
\]
As $a_5$ is fixed under $\tau_0$, we have
\[ a_5:=\mu_0a_0+\mu_1(a_2+a_{-2})+\mu_2(a_4+a_{-4})+\mu_3w,\]
for some $\mu_0,\mu_1,\mu_2,\mu_3\in \F$. Since $a_2-a_{-2}$ is a $\bt$-eigenvector for $a_5$, we get
\[\begin{split}
0&=a_5(a_2-a_{-2})-\bt(a_2-a_{-2})\\
&=\bt\mu_0(a_2-a_{-2})+\mu_1(a_2-a_{-2})+\mu_2(a_2-a_{-2})(a_4+a_{-4})+\mu_3w(a_2-a_{-2})\\
&\phantom{{}+{}} \quad -\bt(a_2-a_{-2})\\
&=(\bt\mu_0+\mu_1-\bt)(a_2-a_{-2})+2\bt\mu_2(a_{2}-a_{-2})\\
&\phantom{{}+{}} \quad +\mu_3(\al-\bt)\big(2(a_{4}-a_{-4})+a_{2}-a_{-2}\big)\\
&=\big(\bt\mu_0+\mu_1-\bt+2\bt\mu_2+(\al-\bt)\mu_3\big)(a_2-a_{-2})+2(\al-\bt)\mu_3(a_{4}-a_{-4}).
\end{split}\]
Since $a_{-4}$, $a_{-2}$, $a_2$, and $a_4$ are linearly independent, we get
\begin{equation}\label{mu3}
   \mu_3=0 \quad \mbox{ and } \quad  \bt\mu_0+\mu_1+2\bt\mu_2-\bt+(\al-\bt)\mu_3=0
\end{equation}
As $a_4-a_{-4}$ is also a $\bt$-eigenvector for $a_5$, a similar calculation gives
\[
\begin{split}
0&=a_5(a_4-a_{-4})-\bt(a_4-a_{-4})\\
&=\bt\mu_0(a_4-a_{-4})+\mu_1(a_2+a_{-2})(a_4-a_{-4})+\mu_2(a_4-a_{-4})-\bt(a_4-a_{-4})\\
&=(\bt\mu_0+2\bt\mu_1+\mu_2-\bt)(a_4-a_{-4}),
\end{split}\]
whence 
\begin{equation}\label{mu0}
    \bt\mu_0+2\bt\mu_1+\mu_2-\bt=0.
\end{equation}
Taking the difference between the second equation in Equation~\eqref{mu3} and Equation~\eqref{mu0} we get
\[ 
0=(\bt\mu_0+\mu_1+2\bt\mu_2-\bt)-(\bt\mu_0+2\bt\mu_1+\mu_2-\bt)=(2\bt-1)(\mu_2-\mu_1).
\]
Since by hypothesis $\bt=\tfrac{5\al-1}{8}$, $\al\neq 1$ and $\ch(\F)\neq 5$, it follows that $\bt\neq \tfrac{1}{2}$.  Therefore $\mu_1=\mu_2$.

By Note~\ref{table5A id} in Table~\ref{table5A}, $\1=\tfrac{1}{5(\al-\bt)}(a_0+a_2+a_4+a_6+a_8)$ is the identity element of $V$. Note that 
\[
    a_5=\gm_0 a_0+ \gm_1 (\1-a_0),
\]
where 
\[ \begin{split}
\gm_0&:=\mu_0+(5(\al-\bt)-1)\mu_1;\\
\gm_1&:=5(\al-\bt)\mu_1.
\end{split}\]
Recall that $a_0(\1-a_0)=0$ and $(\1-a_0)^2=(\1-a_0)$. As $a_5$ is an idempotent, we get
\[\begin{split}
0=a_5^2-a_5&=(\gm_0^2 a_0+\gm_1^2(\1-a_0))-(\gm_0 a_0+\gm_1(\1-a_0))\\
&=\gm_0(\gm_0-1)a_0+\gm_1(\gm_1-1)(\1-a_0).  
\end{split}\]
Hence $\gm_0,\gm_1\in \{0,1\}$ and $a_5\in \{0,a_0,\1-a_0,\1\}$. By \cite[Lemma~6.2]{MM}, the $0$-eigenspace of $\ad_{a_0}$ has dimension two, and so the $1$-eigenspace of $\ad_{\1-a_0}$ has dimension two. Thus  $0$, $\1$, and $\1-a_0$ are not primitive axes and $a_5\not \in \{0,\1, \1-a_0\}$. Hence $a_5=a_0$, a contradiction.
\end{proof}

\begin{lemma}\label{Ve=6}
The algebra $\V_e$ is not isomorphic to one of the following algebras: 
\[
6\A\left (\al,-\tfrac{\al^2}{4(2\al-1)}\right ),\quad 
 6\A\left (\tfrac{2}{3},-\tfrac{1}{3}\right )^\times,\quad  \mbox { or }\quad  6\A\left (\tfrac{1\pm \sqrt{97}}{24},\tfrac{53\pm 5\sqrt{97}}{192}\right )^\times.
\]
\end{lemma}
\begin{proof}
For the sake of contradiction, assume $\V_e$ is isomorphic to one of the algebras in the statement. By Table~\ref{table6A} and the Notes ~\ref{table6A id ann}--~\ref{table6A quotient sqrt97} to Table~\ref{table6A}, $V=V_e$ is spanned by the set  
\[
\{a_{-4},a_{-2},a_0,a_2,a_4,a_6, c,z\}.
\]
%
As $a_1$ is fixed under its own Miyamoto involution, there are $\mu_0$, $\mu_1$, $\mu_2$, $\mu_3$, $\mu_4\in \F$ such that
\[ a_1=\mu_0(a_0+a_2)+\mu_1(a_4+a_{-2})+\mu_2(a_6+a_{-4})+\mu_3c+\mu_4z.\]
Applying $\tau_4$, we get:
\[ 
a_7=(a_1)^{\tau_4}=\mu_0(a_6+a_{-4})+\mu_1(a_4+a_{-2})+\mu_2(a_0+a_2)+\mu_3c+\mu_4z.
\]
Since $a_0-a_2$ is a $\bt$-eigenvector for both $\ad_{a_1}$ and $\ad_{a_7}$, by arguing as before, we get
\[
    0=(a_1-a_7)(a_0-a_2)=(\mu_0-\mu_2)(a_0+a_2-a_6-a_{-4})(a_0-a_2).
\]
Note that $\mu_0 \neq \mu_2$, otherwise $a_1=a_7$, a contradiction. So, the above gives
\begin{align*}
0 &= (a_0+a_2-a_6-a_{-4})(a_0-a_2) \\
&=(a_0-a_2)-s_{\tz,6}-\bt(a_0+a_6)-s_{\tz,4}-\bt(a_0+a_{-4})\\
&\phantom{{}={}} \quad +s_{\td,4}+\bt(a_2+a_6)+s_{\td,6}+\bt(a_2+a_{-4})\\
&= (1-2\bt)(a_0-a_2)-s_{\tz,6}-s_{\tz,4}+s_{\td,4}+s_{\td,6}.
\end{align*}
Taking the difference between the above and its image under $\tau_0$ we get
\begin{equation}\label{NN}
 \begin{split}
  0 &= (1-2\bt)(a_0-a_2)-s_{\tz,6}-s_{\tz,4}+s_{\td,4}+s_{\td,6}\\
  &\phantom{{}={}} \quad -\big((1-2\bt)(a_0-a_{-2})-s_{\tz,6}-s_{\tz,4}+s_{\td,4}+s_{\bar{4},6}\big)\\
  &= (2\bt-1)(a_2-a_{-2})+s_{\td,6}-s_{\bar{4},6}
 \end{split}   
\end{equation}
Summing this with its images under $\rho$ and $\rho^2$ we get 
\[
0 = (1-2\bt)(a_{-2}+a_0-a_4-a_6).
\]
Since $a_{-2}$, $a_0$, $a_4$, and $a_6$ are linearly independent, it follows that $\bt=\tfrac{1}{2}$.  Thus, by Equation~\eqref{NN} and Table~\ref{table6A}, 
\[
\begin{split}
0& =s_{\td,6}-s_{\bar{4},6}\\
&= a_{2}a_{-4}-\bt(a_2+a_{-4})- a_{4}a_{-2}+\bt(a_4+a_{-2})\\
&= (\tfrac{\al}{2}-\bt)(a_2 + a_{-4} - a_{4} - a_{-2}),
\end{split}
\]
Since $\{a_2, a_{-4}, a_{4}, a_{-2}\}$ are linearly independent in all the algebras, we have $0 = \tfrac{\al}{2}-\bt = \tfrac{\al}{2}-\frac{1}{2}$ and so $\al = 1$, a contradiction.
\end{proof}
\begin{lemma}\label{lem:Ve 6Y}
The algebra $\V_e$ is not isomorphic to $ 6\Y(\tfrac{1}{2},2)$ or $ 6\Y(\tfrac{1}{2},2)^\times$.
\end{lemma}
\begin{proof}
For a contradiction, suppose $\V_e$ is isomorphic to either $6\Y(\tfrac{1}{2},2)$ or $6\Y(\tfrac{1}{2},2)^\times$. By Table \ref{table6Y} and Note ~\ref{table6Y quotient} in Table \ref{table6Y}, $V=V_e$ has a spanning set
$\{a_0,a_4,a_8,d,z\}$. By Note~\ref{table6Y axis relation} in Table~\ref{table6Y}, $a_2=a_8+d$, $a_6=a_0+d$ and $a_{10}=a_4+d$. Since $a_1$ is invariant under its own Miyamoto map, there are $\mu_0$, $\mu_1$, $\mu_2$, $\mu_3$, and $\mu_4\in \F$ such that 
\[
\begin{split}
a_1&=\mu_0(a_0+a_2)+\mu_1(a_4+a_{10})+\mu_2(a_6+a_{8})+\mu_3 d+\mu_4z\\
&=\mu_0(a_0+a_8+d)+\mu_1(a_4+a_4+d)+\mu_2(a_0+d+a_8)+\mu_3 d+\mu_4z\\
&=(\mu_0+\mu_2)(a_0+a_8)+2\mu_1a_4+(\mu_0+\mu_1+\mu_2+\mu_3) d+\mu_4z.
\end{split}
\]
Since, by \cite[Lemma 7.19]{MM}, $\tau_4$ fixes $d$ and $z$,
\[
a_7=a_1^{\tau_4}=(\mu_0+\mu_2)(a_8+a_0)+2\mu_1a_4+(\mu_0+\mu_1+\mu_2+\mu_3) d+\mu_4z=a_1.
\]
A contradiction. 
\end{proof}

\begin{proof}[Proof of Theorem~\ref{teonsa}]
Let $\V$  be an axial algebra satisfying the conditions of Theorem ~\ref{teonsa}. By Theorems~\ref{al=2bt}, \ref{skew}, and Lemma~\ref{lhyp1}, we may assume that $\V$ satisfy Hypothesis \ref{Hypns}. By Lemma ~\ref{lem:VeVo}, $\V_e$ is a symmetric $\Mab$-axial algebra. By the Classification Theorem of the Symmetric Algebras, there is a complete list of what $\V_e$ is. By Lemma ~\ref{lem:JT}, if $\V_e$ is an axial algebra of Jordan type, then $\V$ is symmetric. By Lemma~\ref{Ve=IY3} and ~\ref{Ve=IY5}, if $\V_e$ is isomorphic to either $\IY_3(\al,\tfrac{1}{2};\mu)$, $\IY_5(\al,\tfrac{1}{2})$, or any of their quotients, then $\V$ is symmetric. By Corollary~\ref{cor:H}, if $\V_e$ is isomorphic to a quotient of $\hatH$, then $\V$ is symmetric. Finally, Lemmas ~\ref{lem:Ve 3A}--~\ref{lem:Ve 6Y} show that $\V_e$ cannot be any of the remaining algebras. Since these cases exhaust all possibilities of $\V_e$, the result follows. 
\end{proof}


\chapter{A result on algebras with regular axet and small axial dimension}\label{ch:infinite}

Let 
\[
{\mathcal L}:=\{ \J(\delta), \; \J(0)^\times, \; \IY_3(\al, \tfrac{1}{2}; \mu),\; \IY_3(\al, \tfrac{1}{2}; 1)^\times,\; \IY_3(-1, \tfrac{1}{2}; \mu)^\times :  \delta, \al, \mu \in \F\}.
\]
 
In this chapter we study the case when $\bt=\tfrac{1}{2}$ and $\{\V_e, \V_o\}\subseteq \mathcal L$.
By Theorems~\ref{al=2bt},~\ref{teoq},~\ref{skew},~\ref{teonsa}, and Lemma~\ref{lhyp1}, we can restrict ourselves to studying the $2$-generated $\mathcal{M}(\al,\bt)$-axial algebras $\V$ over $\F$  satisfying the following conditions: 
\begin{hyp}\label{hyp1}~
\begin{enumerate}[enum_arabic]
    \item $\bt=\tfrac{1}{2}$ \textup{(}hence $\al\neq 2\bt$\textup{)}, \label{hyp1_1}
    \item  $\V$ is not isomorphic to a quotient of a symmetric algebra, \label{hyp1_2}
    \item $V\not \in \{V_e, V_o\}$, \label{hyp1_3}
    \item $a_0\neq a_2$ and  $ a_1\neq a_{-1}$, \label{hyp1_4}
    \item $\V$ has regular axet. \label{hyp1_5}
\end{enumerate}
\end{hyp}

Note that, as we assume $\bt=\tfrac{1}{2}$, $\al\not \in\{0,1,\tfrac{1}{2}\}$. Moreover, for $W\in \mathcal L$, by Lemmas~\ref{subJ} and~\ref{subIY3}, $\adim(W)\leq 3$.  Since $\ch(\F)\neq 2$, by~\cite[Lemma~5.3]{split} and~\cite[\S 7.3]{axet},  $W$ has axet $X(n)$ with $n\geq 3$. In this chapter, it will be  useful to have at hand the possible dimensions of $W$ and of its subalgebras $W^\ast$ and $W^{\ast\ast}$ (as defined in Chapter~\ref{known}). These dimensions can be deduced immediately from Lemmas~\ref{subJ} and~\ref{subIY3}. For the convenience of the reader, we list them in the following table.

{\renewcommand{\arraystretch}{1.5}
\begin{table}[H] 
{\Small
\[
\begin{array}{|l|c|c|c|}
\hline

W& \dim(W) &\dim(W^\ast) &\dim(W^{\ast \ast})\\
\hline
\J(\delta),\;
\delta\not \in \{0,-\tfrac{1}{2}\}& 3 & 3 & 3 \\
\hline
\J(0)& 3& 2& 2 \\
\hline
\J(-\tfrac{1}{2})& 3& 3& 1 \\
\hline
\J(0)^\times & 2& 1& 1 \\
\hline
\IY_3(\al,\tfrac{1}{2};\mu), \; \al\neq 2, \; \mu\not\in\{1,-1\} & 4& 4& 4 \\
\hline
\IY_3(2,\tfrac{1}{2};\mu), \; \mu\neq -1 & 4& 3& 3 \\
\hline
\IY_3(2,\tfrac{1}{2};-1) & 4& 3& 1 \\
\hline
\IY_3(\al,\tfrac{1}{2};-1),\; \al\neq 2 & 4& 4& 1 \\
\hline
\IY_3(\al,\tfrac{1}{2};1)& 4& 3& 3 \\
\hline
\IY_3(-1,\tfrac{1}{2};\mu)^\times, \; \mu\not\in\{1,-1\} & 3& 3& 3 \\
\hline
\IY_3(-1,\tfrac{1}{2};-1)^\times & 3& 3& 1 \\
\hline
\IY_3(\al,\tfrac{1}{2};1)^\times  & 3& 2 & 2\\
\hline
\end{array} 
 \]}
  \caption{Dimensions of $W$, $W^\ast$, and $W^{\ast\ast}$, for $W\in \mathcal L$.}\label{table3}
 \end{table}
 }

We  prove the following result.
\begin{theorem}\label{case4}
 Let $\V$ be a $2$-generated $\mathcal{M}(\al,\bt)$-axial algebra over $\F$ satisfying Hypothesis~$\ref{hyp1}$. Then $\{\V_e,\V_o\}\not \subseteq \mathcal L$.
\end{theorem}

For the remainder of this chapter we assume $\V$ satisfies the hypothesis of Theorem~\ref{case4} and we suppose for the sake of contradiction that $\{\V_e,\V_o\} \subseteq \mathcal L$. 

For $x\in \{e,o\}$, define
\[
    \varepsilon_x:= \begin{cases}
        8\delta+3 & \mbox{ if } \V_x \quad \mbox{ is isomorphic to a quotient of } \J(\delta); \\
        2\mu+1 & \mbox{ if } \V_x \quad \mbox{ is isomorphic to a quotient of }\quad  \IY_3(\al, \tfrac{1}{2}; \mu).
    \end{cases} 
\]

\begin{lemma}\label{a3-am3}
 The following identities hold in the algebra $\V$: 
\begin{enumerate}
\item  $a_{-2}-a_{4}=\varepsilon_e (a_{0}-a_{2})$, \label{suba3-am3_1}
    \item  $a_3-a_{-3}=\varepsilon_o(a_{1}-a_{-1})$.  \label{suba3-am3_2}
    \end{enumerate}
\end{lemma}
\begin{proof}
 Suppose $\V_e\cong \J(\delta)$. By Lemma~\ref{subJ}, $a_{4}=-a_0+(8\delta+2) a_{2}-4s_{0,2}$ and $a_{-2}=-a_{2}+(8\delta +2)a_{0}-4s_{0,2}$. Taking the difference of the above equations, we get  
 $$
a_{-2}-a_{4}=(8\delta+3)(a_0-a_{2}). 
 $$
The proof when $\V_e\cong \J(0)^\times$ is similar. Suppose $\V_e$ is isomorphic to a quotient of $\IY_3(\al, \tfrac{1}{2};\mu)$. Then, by Lemma~\ref{subIY3}, $a_{-2}=a_{4}+(2\mu+1)(a_{0}-a_2)$, giving~\ref{suba3-am3_1}.
The proof of \ref{suba3-am3_2} is obtained considering $\V_o$ in the place of $\V_e$.
\end{proof}

Let $P$, $Q$, $Q^f$, $R$, $R^f$, $A$, $A^f$, $B$, $B^f$, $C$, and $C^f$ be as defined in Chapter~\ref{ch1} on pages~\pageref{primo}, ~\pageref{A}, and~\pageref{evaluation}.

\begin{cor}\label{eq8}
The following identities hold in the algebra $\V$:
\begin{equation}\label{Pee}
    (P\varepsilon_e -P+R)(a_0-a_2)=Q(a_{3}-a_{-1})
\end{equation}
and
\begin{equation}\label{Peo}
    (P\varepsilon_o -P+R^f)(a_1-a_{-1})=Q^f(a_{-2}-a_2).
\end{equation}
In particular, if $Q\neq 0$ (respectively $Q^f\neq 0$), then \begin{equation*}
    V_e^\ast=V_o^{\ast\ast}\quad  (\mbox{respectively}\quad  V_o^\ast=V_e^{\ast\ast}).
\end{equation*}
\end{cor}
\begin{proof}
Equation~\eqref{Pee} (respectively~\eqref{Peo}) follows at once from Equation~\eqref{equa1} (respectively \eqref{equa2}) on page~\pageref{equa1} and Lemma~\ref{a3-am3}. Assume $Q\neq 0$. Since, for $\V_e\in \mathcal L$, $a_{3}\neq a_{-1}$, by Equation~\eqref{Pee} and  Lemma~\ref{lem:VeVo} we get $V_e^\ast=V_o^{\ast\ast}$.  Similarly, if $Q^f\neq 0$, the result follows using  Equation~\eqref{Peo}. 
\end{proof}

By Hypothesis~\ref{hyp1}\ref{hyp1_3}, $a_0\neq a_2$ and $a_1\neq a_{-1}$. Thus,  if $\{Q, Q^f\}\neq\{0\}$, then Equations~\eqref{Pee} and~\eqref{Peo} give relations between the generating axes of the odd and the even subalgebras $\V_o$ and $\V_e$. 
On the other hand, if $\{Q, Q^f\}=\{0\}$, then we get the identities 
\begin{equation}\label{QQf=0}
    P=P\varepsilon_e+R=P\varepsilon_o+R^f.
\end{equation}
This leads to the dichotomy $\{Q, Q^f\}\neq\{0\}$ or $\{Q, Q^f\}=\{0\}$. We shall show that both cases lead to a contradiction.

\begin{proposition}
\label{*=**}
Assume $\{Q, Q^f\}\neq \{0\}$ or $\al\neq 4\bt$ and $\{B, B^f\}\neq \{0\}$. Then
$V_e^{\ast\ast}=V_e^\ast=V_o^{\ast}=V_o^{\ast\ast}$.
\end{proposition}
\begin{proof}
Suppose, for a contradiction, that $\V$ does not satisfy the condition
$$V_e^{\ast\ast}=V_e^\ast=V_o^{\ast}=V_o^{\ast\ast}.
$$
Assume $Q\neq 0$ or $\al\neq 4\bt$ and $B^f\neq 0$. Then, by Corollary~\ref{eq8} or by Lemma~\ref{Ve*Vo*} (respectively), $V_e^\ast=V_o^{\ast\ast}$. Whence, Hypothesis~\ref{hyp1}\ref{hyp1_3} implies $V_e^\ast<V_e$. 
We claim that 
\begin{equation}\label{eq*<**}
V_e^{\ast\ast}<V_e^\ast=V_o^{\ast\ast}=V_o^\ast.
\end{equation}
Suppose, for the sake of contradiction, that $V_o^{\ast\ast}<V_o^\ast$. Then, by Table~\ref{table3}, $\V_o$ is isomorphic to one of
$$
\J(-\tfrac{1}{2}), \quad \IY_3(\al, \tfrac{1}{2}; -1), \quad \IY_3(-1, \tfrac{1}{2}; -1)^\times,
$$
in particular $\dim(V_e^\ast)=\dim(V_o^{\ast\ast})=1$, whence,  again by Table~\ref{table3}, $\V_e\cong \J(0)^\times$. By Lemma~\ref{subJ}\ref{subJ_3}, 
 $V_e^\ast=\langle a_{0}-a_2\rangle$,  whence, by Lemma~\ref{subJ}\ref{subJ_2},
\begin{equation}\label{action-1}
(a_0-a_{2})^{\tau_0}=a_0^{\tau_0}-a_{2}^{\tau_0}=a_0-a_{-2}=a_0-(-a_2+2a_0)=-(a_0-a_{2}).
\end{equation}
If  $\V_o\cong \J(-\tfrac{1}{2})$, then, by Lemma~\ref{subJ}, $V_o^{\ast\ast}=\langle a_3-a_{-1}\rangle$ and 
$$
a_3-a_{-1}=-2(a_1+a_{-1}+2s_{\tu,2})=-2(a_1+a_{-1}+2s_{\tu,2})^{\tau_0}=(a_3-a_{-1})^{\tau_0}. 
$$ 
Since $V_o^{\ast \ast}=V_e^\ast=\langle a_{0}-a_2\rangle$, this contradicts Equation~\eqref{action-1}. 
 Similarly, if $\V_o\cong \IY_3(\al,\tfrac{1}{2}; -1)$ or $\V_o\cong \IY_3(-1,\tfrac{1}{2}; -1)^\times$, then, by parts \ref{subIY3_2} and \ref{subIY3_4} of Lemma~\ref{subIY3},
 $$
 a_{-3}-a_{1}= a_3-a_{-1} \quad \mbox{ and }\quad V_o^{\ast\ast}=\langle a_3-a_{-1}\rangle,
 $$ 
 whence
 $$
  (a_{3}-a_{-1})^{\tau_0}=a_{3}^{\tau_0}-a_{-1}^{\tau_0}=a_{-3}-a_{1}= a_3-a_{-1}.
 $$
As above, it follows that $\tau_0$ acts trivially on $V_e^\ast$, contradicting Equation~\eqref{action-1}, proving $V_o^{\ast\ast}=V_o^\ast$ and Equation~\eqref{eq*<**}. 
Now Equation~\eqref{eq*<**} and Hypothesis~\ref{hyp1}\ref{hyp1_3} imply $V_o^\ast<V_o$. Thus, comparing the dimensions of $V_e^\ast$, $V_e^{\ast\ast}$, $V_o^\ast$,  and $V_o^{\ast\ast}$ in Table~\ref{table3}, it follows that 
$\V_e\cong \IY_3( 2, \tfrac{1}{2}; -1)$ and $\V_o\cong \IY_3( 2, \tfrac{1}{2}; \mu)$ with $\mu\neq -1$. 

Clearly, swapping the role of $\V_e$ and $\V_o$, the above argument shows that if $Q^f\neq 0$ or $\al\neq 4\bt$ and $B\neq 0$, then 
\begin{equation}\label{eq*=**bis}
V_o^{\ast\ast}<V_o^\ast=V_e^{\ast\ast}=V_e^\ast.
    \end{equation}
and $\V_o\cong \IY_3( 2, \tfrac{1}{2}; -1)$ and $\V_e\cong \IY_3( 2, \tfrac{1}{2}; \mu)$ with $\mu\neq -1$. 

Since Equations~\eqref{eq*<**} and~\eqref{eq*=**bis} are incompatible, we cannot have $Q\neq 0\neq Q^f$. Assume $Q\neq 0$ and $Q^f=0$. By the above discussion $\V_e\cong \IY_3( 2, \tfrac{1}{2}; -1)$ and $\V_o\cong \IY_3( 2, \tfrac{1}{2}; \mu)$ with $\mu\neq -1$, in particular $(\al, \bt)=(2,\tfrac{1}{2})$ and,  by Lemma~\ref{subIY3}\ref{subIY3_1}, $\lmd=\lmdf=1$. Substituting  $\al$ and $\bt$ by $2$ and $\tfrac{1}{2}$ in  Lemma~\ref{newQ-Q^f}\ref{eqQ=0}, we obtain $\lmf=1$. 
Furthermore, Equation~\eqref{eq*<**} and Lemma~\ref{Ve*Vo*} imply $C^f=0$. Evaluating $C^f$ in $(\al, \bt, \lmf, \lmdf)=(2, \tfrac{1}{2}, 1,1)$, we get $\lmu=1$. 
By Theorem~\ref{HWthm}, $\V$ is isomorphic to a quotient of a symmetric algebra, contradicting Hypothesis~\ref{hyp1}\ref{hyp1_2}.
\end{proof}





\section{The case $\{Q, Q^f\}=\{0\}$}
\label{sec:Q=Qf=0}

\begin{lemma}\label{lm<>lmf}
Assume that $Q=Q^f=0$. Then 
\begin{enumerate}
    \item  $\lmf=\lmu$; \label{lm<>lmf_1}
    \item  $\lmd-\lmdf=-\tfrac{(2\al-1)(\al-2)}{4\al}(\varepsilon_e-\varepsilon_o)$. \label{lm<>lmf_2}
\end{enumerate} 
\end{lemma}
\begin{proof}
In order to prove \ref{lm<>lmf_1}, suppose, for the sake of contradiction, that $\lmu\neq \lmf$. Since, by Hypothesis~\ref{hyp1}\ref{hyp1_3}, $V_e\neq V\neq V_o$,   $V_e\cap V_o$ is  properly contained in both $V_e$ and $V_o$.

Since $Q=Q^f=0$ and $\lmu\neq \lmf$, by Lemma~\ref{newQ-Q^f}\ref{newQ-Q^f_1}, $\al\neq 2$ and $\bt=\tfrac{\al(\al-1)}{2(\al-2)}$, whence $\bt=\tfrac{1}{2}$ implies 
\begin{equation}\label{alcond}
    \al^2-2\al+2=0.
\end{equation} 
In particular,
\begin{equation}\label{alcond1}
\al^2=2\al-2 \quad \mbox{ and } \quad 
\al^{-1}=-\tfrac{1}{2}\al + 1.
\end{equation} 
Substituting $\bt$ by $\tfrac{1}{2}$ and the above values for $\al^2$ and $\al^{-1}$ in Lemma~\ref{newQ-Q^f}\ref{eqQ=0},  we get 
 \begin{equation}\label{L}
 \lmu=-\lmf+\tfrac{1}{2}(\al+1).
     \end{equation}     
We claim that 
\begin{equation}\label{cond}
    (\lmd,\varepsilon_e)\neq (\lmdf, \varepsilon_o), \mbox{ in particular } \V_e\not \cong \V_o.
\end{equation} 
Assume for a contradiction that $(\lmd,\varepsilon_e)=(\lmdf, \varepsilon_o)$. Then 
$\lmd=\lmdf$, whence Lemma~\ref{newQ-Q^f}\ref{newQ-Q^f_3} yields   
$$
4(\al-1)\lmu-(2\al-3)= R-R^f .
$$ 
 On the other hand, since $\varepsilon_e=\varepsilon_o$, by Equation~\eqref{QQf=0}, we get 
\[
R-R^f=0,
\]
whence, by Equation~\eqref{alcond1}, 
\[
4(\al-1)\lmu=2\al-3 =\al^2-1.
\]
Using this and $4(\al-1)$ times Equation~\eqref{L}, we get
$$
4(\al-1)\lmf=-4(\al-1)\lmu+2\al^2-2= -\al^2+1+2\al^2-2=\al^2-1= 4(\al-1)\lmu.
$$
Since $\al\neq 1$, this implies $\lmu=\lmf$, contradicting our assumption. This proves Equation~\eqref{cond}.

Suppose $B\neq 0$. Then,  Proposition~\ref{*=**} implies
\[
V_o^{\ast\ast}=V_o^\ast=V_e^{\ast}=V_e^{\ast\ast}\subseteq V_e\cap V_o.
\] 
Since $V_e\cap V_o$ is properly contained in both $V_e$ and $V_o$, by Table~\ref{table3}, we get that  $\V_o$ and $\V_e$ are isomorphic to one of the following
\begin{equation*}
\J(0),\quad  \J(0)^\times,\quad  \IY_3(\al, \tfrac{1}{2}; 1), \quad \IY_3(\al, \tfrac{1}{2}; 1)^\times.
\end{equation*}
By Lemmas~\ref{subJ} and~\ref{subIY3}, and the definition of $\varepsilon_e$, $\varepsilon_o$, it follows 
$(\lmdf, \varepsilon_o)=(1,3)=(\lmd, \varepsilon_e)$, contradicting Equation~\eqref{cond}.
Therefore $B=0$. By repeating the above argument swapping the roles of $V_e$ and $V_o$ we get also $B^f=0$. Then, substituting in the defining formulas of $B$ and $B^f$ the value $\bt=\tfrac{1}{2}$, and using Equations~\eqref{alcond} and~\eqref{L}, we get
$$
0= B+B^f=16\al(\lmu-\tfrac{1}{4}(\al+1))^2,
$$
whence $\lmu=\tfrac{1}{4}(\al+1)$ and so, by Equation~\eqref{L},  $\lmf=\lmu$, contradicting our assumption. This proves \ref{lm<>lmf_1}.

Since $\lmu=\lmf$, the equation in Lemma~\ref{newQ-Q^f}\ref{newQ-Q^f_3} with $\bt=\tfrac{1}{2}$ (and $\al\neq \bt$) implies  
 \begin{equation}\label{lm-lmf1}
\lmd-\lmdf=\tfrac{2}{\al(2\al-1)}(R-R^f). 
\end{equation}
 On the other hand, by Equation~\eqref{QQf=0} we get 
$$
 R-R^f=-P(\varepsilon_e - \varepsilon_o). 
$$
 Substituting  this value for $ R-R^f$ in Equation~\eqref{lm-lmf1}, we get
\begin{equation*}\label{lm2-lm2f}
 \lmd-\lmdf=-\tfrac{2}{\al(2\al-1)}P(\varepsilon_e - \varepsilon_o),  
\end{equation*} 
which implies \ref{lm<>lmf_2} by the definition of $P$. 
\end{proof}

\begin{lemma}\label{lm=lmf}
  $\{Q, Q^f\}\neq \{0\}$. 
\end{lemma}
\begin{proof}
Assume, for the sake of contradiction, that $\{Q, Q^f\}= \{0\}$. Then Lemma~\ref{lm<>lmf}\ref{lm<>lmf_1} implies $\lmu=\lmf$.
If $\al=4\bt=2$, then by Lemma~\ref{lm<>lmf}\ref{lm<>lmf_2}, $\lmd=\lmdf$ and, by Corollary~\ref{reduction}, either $\V$ is   isomorphic to a quotient of a symmetric algebra, contradicting Hypothesis~\ref{hyp1}\ref{hyp1_2}, or $\V$ has a skew axet, contradicting Hypothesis~\ref{hyp1}\ref{hyp1_5}.
Thus $\al\neq 4\bt$, in particular 
\begin{equation}\label{al<>2}
 \al\neq 2. 
\end{equation}
Then, by Proposition~\ref{reg},
\begin{equation}\label{hyplamba2}
\lmd\neq \lmdf.
\end{equation}
By Lemma~\ref{quotient}, it follows that $\V_e$ and $\V_o $ are not isomorphic to quotients of the same algebra, and,  
by Lemma~\ref{lm<>lmf}\ref{lm<>lmf_2}, 
\begin{equation}\label{eeo}
    \varepsilon_e\neq \varepsilon_o.
\end{equation}
Since $\lmu=\lmf$, by Equation~\eqref{A} (see also the paragraph at the beginning of Section~\ref{evaluation}), 
$$
A=-A^f=\tfrac{1}{16}\al(2\al-1)^2(\al-2)(\lmd-\lmdf).
$$
Since  $\al\not \in \{0, \bt\}$ and $\bt=\tfrac{1}{2}$,  Equations~\eqref{al<>2} and~\eqref{hyplamba2} imply
\begin{equation*}\label{eqA}
    A=-A^f\neq 0.
\end{equation*}
By the remark before Table~\ref{table3}, $\V_e$ and $\V_o$ have axet $X(n)$ with $n>2$, whence $a_{-2}\neq a_2$ and $a_{-3}\neq a_1$. Thus, by Lemma~\ref{s1s2}, it follows that $V_e^{\ast}=V_o^{\ast\ast}$ and $V_o^\ast=V_e^{\ast \ast}$. Since $V_o^{\ast\ast}\subseteq V_o^\ast$ and $V_e^{\ast \ast}\subseteq V_e^\ast$, we get 
\[
V_e^{\ast}=V_e^{\ast\ast}=V_o^\ast=V_o^{\ast \ast}.
\]
If $V_e^\ast=V_e$ (respectively $V_o^\ast=V_o$), then this gives $V_e\subseteq V_o$ (respectively $V_o\subseteq V_e$), whence $V=V_o$ (respectively $V=V_e$), a contradiction to Hypothesis~\ref{hyp1}\ref{hyp1_3}. Hence 
$$
V_e^\ast<V_e \quad \mbox{ and }\quad  V_o^\ast<V_o.
$$
Then, by Table~\ref{table3}, $\V_e$ (respectively $\V_o$) is isomorphic to one of the following
$$
\J(0), \quad \J(0)^\times, \quad \IY_3(\al,\tfrac{1}{2}; 1), \quad \IY_3(\al,\tfrac{1}{2}; 1)^\times, 
$$
which implies $\{\varepsilon_e, \varepsilon_o\}=\{3\}$, contradicting Equation~\eqref{eeo}.
\end{proof}

\section{The case $\{Q, Q^f\}\neq\{0\}$}
\label{sec:largeint}

Throughout this section we assume $\{Q, Q^f\}\neq\{0\}$.
By Proposition~\ref{*=**} and  Hypothesis~\ref{hyp1}\ref{hyp1_3}, 
\begin{equation}
    \label{eq*=**}
  V_o\neq V_o^{\ast}=V_o^{\ast\ast}=V_e^{\ast\ast}=V_e^{\ast}\neq V_e.  
\end{equation}

\begin{lemma}\label{possipairs}
Assume $\{\V_e, \V_o\}\subseteq \mathcal {L}$ and Equation~\eqref{eq*=**} is satisfied. Then 
\begin{enumerate}
    \item  the unordered pair $(\V_e, \V_o)$ is (up to isomorphism) one of those listed in the first column of Table~$\ref{tableposs}$;
    \item  for each possible pair, the dimension of $V_e^\ast$ is given in the second column of Table~$\ref{tableposs}$.
\end{enumerate}
\end{lemma}
{\renewcommand{\arraystretch}{1.5}
\begin{table}[H]
{\Small
\[
\begin{array}{|l|c|}
\hline
\begin{array}{l}(\V_e, \V_o)\end{array} & \dim(V_e^\ast) \\
\hline
  \begin{array}{l} \left (\J(0)^\times,\; \J(0)^\times\right )\end{array}  & 1\\
 \hline
 \begin{array}{l}
 \left (\J(0),\; \J(0)\right ),\\
 \left (\J(0), \;\IY_3(\al,\tfrac{1}{2}; 1)^\times\right ), \\\left (\IY_3(\al,\tfrac{1}{2}; 1)^\times, \;\IY_3(\al,\tfrac{1}{2}; 1)^\times\right )
 \end{array} & 2
 \\
 \hline
 \begin{array}{l}
 \left (\IY_3(2,\tfrac{1}{2}; \mu),\; \IY_3(2,\tfrac{1}{2}; \mu^\prime)\right ),\quad  \mbox{ with }\quad  \{\mu,\mu^\prime\}\subseteq \F\setminus \{-1\},\\
 \left (\IY_3(\al,\tfrac{1}{2}; 1), \;\IY_3(\al,\tfrac{1}{2}; 1)\right ),\quad   \mbox{ with }\quad  \al\in \F\setminus \{2\}.
 \end{array} & 3
\\
\hline
\end{array}
\]

\caption{{ Possible pairs $(\V_e, \V_o)$ with $
 V_o\neq V_o^{\ast\ast}=V_o^\ast=V_e^{\ast}=V_e^{\ast\ast}\neq V_e.
$}
}\label{tableposs}}
\end{table}
}

\begin{remark}\label{remlambda1}
By Lemmas~$\ref{subJ}$\ref{subJ_1} and~$\ref{subIY3}$\ref{subIY3_1}, in all the algebras appearing in Table~$\ref{tableposs}$, $\lmu=1$. Hence, in the algebra $\V$, $\lmd=\lmdf=1$. Since, by Hypothesis~$\ref{hyp1}$\ref{hyp1_2}, $\V$ is not isomorphic to quotient of a  symmetric algebra, by Proposition~$\ref{reg}$ it follows that 
\begin{equation}\label{eqrem}
 \mbox{either}\quad  \lmu\neq \lmf \quad \mbox{or}\quad \al=2. 
\end{equation}
\end{remark}

In the following lemmas, we shall show that each of the above pairs leads to a contradiction. We start with the following observation. 

\begin{lemma}\label{smult}
Assume $s_{\tu,2}=ys_{\tz,2}$, for some $y\in \F\setminus \{0\}$, then $\al=2$ and $\lmu=\lmf$.   \end{lemma}
\begin{proof}
Since $\bt=\tfrac{1}{2}$ and, as noted in Remark~\ref{remlambda1}, $\lmd=\lmdf=1$, 
 Lemma~\ref{s}\ref{s_2} implies $\lm_{a_0}(s_{\tz,2})=0$ and $\lm_{a_1}(s_{\tu,2})=0$, whence 
$$
\lm_{a_0}(s_{\tu,2})=y\lm_{a_0}(s_{\tz,2})=0\quad \mbox{ and }\quad 
 \lm_{a_1}(s_{\tz,2})=y^{-1}\lm_{a_1}(s_{\tu,2})=0.
$$
Thus, by Corollary~\ref{cor:lm2f},
$$0=\lm_{a_0}(s_{\tu,2})=\tfrac{2(\al-1)}{\al-\bt}\lmu(\lmu-\lmf)+(1-2\bt)\lmu+\bt\lmd-\bt=\tfrac{2(\al-1)}{\al-\bt}\lmu(\lmu-\lmf)$$
$$0= \lm_{a_1}(s_{\tz,2})=\tfrac{2(\al-1)}{\al-\bt}\lmf(\lmf-\lmu)+(1-2\bt)\lmf+\bt\lmdf-\bt=\tfrac{2(\al-1)}{\al-\bt}\lmf(\lmf-\lmu),$$
whence $\lmu=\lmf$. By Equation~\eqref{eqrem}, $\al=2$.   
\end{proof}

\begin{lemma}\label{Jx} $(\V_e, \V_o)\neq (\J(0)^\times, \J(0)^\times)$. In particular, $\dim(V_e^\ast)\neq 1$.
\end{lemma}
\begin{proof}
Assume, for a contradiction, that 
$$\V_e\cong \V_o\cong \J(0)^\times.$$
By Lemma~\ref{subJ} and Equation~\eqref{eq*=**},  $\langle a_0-a_2\rangle= V_e^\ast=V_o^\ast=\langle a_1-a_{-1}\rangle$. Hence, there exists $x\in \F\setminus \{0\}$ such that 
\begin{equation}\label{a1m1=xa1m2}
 a_1-a_{-1}=x(a_2-a_0).
\end{equation}
Thus
$a_{-1}=a_1-x(a_2-a_0)$ and so, by Lemma~\ref{primo}, it follows that 
\begin{equation}\label{eqdim}
V=\langle a_0, a_1, a_2, s_{\tz,1}\rangle.
\end{equation} 
By Equation~\eqref{defs} on page~\pageref{defs}
\begin{equation*}\label{s2f=a1-am1}
(a_{-1}-a_1)^2=a_{-1}+a_1-2a_{-1}a_1=a_{-1}+a_1-a_{-1}-a_1-2s_{\tu,2}=-2s_{\tu,2}
\end{equation*}
and similarly
\begin{equation*}\label{s2=a0-a2}
(a_{0}-a_2)^2=-2s_{\tz,2}.
\end{equation*}
Hence Equation~\eqref{a1m1=xa1m2} implies
$$
s_{\tu,2}=-\tfrac{1}{2}(a_{-1}-a_1)^2=-\tfrac{1}{2}x^2(a_2-a_0)^2=x^2s_{\tz,2},
$$
whence, by Lemma~\ref{smult}, $\lmu=\lmf$. Then, by Equation~\eqref{eqdim} and~\cite[Corollary~6.15]{FMS3}, $\V$ is isomorphic to a quotient of a symmetric algebra, contradicting Hypothesis~\ref{hyp1}\ref{hyp1_2}. This prove that $(\V_e, \V_o)\neq (\J(0)^\times, \J(0)^\times)$. By Lemma~\ref{possipairs}, $\dim(V_e^\ast)\neq 1$.
\end{proof}

\begin{lemma}\label{JJ} 
$(\V_e,\V_o)\not \in \{(\J(0),\J(0)), (\J(0), \IY_3(\al,\tfrac{1}{2}; 1)^\times)\}$.
\end{lemma}
\begin{proof}
Up to swapping $\V_e$ with $\V_o$, we may assume, for the sake of contradiction, that $\V_e\cong \J(0)$ and $\V_o$ is isomorphic either to $\J(0)$ or to $\IY_3(\al,\tfrac{1}{2}; 1)^\times$.  
By Table~\ref{table2},  $V_e$ has basis $(a_0, a_2, s_{\tz,2})$. 
By Lemma~\ref{subJ}, 
\begin{equation}\label{lambdaJ0}
    \lmd=1, \quad \dim(V_e^\ast)=2,\quad  a_{-2}=-a_2+2a_0-4s_{\tz,2},
\end{equation}
and 
\begin{equation*}\label{V*IY3m1x}
    V_e^\ast=\langle a_0-a_2, s_{\tz,2}\rangle.
\end{equation*}
As $\V_o$ is isomorphic either to $\J(0)$ or to $\IY_3(\al,\tfrac{1}{2}; 1)^\times$, by Lemmas~\ref{subJ} and \ref{subIY3},
$$
\lmdf=1=\lmd\quad  \mbox{ and }\quad  V_e^\ast=V_o^\ast=\langle a_{-1}-a_1, a_1-a_3\rangle.
$$
Since 
\begin{equation}\label{eqa2am2}
a_{-2}-a_2=-a_2+2a_0-4s_{\tz,2}-a_2=2(a_0-a_2)-4s_{\tz,2}\in V_e^\ast,
\end{equation}
it follows that $(a_{-2}-a_2, s_{\tz,2})$ is a basis for $V_e^\ast$.
Hence $a_1-a_{-1}$ is a linear combination of $a_{-2}-a_2$ and $s_{\tz, 2}$. Since $\tau_0$ fixes  $s_{\tz,2}$ and negates  $a_1-a_{-1}$ and $a_{-2}-a_2$, there exists $x\in \F\setminus \{0\}$ such that 
\begin{equation}\label{s12bis}
  a_1-a_{-1}=x(a_2-a_{-2}).
\end{equation}
Substituting the expression for $a_{-2}$ given in Equation~\eqref{lambdaJ0} into this, we get 
\begin{equation}\label{am1}
   a_{-1}=a_1-2x(a_2-a_{0}+2s_{\tz,2})
\end{equation}
and 
\begin{equation}\label{a3}
a_3=(a_{-1})^{\tau_1}=\left (a_1-2x(a_2-a_{0}+2s_{\tz,2})\right )^{\tau_1}=a_1-2x(a_0-a_2+2s_{\tz,2}).
\end{equation}
By Equations~\eqref{defs} on page~\pageref{defs}, \eqref{eqa2am2}, ~\eqref{s12bis}, and Table~\ref{table3}, 
\begin{align}\label{s12n}
s_{\tu,2}&= -\tfrac{1}{2}(a_1-a_{-1})^2 \nonumber \\
&= -\tfrac{1}{2}x^2(a_2-a_{-2})^2 \nonumber \\
&= -\tfrac{1}{2}x^2(2(a_0-a_2)-4s_{\tz,2})^2  \\
&= -2x^2(a_0-a_2)^2\nonumber \\
&= 4x^2s_{\tz,2}.\nonumber 
\end{align}
By Lemma~\ref{smult}, $\al=2=4\bt$ and $\lmu=\lmf$. Since, by Tables~\ref{table2} and~\ref{tableIY3bis}, $V_o=\langle a_{-1}, a_1, a_3, s_{\tu,2}\rangle$, by Equations~\eqref{am1}, \eqref{a3}, 
\eqref{s12n}, Equation~\eqref{defs} on page~\pageref{defs}, Lemma~\ref{s}, Table~\ref{table2}, and~\cite[Lemma~7.6]{FMS3}, it follows that
$$
V=\langle a_0, a_1, a_2, s_{\tz,1}, s_{\tz,2}\rangle .
$$
Hence, by~\cite[Corollary ~5.10]{FMS3}, $\V$ is isomorphic to a quotient of a symmetric algebra, contradicting Hypothesis~\ref{hyp1}\ref{hyp1_2}.
\end{proof}

\begin{lemma}\label{Ve=IY3x}
   $(\V_e,\V_o)\neq ( \IY_3(\al, \tfrac{1}{2}; 1)^\times, \IY_3(\al, \tfrac{1}{2}; 1)^\times)$. 
\end{lemma}
\begin{proof}
Suppose for a contradiction that $\V_e\cong \V_o\cong \IY_3(\al, \tfrac{1}{2}; 1)^\times$. 
By Table~\ref{tableIY3bis},
$V_e$ has basis $(a_{-2}, a_0, a_2)$, and $s_{\tz,2}=\tfrac{2\al-1}{4}(-2a_0+a_2+a_{-2})$. Thus 
\begin{equation}\label{a2IY3m1x}
    a_{-2}=2a_0-a_2+\tfrac{4}{2\al-1}s_{\tz,2} 
\end{equation}
By Lemma~\ref{subJ}, 
\begin{equation*}\label{V*IY3m1xbis}
    V_e^\ast=\langle a_0-a_2, a_0-a_{-2}\rangle=\langle a_0-a_2, s_{\tz,2}\rangle.
\end{equation*}
Similarly,  
$$
V_o^\ast=\langle a_{-1}-a_1, s_{\tu,2}\rangle.
$$
By arguing as in the proof of Lemma~\ref{JJ}, we get that there exists $x\in \F\setminus\{0\}$ such that
\begin{equation}\label{a1am1}
    a_1-a_{-1}=x(a_2-a_{-2}).
\end{equation}
Substituting the expression for $a_{-2}$ given in Equation~\eqref{a2IY3m1x} into this, we get
\begin{equation}\label{am1bis}
    a_{-1}=a_1-2x(a_2-a_0-\tfrac{2}{2\al-1}s_{\tz,2})
\end{equation}
and
\begin{align}\label{a3IY3x}
  a_3&= (a_{-1})^{\tau_1}=\left (a_1-2x(a_2-a_0-\tfrac{2}{2\al-1}s_{\tz,2})\right )^{\tau_1}\\
  &= a_1-2x(a_0-a_2-\tfrac{2}{2\al-1}s_{\tz,2}).  \nonumber
\end{align}
Now note that,
by Table~\ref{tableIY3bis},
\begin{equation}\label{s2s2}
  s_{\tz,2}\cdot s_{\tz,2}=0,\quad \mbox{ and }\quad (a_0-a_2)s_{\tz,2}=0.
\end{equation}
Thus, by Equations~\eqref{defs}, \eqref{a1am1} and \eqref{a2IY3m1x},
\begin{align}\label{s12ter}
 s_{\tu,2} &= -\tfrac{1}{2}(a_1-a_{-1})^2\nonumber \\
 &= -\tfrac{1}{2}x^2(a_2-a_{-2})^2  \nonumber\\
 &= -\tfrac{1}{2}x^2\left ( 2a_2-2a_0-\tfrac{4}{2\al-1}s_{\tz,2}\right )^2\\
 &= -2x^2(a_2-a_0)^2  \nonumber\\
 &= 4x^2s_{\tz,2} \nonumber
\end{align}
By Lemma~\ref{smult}, $\al=2=4\bt$ and  $\lmu=\lmf$.  Thus, by Table~\ref{tableIY3bis}, Equations~\eqref{a2IY3m1x}, \eqref{am1bis}, \eqref{a3IY3x}, \eqref{s2s2}, \eqref{s12ter}, Lemma~\ref{primo}, and~\cite[ Lemma~7.6]{FMS3} we get  $V=\langle a_0, a_1, a_2, s_{\tz,2}\rangle$. Hence, by~\cite[Corollary~5.10]{FMS3}, $\V$ is isomorphic to a quotient of a symmetric algebra, contradicting Hypothesis~\ref{hyp1}\ref{hyp1_2}. 
\end{proof}

\begin{cor}\label{corveast2}
$\dim(V_e^\ast)\neq 2$
\end{cor}

\begin{proof}
This follows from Lemmas~\ref{possipairs},~\ref{JJ} and~\ref{Ve=IY3x}. 
\end{proof}

\begin{lemma}\label{IY3al1}
$\dim(V_e^\ast) \neq 3$.
\end{lemma}
\begin{proof}
Suppose for a contradiction that $\dim(V_e^\ast) = 3$.
By Lemma~\ref{possipairs}, 
\begin{enumerate}
\item either $(\V_e,\V_o) = (\IY_3(\al, \tfrac{1}{2}; 1), \IY_3(\al, \tfrac{1}{2}; 1))$,  where $\al\neq 2$; or
\item $(\V_e,\V_o)=  (\IY_3(2, \tfrac{1}{2}; \mu), \IY_3(2, \tfrac{1}{2}; \mu^\prime))$,  where $\{\mu, \mu^\prime\}\subseteq \F\setminus \{ -1\}$.
\end{enumerate} 
In particular, by Lemma~\ref{subIY3}\ref{subIY3_1}, 
\begin{equation*}\label{lm2=lm2f}
    \lmd=\lmdf=1.
\end{equation*}
By Table~\ref{tableIY3bis}, $V_e$ has basis $(a_{-2}, a_0, a_2, s_{\tz,2})$ and  $V_o$ has basis 
 $(a_{-1}, a_1, a_3, s_{\tu,2})$. 
    By Lemma~\ref{subIY3} and Equation~\eqref{eq*=**} on page~\pageref{eq*=**}, 
\begin{equation*}
    \label{*IY3}
    \langle a_0-a_{-2}, a_0-a_2, s_{\tz,2}\rangle =V_e^\ast=V_o^\ast=\langle a_1-a_{-1}, a_1-a_3, s_{\tu,2}\rangle . 
\end{equation*}
In particular, $a_1-a_{-1}\in V_e^\ast$. Note that $V_e^\ast$ has basis also
$$(a_2-a_{-2}, 2a_0-a_2-a_{-2}, s_{\tz,2}).
$$ Since $\tau_0$ fixes $2a_0-a_2-a_{-2}$ and $s_{\tz,2}$ and negates $a_1-a_{-1}$ and  $a_2-a_{-2}$, there exists $x\in \F\setminus \{0\}$ such that 
\begin{equation}
    \label{a1am1ter}
    a_1-a_{-1}=x(a_2-a_{-2}), \quad \mbox{ whence }\quad a_{-1}=a_1-x(a_2-a_{-2}).
\end{equation}
Thus, by 
Equations~\eqref{defs} and ~\eqref{a1am1ter}, and Table~\ref{tableIY3bis},
\begin{align}
    \label{s12ter3}
    s_{\tu,2} &= -\tfrac{1}{2}(a_{1}-a_{-1})^2=-\tfrac{1}{2}x^2(a_2-a_{-2})^2=-\tfrac{1}{2}x^2(-4(\mu+1)s_{\tz,2}) \\
    &= 2x^2(\mu+1)s_{\tz,2}.\nonumber
\end{align}
By Lemma~\ref{smult}, $\al=2=4\bt$ and  $\lmu=\lmf$. We claim  that  
\begin{equation}\label{genVclaim}
    V=\langle  a_{-2}, a_0, a_1, a_2, s_{\tz,1}, s_{{\tz},2}\rangle . 
\end{equation}
Since, by Lemma~\ref{subIY3}, $a_4=a_{-2}-(2\mu+1)(a_0-a_2)$, Equation~\eqref{a1am1ter} implies
\begin{align}
     \label{a3ter}
    a_3&= (a_{-1})^{\tau_1}=\left (a_1-x(a_2-a_{-2})\right )^{\tau_1} =a_1-x(a_0-a_{4})\\
    &= xa_{-2}-2x(\mu+1)a_0+a_1+(2\mu+1)xa_2 \nonumber
\end{align}
Further, by Equations~\eqref{a1am1ter} and~\eqref{defs}, and Table~\ref{tableIY3bis},
\begin{align}
    a_{-2}a_1&=  a_{-2}\left (a_{-1}+x(a_2-a_{-2})\right )\nonumber \\
    &= a_{-2}a_{-1}+x(a_{2}-a_{-2})a_{-2}\nonumber \\
    &= \tfrac{1}{2}(a_{-2}+a_{-1})+s_{\tz,1}+x\left (\tfrac{1}{2}(a_{2}+a_{-2})+2(\mu+1)s_{\tz,2}\right )-xa_{-2}\nonumber\\
    &= \tfrac{1}{2}(a_{-2}+a_{1})+s_{\tz,1}+ 2(\mu+1)xs_{\tz,2},\nonumber 
\end{align}
whence, by Equation~\eqref{defs},
\begin{align}
    \label{s13}
    s_{\tu,3} 
    &= s_{\tz,1}+ 2(\mu+1)xs_{\tz,2},
\intertext{and}
    \label{s23}
    s_{\bar 2,3} &= (s_{\tu,3})^{\tau_0}=(s_{\tz,1}+ 2(\mu+1)xs_{\tz,2})^{\tau_0}=s_{\tu,3},\\
    \label{s03}
    s_{\tz,3} &= (s_{\bar 2,3})^{\tau_1}=(s_{\tz,1}+ 2(\mu+1)xs_{\tz,2})^{\tau_1}=s_{\tu,3}.
\end{align}
Therefore, Equation~\eqref{genVclaim} follows from Table~\ref{tableIY3bis}, Equations~\eqref{a1am1ter}, \eqref{a3ter}, ~\eqref{s12ter3}, Lemma~\ref{primo}, \cite[Lemma~7.6]{FMS3} and Equations~\eqref{s13}, \eqref{s23}, and~\eqref{s03}.
Thus, by Proposition~\ref{reg}, we get a contradiction to Hypothesis~\ref{hyp1}\ref{hyp1_2} once we prove that
\begin{equation}\label{newl}
 \lm_3^f=\lm_3.
\end{equation}
By Equation~\eqref{a3ter} and Lemma~\ref{s},
$$
\lm_3=x\lmd-2x(\mu+1)+\lmu+(2\mu+1)x\lmd=\lmu
$$
On the other hand, by Equations~\eqref{s13} and~\eqref{s12ter},
\begin{align*}
   \lm_3^f-\bt-\bt\lm_3^f &= \lm_{a_1}(s_{\tu,3})\\
   &= \lm_{a_1}(s_{\tz,1}+ 2(\mu+1)xs_{\tz,2})\\
   &= \lmf-\bt-\bt\lmf+\tfrac{1}{x}\lm_{a_1}(s_{\tu,2})\\
   &= \lmf-\bt-\bt\lmf
\end{align*}
whence, $\lm_3^f=\lmf=\lmu=\lm_3$, giving Equation~\eqref{newl}. 
\end{proof}

\begin{proof}[Proof of Theorem~\ref{case4}] 
Let $\V$ be as in Theorem~\ref{case4}. Then, by Lemma~\ref{lm=lmf}, $\{Q,Q^f\}\neq \{0\}$. Hence, the result follows by Lemmas~\ref{possipairs},~\ref{Jx}, Corollary~\ref{corveast2}, and Lemma~\ref{IY3al1}, and Corollary~\ref{corveast2}. 
\end{proof}

\chapter{Algebras with regular axet  }\label{chap:regular}

In this chapter we prove Theorems~\ref{teo2}, \ref{teo3}, and~\ref{teolarger}. By Theorems~\ref{al=2bt},~\ref{teoq},~\ref{skew},~\ref{teonsa} 
we can restrict ourselves to studying the $2$-generated $\Mab$-axial algebras $\V$ satisfying the following conditions:
\begin{hyp}\label{Hyp2}~
\begin{enumerate}[enum_arabic]
    \item $\al\neq 2\bt$, \label{Hyp2_1}
    \item  $\V$ is not isomorphic to a quotient of a symmetric algebra, \label{Hyp2_2}
    \item $V\not \in \{V_e, V_o\}$,  \label{Hyp2_3}
    \item $ \V$ has axet $X(n+n)$ for $n\in\N\cup \{\infty\}$ and $n\geq 2$. \label{Hyp2_4}
\end{enumerate}
\end{hyp}

\begin{remark}\label{rem:reg}
Hypotheses~$\ref{Hyp2}$\ref{Hyp2_3} and~\ref{Hyp2_4} are equivalent to saying that $\V_e$ and $\V_o$ are distinct and
both  have axet $X(n)$.    
\end{remark}

\section{Algebras with axet $X(2+2)$}\label{chap:axet2}

In this section we assume that $\V$ satisfies Hypothesis~\ref{Hyp2} with $n=2$.
By Remark~\ref{rem:reg}, 
\begin{equation}\label{hyp2}
 a_0\neq a_2,\quad  a_1\neq a_{-1}, \quad  \mbox{ and }\quad  a_{i}=a_{i+4},  \quad \mbox{  for every  }\quad  i\in \Z .
\end{equation}

\begin{lemma}\label{jordan}
Under the above hypothesis, $\V_e$ and $\V_o$  are $\mathcal J(\al)$-axial algebras.
\end{lemma}
\begin{proof}
By Equation~\eqref{hyp2},  the Miyamoto maps $\tau_0$ and $\tau_2$ (respectively $\tau_1$ and $\tau_3$) act trivially on $\V_e$  (respectively $\V_o$),  hence the $\beta$-eigenspaces of $\ad_{a_0}$ and $\ad_{a_2}$ in $V_e$ (respectively $\ad_{a_1}$ and $\ad_{a_3}$ in $V_o$) are trivial, that is $\V_e$ and $\V_o$  are $\mathcal J(\al)$-axial algebras.
\end{proof}

Recall from Chapter~\ref{ch1} (pages~\pageref{primo}, ~\pageref{A}, and~\pageref{evaluation}) the definitions of $P$, $Q$, $Q^f$, $R$, $R^f$, $A$, $A^f$, $B$, $B^f$, $C$, and $C^f$.

\begin{lemma}\label{lem2+2:1}
In the algebra $\V$ the following hold:
\begin{equation}\label{eq2+2:1}
 R=R^f=2P,
 \end{equation}
\begin{equation}\label{eq2+2:3}
0=(2P-Q^f)(a_0+a_2)+(Q-2P)(a_1+a_3)+(S-S^f)s_{\tz,1}+(T-U)(s_{\tz,2}-s_{\tu,2}),
\end{equation} 
\begin{equation}\label{eq2+2:2}
   \mbox{if } \al\neq  4\bt,\: B=B^f=0, \quad  \mbox{ if } \al= 4\bt, \: C=C^f=0,
\end{equation} 
\begin{align}\label{eqlm2}
\lmd &= \tfrac{2(\lmu-\bt)}{\al\bt(\al-\bt)}\Big ((3\al^2-4\al\bt-\al+2\bt)\lmu+\al(\al-1)\lmf \\
&\phantom{{}+{} \tfrac{2(\lmu-\bt)}{\al\bt(\al-\bt)}\Big (} \qquad\qquad -(\al^3+\al^2\bt-2\al\bt^2-2\al\bt+2\bt^2)\Big ), \nonumber 
\end{align} 
\begin{align}\label{eqlm2f}
\lmdf &= \tfrac{2(\lmf-\bt)}{\al\bt(\al-\bt)}\Big ((3\al^2-4\al\bt-\al+2\bt)\lmf+\al(\al-1)\lmu \\
&\phantom{{}+{} \tfrac{2(\lmf-\bt)}{\al\bt(\al-\bt)}\Big (} \qquad\qquad -(\al^3+\al^2\bt-2\al\bt^2-2\al\bt+2\bt^2)\Big),\nonumber 
\end{align}
\begin{equation}\label{lem2+2sym}
\lmu\neq \lmf.
\end{equation}
\end{lemma}
\begin{proof}
By Equation~\eqref{hyp2}, $$a_2-a_0\neq 0, \quad a_{-1}-a_1\neq 0, \quad 
a_{-2}=a_2,\quad \;\mbox{ and }\quad a_3=a_{-1}.$$
Substituting these values in Equations~\eqref{equa1} and~\eqref{equa2} on page~\pageref{equa1} we get 
$$(2P-R)(a_2-a_0)=0\;\;\;\mbox{ and  }\;\;(2P-R^f)(a_{-1}-a_1)=0,$$
respectively, whence  
Equation~\eqref{eq2+2:1} follows. Now, Equation~\eqref{eq2+2:3} follows by replacing, in  Equation~\eqref{equa3} on page~\pageref{equa3}, $R$ and $R^f$ by $2P$, $a_{-2}$ by $a_2$, and $a_{-1}$ by $a_3$. Similarly, Equation~\eqref{eq2+2:2} follows from Lemmas~\ref{s1s2} and~\ref{l7.3}. Equations~\eqref{eqlm2} and~\eqref{eqlm2f} are obtained by rearranging Equation~\eqref{eq2+2:1}. Finally, assume for a contradiction that 
\[
\lmu=\lmf.
\]
By Equation~\eqref{eq2+2:1}, $R-R^f=0$ and so, by Lemma~\ref{newQ-Q^f}\ref{newQ-Q^f_3}, 
  \begin{equation}\label{lm2lm2axet2}
  \lmd=\lmdf.
  \end{equation}
By Proposition~\ref{reg}\ref{reg_1} and  Hypothesis~\ref{Hyp2}\ref{Hyp2_2}  we get that $(\al,\bt)=(2,\tfrac{1}{2})$.  
Let 
$$
\overline{V}:=\langle  a_{-1}, a_0, a_1, a_2,  s_{\tz,1}, s_{\tz,2}\rangle.
$$
We claim that 
\begin{equation}
    \label{VV}
   \overline{V}= V.
\end{equation} Since $\overline{V}$ contains $a_0$ and $a_1$, this  is equivalent to say that $\overline{V}$ is closed under the algebra multiplication. Since $\al=4\bt$, \cite[Corollary 7.2]{FMS3} implies that  $s_{\tu,2}\in \overline{V}$. By Equation~\eqref{defs} on page~\pageref{defs} and Equation~\eqref{hyp2}, it follows that  
$$
a_ia_j \in \overline{V} \quad \mbox{ for every }\quad \{i,j\}\subseteq \{-1,0,1,2\}.
$$
By Lemma~\ref{primo} and~\cite[Corollary 7.4]{FMS3}, 
$$
\{a_is_{\tz,1}, a_is_{\tz,2}\}
\subseteq \overline{V} \quad \mbox{  for every  }\quad i\in \{-1,0,1,2\}.
$$ 
By Lemma~\ref{primo}\ref{primo_6}, $s_{\tz,1}s_{\tz,1}\in \overline{V}$. 
Equation~\eqref{hyp2} implies 
$$
\{s_{\tz,3}, s_{\tu,3}, s_{\td, 3}\}=\{s_{\tz,1}\}\subseteq \overline{V},
$$
whence, by~\cite[Lemma~7.6, and Lemma~7.7]{FMS3}, $\{s_{\tz,1}s_{\tz,2}, s_{\tz,2}s_{\tz,2}\}\subseteq \overline{V}$. Thus $\overline{V}$ is a subalgebra of $V$. 
Again by Equation~\eqref{hyp2}, 
\begin{equation}
    \label{lm3axet2}
    \lm_3=\lm_{a_0}(a_3)=\lm_{a_0}(a_{-1})=\lmu=\lmf=\lm_{a_1}(a_0)=\lm_{a_1}(a_4)=\lm_3^f.
\end{equation}
Thus, Equations~\eqref{lm3axet2}, \eqref{lm2lm2axet2}, and \eqref{VV} and Proposition~\ref{reg}\ref{reg_2} yield that $\V$ is isomorphic to a quotient of a symmetric algebra, contradicting 
Hypothesis~\ref{Hyp2}\ref{Hyp2_2}.
\end{proof}

\begin{lemma}
In the algebra $\V$ the following identities hold:
\begin{align}
0 &= (2P-Q^f)(1+\lmd)+2(Q-2P)\lmu+(S-S^f)((1-\bt)\lmu-\bt)\label{eq2+2:proj1}\\
&\phantom{{}+{}} \quad +(T-U)\left(-\tfrac{2(\al-1)}{\al-\bt}\lmu(\lmu-\lmf)-(1-2\bt)\lmu+(1-2\bt)\lmd \right), \notag\\
0 &= 2(2P-Q^f)\lmf+(Q-2P)(1+\lmdf)+(S-S^f)((1-\bt)\lmf-\bt)\label{eq2+2:proj2}\\
&\phantom{{}+{}} \quad + (T-U)\left(\tfrac{2(\al-1)}{\al-\bt}\lmf(\lmf-\lmu)+(1-2\bt)\lmf-(1-2\bt)\lmdf \right), \notag\\
    0&=2\bt(\al-2\bt)(Q-2P)+(\al-2\bt)(2P-Q^f)\label{eq2+2:a0 mult}\\
    &\phantom{{}+{}} \quad +(\al-2\bt)(S-S^f)((1-\al)\lmu+\bt(\al-\bt-1))\notag\\
    &\phantom{{}+{}} \quad +(\al-2\bt)(T-U)((1-\al)\lmd-\bt)-(T-U)(J-\bt L), \notag\\
\intertext{and}
    0&=2\bt(\al-2\bt)(2P-Q^f)+(\al-2\bt)(Q-2P)\label{eq2+2:a1 mult}\\
    &\phantom{{}+{}} \quad +(\al-2\bt)(S-S^f)((1-\al)\lmf+\bt(\al-\bt-1)) \notag\\
    &\phantom{{}+{}} \quad -(\al-2\bt)(T-U)((1-\al)\lmdf-\bt)+(T-U)(J^f-\bt L). \notag
\end{align}
%
\end{lemma}
\begin{proof}
Taking the image of Equation~\eqref{eq2+2:3} under the linear map $\lm_{a_0}$ and using Equation~\eqref{deflmlmf} and Lemma~\ref{s}\ref{s_3}, we get 
\begin{equation}\label{eq2+2:proj1bis}
\begin{split}
0&=(2P-Q^f)(1+\lmd)+2(Q-2P)\lmu+(S-S^f)\lm_{a_0}(s_{\tz,1})\\
&\phantom{{}+{}} \quad +(T-U)\left ( \lm_{a_0}(s_{\tz,2})-\lm_{a_0}(s_{\tu,2})\right ).     
\end{split}
\end{equation}
By Lemma~\ref{s}\ref{s_2}, 
\begin{equation}\label{x1}
\lm_{a_0}(s_{\tz,1})=(1-\bt)\lmu-\bt\;\;\;\mbox{ and } \quad 
\lm_{a_0}(s_{\tz,2})=(1-\bt)\lmd-\bt.
\end{equation}
By Corollary~\ref{cor:lm2f}, 
\begin{equation}\label{x2}
 \lm_{a_0}(s_{\tu,2})=\frac{2(\al-1)}{\al-\bt}\lmu(\lmu-\lmf)+(1-2\bt)\lmu+\bt\lmd-\bt.   
\end{equation}
Equation~\eqref{eq2+2:3} then follows by Equations~\eqref{eq2+2:proj1bis},~\eqref{x1}, and~\eqref{x2}. 
Equation~\eqref{eq2+2:proj2} is obtained in a similar way, by taking the image of Equation~\eqref{eq2+2:3} under the linear map $\lm_{a_1}$.

Multiplying by $(\al-2\bt)a_0$ both terms of Equation~\eqref{eq2+2:3} (where the products between $a_0$ and the elements of $V$ appearing in that equation are given by  Equation~\eqref{defs} on page~\pageref{defs}, Lemma~\ref{primo}, we get  
\begin{equation}\label{eq2+2:4}
0=W_0a_0+W_1(\bt a_2+s_{0,2})+W_2\left(\tfrac{1}{2}\bt(a_1+a_{-1})+s_{0,1}\right),
\end{equation}
where 
\begin{align*}
    W_0&:= (\al-2\bt)(2P-Q^f)(1+\bt)+2\bt(\al-2\bt)(Q-2P)\\
    &\phantom{{}={}} \quad +(\al-2\bt)(S-S^f)((1-\al)\lmu+\bt(\al-\bt-1))\\
    &\phantom{{}={}} \quad +(\al-2\bt)(T-U)((1-\al)\lmd+\bt(\al-\bt-1))-(T-U)J,\\
    W_1&:= (\al-2\bt)(2P-Q^f)+(\al-2\bt)(T-U)(\al-\bt)-(T-U)L,\\
    W_2&:= 2(Q-2P)(\al-2\bt)+(\al-2\bt)(S-S^f)(\al-\bt)-(T-U)K.
\end{align*}
Taking the difference between Equation \eqref{eq2+2:4} and its image under $\tau_1$, we get 
\[ (W_0-\bt W_1)(a_0-a_2)=0,\]
whence $W_0-\bt W_1=0$, which is exactly Equation \eqref{eq2+2:a0 mult}. 
Equation~\eqref{eq2+2:a1 mult} is obtained in a similar way, multiplying by $(\al-2\bt)a_1$ both terms of Equation~\eqref{eq2+2:3} and using Equation~\eqref{4.7.1f} on page~\pageref{4.7.1f}.
\end{proof}

\begin{lemma}\label{lem2+2:al and bt}
In the algebra $\V$, either $\al=4\bt$ or $\bt=\tfrac{\al^2}{2(\al-1)}$.
\end{lemma}
\begin{proof}
Assume $\al\neq 4\bt$. By Equation~\eqref{eq2+2:2}, $B=B^f=0$. Substituting the expressions of $\lmd$  and $\lmdf$ given in Equations~\eqref{eqlm2} and~\eqref{eqlm2f} respectively, equation $B+B^f=0$ simplifies to 
\begin{equation*}
    8(\al-1)(\al-2\bt)(\lmu-\lmf)^2(\al-4\bt)(\al^2-2\al\bt+2\bt)=0.
\end{equation*}
By Equation~\eqref{lem2+2sym}, $\lmu\neq \lmf$; by the initial assumption and Hypothesis~\ref{Hyp2}\ref{Hyp2_1}, $\al\not \in \{ 2\bt, 4\bt\}$, and $\al\neq 1$, so the above equation implies $\al^2-2\al\bt+2\bt=0$, whence $\bt=\tfrac{\al^2}{2(\al-1)}$.
\end{proof}

\subsection{The half case}

In this subsection, we assume $\al=\frac{1}{2}$. By
Lemma~\ref{lem2+2:al and bt}, either $\bt=\tfrac{1}{8}$ or $\bt=-\tfrac{1}{4}.$ We prove that in both cases we get a contradiction.


\begin{lemma}\label{lem2+2:half 2}
In the algebra $\V$,  $\bt\neq-\tfrac{1}{4}$.
\end{lemma}
\begin{proof}
For the sake of contradiction, assume that $\bt=-\frac{1}{4}$. Note that, in this case, since $\al\neq \bt$, $\ch(\F)\neq 3$. Substituting, in Equations~\eqref{eq2+2:a0 mult} and~\eqref{eq2+2:a1 mult}, the values  of $\lmd$ and $\lmdf$ given in Equations~\eqref{eqlm2} and~\eqref{eqlm2f}, we get
\begin{equation}\label{eqlm14}
    0=\tfrac{1}{2}(\lmu-\tfrac{1}{4})(\lmu-\lmf+\tfrac{9}{8})
\end{equation}
and
\begin{equation}\label{eqlmf14}
    0=\tfrac{1}{2}(\lmf-\tfrac{1}{4})(\lmu-\lmf-\tfrac{9}{8}).
\end{equation}
From Equation~\eqref{eqlm14} we get that either $\lmu=\tfrac{1}{4}$ or $\lmu=\lmf-\tfrac{9}{8}$. In the former case, by Equation~\eqref{lem2+2sym}, $\lmu\neq \lmf$, hence  Equation~\eqref{eqlmf14} implies $\lmf=\lmu-\tfrac{9}{8}=-\tfrac{7}{8}$. In the latter case, since $\ch(\F)\neq 3$, Equation~\eqref{eqlmf14} implies $\lmf=\tfrac{1}{4}$, whence $\lmu=-\tfrac{7}{8}$. 
Thus, by Equations~\eqref{eqlm2} and~\eqref{eqlm2f}, we get 
$$
(\lmu, \lmf, \lmd, \lmdf)\in \left \{\left (\tfrac{1}{4}, -\tfrac{7}{8}, 1, \tfrac{35}{4}\right ), \left (-\tfrac{7}{8}, \tfrac{1}{4}, \tfrac{35}{4}, 1 \right) \right \}.
$$
Substituting these values in Equations~\eqref{eq2+2:proj1} and~\eqref{eq2+2:proj2}, respectively, we get
\[
0=\tfrac{27}{32} ,
\]
which is a contradiction, as $\ch(\F)\neq 3$. 
\end{proof}

\begin{lemma}\label{lem2+2:half 3}
In the algebra $\V$, 
$\bt\neq \tfrac{1}{8}$. 
\end{lemma}
\begin{proof}
For the sake of contradiction, assume  $\bt=\frac{1}{8}$. In particular, since $\tfrac{1}{2}=\al\neq \bt$, $\ch(\F)\neq 3$. Substituting in Equations~\eqref{eq2+2:a0 mult} and~\eqref{eq2+2:a1 mult} the expressions of $\lmd$ and $\lmdf$ given in Equations~\eqref{eqlm2} and~\eqref{eqlm2f}, we get
\begin{equation}
    0=\tfrac{1}{8}(\lmu-\tfrac{1}{16})(\lmu-\lmf+\tfrac{9}{32})
\end{equation}
and
\begin{equation}
    0=\tfrac{1}{8}(\lmu-\tfrac{1}{16})(\lmu-\lmf-\tfrac{9}{32}).
\end{equation}
As in the proof of Lemma~\ref{lem2+2:half 2}, we get 
\[
(\lmu, \lmf, \lmd, \lmdf)\in \left \{\left (\tfrac{1}{16}, -\tfrac{7}{32}, -\tfrac{1}{8}, \tfrac{55}{16}\right ), \left (-\tfrac{7}{32}, \tfrac{1}{16}, \tfrac{55}{16}, -\tfrac{1}{8} \right ) \right \}.
\]
Substituting these values in Equations~\eqref{eq2+2:proj1} and~\eqref{eq2+2:proj2} respectively, we get
\[0=-\tfrac{27}{2048},\]
which is a contradiction, as $\ch(\F)\neq 3$. \end{proof}

\begin{cor}\label{alneqonehalf}
In the algebra $\V$, $\al\neq\tfrac{1}{2}$.
\end{cor}

\subsection{The non-half case}

 By Corollary~\ref{alneqonehalf}, $\al\neq \tfrac{1}{2}$, hence, by Lemma~\ref{jordan} and Table~\ref{table2},   we may assume that  
\begin{equation}\label{hyp2B3C}
\{\V_e, \V_o\}\subseteq \{ 2\B, 3\C(\al), 3\C(-1)^\times\}.
\end{equation}
\begin{lemma}\label{lem:2+2lm}
We have 
\begin{equation}\label{eq:2+2lm1}
\lmu=\bt(1-\lmd)+\lm_0(a_2a_1)  
\end{equation}
and
\begin{equation}\label{eq:2+2lm2}
\lmf=\bt(1-\lmdf)+\lm_1(a_3a_0). 
\end{equation}
\end{lemma}
\begin{proof}
By Equation~\eqref{defs} on page~\pageref{defs} and Lemma~\ref{invariant}, $ 
s_{\tz,1}=(s_{\tz,1})^{\tau_1}=a_2 a_1-\bt(a_2+a_1)$. 
Thus, by Lemma~\ref{s},
\[ 
\lmu-\bt-\bt\lmu=\lm_{a_0}(s_{\tz,1})=\lm_{a_0}(a_2a_1)-\bt(\lmd+\lmu).
\]
Rearranging the above equation, we get Equation \eqref{eq:2+2lm1}. Equation \eqref{eq:2+2lm2} follows in a similar way,  by applying  $\lm_{a_1}$ and $\tau_0$ to $s_{\tz,1}$.
\end{proof}

\begin{cor}\label{lem:2+2 orthogonal case}
If $\V_e\cong 2\B$, then $(\lmu,\lmd)=(\bt,0)$. If $\V_o\cong 2\B$, then $(\lmf,\lmdf)=(\bt,0)$.
\end{cor}
\begin{proof}
Assume $\V_e\cong 2\B$.  Then, by Table~\ref{table2}, $a_2$ is a $0$-eigenvector for $\ad_{a_0}$, whence, $\lmd=0$ and, by the fusion law, $a_2 a_1\in V_{\{0,\al,\bt\}}^{a_0}$.  Thus $\lm_{a_0}(a_2a_1)=0$ and, by Equation~\eqref{eq:2+2lm1}, $\lm_1=\bt(1-0)+0=\bt$. A similar argument holds when $\V_o\cong 2\B$.
\end{proof}

\begin{cor}\label{cor2B}
$\{\V_e,\V_o\}\neq \{2\B\}$. 
\end{cor}
\begin{proof}
If $\{\V_e,\V_o\}= \{2\B\}$, then 
by Lemma~\ref{lem:2+2 orthogonal case}, we get $\lmu=\lmf=\bt$, contradicting Equation~\eqref{lem2+2sym}.
\end{proof}

\begin{lemma}\label{lem3C}
$\{\V_e,\V_o\}\subseteq \{ 3\C(\al), 3\C(-1)^\times\}$.
\end{lemma}
\begin{proof}
For the sake of contradiction, assume $\{\V_o,\V_e\}\not \subseteq \{ 3\C(\al), 3\C(-1)^\times\}$. Up to swapping $\V_e$ and $\V_o$, we may suppose $\V_o\cong 2\B$. Then, by Lemma~\ref{lem:2+2 orthogonal case}, \begin{equation}\label{xyz}
(\lmf,\lmdf)=(\bt,0).
\end{equation}
Further,  $a_1a_{-1}=0$, whence 
$$
s_{\tu,2}=a_1a_{-1}-\bt(a_1+a_{-1})=-\bt(a_1+a_{-1}).
$$ 
Therefore
\begin{align}\label{a0s12}
(\al-2\bt)a_0s_{\tu,2} &= -\bt(\al-2\bt)a_0(a_1+a_{-1})\\
&= -\bt(\al-2\bt)(2s_{\tz,1}+2\bt a_0+\bt(a_1+a_{-1})) . \nonumber 
\end{align}
Taking the difference between Equation~\eqref{a0s12} and its image under $\tau_1$, we get
\begin{equation}\label{eq2+2:2B3C 1}
(\al-2\bt)(a_0-a_2)s_{\tu,2}=-2\bt^2(\al-2\bt)(a_0-a_2).
\end{equation}
Taking the difference between Lemma~\ref{primo}\ref{primo_5} and its image under $\tau_1$, we get
\begin{equation}\label{eq2+2:2B3C 2}
        (\al-2\bt)(a_0-a_2)s_{\tu,2}=(J-2H)(a_0-a_2).
\end{equation}
Since $a_0-a_2\neq 0$, multiplying by $(\al-\bt)$ the difference between   Equations~\eqref{eq2+2:2B3C 1} and \eqref{eq2+2:2B3C 2}, we obtain
\begin{equation*}
\begin{split}\label{eq2+2:2B3C 3}
0&=(\al-\bt)(J-2H+2\bt^2(\al-2\bt)),
\end{split}
\end{equation*}
whence, by Equation~\eqref{eq2+2:1}, we get 
\begin{equation}\label{bella}
    0=\tfrac{1}{2}(R-2P)+(\al-\bt)(J-2H+2\bt^2(\al-2\bt)).
\end{equation}
By Equation~\eqref{hyp2B3C} and Corollary~\ref{cor2B}, $\V_e$ is isomorphic to $3\C(\al)$ or to $3\C(-1)^\times$, hence by Lemma~\ref{sub3C}, 
\begin{equation}\label{xyzt}
    \lmd=\tfrac{\al}{2}.
\end{equation} 
Replacing, in Equation~\eqref{bella},   $R$, $P$, $J$, and $H$ by their expressions given on page~\pageref{primo} in Chapter~\ref{ch1}, and substituting, in each of these expressions, $\lmf$, $\lmd$, and $\lmdf$ by the values given in Equations~\eqref{xyz} and~\eqref{xyzt},  we get  
  \begin{equation}
0=-2(\al-1)(\al-2\bt)^2(\lmu-\bt). 
\end{equation}
Since, by hypothesis,   $\al\not\in \{ 1, 2\bt\}$, it follows that $\lmu=\bt=\lmf$, contradicting  Equation~\eqref{lem2+2sym}.  \end{proof}

\begin{lemma}\label{lob}
$\al= 4\bt$. 
\end{lemma}
\begin{proof}
Assume for a contradiction that  
$$
\al\neq 4\bt.
$$ 
By Lemma~\ref{lem3C}, $\{\V_o,\V_e\}\subseteq \{ 3\C(\al), 3\C(-1)^\times\}$, whence, by Lemma~\ref{subJ}\ref{subJ_1},
\begin{equation}\label{mario1}
    \lmd=\lmdf=\tfrac{\al}{2}
\end{equation} 
and, 
by Lemma~\ref{lem2+2:al and bt}, 
\begin{equation}\label{mario2}
 \bt=\tfrac{\al^2}{2\al-2}.   
\end{equation} 
  By Equation~\eqref{lem2+2sym}, $\lmu\neq \lmf$, hence, by Equation~\eqref{eq2+2:1} and Lemma~\ref{newQ-Q^f}\ref{newQ-Q^f_3}, we get
 \begin{equation}
    \label{fr}
 \begin{split}
  0&=\tfrac{(\al-1)^2}{2\al(\lmu-\lmf)}(R-R^f) \\
  &=
  2(\al-1)(\al^2-3\al+1)(\lm+\lmf)-\al^2(2\al^2-7\al+4).
  \end{split}
 \end{equation}
 Note that $$(\al-1)(\al^2-3\al+1)\neq 0,$$ otherwise, since $\al\not \in \{0,1\}$, by Equation~\eqref{fr},  $2\al^2-7\al+4=0$ and $\al^2-3\al+1=0$, a contradiction, since these two polynomials have no common root, their resultant being $-1$.
Thus, by Equation~\eqref{fr}, 
 \begin{equation}\label{2+2:lmf}
    \lmu+ \lmf=\tfrac{\al^2(2\al^2-7\al+4)}{2(\al-1)(\al^2-3\al+1)}.
 \end{equation}
Taking the sum of Equations~\eqref{eq2+2:proj1} and~\eqref{eq2+2:proj2} and evaluating the expressions of $P$, $Q$, $Q^f$, $S$, $S^f$, $T$, and $U$ given on page~\pageref{equa1} for the values of $\bt$, $\lmd$, $\lmdf$, and $\lmu+\lmf$ given in Equations~\eqref{mario1}, \eqref{mario2}, and~\eqref{2+2:lmf}, respectively, we obtain
\begin{equation}\label{eq173+eq174}
0=\tfrac{\al^2(\lmu-\lmf)}{2(\al-1)^2(\al^2-3\al+1)}p(\al),
\end{equation}
where 
\begin{equation*}\label{p}
    p(\al)=4\al^6-19\al^5+26\al^4-12\al^3-\al^2+4\al-4.
\end{equation*}
Since $\al\neq 0$ and, by Equation~\eqref{lem2+2sym}, $\lmu-\lmf\neq 0$, Equation~\eqref{eq173+eq174} implies 
\begin{equation}\label{p=0}
p(\al)=0.
\end{equation}
Similarly,  taking the sum of Equations~\eqref{eq2+2:a0 mult} and~\eqref{eq2+2:a1 mult} and evaluating the expressions of $P$, $Q$, $Q^f$, $S$, $S^f$, $T$, $U$, $J$, $J^f$, and $L$ given on page~\pageref{equa1} for the values of $\bt$, $\lmd$, $\lmdf$, and $\lmu+\lmf$ given in Equations~\eqref{mario1}, \eqref{mario2}, and~\eqref{2+2:lmf}, respectively, we obtain
\begin{equation}\label{eq176+eq177}
0=-\tfrac{\al^3(\lmu-\lmf)}{(\al-1)^4(\al^2-3\al+1)}q(\al),
\end{equation}
where 
\begin{equation*}
    \label{q}
     q(\al)= \al^7-4\al^6+3\al^5+3\al^4-5\al^3+4\al^2-5\al+2.
\end{equation*}
As above, since $\al\neq 0$ and, by Equation~\eqref{lem2+2sym}, $\lmu-\lmf\neq 0$, Equation~\eqref{eq176+eq177} implies 
\begin{equation}\label{q=0}
q(\al)=0.
\end{equation}
By Equation~\eqref{eq2+2:1} and Lemma~\ref{s1s2}, we get   
 \begin{equation}\label{condR-2P}
0=\tfrac{4(\al-1)^3(\al^2-3\al+1)}{\al^2}(R-2P)-\tfrac{8(\al-1)^4(\al^2-3\al+1)}{\al^3(\lmu-\lmf)}B .
\end{equation}
%
Evaluating the expressions of $R$, $P$, and $B$ given in Chapter~\ref{ch1}, for  the values of $\bt$, $\lmf$, $\lmd$, and $\lmdf$ given in Equations~\eqref{mario2}, \eqref{2+2:lmf}, and \eqref{mario1}, respectively, Equation~\eqref{condR-2P}  becomes  
 \begin{equation}\label{condR-2Pconti}
 \begin{split}
0
&= \left (16(\al-1)(\al^2-3\al+1)\lmu^2-8\al^2(2\al^2-7\al+4)\lmu \right .\\
&\phantom{{}+{}} \quad \left . +\al^3(\al^3-\al^2-5\al-2)\right )\\
&\phantom{{}+{}} \quad -\left (16(\al-1)^2(\al^2-3\al+1)\lmu^2-8(\al-1)\al^2(2\al^2-7\al+4)\lmu \right .\\
&\phantom{{}+{}} \quad \left .+\al^3(4\al^6-16\al^5+11\al^4+12\al^3-24\al^2+\al+4) \right ) \\
&=-2\al^3(\al+1)t(\al)
\end{split}
 \end{equation}
%
where 
\begin{equation*}
    \label{t}
    t(\al):=
  2\al^5-10\al^4+15\al^3-8\al^2-2\al+1.
\end{equation*}
Since $p(-1)=52$ and $q(-1)=-11$, $-1$ is not a common root of $p$ and $q$. Therefore, Equations~\eqref{p=0} and~\eqref{q=0} imply that $\al\neq -1$. 
  Since also $\al\neq 0$, by Equation~\eqref{condR-2Pconti} it follows that
\begin{equation}\label{condal}
   t(\al)=0.
\end{equation}
 The resultant between $p(\al)$ and $q(\al)$ is equal to $2^9\cdot 36559$, while the resultant between $p(\al)$ and $t(\al)$ is equal to $2^7\cdot 22709$. Since their only common prime factor is $2$ and $\ch(\F)\neq 2$, there is no value of $\al$ satisfying Equations~\eqref{p=0}, ~\eqref{q=0}, and \eqref{condal}, a contradiction. \end{proof}

\begin{lemma}\label{lob2}
 $\al\neq  4\bt$. 
 \end{lemma}
\begin{proof}
Assume for a contradiction that  
\begin{equation}\label{alfa}
\al=4\bt
\end{equation} 
(whence $\ch(\F)\neq 3$, since $\al\neq \bt$). By Lemma~\ref{lem3C}, $\{\V_o,\V_e\}\subseteq \{ 3\C(\al), 3\C(-1)^\times\}$, whence 
\begin{equation}
    \label{Mario1}
    \lmd=\lmdf=\tfrac{\al}{2}.
\end{equation} 
By Equations~\eqref{eq2+2:1}, \eqref{lem2+2sym}, and Lemma~\ref{newQ-Q^f}\ref{newQ-Q^f_3}, we get
 \begin{equation}
  0=\tfrac{1}{8\bt(\lmu-\lmf)}(R-R^f)=
  (16\bt-1)(\lmu+\lmf)-2\bt(22\bt-1). 
 \end{equation}
 As in the proof of Lemma~\ref{lob}, $16\bt-1\neq0$, otherwise  $0=2\bt(22\bt-1)=-\tfrac{3}{64}$, a contradiction, since $\ch(\F)\neq 3$. Hence 
 \begin{equation}\label{2+2:lmf4bt}
     \lm+\lmf=\tfrac{2\bt(22\bt-1)}{16\bt-1}.
 \end{equation}
 Taking the sum of Equations~\eqref{eq2+2:proj1} and~\eqref{eq2+2:proj2} and evaluating the expressions of $P$, $Q$, $Q^f$, $S$, $S^f$, $T$, and $U$ given on page~\pageref{equa1} for the values of $\al$, $\lmd$, $\lmdf$, and $\lmu+\lmf$ given in Equations~\eqref{alfa}, \eqref{Mario1}, and~\eqref{2+2:lmf4bt}, respectively, we obtain
\begin{equation}\label{eq173+eq174bis}
0=\tfrac{16\bt^3(\lmu-\lmf)}{16\bt-1}(248\bt^2+2\bt+11).
\end{equation}
Similarly, taking the sum of Equations~\eqref{eq2+2:a0 mult} and~\eqref{eq2+2:a1 mult} and evaluating the expressions of $P$, $Q$, $Q^f$, $S$, $S^f$, $T$, $U$, $J$, $J^f$, and $L$ given on page~\pageref{equa1} for the values of $\al$, $\lmd$, $\lmdf$, and $\lmu+\lmf$ given in Equations~\eqref{alfa}, \eqref{Mario1}, and~\eqref{2+2:lmf4bt}, respectively, we obtain 
\begin{equation}\label{eq176+eq177bis}
0=-\tfrac{16\bt^4(\lmu-\lmf)}{16\bt-1}(320\bt^2-40\bt-19).
\end{equation}
Since, by hypothesis, $\bt\neq 0$ and, by Equation~\eqref{lem2+2sym}, $\lmf\neq \lmu$, from Equations~\eqref{eq173+eq174bis} and~
\eqref{eq176+eq177bis} it follows that 
\begin{equation}\label{ch479}
248\bt^2+2\bt+11=0 \quad \mbox{ and }\quad 320\bt^2-40\bt-19=0.
\end{equation}
The resultant between $ 248\bt^2+2\bt+11$ and $320\bt^2-40\bt-19$ is equal to $2^8\cdot  3^4\cdot 29\cdot 479$. Hence $\ch(\F)\in \{29, 479\}$. If $\ch(\F)=29$, then 
\[
0=248\bt^2+2\bt+11=3(\bt+12)(\bt-1) \quad \mbox{ and }\quad 0=320\bt^2-40\bt-19=(\bt-1)(\bt-10)
\]
imply $\bt=1$, which not allowed. Finally, let $\ch(\F)=479$. Then Equation~\eqref{ch479} becomes
\[
0=17(\bt+82)(\bt+200) \quad \mbox{ and }\quad 0=-159(\bt+200)(\bt+219),
\]
whence $\bt=-200$, $\al=158$, $\lmd=\lmdf=79$, and $\lmf=-\lmu+136$. Substituting these values in Equations~\eqref{eq2+2:proj1} and~
\eqref{eq2+2:a0 mult} we get that $\lmu$ is a common root of the polynomials
\[
83\lmu^2+208\lmu+4 \quad \mbox{ and } \quad 50\lmu^2-94\lmu-172,
\]
a contradiction since these polynomials have non zero resultant in characteristic $479$.
\end{proof}

\begin{proof}[Proof of Theorem~\ref{teo2}]
Let $\V$ be as in Theorem~\ref{teo2}, that is $\al\neq 2\bt$ and $\V$ has axet $X(2+2)$.
The result follows, since, by Corollary~\ref{alneqonehalf} and Lemmas~\ref{lob} and~\ref{lob2}, there is no algebra $\V$ satisfying Hypothesis~\ref{Hyp2} with $n=2$. 
    \end{proof}

\section{Algebras with axet $X(3+3)$}\label{chap:axet3}

In this section we assume that $\V$ satisfies Hypothesis~\ref{Hyp2} with $n=3$. 
By Remark~\ref{rem:reg},
\begin{equation}\label{hyp}
\begin{split}
 & a_{i}=a_{j}  \quad \mbox{for every}\quad i,j\in \Z \quad \mbox{with}\quad i\equiv_6 j,  \\
& a_i\neq a_j \quad \mbox{for every}\quad i,j\in \Z \quad \mbox{with}\quad i\not \equiv_6 j .
\end{split}
\end{equation}
In particular, $a_0-a_{-2}=a_0-a_4$ and $a_1-a_{-1}=a_1-a_5$, whence 
\begin{equation}\label{odd**}
V_e^{\ast\ast}=V_e^\ast\quad \mbox{and}\quad V_o^{\ast\ast}=V_o^\ast.
\end{equation}
We start by listing the symmetric algebras with axet $X(3)$. 

\begin{lemma}\label{axet3}
Let $\mathcal{W}$ be a symmetric $2$-generated $\mathcal{M}(\al,\bt)$-axial algebra over a field $\F$, with  axet $X(3)$. Then $\mathcal{W}$ is isomorphic to one of the following
\begin{enumerate}
 \item  $3\C(\bt)$ or, when $\bt=-1$,   $3\C(-1)^\times$;
 \item  $3\A(\al,\bt)$ or $3\A(\al, \tfrac{1-3\al^2}{3\al-1})^\times$.
\end{enumerate}
\end{lemma}
\begin{proof}
The result follows by the Classification Theorem of the Symmetric Algebras (page~\pageref{symmetric} in the Introduction), Section~\ref{thetables} and~\cite[Section 7.3]{axet}. Note that the maximal quotient of the  algebra $\hatH$ with axet $X(3)$ is $\hatH_3$ (see the definition in~\cite[p. 469]{HWQ}) which, by~\cite[Corollary~10.1 and  Lemma~11.4]{HWQ}, is equal to $\mathcal H_3\cong 3\A(2,\tfrac{1}{2})$. Moreover, by Note~\ref{tableIY3 3A} in Table~\ref{tableIY3bis},  $\IY_3(\al, \tfrac{1}{2}; -\tfrac{1}{2})\cong 3\A(\al,\tfrac{1}{2})$.
\end{proof}

Let $P$, $Q$, $Q^f$, $R$, $R^f$, $A$, $A^f$, $B$, $B^f$ be as defined in Chapter~\ref{ch1} on pages~\pageref{primo}, \pageref{A}, and \pageref{evaluation}.

\begin{lemma}\label{cases}
The following identities hold in the algebra $\V$:
\begin{equation}\label{P-R}
(P-R)(a_2-a_0)=Q(a_3-a_{-1})
\end{equation}
and 
\begin{equation}\label{P-R^f}
(P-R^f)(a_{-1}-a_1)=Q^f(a_{2}-a_{-2}),
\end{equation}
\end{lemma}
\begin{proof}
Equations~\eqref{P-R} and~\eqref{P-R^f} follow immediately from Equations~\eqref{equa1}, \eqref{equa2} on page~\pageref{equa1}, and~\eqref{hyp}.  
\end{proof}

\begin{lemma}\label{end}
$V_e^\ast=V_o^\ast$.
\end{lemma}
\begin{proof}
Assume first that $$Q\neq 0.$$ Since by Equation~\eqref{hyp}, $a_2-a_0\neq 0$ and $a_3- a_{-1}\neq 0$, Equation~\eqref{P-R} implies that $P-R\neq 0$ and $\langle a_2-a_0\rangle=\langle a_3-a_{-1}\rangle$. By Lemma~\ref{lem:VeVo} it follows that $V_e^\ast=V_o^{\ast\ast}$  and  Equation~\eqref{odd**} implies that $V_e^{\ast}=V_o^\ast$. A similar argument applies if $\{P-R, P-R^f, Q, Q^f\}\neq \{0\}$, giving the result.
So we may assume that   
\begin{equation}\label{allzero}
\{ P-R,P-R^f, Q,Q^f \} = \{0\}.
\end{equation}
Since $\V$ has regular axet, Hypothesis~\ref{Hyp2}\ref{Hyp2_2} and Corollary~\ref{reduction} imply that 
\begin{equation*}\label{alneq4bt}
\al\neq 4\bt.
\end{equation*}
We claim that
\begin{equation}\label{claimlmnotlmf}
    \lmu\neq \lmf.
\end{equation}
Assume, for the sake of contradiction, that $\lmu=\lmf$.
Since $R-R^f=P-R^f-(P-R)$, Equation~\eqref{allzero} implies that 
$$R-R^f=0,$$
 thus, by Lemma~\ref{newQ-Q^f}\ref{newQ-Q^f_3},  $\lmd=\lmdf$. Then, by Proposition~\ref{reg}, $\V$ is isomorphic to a quotient of a symmetric algebra, contradicting  Hypothesis~\ref{Hyp2}\ref{Hyp2_2}, proving Equation~\eqref{claimlmnotlmf}. 
%
Let $A$, $A^f$, $B$, $B^f$ be as defined on page~\pageref{A} in Chapter~\ref{ch1}. Since $a_2\neq a_{-2}$ and $a_3\neq a_{-1}$, by Lemma~\ref{s1s2}, $A=0$ (respectively $A^f=0$) if and only if $B=0$  (respectively $B^f=0$). Thus, if $\{B,B^f\}=\{0\}$, then  $\{A,A^f, B, B^f\}= \{0\}$ and Lemma~\ref{lemma:A+Af} implies $\lmu=\lmf$, contradicting Equation~\eqref{claimlmnotlmf}. So $\{B,B^f\}\neq \{0\}$ and the result follows from  Lemma~\ref{Ve*Vo*} and Equation~\eqref{hyp}. 
\end{proof}


By Lemma~\ref{end} and Hypothesis~\ref{Hyp2}\ref{Hyp2_3}, we are left with the case when 
\begin{equation}
    \label{eq*=**3}
  V_e^\ast < V_e,\;\;\;V_o^{\ast}< V_o, \;\;\mbox{ and }\;\; V_e^\ast=V_o^\ast .  
\end{equation}
\begin{lemma}\label{possipairs3}
Assume Equation~\eqref{eq*=**3} is satisfied. Then 
\begin{enumerate}
    \item  up to isomorphism of its components, the unordered pair $(\V_e, \V_o)$ is  one of those listed in the first column of Table~\ref{tableposs3};
    \item  for each possible pair, the dimension of $V_e^\ast$ is given in the second column of Table~\ref{tableposs3}.
\end{enumerate}
\end{lemma}

{\renewcommand{\arraystretch}{1.5}
\begin{table}[H]
{\Small
\[
\begin{array}{|l|c|}
\hline
\begin{array}{l}(\V_e, \V_o)\end{array} & \dim(V_e^\ast) \\
\hline
  \begin{array}{l} 
 \left (3\C(-1)^\times, \;3\C(-1)^\times\right ), \quad  \mbox{ with }\quad \bt=-1, \quad \ch(\F)=3
 \end{array} & 1 \\
 \hline
 \begin{array}{l}
 \left (3\C(2), \;3\C(2))\right ), \quad  \mbox{ with }\quad \bt=2
 \end{array} & 2
 \\
 \hline
 \begin{array}{l}
 \left ( 3\A(\al, \bt), 3\A(\al, \bt)\right ), \quad  \mbox{ with }\quad (3\al^2+3\al\bt-9\al-2\bt+4)(3\al+\bt-2)=0
 \end{array} & 3
\\
\hline
\end{array}
\]

\caption{{\small Possible pairs $(\V_e, \V_o)$ with $
 V_o\neq V_o^\ast=V_e^{\ast}\neq V_e.
$}
}\label{tableposs3}}
\end{table}

\begin{proof} By Lemma~\ref{lem:VeVo}, $\V_e$ and $\V_o$ are symmetric. Since they have both axet $X(3)$, by  Lemma~\ref{axet3}, $\V_e$ and $\V_o$ are isomorphic to one of the following algebras:
\[ 
3\C(\bt), \quad  3\C(-1)^\times, \quad 3\A(\al, \bt), \quad 3\A(\al, \tfrac{1-3\al^2}{3\al-1})^\times. 
 \]
Since, by Equation~\eqref{eq*=**3}, $V_e^\ast<V_e$ and $V_o^\ast<V_o$, Lemmas~\ref{sub3C},~\ref{subJ}, and~\ref{sub3A} imply that $\V_e$ (respectively $\V_o$) is isomorphic to one of the algebras appearing in the pairs listed in Table~\ref{tableposs3}. Since $V_o^\ast=V_e^\ast$, comparing the dimensions of $V_e^\ast$ and $V_o^\ast$, we get the result. 
\end{proof}

\begin{cor}\label{J(0)3}
$\dim(V_e^\ast)\neq 1$.
\end{cor}
\begin{proof}
Assume for a contradiction that $\dim(V_e^\ast)= 1$.
By Lemma~\ref{possipairs3},
$$
\V_e\cong \V_o \cong 3\C(-1)^\times, \;\;\;  \ch(\F)=3\;\;\;\mbox{ and }\bt=-1.
$$ 
In this case, by Table~\ref{table2},  $3\C(-1)^\times\cong \J(0)^\times$. The result then follows from  
Lemma~\ref{Jx}.  
\end{proof}

\begin{lemma}\label{bt=2}
$\dim(V_e^\ast)\neq 2$.
\end{lemma}
\begin{proof}
Suppose for a contradiction that $\dim(V_e^\ast)= 2$.  By Lemma~\ref{possipairs3}, $\bt=2$ and $\V_e\cong \V_o\cong 3\C(2)$. Then, by  Lemma~\ref{sub3C}\ref{sub3C_1}
\begin{equation}
    \label{lm2=lm2fcaso3}
    \lmd=\lmdf=1
\end{equation}
and, by  Lemma~\ref{sub3C}\ref{sub3C_3}  and Equation~\eqref{eq*=**3},
\begin{equation*}
  \langle a_0-a_2, a_0-a_{-2}\rangle=V_e^\ast=V_o^\ast=\langle a_{-1}-a_1, a_{-1}-a_3\rangle .  
\end{equation*} 
Then 
$2a_0-a_2-a_{-2}$ is a $1$-eigenvector in $V_e^\ast$ for $\tau_0$ while $a_2-a_{-2}$  and  $a_1-a_{-1}$ are $-1$-eigenvectors for $\tau_0$. It follows that there exists $x\in \F$ such that 
\begin{equation}\label{bt=2:eq1}
 a_1-a_{-1}=x(a_2-a_{-2}), \quad \mbox{ with } x\neq 0. 
\end{equation}
Thus 
\begin{equation}\label{am1a3}
a_{-1}=a_1-x(a_2-a_{-2}) \quad \mbox{ and } \quad  a_3=(a_1)^{\tau_1}=a_1-x(a_0-a_{-2}).
\end{equation}
Since, by Equation~\eqref{am1a3}, $a_1=a_3+x(a_0-a_{-2})$ and, by Equation~\eqref{hyp}, $a_{-2}=a_4$, by Equation~\eqref{defs} on page~\pageref{defs} and Table~\ref{table2}, we get
\begin{align}
    \label{newa1am2}
    a_1a_{-2}&= \left (a_3+x(a_0-a_{-2})\right )a_{-2} \nonumber\\
    &= a_3a_{-2}+x(a_0-a_{-2})a_{-2}\nonumber\\
    &= s_{\tz,1}+2(a_3+a_{-2})+x(a_0+a_{-2}-a_2-a_{-2}) \\
    &= s_{\tz,1}+2\left ((a_1-x(a_0-a_{-2}))+a_{-2} \right )+x(a_0+a_{-2}-a_2-a_{-2})\nonumber \\
    &= s_{\tz,1}+2a_1-x(a_0+a_2)+2(x+1)a_{-2}.\nonumber
\end{align}
By Equation~\eqref{defs} on page~\pageref{defs} and Equation~\eqref{newa1am2}, it follows that 
\begin{align*}
s_{\tu,3} &= a_1a_{-2}-2(a_1+a_{-2})\\
&= s_{\tz,1}+2a_1-x(a_0+a_2)+2(x+1)a_{-2}-2(a_1+a_{-2})\\
&= s_{\tz,1}-x(a_0+a_2-2a_{-2}),
\end{align*}
whence, by Equation~\eqref{am1a3}, 
\begin{equation}\label{pam1a3}
\begin{split}
a_{-1}a_3&=\left (a_1-x(a_2-a_{-2})\right )\left (a_1-x(a_0-a_{-2})\right )\\
&=a_1-xa_1(a_2-a_{-2})-xa_1(a_0-a_{-2})+x^2(a_2-a_{-2})(a_0-a_{-2})\\
&=(a_1-x\left (a_1a_2+a_1a_0-2a_1a_{-2}\right ))+x^2(a_2-a_{-2})(a_0-a_{-2})\\
&=(a_1-x\left (2s_{\tz,1}+2(2a_1+a_0+a_2))-(2s_{\tu,3}-4(a_1+a_{-2})\right )\\
&\phantom{{}+{}} \quad -x^2\left (2a_2-2a_{-2}-a_2-a_{-2}+a_0+a_{-2}\right ))\\
&=a_1-x\left (2s_{\tz,1}+2(a_0+a_2)-2s_{\tu,3}-4a_{-2}\right )+x^2\left (a_2-2a_{-2}+a_0\right )\\
&=a_1-x\left (2s_{\tz,1}+2(a_0+a_2)-2s_{\tz,1}+2x(a_0+a_2-2a_{-2})-4a_{-2}\right )\\
&\phantom{{}+{}} \quad +x^2\left (a_2-2a_{-2}+a_0\right )\\
&=a_1-x\left (2(a_0+a_2)+2x(a_0+a_2-2a_{-2})-4a_{-2}\right )\\
&\phantom{{}+{}} \quad +x^2\left (a_2-2a_{-2}+a_0\right )\\
&=a_1+2x(x+2) a_{-2}-x(x+2)  (a_0+a_2).
\end{split}
\end{equation}
On the other hand, since, by hypothesis, $\V_o\cong 3\C(2)$, Table~\ref{table2} and Equation~\eqref{bt=2:eq1} imply 
\begin{equation}\label{pam1a3bis}
   a_{-1}a_3=a_{-1}+a_3-a_1=a_1-x(a_0+a_2)+2xa_{-2}. 
\end{equation}
Taking the difference between  Equations~\eqref{pam1a3} and~\eqref{pam1a3bis}, we get
$$
x(x+1)(a_0+a_2-a_{-2})=0.
$$
Since $x\neq 0$ and $a_0+a_2-a_{-2}$ is not the zero vector in $3\C(2)$, it follows that $x=-1$. Thus, by Equation~\eqref{bt=2:eq1} and Corollary~\ref{a1xa2}, we get
$$
0=(1-\al)(\lmf-\lmu)+\tfrac{\bt}{2}(\al-\bt)(\tfrac{1}{x}-x)=(1-\al)(\lmf-\lmu),
$$
whence
\begin{equation}
    \label{lmlm}
    \lmu=\lmf.
\end{equation}
Since $\bt=2$ and $\bt\neq \al$, $(\al,\bt)\neq (2, \tfrac{1}{2})$. Hence, by Equations~\eqref{lm2=lm2fcaso3} and \eqref{lmlm} and by Proposition~\ref{reg}, $\V$ is isomorphic to a quotient of a symmetric algebra, contradicting Hypothesis~\ref{Hyp2}\ref{Hyp2_2}.
\end{proof}

\begin{lemma}\label{3A}
$\dim(V_e)\neq 3$.
\end{lemma}
\begin{proof}
Assume, for a contradiction, that $\dim(V_e^\ast)= 3$. Then, by Lemma~\ref{possipairs3}, $(3\al^2+3\al\bt-9\al-2\bt+4)(3\al+\bt-2)=0$ and  $\V_e\cong\V_o\cong 3\A(\al,\bt)$.
In particular, 
\begin{equation}\label{lmdtris}
   \lmd=\lmdf . 
\end{equation}
Moreover, by Equation~\eqref{eq*=**3} and Lemma~\ref{sub3A}\ref{sub3A_4},   
\begin{equation*}
  \langle a_0-a_2, a_0-a_{-2}, (2\bt-1)a_0+s_{\tz,2}\rangle=V_e^\ast=V_o^\ast=\langle a_{-1}-a_1, a_{1}-a_3, (2\bt-1)a_1+s_{\tu,2}\rangle .  
\end{equation*} 
Then $a_1-a_{-1}$ belongs to the $\bt$-eigenspace of $\ad_{a_0}$ in $V_e$ which is generated by $a_2-a_{-2}$. Hence \begin{equation}\label{a1m13A}
    a_1-a_{-1}=x(a_2-a_{-2})\quad \mbox{  for some  }\quad x\in \F\setminus \{0\},
\end{equation} 
and
\begin{equation}\label{eq25sec}
a_{-1}=a_1-x(a_2-a_{-2}), \qquad a_{-3}=a_3=a_1-x(a_0-a_{-2}).
\end{equation}
By Equation~\eqref{defs} on page~\pageref{defs} and Equation~\eqref{eq25sec},
\begin{equation}
\begin{aligned}\label{s13bis}
s_{\tu,3}&= a_1a_{-2}-\bt(a_1+a_{-2}) \\
&= \left (a_3+x(a_0-a_{-2})\right )a_{-2}-\bt(a_1+a_{-2}) \\
&= s_{\tz,1}+\bt(a_3+a_{-2})+x\left (s_{\tz,2}+\bt(a_0+a_{-2})-a_{-2}\right )-\bt(a_1+a_{-2}) \\
&= s_{\tz,1}+\bt a_3+xs_{\tz,2}+x(\bt-1)a_{-2}+x\bt a_0-\bt a_1\\
&= s_{\tz,1}+\bt \left (a_1-x(a_0-a_{-2})\right )+xs_{\tz,2}+x(\bt-1)a_{-2}+x\bt a_0-\bt a_1 \\
&= s_{\tz,1}+x\left ((2\bt-1)a_{-2}+s_{0,2}\right ).
\end{aligned}
\end{equation}
Using the above expression for $s_{\tu,3}$ and Table~\ref{table3A},
\begin{align}\label{prodm13}
a_{-1}a_3 &= \left (a_1-x(a_2-a_{-2})\right )\left (a_1-x(a_0-a_{-2})\right ) \nonumber \\
&= a_1-xa_1(a_2-a_{-2})-xa_1(a_0-a_{-2})+x^2(a_2-a_{-2})(a_0-a_{-2})\nonumber\\
&= a_1-x\left (2s_{\tz,1}+\bt(a_0+2a_1+a_2)-2s_{\tu,3}-2\bt(a_1+a_{-2})\right ) \nonumber\\
&\phantom{{}={}} \quad +x^2\left (\bt(a_2-a_{-2})-s_{\tz,2}-\bt(a_2+a_{-2})+a_{-2}\right )\\
&= a_1-x\left (-2x(2\bt-1)a_{-2}-2xs_{\tz,2} +\bt(a_0+a_2-2a_{-2})\right )\nonumber\\
&\phantom{{}={}} \quad -x^2\left ((2\bt-1)a_{-2}+s_{\tz,2}\right )\nonumber\\
&= a_1+x^2(2\bt-1)a_{-2}+x^2s_{\tz,2} -x\bt(a_0+a_2-2a_{-2}).\nonumber
\end{align}
On the other hand, by the multiplication table of $\V_o\cong 3\A(\al, \bt)$ and Equation~\eqref{eq25sec}, $$a_{-1}a_3=s_{\tu,2}+\bt(a_{-1}+a_3)=s_{\tu,2}+2\bt a_1-x\bt(a_0+a_2-2a_{-2}),
$$ 
whence, by comparing the above expression for $a_{-1}a_3$ with that given in Equation~\eqref{prodm13}, we get 
\begin{equation}\label{s123A}
   s_{\tu,2}=(1-2\bt)a_1+x^2(2\bt-1)a_{-2}+x^2s_{\tz,2}.  
\end{equation}
Since $s_{\tz,2}$ and $s_{\tu,2}$ are $\tau_0$-invariant, by Equations~\eqref{s123A} and~\eqref{eq25sec}, we get 
\begin{align*}
0&= s_{\tu,2}-(s_{\tu,2})^{\tau_0}\\
&= (1-2\bt)(a_1-a_{-1})+x^2(2\bt-1)(a_{-2}-a_2)\\
&= x(x+1)(2\bt-1)(a_{-2}-a_2),
\end{align*}
whence $x(x+1)(2\bt-1)=0$ and 
\begin{equation}\label{eqcases}
\mbox{ either }\quad  x=-1 \quad \mbox{ or }\quad  \bt=\tfrac{1}{2}.
\end{equation}
By Equation~\eqref{a1m13A} and Corollary~\ref{a1xa2}\ref{a1xa2_1}, as $a_{-3}=a_3$, we get 
\begin{equation}\label{eqfin}
0=(1-\al)(\lmf-\lmu)+\tfrac{\bt}{2}(\al-\bt)(\tfrac{1}{x}-x).
\end{equation}
Assume $\bt\neq \tfrac{1}{2}$. Then, by Equation~\eqref{eqcases}, $x=-1$ and by Equation~\eqref{eqfin}, $\lmf=\lmu$. Hence, by Equation~\eqref{lmdtris} and Proposition~\ref{reg}\ref{reg_1}, $\V$ is isomorphic to a quotient of a symmetric algebra, contradicting Hypothesis~\ref{Hyp2}\ref{Hyp2_2}. 

Therefore $\bt=\tfrac{1}{2}$. Since by hypothesis $(3\al^2+3\al\bt-9\al-2\bt+4)(3\al+\bt-2)=0$ and $\al\neq \bt$, we get $\al=2$ and, by Lemma~\ref{sub3A}\ref{sub3A_1}, 
\begin{equation}\label{3Clm2}
    \lmd=\lmdf=1.
\end{equation}
From Equation~\eqref{eq25sec} it follows that
\begin{equation}\label{lambda3}
\lm_3=\lm_{a_0}(a_3)=\lm_{a_0}(a_1-x(a_0-a_{-2}))=\lmu-x(1-\lmd)=\lmu.
\end{equation}
and $\lmd^f=\lm_{a_1}(a_3)=\lm_{a_1}(a_1-x(a_0-a_{-2}))=1-x(\lmf-\lm_{a_1}(a_{-2}))$, whence 
\begin{equation}\label{lambda3f3A}
\lm_3^f=\lm_{a_1}(a_{-2})=\lmf-x^{-1}(1-\lmd^f)=\lmf.
\end{equation}
%
Since, by Equation~\eqref{hyp},  $a_{-2}=a_4$, $a_{-2}$ is $\tau_1$-invariant, hence, by the fusion law, $a_1a_{-2}-\lm_3^fa_1$ is an $\al$-eigenvector for $\ad_{a_1}$. By Equation~\eqref{eq25sec}
\begin{align*}
a_1a_{-2}&= \left (a_3+x(a_0-a_{-2})\right )a_{-2}\\
&= s_{\tz,1}+\bt(a_3+a_{-2})+x\left (s_{\tz,2}+\bt(a_0+a_{-2})-a_{-2}\right )\\
&= s_{\tz,1}+\bt a_3+xs_{\tz,2}+\left (\bt+x(\bt-1)\right )a_{-2}+x\bt a_0\\
&= s_{\tz,1}+\bt \left (a_1-x(a_0-a_{-2})\right )+xs_{\tz,2}+\left (\bt+x(\bt-1)\right )a_{-2}+x\bt a_0\\
&= s_{\tz,1}+\bt a_1+xs_{\tz,2}+\left (\bt+x(2\bt-1)\right )a_{-2},
\end{align*}
whence it follows that
\begin{equation*}
    0=a_1(a_1a_{-2}-\lm_3^fa_1)-\al(a_1a_{-2}-\lm_3^fa_1)= \tfrac{3}{8}(x^2-1)(2a_{-2}-a_0-a_2+4s_{\tz,2}).
\end{equation*}
Since $2a_{-2}-a_0-a_2+4s_{\tz,2}\neq 0$ in $V_e$, we get  $x^2=1$. By Equation~\eqref{eqfin}  
\begin{equation}\label{3Clm1}
    \lmf=\lmu
\end{equation}
and by Equations~\eqref{lambda3} and~\eqref{lambda3f3A}, \begin{equation}\label{lm3=lm3f}
   \lm_3^f=\lm_3. 
\end{equation}
Let 
$$
\overline{V}:=\langle a_{-2}, a_{-1}, a_0, a_1, a_2, a_3, s_{\tz,1}, s_{\tz,2}\rangle.
$$
We claim that 
\begin{equation}
    \label{genV}
   \overline{V}= V.
\end{equation} 
From Equations~\eqref{s13bis} and~\eqref{s123A} it follows  that 
$$
\{s_{\tu,2}, s_{\tu,3}\}\subseteq \overline{V}.
$$
Since, by Equation~\eqref{hyp}, $a_3=a_{-3}$ and $a_4=a_{-2}$, 
$$\overline{V}^{\tau_0}=\overline{V} \quad  \mbox{ and }\quad \overline{V}^{\tau_1}=\overline{V},
$$ 
whence $s_{\bar 2,3}=(s_{\tu, 3})^{\tau_0}\in \overline{V}$ and $s_{\bar 0,3}=(s_{\bar 2, 3})^{\tau_1}\in \overline{V}$.
Hence, by~\cite[Corollary 7.4, Lemma~7.6, and Lemma~7.7]{FMS3}, $\overline{V}$ is a subalgebra of $V$. Since $\overline{V}$ contains $a_0$ and $a_1$, we get $\overline{V}=V$. By Equations~\eqref{3Clm1}, \eqref{3Clm2}, \eqref{lm3=lm3f}, 
$$
\lmu=\lmf, \quad  \lmd=\lmdf, \quad \lm_3=\lm_3^f,
$$
respectively. 
By Equation~\eqref{genV} and Proposition~\ref{reg}\ref{reg_2}, $\V$ is isomorphic to a quotient of a symmetric algebra, contradicting Hypothesis~\ref{Hyp2}\ref{Hyp2_2}.
\end{proof}



\begin{proof}[Proof of Theorem~\ref{teo3}]
Let $\V$ be as in Theorem~\ref{teo3}, that is $\al\neq 2\bt$ and $\V$ has axet $X(3+3)$. By Lemmas~\ref{end},~\ref{possipairs3}, Corollary~\ref{J(0)3}, Lemma~\ref{bt=2} and Lemma~\ref{3A}, $\V$ does not satisfy Hypothesis~\ref{Hyp2}. Hence, either $\V$ is isomorphic to a quotient of a symmetric algebra, thus satisfying the hypotheses of Theorem~\ref{teoq}, or $V\in \{V_e, V_o\}$, satisfying the hypotheses of Theorem~\ref{teonsa}. 
\end{proof}

\section{Algebras with larger regular axets} \label{chap:larger}

In this section we assume that $\V$ satisfies Hypothesis~\ref{Hyp2} with $n\in \N\cup \{\infty\}$ and $n\geq 4$.
By Remark~\ref{rem:reg},
\begin{equation*}\label{hyplarger}
a_{-2}\neq a_2 \quad \mbox{  and }\quad  a_{-1}\neq a_3.
\end{equation*}

Let
\[
{\mathcal L}:=\{ \J(\delta), \; \J(0)^\times, \; \IY_3(\al, \tfrac{1}{2}; \mu),\; \IY_3(\al, \tfrac{1}{2}; 1)^\times,\; \IY_3(-1, \tfrac{1}{2}; \mu)^\times :  \delta, \al, \mu \in \F\}.
\]
as defined at the beginning of Chapter~\ref{ch:infinite}, on page~\pageref{ch:infinite}.

\begin{lemma}\label{axialdim}
Let $\W$ be a symmetric $2$-generated $\mathcal{M}(\al,\bt)$-axial algebra over a field $\F$, with $\al\neq 2\bt$ and  axet $X(n)$, where $n\geq 4$. Then
\begin{enumerate}
\item   if $\adim(\W)\leq 3$, then $\bt=\tfrac{1}{2}$ and $\W$ is isomorphic to an algebra in $\mathcal L\cup \{I\mathcal H_3 \}$,  if further $W\neq W^\ast$, then  $\dim(W^\ast)\leq 3$; \label{axialdim_1}
\item  if  $\adim(\W)>3$ and  $W\neq W^{\ast\ast}$, then $\dim(W^{\ast\ast})\geq 4$ unless $\W$ is isomorphic to a quotient of $\hatH$ or to $4\Y(2, -\tfrac{3}{2})$, $6\Y(\frac{1}{2}, 2)$, or $6\Y(\frac{1}{2}, 2)^\times$. \label{axialdim_2}
    \end{enumerate}
\end{lemma}
\begin{proof}
Assume $\adim(\W)\leq 3$. Then there are at most three axes in $W$ that are linearly independent. By checking the bases given in Tables~\ref{table2} to~\ref{tablehatH} and by Lemma~\ref{subH}, it follows that either $\W$ is isomorphic to an algebra in $\mathcal L\cup\{I{\mathcal H}_3 \}$,  or $\W$ is isomorphic to one of the following algebras:
\begin{equation*}\label{candidates}
2\B, \;3\C(\al),\;3\C(\bt),\; 3\C(-1)^\times,\; 3\A(\al,\bt),\; 3\A(\al, \tfrac{1-3\al^2}{3\al-1})^\times,   
\end{equation*} which are excluded,   since they have axet $X(n)$ with $n\leq 3$. If further $W\neq W^\ast$, by Table~\ref{table3} and Lemma~\ref{subIH3}, we get $\dim(W^\ast)\leq 3$.

Now assume $\adim(\W)>3$ and suppose that $\W$ is not isomorphic to a quotient of either $\hatH$. By checking the bases given in Tables~\ref{table2} to~\ref{tablehatH}, it follows that $\W$ is isomorphic to a quotient of one of the following algebras: 
\begin{equation*}\label{candidates2}
\begin{split}
   & 4\A(\tfrac{1}{4}, \bt), \;4\B(\al, \tfrac{\al^2}{2}), \;4\Y( \tfrac{1}{2}, \bt), \;4\Y(\al, \tfrac{1-\al^2}{2}), 
   \\
   & 
  5\A(\al, \tfrac{5\al-1}{8}),\; 6\A(\al, \tfrac{\al^2}{4(2\al-1)}), \; 6\Y(\tfrac{1}{2}, 2),\; \IY_5(\al, \tfrac{1}{2}). 
\end{split}   
\end{equation*}
By Lemmas~\ref{sub4A}, \ref{sub4B}, \ref{sub4Y1/2}, \ref{sub4Y}, \ref{sub5A}, \ref{subA6}, \ref{sub6Y} and \ref{subIY5}, only the algebras $4\Y(2,-\tfrac{3}{2})$, $6\Y(\frac{1}{2}, 2)$ and $6\Y(\frac{1}{2}, 2)^\times$ satisfy the condition $W\neq W^{\ast\ast}$  and $\dim(W^{\ast\ast})<4$.
\end{proof}

\begin{lemma}\label{AQ}
Let $A,A^f, B, B^f, C, C^f, Q, Q^f$ be as in Sections~$\ref{universal}$, $\ref{relVeVo}$, and $\ref{evaluation}$. Then
 \begin{enumerate}
     \item  if $\al\neq 4\bt$, then $\{A, A^f\}\neq \{0\}\neq \{B, B^f\}$; \label{AQ_1}
     \item  if $\al= 4\bt$, then $\{Q,Q^f\}\neq \{0\}\neq \{C, C^f\}$. \label{AQ_2}
 \end{enumerate} 
\end{lemma}
\begin{proof}
Assume $\al\neq 4\bt$. Since $a_{-2} \neq a_2$ and $a_{-1} \neq a_3$, by Lemma~\ref{s1s2}, $A=0$ (respectively $A^f=0$) if and only if $B=0$  (respectively $B^f=0$). Thus, if $\{B,B^f\}=\{0\}$ or $\{A, A^f\}=\{0\}$, then  $\{A,A^f, B, B^f\}= \{0\}$ and, by Lemma~\ref{lemma:A+Af}, $\V$ is isomorphic to a quotient of a symmetric algebra, contradicting Hypothesis~\ref{Hyp2}\ref{Hyp2_2}. This proves \ref{AQ_1}. 

Now let $\al= 4\bt$. Since, by assumption, $\V$ has a regular axet and is not a quotient of a symmetric algebra, by Corollary~\ref{reduction}, $\{Q,Q^f\}\neq \{0\}$. Now, similarly to the previous case, by Lemma~\ref{l7.3} we get that $\{C, C^f\}\neq \{0\}$.
\end{proof}

\begin{lemma}\label{dim33}
$    
\{\adim(\V_e),\adim(\V_o)\}\not \subseteq \{1,2,3\} .
$  
\end{lemma}
\begin{proof}
    Assume, for a  contradiction, that 
    $$
\adim(\V_e)\leq 3\quad \mbox{ and }\quad \adim(\V_o)\leq 3.
$$ 
Then, by Lemma~\ref{axialdim}, $\{\V_e, \V_o\}\subseteq \mathcal L\cup \{I\mathcal H_3\}$.  By Theorem~\ref{case4}, $\{\V_e, \V_o \}\not \subseteq \mathcal L$, and so, without loss of generality, we may assume that
\begin{equation*}\label{Ve=IH3}
    \V_e\cong I\mathcal H_3,
\end{equation*}
whence \begin{equation*}\label{albeta}
    (\al,\bt)=(2, \tfrac{1}{2})\quad \mbox{ and }\quad \lmd=\lmdf=1.
\end{equation*} 
Hence, by Theorem~\ref{HWthm} and our assumption in Hypothesis~\ref{Hyp2}\ref{Hyp2_2} that $\V$ is not a quotient of a symmetric algebra,
\begin{equation*}
    \label{ll1}
    \{\lmu, \lmf\}\neq \{1\}.
\end{equation*}
Substituting $\al=2$ and $\bt =\frac{1}{2}$ into the definitions of $Q$ and $C$, we get
\begin{equation}\label{QandC}
Q = 6(1-\lmu) \quad \mbox{and} \quad C = 2\lmu\lmf + 7\lmu^2 - \lmf-11\lmu+3
\end{equation}
Note that, as $\ch \F \neq 3$, $\lmu = 1$ if and only if $Q = 0$ and so, $\lmf = 1$ if and only if $Q^f=0$.

Since $a_{-2} \neq a_2$, $a_2 \neq a_0$, $a_3 \neq a_{-1}$ and $a_{-1} \neq a_1$ by Hypothesis~\ref{Hyp2}\ref{Hyp2_4}, Lemma~\ref{l7.3} gives $Q = 0$ if and only if $C = 0$ and $Q^f=0$ if and only if $C^f = 0$.  In addition,
\begin{equation}\label{l7.3 consequences}
Q \neq 0 \mbox{ implies } V_e^\ast=V_o^{\ast\ast}, \quad \mbox{and} \quad Q^f \neq 0 \mbox{ implies } V_o^\ast=V_e^{\ast\ast}.
\end{equation}
Now, as $\al = 4\bt$, by Lemma~\ref{AQ}, $Q$ and $Q^f$ are not both zero.  So either $V_e^\ast=V_o^{\ast\ast}$, or $V_o^\ast=V_e^{\ast\ast}$.

Suppose first that $V_o^\ast=V_e^{\ast\ast}$.  As $\V_e \cong I\mathcal{H}_3$, by Lemma~\ref{subIH3}, $V_e^{\ast\ast}<V_e^\ast$ and so
\[
V_o^{\ast\ast} \leq V_o^\ast=V_e^{\ast\ast} < V_e^\ast
\]
hence $V_o^{\ast\ast} \neq V_e^\ast$. However, by Equation \eqref{l7.3 consequences}, this implies that $Q=0$ and so $C=0$.  Now, from Equation \eqref{QandC}, we get $\lmu=1$ and so $0 = C = \lmf -1$ and $\lmf=1$, a contradiction.

Similarly, if $V_e^\ast=V_o^{\ast\ast}$, then as $\V_e \cong I\mathcal{H}_3$, we get $V_e^{\ast\ast}<V_e^\ast=V_o^{\ast\ast}\leq V_o^\ast$.  So $V_o^\ast\neq V_e^{\ast\ast}$ and hence $Q^f = 0 = C^f$.  This then yields $\lmf = 1 = \lmu$, another contradiction.
%
\end{proof}

\begin{lemma}
    \label{dim>4}
 $
\{\adim(\V_e),\adim(\V_o)\}\cap \{1,2,3\}\neq \emptyset .
$   
\end{lemma}
\begin{proof}
Assume the assertion were false. 
By Lemma~\ref{supernew}, if $\al=4\bt$, then $\{Q,Q^f\}=\{0\}$, and otherwise $\{A,A^f,B,B^f\}=\{0\}$. In both cases, Lemma~\ref{AQ} leads to a contradiction. 
\end{proof}

\begin{proof}[Proof of Theorem~\ref{teolarger}]
Let $\V$ be as in Theorem~\ref{teolarger}, that is $\al\neq 2\bt$ and $\V$ has axet $X(n+n)$ with $n\geq 4$. It is enough to show that $\V$ does not satisfy Hypothesis~\ref{Hyp2}. In order to obtain a contradiction, assume that $\V$ satisfies Hypothesis~\ref{Hyp2}. Then, by Lemmas~\ref{dim33} and~\ref{dim>4}, up to swapping $\V_e$ and $\V_o$ we may assume
\begin{equation}\label{dim34}
    \adim(\V_e)\leq 3\quad \mbox{ and }\quad \adim(\V_o)\geq 4.
\end{equation} 
By Lemma~\ref{axialdim}\ref{axialdim_1}, this implies that $\V_e\in \mathcal L \cup \{I\mathcal H_3\}$ and $\bt = \frac{1}{2}$.

We claim that 
\begin{equation}\label{lastclaim}
V_o^{\ast\ast}=V_e^\ast.
\end{equation}
First suppose that $\al=4\bt$.  Since $\adim(\V_o) \geq 4$, Lemma~\ref{supernew}\ref{supernew_2} implies that
\[
Q^f=C^f=0.
\]
By Lemma~\ref{AQ}, it follows that $C\neq 0$, whence, by Lemma~\ref{Ve*Vo*}, we get Equation~\eqref{lastclaim}. Similarly, if $\al\neq 4\bt$, then Lemma~\ref{supernew}\ref{supernew_2} implies that
\[
A=B=0,
\]
whence, by Lemma~\ref{AQ}, $B^f\neq 0$. Thus, the claim in Equation~\eqref{lastclaim} follows from Lemma~\ref{Ve*Vo*}\ref{Ve*Vo*_2}.

Now, $V_e^\ast<V_e$.  Otherwise, if $V_e = V_e^\ast$, then by Equation \ref{lastclaim}, $V_e = V_e^\ast = V_o^{\ast\ast} \leq V_o$.  Hence $V = V_o$, contradiction our assumptions in Hypothesis~\ref{Hyp2}\ref{Hyp2_3}.
Then, by Lemma~\ref{axialdim}\ref{axialdim_1}, we get
\[
\dim(V_o^{\ast\ast})=\dim(V_e^\ast)\leq 3.
\]
Since $\adim(V_o) \geq 4$, $V_o^{\ast\ast}< V_o$ and so Lemma~\ref{axialdim}\ref{axialdim_2} yields that $\V_o$ is isomorphic 
\[
\mbox{either to one of } 4\Y(2, -\tfrac{3}{2}), 6\Y(\tfrac{1}{2}, 2), 6\Y(\tfrac{1}{2}, 2)^\times,\mbox{ or to a quotient of }\hatH.
\]
Since we saw above that $\bt = \frac{1}{2}$, we must have that $\V_o$ is a quotient of $\hatH$.  So, applying Lemma \ref{lambdasut} to $\V_o$, we get $\lmdf = 1$.  Similarly, as $\V_e\in \mathcal L \cup \{I\mathcal H_3\}$, Lemmas~\ref{subJ}, \ref{subIY3}, and \ref{lambdasut} yield $\lmd = 1$.  

Since $\V_o$ is a quotient of $\hatH$ and $\adim(V_o) \geq 4$, it is not of Jordan type and so we see that $\al = 2 = 4\bt$.  As we saw above, by Lemma~\ref{supernew}\ref{supernew_2}, 
\[
Q^f=C^f=0.
\]
%
%
Substituting into $Q^f$ and $C^f$ the values $\al=2$, $\bt=\tfrac{1}{2}$, $\lmd=1$, $\lmdf=1$, we get $\lmu=\lmf=1$.  Finally, as $\lmu = \lmf = \lmd = \lmdf = 1$, by Theorem~\ref{HWthm}, $\V$ is isomorphic to a quotient of a symmetric algebra, contradicting Hypothesis~\ref{Hyp2}\ref{Hyp2_2}.
\end{proof}

\chapter{Proofs of the Main Theorem and its corollaries}\label{proof}

\begin{proof}[Proof of the Main Theorem]
For the sake of contradiction, suppose $\V$ is an algebra that is not listed in the theorem. By Theorems~\ref{al=2bt} and \ref{skew}, $\al\neq 2\bt$ and $\V$ does not have a skew axet. Therefore $\V$ has a regular axet of size $2n$, with $n\geq 1$, and it does not satisfy the hypotheses of Theorems~\ref{teoq} and~\ref{teonsa}.  By Lemma~\ref{lhyp1} and Theorems~\ref{teo2}, \ref{teo3}, and \ref{teolarger}, we get the desired contradiction.  
\end{proof}
\begin{proof}[Proof of  Corollaries \ref{mcor:FF} and \ref{mcor:bt<>1/2}] 
This follows by direct inspection using the tables in Section~\ref{thetables}.
\end{proof}
\begin{proof}[Proof of Corollary \ref{cor6trans}]
Recall that a $6$-transposition group is a  group generated by a normal set $\mathcal T$ of involutions such that the product of every two elements of $\mathcal T$ has order less or equal to $6$ (see~\cite[\S 5.8.3]{Wilson}).  
Taking $\mathcal T=X$, with $X$ as in Equation~\eqref{defX}, the result follows by Corollary~\ref{mcor:bt<>1/2} with $\bt\neq \tfrac{1}{2}$, since every subalgebra generated by two axes in $X$ will have axet of size less than or equal to $6$. 
\end{proof}

\begin{lemma}\label{JM}
Let $\V$ be a $2$-generated $\mathcal M(\tfrac{1}{4}, \tfrac{1}{32})$-algebra over a field $\F$. If $\V$ is not a Norton-Sakuma algebra, then one of the following holds
\begin{enumerate}
    \item $\ch(\F)=5$, $(\tfrac{1}{4},\tfrac{1}{32})=(-1,\tfrac{1}{2})=(\tfrac{2}{3},-\tfrac{1}{3})$ and $\V$ is isomorphic to one of 
    \begin{enumerate}
        \item $3\C(-1)^\times$, 
        \item $\J(\dl)$ with $\dl\neq -\tfrac{3}{8}$, 
        \item $\J(0)^\times$, 
        \item $3\A(-1, \frac{1}{2})^\times$, 
        \item $4\A(\tfrac{1}{4},\tfrac{1}{2})^\times$, 
        \item $4\B(-1,\tfrac{1}{2})^\times$, 
        \item $4\B(-1,\tfrac{1}{2};\nu)^\times$, 
        \item $6\A(\tfrac{2}{3},-\tfrac{1}{3})^\times$, 
        \item $\IY_3(-1,\tfrac{1}{2};\mu)$ with $\mu\neq -\tfrac{1}{2}$
        \item $\IY_3(-1,\tfrac{1}{2}; \mu)^\times$ with $\mu\neq -\tfrac{1}{2}$,
        \item $\IY_5(-1,\tfrac{1}{2})^\times$, 
    \end{enumerate}
    \item $\ch(\F)=11$, $(\tfrac{1}{4},\tfrac{1}{32})=(3,-1)$ and $\V\cong 3\C(-1)^\times$; 
    \item $\ch(\F)=23$, $(\tfrac{1}{4},\tfrac{1}{32})=(6,-5)$ and $\V\cong 3\C^\prime (6,-5)$.
\end{enumerate}
 \end{lemma}
 \begin{proof}
Since $\V$ is not a Norton-Sakuma algebra, by~\cite[Theorem 1.5]{FMS3}, $\F$ has positive characteristic. If $\ch(\F)\in \{2,3,7,31\}$, then the fusion law $\M(\tfrac{1}{4},\tfrac{1}{32})$ is not defined, so we exclude these characteristics. 

Assume first that $\V$ is isomorphic to a quotient  of an axial algebra of  Jordan type.  By Table~\ref{table2}, $$\V\cong 3\C(-1)^\times,\quad \J(\dl),\quad  \mbox{or}\quad   \J(0)^\times. $$ 

If $\V\cong 3\C(-1)^\times$, then $\tfrac{1}{4}$ or $\tfrac{1}{32}$ is equal to $-1$. We have $\tfrac{1}{4}=-1$ if and only if $\ch(\F)=5$, while $\tfrac{1}{32}=-1$ if and only if $\ch(\F)=11$, since characteristic three is already ruled out. 

If $\V$ is isomorphic to $\J(\dl)$ or $\J(0)$, then $\tfrac{1}{4}$ or $\tfrac{1}{32}$ is equal to $\tfrac{1}{2}$. Note that $\tfrac{1}{4}=\tfrac{1}{2}$ cannot happen in any characteristic, while $\tfrac{1}{32}=\tfrac{1}{2}$ happens only when $\ch(\F)=5$. By Note~\ref{table2 3C J} in Table~\ref{table2}, $\J(-\tfrac{3}{8})\cong 3\C(\tfrac{1}{2})$.

Next assume $\V$ is not isomorphic to a quotient of an axial algebra of Jordan type. Note that, since $\tfrac{1}{4}\not \in \{ \tfrac{1}{2}, \tfrac{1}{8}, \tfrac{1}{16}\}$ in every characteristic other than $2$ and $3$, 
$$\al \not \in \{\tfrac{1}{2}, 2\bt, 4\bt\}.$$ 
By Tables~\ref{table3A}-\ref{4Bq}, $\V$ is isomorphic to a quotient of one of the following
\[
3\A(\al, \tfrac{1-3\al^2}{3\al-1})^\times, \quad 4\A(\tfrac{1}{4}, \tfrac{1}{2})^\times, \quad 4\B(-1, \tfrac{1}{2})^\times, \quad 4\Y(\al, \tfrac{1-\al^2}{2} ), \quad 6\A(\tfrac{2}{3},-\tfrac{1}{3})^\times,
\]
\[
6\A(\tfrac{1\pm \sqrt{97}}{24},\tfrac{53\pm 5\sqrt{97}}{192})^\times, \quad 
\IY_3(\al, \tfrac{1}{2}, \mu),  \quad \IY_5(\al, \tfrac{1}{2}), \quad Q_2^\prime(\tfrac{1}{3}, \tfrac{2}{3}),\]
\[3\C^\prime(\al, 1-\al), \quad \mbox{or}\quad 4\B(-1, \tfrac{1}{2}, \nu)^\times .
\]
If $\V\cong 3\A(\al,\tfrac{1-3\al^2}{3\al-1})^\times$, then 
$\tfrac{1}{32}=\bt =\tfrac{1-3\al^2}{3\al-1}= -\tfrac{13}{4}$, which  implies   $\ch(\F)=5$. Note that, in this case, $\bt=\tfrac{1}{2}$ and, by Note~\ref{tableIY3 3A} in Table~\ref{tableIY3bis}, $\V\cong \IY_3(-1,\tfrac{1}{2};-\tfrac{1}{2})^\times$. 

If $\V$ is isomorphic to a quotient of $4\A(\tfrac{1}{4},\tfrac{1}{2})^\times$, $4\B(-1,\tfrac{1}{2})^\times$, $4\B(-1,\tfrac{1}{2};\nu)^\times$, $\IY_3(\al,\tfrac{1}{2};\mu)$, $\IY_5(\al,\tfrac{1}{2})$, then $\bt=\tfrac{1}{2}$ which again implies  $\ch(\F)=5$, and then $\frac{1}{4} = \al =-1$.   By Note~\ref{tableIY3 3A} in Table~\ref{tableIY3bis}, $\IY_3(-1,\tfrac{1}{2};-\tfrac{1}{2})\cong 3\A(-1,\tfrac{1}{2})$, which is a Norton-Sakuma algebra. By Note~\ref{tableIY5 char5} of Table~\ref{tableIY5}, $\IY_5(-1,\tfrac{1}{2}) \cong 5\A(-1,\tfrac{1}{2})$ in characteristic five, which is again a Norton-Sakuma algebra.  However,  the quotient $\IY_5(-1,\tfrac{1}{2})^\times\cong 5\A(-1,\tfrac{1}{2})^\times$ is not Norton-Sakuma.

If $\V\cong 4\Y(\al,\tfrac{1-\al^2}{2})$, then $\al^2+2\bt-1=0$. Substituting $(\al,\bt)=(\tfrac{1}{4},\tfrac{1}{32})$ into the polynomial, we get $-\tfrac{7}{8}=0$, a contradiction as characteristic $7$ is ruled out above.

  If $\V\cong 6\A(\tfrac{2}{3},-\tfrac{1}{3})^\times$, then $(\tfrac{1}{4},\tfrac{1}{32})=(\tfrac{2}{3},-\tfrac{1}{3})$. This only happens in characteristic five. If $\V\cong 6\A(\tfrac{1\pm \sqrt{97}}{24},\tfrac{53\pm 5\sqrt{97}}{192})^\times$, then $12\al^2-\al-2=0$. Evaluating this expression by $\al=\tfrac{1}{4}$, we get $\tfrac{5}{2}=0$ thus $\ch(\F)=5$. However, 
$$\bt=\tfrac{53\pm 5\sqrt{97}}{192}=\tfrac{3}{2}=-1=\tfrac{1}{4}=\al$$ 
which cannot happen. Hence $\V\not \cong 6\A(\tfrac{1\pm \sqrt{97}}{24},\tfrac{53\pm 5\sqrt{97}}{192})^\times$.

Finally, if $\V$ is isomorphic to $3\QQ^\prime(\tfrac{1}{3},\tfrac{2}{3})$ or $3\C^\prime (\al,1-\al)$, then $\al+\bt-1=0$. Substituting $(\al,\bt)=(\tfrac{1}{4},\tfrac{1}{32})$ into the expression, we get $-\tfrac{23}{32}=0$ which only happens for  $\ch(\F)=23$. Since  $(\tfrac{1}{4},\tfrac{1}{32})\neq (\tfrac{1}{3},\tfrac{2}{3})$ in characteristic $23$, we get  $\V\cong 3\C^\prime (\tfrac{1}{4},\tfrac{1}{32})$.

Conversely, a direct check shows that none of the algebras listed in the statement is isomorphic to a Norton-Sakuma algebra.
\end{proof}

\begin{proof}[Proof of Corollary \ref{mcor:NS char}]
This follows by Lemma~\ref{JM}.
\end{proof}

\printindex
\printnomenclature[4cm] 

\end{document}